\documentclass[12pt,final]{amsart}
\usepackage{csquotes}
\usepackage[style=alphabetic,url=false, isbn=false, doi=false,maxbibnames=99]{biblatex}
\renewbibmacro{in:}{}
\AtEveryBibitem{
    \clearfield{urlyear}
    \clearfield{urlmonth}
    \clearfield{month}
    \clearfield{day}    
    }
\usepackage{xspace,amssymb,amsfonts,epsfig,euscript,enumitem,amsmath,nicefrac}
\usepackage{tikz,bbm,xspace,afterpage, stmaryrd}
 \usepackage{epstopdf,ifpdf}
\usepackage{pxfonts,multicol}
\usepackage{stackengine}

\usepackage{scalerel}
\usepackage[eulergreek]{sansmath}

\stackMath
\usepackage{xr,xr-hyper}
\usepackage[margin=3cm,footskip=25pt,headheight=20pt]{geometry}

\newcommand{\doubletilde}[1]{
  \tilde{{\tilde{#1}}}}
\usepackage{fancyhdr,mathrsfs,hyperref}
\usepackage[english]{babel}
\usepackage{mathtools}
\usepackage{zref-clever}
\usepackage{aliascnt}
\zcsetup{cap}
\providecommand{\cref}[1]{\zcref{#1}}
\providecommand{\Cref}[1]{\zcref{#1}}
\ExplSyntaxOn
\cs_new:Npn \bw_is_section_label_p:n #1
  { \str_if_eq_p:ee { \tl_range:nnn {#1} {1} {4} } { sec: } }
\cs_new_protected:Npn \bw_crefrange:nn #1#2
  {
    \bool_if:nTF
      {
        \bw_is_section_label_p:n {#1}
        &&
        \bw_is_section_label_p:n {#2}
      }
      { \cref{#1} to \cref{#2} }
      { (\cref{#1}--\cref{#2}) }
  }
\NewDocumentCommand{\bwcrefrange}{mm}{\bw_crefrange:nn{#1}{#2}}
\ExplSyntaxOff
\providecommand{\crefrange}[2]{\bwcrefrange{#1}{#2}}

\zcRefTypeSetup{equation}{
  name-sg=,
  Name-sg=,
  name-pl=,
  Name-pl=
}
\newif\ifanindex
\anindextrue
\ifanindex
\usepackage{anindex}
\anindexsetup{
    heading,
    lines,
}

\fi
\usepackage{collectbox}
\makeatletter
\newcommand{\mybox}{\collectbox{\setlength{\fboxsep}{1pt}\fbox{\BOXCONTENT}}}
\makeatother

\DeclareMathOperator*{\bigboxtimes}{\scalerel*{\boxtimes}{\sum}}

\newcounter{dummy}

\begin{document}

\ifanindex
\else 
\newcommand{\notation}[2]{}
\fi
\newtheoremstyle{all}{11pt}{11pt}{\slshape}{}{\bfseries}{}{.5em}{}

\NewDocumentCommand{\NewTheoremWithZref}{ m o m m }{%
  \IfNoValueTF{#2}{%
    \newtheorem{#1}{#3}%
    \zcRefTypeSetup{#1}{%
      name-sg=#3, Name-sg=#3, name-pl=#4, Name-pl=#4%
    }%
  }{%
    \newaliascnt{#1}{#2}%
    \newtheorem{#1}[#1]{#3}%
    \aliascntresetthe{#1}%
    \zcRefTypeSetup{#1}{%
      name-sg=#3, Name-sg=#3, name-pl=#4, Name-pl=#4%
    }%
  }%
}

\theoremstyle{all}
\newtheorem{itheorem}{Theorem}
\newtheorem{theorem}{Theorem}[section]
\newtheorem*{theoremfourfive}{Theorem 4.5}
\newtheorem*{proposition*}{Proposition}
\NewTheoremWithZref{proposition}[theorem]{Proposition}{Propositions}
\NewTheoremWithZref{corollary}[theorem]{Corollary}{Corollaries}
\NewTheoremWithZref{lemma}[theorem]{Lemma}{Lemmas}
\NewTheoremWithZref{definition}[theorem]{Definition}{Definitions}
\NewTheoremWithZref{ques}[theorem]{Question}{Questions}
\NewTheoremWithZref{conj}[theorem]{Conjecture}{Conjectures}

\theoremstyle{remark}
\NewTheoremWithZref{remark}[theorem]{Remark}{Remarks}
\NewTheoremWithZref{example}[theorem]{Example}{Examples}

\zcRefTypeSetup{itheorem}{
  name-sg=theorem, Name-sg=Theorem, name-pl=theorems, Name-pl=Theorems
}
\zcRefTypeSetup{theorem}{
  name-sg=theorem, Name-sg=Theorem, name-pl=theorems, Name-pl=Theorems
}

\newcommand{\nc}{\newcommand}
\newcommand{\renc}{\renewcommand}
\newcounter{subeqn}
\renewcommand{\thesubeqn}{\theequation\alph{subeqn}}
\newcommand{\subeqn}{\refstepcounter{subeqn}\tag{\thesubeqn}}\makeatletter
\@addtoreset{subeqn}{equation}
\newcommand{\newseq}{\refstepcounter{equation}}

\def\eps{{\varepsilon}}
\renc{\aa}{u}
\nc{\bb}{v}
\nc{\cc}{w}
  \nc{\kac}{\kappa^C}
  \newcommand{\pwr}{\gamma}
\nc{\alg}{T}
\nc{\km}{\kappa}
\nc{\Lco}{L_{\la}}
\nc{\qD}{q^{\nicefrac 1D}}
\nc{\ocL}{M_{\la}}
\nc{\excise}[1]{}
\nc{\Dbe}{D^{\uparrow}}
\nc{\Dfg}{D^{\mathsf{fg}}}
\nc{\op}{\operatorname{op}}
\nc{\MaxSpec}{\operatorname{MaxSpec}}

\nc{\tr}{\operatorname{tr}}
\newcommand{\Mirkovic}{Mirkovi\'c\xspace}
\nc{\tla}{\mathsf{t}_\la}
\nc{\llrr}{\langle\la,\rho\rangle}
\nc{\lllr}{\langle\la,\la\rangle}
\nc{\K}{\Bbbk}
\nc{\Stosic}{Sto{\v{s}}i{\'c}\xspace}
\nc{\cd}{\mathcal{D}}
\nc{\cT}{\mathcal{T}}
\nc{\vd}{\mathbb{D}}
\nc{\R}{\mathbb{R}}
\renc{\wr}{\operatorname{wr}}
  \nc{\Lam}[3]{\La^{#1}_{#2,#3}}
  \nc{\Lab}[2]{\La^{#1}_{#2}}
  \nc{\Lamvwy}{\Lam\Bv\Bw\By}
  \nc{\Labwv}{\Lab\Bw\Bv}
  \nc{\nak}[3]{\mathcal{N}(#1,#2,#3)}
  \nc{\hw}{highest weight\xspace}
  \nc{\al}{\alpha}
\numberwithin{equation}{section}
\renc{\theequation}{\arabic{section}.\arabic{equation}}
  \nc{\be}{\beta}
  \nc{\bM}{\mathbf{m}}
  \nc{\Bu}{\mathbf{u}}

  \nc{\bkh}{\backslash}
  \nc{\Bi}{\mathbf{i}}
  \nc{\Ba}{\mathbf{a}}
  \nc{\Bm}{\mathbf{m}}
  \nc{\Bj}{\mathbf{j}}
 \nc{\Bk}{\mathbf{k}}

\nc{\bd}{\mathbf{d}}
\nc{\D}{\mathcal{D}}
\nc{\mmod}{\operatorname{-mod}}  
\newcommand{\red}{\mathfrak{r}}

\nc{\RAA}{R^\A_A}
  \nc{\Bv}{\mathbf{v}}
  \nc{\Bw}{\mathbf{w}}
\nc{\Id}{\operatorname{Id}}
  \nc{\By}{\mathbf{y}}
\nc{\eE}{\EuScript{E}}
  \nc{\Bz}{\mathbf{z}}
  \nc{\coker}{\mathrm{coker}\,}
  \nc{\C}{\mathbb{C}}
  \nc{\ch}{\mathrm{ch}}
  \nc{\de}{\delta}
  \nc{\ep}{\epsilon}
  \nc{\Rep}[2]{\mathsf{Rep}_{#1}^{#2}}
  \nc{\Ev}[2]{E_{#1}^{#2}}
  \nc{\fr}[1]{\mathfrak{#1}}
  \nc{\fp}{\fr p}
  \nc{\fq}{\fr q}
  \nc{\fl}{\fr l}
  \nc{\fgl}{\fr{gl}}
\nc{\rad}{\operatorname{rad}}
\nc{\ind}{\operatorname{ind}}
  \nc{\GL}{\mathrm{GL}}
\newcommand{\arxiv}[1]{\href{http://arxiv.org/abs/#1}{\tt arXiv:\nolinkurl{#1}}}
  \nc{\Hom}{\operatorname{Hom}}
  \nc{\im}{\mathrm{im}\,}
  \nc{\La}{\Lambda}
  \nc{\la}{\lambda}
  \nc{\mult}{b^{\mu}_{\la_0}\!}
  \nc{\mc}[1]{\mathcal{#1}}
  \nc{\om}{\omega}
\nc{\gl}{\mathfrak{gl}}
  \nc{\cF}{\mathcal{F}}
\nc{\cC}{\mathcal{C}}
  \nc{\Mor}{\mathsf{Mor}}
  \nc{\Ob}{\mathsf{Ob}}
  \nc{\Vect}{\mathsf{Vect}}
\nc{\gVect}{\mathsf{gVect}}
  \nc{\modu}{\mathsf{-mod}}
\nc{\pmodu}{\mathsf{-pmod}}
  \nc{\qvw}[1]{\La(#1 \Bv,\Bw)}
  \nc{\van}[1]{\nu_{#1}}
  \nc{\Rperp}{R^\vee(X_0)^{\perp}}
  \nc{\si}{\sigma}
  \nc{\croot}[1]{\al^\vee_{#1}}
\nc{\di}{\mathbf{d}}
  \nc{\SL}[1]{\mathrm{SL}_{#1}}
  \nc{\Th}{\theta}
  \nc{\vp}{\varphi}
  \nc{\wt}{\mathsf{wt}}
\nc{\te}{\tilde{e}}
\nc{\tf}{\tilde{f}}
\nc{\hwo}{\mathbb{V}}
\nc{\soc}{\operatorname{soc}}
\nc{\cosoc}{\operatorname{cosoc}}
 \nc{\Q}{\mathbb{Q}}

  \nc{\Z}{\mathbb{Z}}
  \nc{\Znn}{\Z_{\geq 0}}
  \nc{\ver}{\EuScript{V}}
  \nc{\Res}[2]{\operatorname{Res}^{#1}_{#2}}
  \nc{\edge}{\EuScript{E}}
  \nc{\Spec}{\mathrm{Spec}}
  \nc{\tie}{\EuScript{T}}
  \nc{\ml}[1]{\mathbb{D}^{#1}}
  \nc{\fQ}{\mathfrak{Q}}
        \nc{\fg}{\mathfrak{g}}
  \nc{\Uq}{U_q(\fg)}
        \nc{\bom}{\boldsymbol{\omega}}
\nc{\bla}{{\underline{\boldsymbol{\la}}}}
\nc{\bmu}{{\underline{\boldsymbol{\mu}}}}
\nc{\bal}{{\boldsymbol{\al}}}
\nc{\bet}{{\boldsymbol{\eta}}}
\nc{\rola}{X}
\nc{\wela}{Y}
\nc{\fM}{\mathfrak{M}}
\nc{\fX}{\mathfrak{X}}
\nc{\fH}{\mathfrak{H}}
\nc{\fE}{\mathfrak{E}}
\nc{\fF}{\mathfrak{F}}
\nc{\fI}{\mathfrak{I}}
\nc{\qui}[2]{\fM_{#1}^{#2}}
\nc{\cL}{\mathcal{L}}
\nc{\ca}[2]{\fQ_{#1}^{#2}}
\nc{\cat}{\mathcal{V}}
\nc{\cata}{\mathfrak{V}}
\nc{\catf}{\mathscr{V}}
\nc{\hl}{\mathcal{X}}
\nc{\hld}{\EuScript{X}}
\nc{\hldbK}{\EuScript{X}^{\bla}_{\bar{\mathbb{K}}}}

\nc{\pil}{{\boldsymbol{\pi}}^L}
\nc{\pir}{{\boldsymbol{\pi}}^R}
\nc{\cO}{\mathcal{O}}
\nc{\Ko}{\text{\Denarius}}
\nc{\Ei}{\fE_i}
\nc{\Fi}{\fF_i}
\nc{\fil}{\mathcal{H}}
\nc{\brr}[2]{\beta^R_{#1,#2}}
\nc{\brl}[2]{\beta^L_{#1,#2}}
\nc{\so}[2]{\EuScript{Q}^{#1}_{#2}}
\nc{\EW}{\mathbf{W}}
\nc{\rma}[2]{\mathbf{R}_{#1,#2}}
\nc{\Dif}{\EuScript{D}}\nc{\MDif}{\EuScript{E}}
\renc{\mod}{\mathsf{mod}}
\nc{\modg}{\mathsf{mod}^g}
\nc{\fmod}{\mathsf{mod}^{fd}}
\nc{\id}{\operatorname{id}}
\nc{\DR}{\mathbf{DR}}
\nc{\End}{\operatorname{End}}
\nc{\Fun}{\operatorname{Fun}}
\nc{\Ext}{\operatorname{Ext}}
\nc{\tw}{\tau}
\nc{\A}{\EuScript{A}}
\nc{\Loc}{\mathsf{Loc}}
\nc{\eF}{\EuScript{F}}
\nc{\LAA}{\Loc^{\A}_{A}}
\nc{\perv}{\mathsf{Perv}}
\nc{\gfq}[2]{B_{#1}^{#2}}
\nc{\qgf}[1]{A_{#1}}
\nc{\qgr}{\qgf\rho}
\nc{\tqgf}{\tilde A}
\nc{\Tr}{\operatorname{Tr}}
\nc{\Tor}{\operatorname{Tor}}
\nc{\cQ}{\mathcal{Q}}
\nc{\st}[1]{\Delta(#1)}
\nc{\cst}[1]{\nabla(#1)}
\nc{\ei}{\mathbf{e}_i}
\nc{\Be}{\mathbf{e}}
\nc{\Hck}{\mathfrak{H}}
\renc{\P}{\mathbb{P}}
\nc{\bbB}{\mathbb{B}}
\nc{\ssy}{\mathsf{y}}
\nc{\cI}{\mathcal{I}}
\nc{\cG}{\mathcal{G}}
\nc{\cH}{\mathcal{H}}
\nc{\coe}{\mathfrak{K}}
\nc{\pr}{\operatorname{pr}}
\nc{\bra}{\mathfrak{B}}
\nc{\rcl}{\rho^\vee(\la)}
\nc{\tU}{\mathcal{U}}
\nc{\htU}{\widehat{\tU}}
\nc{\tUdot}{\tU_{\bullet}}
\nc{\htUdot}{\widehat{\tUdot}}
\nc{\dU}{{\stackon[8pt]{\tU}{\cdot}}}
\nc{\zw}{\beta}
\nc{\poly}{\mathcal{P}}

\nc{\RHom}{\mathrm{RHom}}
\nc{\tcO}{\tilde{\cO}}
\nc{\Yon}{\mathscr{Y}}
\nc{\sI}{{\mathsf{I}}}
\def\h{\mathbb{O}}
\newcommand{\Ccat}{\mathcal{C}}

\newcommand{\cev}[1]{\reflectbox{\ensuremath{\vec{\reflectbox{\ensuremath{#1}}}}}}
\def\bull{{\scriptstyle\bullet}}
\def\dt{{\color{white}\bullet}\!\!\!\circ}

\def\up{\uparrow}
\def\down{\downarrow}

\def\anticlock{\begin{tikzpicture}[baseline=-.9mm]
\filldraw[white] (0,0) circle (1.72mm);
\draw[-] (0,-0.18) to[out=180,in=-102] (-.178,0.02);
\draw[-] (-0.18,0) to[out=90,in=180] (0,0.18);
\draw[-] (0.18,0) to[out=-90,in=0] (0,-0.18);
\draw[<-] (0,0.18) to[out=0,in=90] (0.18,0);
\end{tikzpicture}\,}
\def\thickanticlock{\begin{tikzpicture}[baseline=-.9mm]
\draw[-,thick] (0,-0.18) to[out=180,in=-102] (-.178,0.02);
\draw[-,thick] (-0.18,0) to[out=90,in=180] (0,0.18);
\draw[-,thick] (0.18,0) to[out=-90,in=0] (0,-0.18);
\draw[<-,thick] (0,0.18) to[out=0,in=90] (0.18,0);
\end{tikzpicture}\,}
\def\anticlockplus{\begin{tikzpicture}[baseline=-.9mm]
\filldraw[white] (0,0) circle (1.72mm);
\draw[-] (0,-0.18) to[out=180,in=-102] (-.178,0.02);
\draw[-] (-0.18,0) to[out=90,in=180] (0,0.18);
\draw[-] (0.18,0) to[out=-90,in=0] (0,-0.18);
\draw[<-] (0,0.18) to[out=0,in=90] (0.18,0);
   \node at (0,0) {$+$};
\end{tikzpicture}\,}
\def\anticlockminus{\begin{tikzpicture}[baseline=-.9mm]
\filldraw[white] (0,0) circle (1.72mm);
\draw[-] (0,-0.18) to[out=180,in=-102] (-.178,0.02);
\draw[-] (-0.18,0) to[out=90,in=180] (0,0.18);
\draw[-] (0.18,0) to[out=-90,in=0] (0,-0.18);
\draw[<-] (0,0.18) to[out=0,in=90] (0.18,0);
   \node at (0,0) {$-$};
\end{tikzpicture}\,}
\def\anticlocki{\begin{tikzpicture}[baseline=-.9mm]
\filldraw[white] (0,0) circle (1.72mm);
\draw[-,thick] (0,-0.18) to[out=180,in=-102] (-.178,0.02);
\draw[-,thick] (-0.18,0) to[out=90,in=180] (0,0.18);
\draw[-,thick] (0.18,0) to[out=-90,in=0] (0,-0.18);
\draw[<-,thick] (0,0.18) to[out=0,in=90] (0.18,0);
   \node at (0,0) {$\scriptstyle{i}$};
\end{tikzpicture}\,}
\def\anticlockj{\begin{tikzpicture}[baseline=-.9mm]
\filldraw[white] (0,0) circle (1.72mm);
\draw[-,thick] (0,-0.18) to[out=180,in=-102] (-.178,0.02);
\draw[-,thick] (-0.18,0) to[out=90,in=180] (0,0.18);
\draw[-,thick] (0.18,0) to[out=-90,in=0] (0,-0.18);
\draw[<-,thick] (0,0.18) to[out=0,in=90] (0.18,0);
   \node at (-0.02,-0.02) {$\scriptstyle{j}$};
\end{tikzpicture}\,}

\def\clock{\begin{tikzpicture}[baseline=-.9mm]
\filldraw[white] (0,0) circle (1.72mm);
\draw[-] (0,-0.18) to[out=180,in=-90] (-.18,0);
\draw[->] (-0.18,0) to[out=90,in=180] (0,0.18);
\draw[-] (-0.02,0.178) to[out=12,in=90] (0.18,0);
\draw[-] (0.18,0) to[out=-90,in=0] (0,-0.18);
\end{tikzpicture}\,}
\def\clockplus{\begin{tikzpicture}[baseline=-.9mm]
\filldraw[white] (0,0) circle (1.72mm);
\draw[-] (0,-0.18) to[out=180,in=-90] (-.18,0);
\draw[->] (-0.18,0) to[out=90,in=180] (0,0.18);
\draw[-] (-0.02,0.178) to[out=12,in=90] (0.18,0);
\draw[-] (0.18,0) to[out=-90,in=0] (0,-0.18);
   \node at (0,0) {$+$};
\end{tikzpicture}\,}
\def\clockminus{\begin{tikzpicture}[baseline=-.9mm]
\filldraw[white] (0,0) circle (1.72mm);
\draw[-] (0,-0.18) to[out=180,in=-90] (-.18,0);
\draw[->] (-0.18,0) to[out=90,in=180] (0,0.18);
\draw[-] (-0.02,0.178) to[out=12,in=90] (0.18,0);
\draw[-] (0.18,0) to[out=-90,in=0] (0,-0.18);
   \node at (0,0) {$-$};
\end{tikzpicture}\,}
\def\clocki{\begin{tikzpicture}[baseline=-.9mm]
\filldraw[white] (0,0) circle (1.72mm);
\draw[-,thick] (0,-0.18) to[out=180,in=-90] (-.18,0);
\draw[->,thick] (-0.18,0) to[out=90,in=180] (0,0.18);
\draw[-,thick] (-0.02,0.178) to[out=12,in=90] (0.18,0);
\draw[-,thick] (0.18,0) to[out=-90,in=0] (0,-0.18);
   \node at (0,0) {$\scriptstyle{i}$};
\end{tikzpicture}\,}
\def\clockj{\begin{tikzpicture}[baseline=-.9mm]
\filldraw[white] (0,0) circle (1.72mm);
\draw[-,thick] (0,-0.18) to[out=180,in=-90] (-.18,0);
\draw[->,thick] (-0.18,0) to[out=90,in=180] (0,0.18);
\draw[-,thick] (-0.02,0.178) to[out=12,in=90] (0.18,0);
\draw[-,thick] (0.18,0) to[out=-90,in=0] (0,-0.18);
   \node at (0,0) {$\scriptstyle{j}$};
\end{tikzpicture}\,}
\def\red#1{{\color{red} #1}}
\def\blue#1{{\color{blue} #1}}
\def\green#1{{\color{green} #1}}

\def\smallclock{\begin{tikzpicture}
\filldraw[white] (0,0) circle (1mm);
\draw[-,thin] (0,-0.1) to[out=180,in=-90] (-.1,0);
\draw[-,thin] (-0.1,0) to[out=90,in=180] (0,0.1);
\draw[->,thin] (0,0.1) to[out=0,in=90] (0.09,-0.04);
\draw[-,thin] (0.09,-0.02) to[out=-96,in=0] (0,-0.1);
\end{tikzpicture}}

\def\smallanticlock{\begin{tikzpicture}
\filldraw[white] (0,0) circle (1mm);
\draw[-,thin] (0,-0.1) to[out=180,in=-84] (-.09,-.02);
\draw[<-,thin] (-0.09,-0.04) to[out=90,in=180] (0,0.1);
\draw[-,thin] (0.1,0) to[out=-90,in=0] (0,-0.1);
\draw[-,thin] (0,0.1) to[out=0,in=90] (0.1,0);
\end{tikzpicture}}

\def\clockright{\begin{tikzpicture}[baseline=-.9mm]
\filldraw[white] (0,0) circle (1.72mm);
\draw[-,thin] (0,-0.18) to[out=180,in=-90] (-.18,0);
\draw[-,thin] (-0.18,0) to[out=90,in=180] (0,0.18);
\draw[->,thin] (0,0.18) to[out=0,in=90] (0.18,0);
\draw[-,thin] (0.178,0.02) to[out=-78,in=0] (0,-0.18);
\end{tikzpicture}}

\def\clockplus{\begin{tikzpicture}[baseline=-.9mm]
\filldraw[white] (0,0) circle (1.72mm);
\draw[-,thin] (0,-0.18) to[out=180,in=-90] (-.18,0);
\draw[-,thin] (-0.18,0) to[out=90,in=180] (0,0.18);
\draw[->,thin] (0,0.18) to[out=0,in=90] (0.18,0);
\draw[-,thin] (0.178,0.02) to[out=-78,in=0] (0,-0.18);
\node at (0,0.01) {$+$};
\end{tikzpicture}}

\def\clockminus{\begin{tikzpicture}[baseline=-.9mm]
\filldraw[white] (0,0) circle (1.72mm);
\draw[-,thin] (0,-0.18) to[out=180,in=-90] (-.18,0);
\draw[-,thin] (-0.18,0) to[out=90,in=180] (0,0.18);
\draw[->,thin] (0,0.18) to[out=0,in=90] (0.18,0);
\draw[-,thin] (0.178,0.02) to[out=-78,in=0] (0,-0.18);
\node at (0,0.01) {$-$};
\end{tikzpicture}}

\def\clockplusminus{\begin{tikzpicture}[baseline=-.9mm]
\filldraw[white] (0,0) circle (1.72mm);
\draw[-,thin] (0,-0.18) to[out=180,in=-90] (-.18,0);
\draw[-,thin] (-0.18,0) to[out=90,in=180] (0,0.18);
\draw[->,thin] (0,0.18) to[out=0,in=90] (0.18,0);
\draw[-,thin] (0.178,0.02) to[out=-78,in=0] (0,-0.18);
\node at (0,0.01) {$\pm$};
\end{tikzpicture}}

\def\clocktop{\begin{tikzpicture}[baseline=-.9mm]
\filldraw[white] (0,0) circle (1.72mm);
\draw[-,thin] (0,-0.18) to[out=180,in=-90] (-.18,0);
\draw[->,thin] (-0.18,0) to[out=90,in=180] (0,0.18);
\draw[-,thin] (-0.02,0.178) to[out=12,in=90] (0.18,0);
\draw[-,thin] (0.18,0) to[out=-90,in=0] (0,-0.18);
\end{tikzpicture}}

\def\anticlockright{\begin{tikzpicture}[baseline=-.9mm]
\filldraw[white] (0,0) circle (1.72mm);
\draw[-,thin] (0,-0.18) to[out=180,in=-90] (-.18,0);
\draw[-,thin] (-0.18,0) to[out=90,in=180] (0,0.18);
\draw[-,thin] (0,0.18) to[out=0,in=78] (0.178,-0.02);
\draw[<-,thin] (0.18,0) to[out=-90,in=0] (0,-0.18);
\end{tikzpicture}}

\def\anticlockleft{\begin{tikzpicture}[baseline=-.9mm]
\filldraw[white] (0,0) circle (1.72mm);
\draw[-,thin] (0,-0.18) to[out=180,in=-102] (-.178,0.02);
\draw[<-,thin] (-0.18,0) to[out=90,in=180] (0,0.18);
\draw[-,thin] (0.18,0) to[out=-90,in=0] (0,-0.18);
\draw[-,thin] (0,0.18) to[out=0,in=90] (0.18,0);
\end{tikzpicture}}

\def\anticlockplus{\begin{tikzpicture}[baseline=-.9mm]
\filldraw[white] (0,0) circle (1.72mm);
\draw[-,thin] (0,-0.18) to[out=180,in=-102] (-.178,0.02);
\draw[<-,thin] (-0.18,0) to[out=90,in=180] (0,0.18);
\draw[-,thin] (0.18,0) to[out=-90,in=0] (0,-0.18);
\draw[-,thin] (0,0.18) to[out=0,in=90] (0.18,0);
\node at (0,0.01) {$+$};
\end{tikzpicture}}

\def\anticlockminus{\begin{tikzpicture}[baseline=-.9mm]
\filldraw[white] (0,0) circle (1.72mm);
\draw[-,thin] (0,-0.18) to[out=180,in=-102] (-.178,0.02);
\draw[<-,thin] (-0.18,0) to[out=90,in=180] (0,0.18);
\draw[-,thin] (0.18,0) to[out=-90,in=0] (0,-0.18);
\draw[-,thin] (0,0.18) to[out=0,in=90] (0.18,0);
\node at (0,0.01) {$-$};
\end{tikzpicture}}

\def\anticlockplusminus{\begin{tikzpicture}[baseline=-.9mm]
\filldraw[white] (0,0) circle (1.72mm);
\draw[-,thin] (0,-0.18) to[out=180,in=-102] (-.178,0.02);
\draw[<-,thin] (-0.18,0) to[out=90,in=180] (0,0.18);
\draw[-,thin] (0.18,0) to[out=-90,in=0] (0,-0.18);
\draw[-,thin] (0,0.18) to[out=0,in=90] (0.18,0);
\node at (0,0.01) {$\pm$};
\end{tikzpicture}}
\tikzset{ centerzero/.style={>=To,baseline={([yshift=-0.5ex](#1))}},
    centerzero/.default={0,0}
}
\newcommand\singdot[2][white]{\filldraw[fill=#1, draw=black] (#2) circle (1.5pt)
}
 \newcommand\dotstrand[1][white]{    \begin{tikzpicture}[centerzero]
        \draw (0,-0.2) -- (0,0.2);
        \singdot[#1]{0,0};
    \end{tikzpicture}
}

\def\thickdown{\pmb\downarrow}
\def\thickup{\pmb\uparrow}
\newcommand{\darkg}{\color{green!70!black}}
\tikzset{darkg/.style={green!70!black}}
\def\signs{\sigma}

\setcounter{tocdepth}{1}
\newcommand{\thetitle}{Unfurling Khovanov-Lauda-Rouquier algebras}

\renc{\theitheorem}{\Alph{itheorem}}

\excise{
\newenvironment{block}
\newenvironment{frame}
\newenvironment{tikzpicture}
\newenvironment{equation*}
}

\baselineskip=1.1\baselineskip

 \usetikzlibrary{decorations.pathreplacing,backgrounds,decorations.markings,shapes.geometric}
\tikzset{wei/.style={draw=red,double=red!40!white,double distance=1.5pt,thin}}
\tikzset{awei/.style={draw=blue,double=blue!40!white,double distance=1.5pt,thin}}
\tikzset{bdot/.style={fill,circle,color=blue,inner sep=3pt,outer
    sep=0}}
\tikzset{dir/.style={postaction={decorate,decoration={markings,
    mark=at position .8 with {\arrow[scale=1.3]{<}}}}}}
\tikzset{rdir/.style={postaction={decorate,decoration={markings,
    mark=at position .8 with {\arrow[scale=1.3]{>}}}}}}
\tikzset{edir/.style={postaction={decorate,decoration={markings,
    mark=at position .2 with {\arrow[scale=1.3]{<}}}}}}\begin{center}
\noindent {\large  \bf Unfurling Khovanov-Lauda-Rouquier algebras}
\medskip

\noindent {\sc Ben Webster}\footnote{Supported by the NSF under Grant
  DMS-1151473. This research was supported in part by Perimeter Institute for Theoretical Physics. Research at Perimeter Institute is supported by the Government of Canada through the Department of Innovation, Science and Economic Development Canada and by the Province of Ontario through the Ministry of Research, Innovation and Science.
}\\  
Department of Pure Mathematics\\ University of Waterloo \&\\
Perimeter Institute for Mathematical Physics\\
Waterloo, ON, Canada\\
Email: {\tt ben.webster@uwaterloo.ca}
\end{center}
\bigskip
{\small
\begin{quote}
\noindent {\em Abstract.}
One famous difficulty of studying algebras and categories presented with explicit generators and relations is the difficulty of checking that they have the ``expected size,'' for example, that an obvious spanning set is in fact a basis.  One particular example which has resisted a simple proof of a basis theorem is the 2-category $\mathcal{U}(\mathfrak{g})$ categorifying a symmetrizable Kac-Moody algebra $\mathfrak{g}$, defined by Khovanov, Lauda, and Rouquier.  Khovanov and Lauda gave a conjectured basis for the 2-morphism spaces in this category, as well as a conjectured isomorphism of its Grothendieck group to $\dot{U}(\mathfrak{g})$, the idempotented universal enveloping algebra, in 2008.  

In this paper, we give a proof of the conjectures above, exploiting the technique of deforming representations of this 2-category;  by showing that these representations have the expected size at the generic point, we can confirm that they are not smaller at the special point of this deformation.
We achieve this by a more general study of the behavior of categorical actions of a Lie
algebra $\mathfrak{g}$ under the deformation of their spectra.  We give conditions
under which the general point of a family of categorical actions of
$\mathfrak{g}$ carry an action of a larger Lie algebra
$\mathfrak{\tilde{g}}$, which we call an {\bf unfurling} of
$\mathfrak{g}$.  This is closely related to the folding of Dynkin
diagrams, but to avoid confusion, we think it is better to use a
different term.
\end{quote}
}

\section{Introduction}
\label{sec:introduction}

Categorification of Lie algebras and their representations has proven to be a rich and
durable subject over the past decade.  This
theory produces a 2-category $\tU$ depending on a Cartan
datum and a choice of parameters; a representation of this 2-category
is called a {\bf categorical Lie algebra action} of the Kac-Moody
algebra $\fg$ corresponding to the Cartan datum.  Many
interesting categories carry categorical Lie algebra actions, although
most well-understood examples are for Cartan data of
(affine) type A.  

However, since the 2-quantum group $\tU$ was first defined by
Khovanov-Lauda 
\cite{khovanovCategorificationQuantum2010} and Rouquier \cite{Rou2KM}, one
serious issue has remained: because $\tU$ is presented by generators and relations,
it is difficult to show that it is not smaller than expected.  Specifically,
\cite{khovanovCategorificationQuantum2010} proves that there is a surjective map from the
modified quantum group $\mathbf{\dot U}$ to the Grothendieck group of
$\tU$, and that the dimension of 2-morphism spaces
between 1-morphisms in $\tU$ is bounded above by a variation on Lusztig's
bilinear form on $\mathbf{\dot U}$.  If
equality holds in this bound, then $\mathbf{\dot U}$ is isomorphic to the
Grothendieck group, and we call the corresponding
categorification {\bf non-degenerate}.  

As a general rule, proving non-degeneracy depends on constructing
appropriate representations in which one can show that no unexpected
relations exist between 2-morphisms in $\tU$.  This is done for
$\mathfrak{sl}_n$ in \cite{khovanovCategorificationQuantum2010}, using an action on the cohomology
of flag varieties.  The next major step was independent proofs by
Kang-Kashiwara \cite{kangCategorificationHighest2012} and the author \cite{Webmerged} that the
simple highest weight representations of $\fg$ possess
categorifications which are non-degenerate in an appropriate sense.  This should be sufficient to prove non-degeneracy for finite-type Cartan data, though as far as the author can tell, this was not carried out in generality beyond $\mathfrak{sl}_n$ (despite the author claiming that he would in \cite{Webmerged}\footnote{This seems to be an artifact of a long and complex revision process. The author can only apologize for the mistakes of his misspent youth.}).

This was significant progress, but it shares an unfortunate limitation with many other
techniques (such as connections to quiver varieties studied in
\cite{cautisCoherentSheaves2013,rouquierQuiverHecke2012,Webqui}).  Recall
that the {\it open Tits cone} of a Cartan datum consists of the elements 
in the orbit of a dominant weight under the Weyl group; for example, if
$\fg$ is affine, then the open Tits cone is the set of weights
of positive level.  This set is convex and
every weight of a highest integrable representation
lies inside the open Tits cone.  Similarly, weights of lowest weight
representations lie in the negative of the Tits cone.  
If the Cartan datum is of infinite type, then no information about a weight $\la$
outside the open Tits cone and its negative (in the affine case, these
are level 0 representations) is contained in any of these
representations, since the corresponding idempotent
$\mathbf{1}_\la\in \mathbf{\dot U}$ kills any integrable highest weight representation.

Thus, if we are to understand non-degeneracy for weights outside the
open Tits cone and its negative, we must have access to representations which are not
highest or lowest weight.  For our purposes, the most promising are
those given by a tensor product of highest and lowest weight
representations (perhaps many of each type).  A construction of such
categorifications $\hl^{\lambda,-\mu}$ of the tensor product of the simple modules with highest weight $\lambda$ and lowest weight $-\mu$ was given in \cite{WebCB}, but the non-degeneracy proof given there is only valid for weights inside the Tits cone.  Thus, to access
these other weights, we must give a new argument for the
non-degeneracy of these tensor product categorifications, which will
then imply the non-degeneracy of $\tU$.  In particular, we prove that:
\begin{itheorem}[\cref{nondegenerate}]\label{thm:main}
Fix any commutative ring $\K'$ where the integers $d_i$ are invertible and consider any Cartan datum $(I,\langle-,-\rangle)$ and choice of the polynomials $Q_{ij}(u,v)\in \K'[u,v]$ which is homogeneous (in the sense discussed in \cref{sec:klr-algebra}).  The associated 2-quantum group
  $\tU$ is non-degenerate and the Grothendieck group of $\tU$ is $\mathbf{\dot U}$.
\end{itheorem}

While this theorem is interesting in its own right, the techniques
used to prove it are also independently useful.  A simple, but often underused, method for this kind
of non-degeneracy argument is upper semicontinuity of dimension
under deformation.  Perhaps calling this ``underused'' is
unfair, since it is certainly a
well-known and much-used trick, but at least this author wishes
he had learned to exploit it systematically much earlier.

Thus, much of this paper is devoted to an exploration of the
behavior of categorical actions under deformation.  Let $R$ be the KLR
algebra of the Cartan datum $(I,\langle-,-\rangle)$ (defined in
\cref{sec:klr-algebra}).  We wish to
consider quotients of this algebra where the dots (the elements usually denoted $y_k\in R$) have a fixed
spectrum; of course, all of these quotients can be packaged together
into a completion $\hat{R}$. Most often, people have studied
representations where the elements $y_k$ act nilpotently (all
gradable finite-dimensional representations have this property), but we can
also have them act with certain fixed non-zero eigenvalues. Given a choice of spectrum for the dots, we have an
associated graph with the vertex set $\tilde{I}$ and associated Cartan datum.
There is a natural map $\tilde{I} \to I$, which one can informally think of as a ``branched
cover'' of the Cartan datum $(I,\langle-,-\rangle)$.

This is closely related to the
phenomenon of {\bf folding} of Dynkin diagrams, but due to some
technical differences, we think it would be misleading to use the term
``folding'' here.  Thus, we call $\tilde{I}$ (with its induced graph structure) an {\bf unfurling} of $I$
(and $I$ a {\bf furling} of $\tilde{I}$). Note that while $I$
is not necessarily symmetric as a Cartan datum (it may have roots of
different lengths), we define $\tilde{I}$ in such a way that it is symmetric.  To give the reader a
sense of this operation, let us discuss some examples:
\begin{itemize}
    \item  If $I$ is simply laced, then $\tilde{I}$ will be a topological cover of $I$, such as an $A_\infty$ graph covering an $n$-cycle or the trivial cover $\tilde I \cong I\times U$ for some set $U$.  A relevant example is discussed in \cref{example1}. 
\item If $I$ is not simply laced, we will obtain a symmetric Cartan datum $\tilde{I}$, which in many examples is simply-laced, with
  an isomorphism $\fg\cong \tilde{\fg}^\sigma$ for some diagram automorphism $\sigma$.  If $\tilde{I}$ is simply-laced, this means that $I$ is the Langlands dual of what is usually called a folding of $\tilde{I}$ for the automorphism $\sigma$.  A relevant example is discussed in \cref{example2}. 
\end{itemize}
At present, it is unclear to the author what, if any, the
relationship between this work and that of McNamara \cite{mcnamaraFoldingKLR2019} and
Elias \cite{eliasFoldingSoergel2017},
which also combine the ideas of categorification and folding.  This
would be an interesting topic for future consideration.
 
We always have a map of Lie algebras $\fg\hookrightarrow \tilde{\fg}$ (after appropriate completion if $\tilde{I}$ is infinite),
and this map has a categorical analogue, proven in \cref{lem:g-action,equivalence}:
\begin{itheorem}\label{th:B}
	We have a functor $\tU(\fg)\to \htU(\tilde{\fg})$, the completion of the additive closure of $\tU(\tilde{\fg})$ where dots are nilpotent.  We can always choose $\tilde{I}$ so that the categorification $\hl^{\lambda,-\mu}$ flatly deforms to the pullback of the categorification of a simple under this functor. 
\end{itheorem}

Since we know a basis theorem for morphism spaces in categorified simple modules over $\tU(\tilde{\fg})$, this allows us to prove a similar theorem for $\hl^{\lambda,-\mu}$ (\cref{tricolore-nondegenerate}).

We first give the most direct proof of \cref{thm:main} based on this result, and then proceed to complementary results. In particular, \cref{th:B} has a converse: we can define a $\tU(\tilde{\fg})$-action on a category with a $\tU(\fg)$-action where the dots are not nilpotent (\cref{thm:deform-action}).  

Finally, we apply these results to the categorifications of tensor products introduced in \cite[\S 4]{WebCB}.  The theory developed there is ultimately dependent on certain non-degeneracy results that we were only able to prove using the deformation techniques of this paper.  

We also include an appendix discussing the theory of valued graphs and how it relates to the graph structures on the sets $\tilde{I}$.  This material is not strictly required in the main body of the paper, but it connects some of the combinatorics we use with existing constructions.  

Since this paper has been substantially rewritten since it was first posted on the arXiv, we include a brief explanation of how this version differs from earlier ones.  The ideas at the center of the first version of this paper ultimately led to a very fruitful collaboration with Jon Brundan and Alistair Savage \cite{brundanDegenerateHeisenberg2023,brundanDefinitionQuantum2020,brundanHeisenbergKacMoody2020,brundanQuantumFrobenius2022,brundanFoundationsFrobenius2021}.  
In particular, closely analogous ideas were central to the new proof of the basis theorem for Heisenberg categories in \cite[Th. 6.4]{brundanDegenerateHeisenberg2023} and to the relation between Heisenberg and Kac-Moody categorifications in \cite[Th. A]{brundanHeisenbergKacMoody2020}.  The aim of this revision is to apply similar ideas, replacing Heisenberg categorifications with Kac-Moody categorifications.  Thus, \cref{sec:proof} is largely new, and heavily based on \cite{brundanDegenerateHeisenberg2023,brundanHeisenbergKacMoody2020}. This has created a slightly tangled relationship: those papers were inspired by the original version of this paper, but the underlying ideas and calculations were refined substantially in the process of writing them (mostly due to the positive influence of my coauthors), and I apply those improvements here to reorganize this paper.  In particular, there is a strong analogy between:
\begin{enumerate}
  \item \cref{thm:deform-action} and \cite[Th. A]{brundanHeisenbergKacMoody2020}.  Both take an action of a 2-category (Kac-Moody or Heisenberg), apply a spectral decomposition to the functors, and find a new, more ``homogeneous'' 2-category action.
	\item \cref{th:B} and \cite[Th. 5.4]{brundanDegenerateHeisenberg2023} (and also \cite[Th. 5.22]{brundanHeisenbergKacMoody2020}).  Both define a functor from one 2-category into a localization of another, allowing for the construction of ``large'' modules over the source 2-category via pullback.

 	\item \cref{thm:main} and \cite[Th. 6.4]{brundanDegenerateHeisenberg2023}.  Both use the ``large'' modules constructed via the pullback above to check a basis theorem for the morphisms in the 2-category.  
\end{enumerate}

\subsection*{Acknowledgements}
\label{sec:acknowledgements}

As discussed in the introduction above, the current version of this paper would have been impossible without what I have learned through the wisdom and hard work of my collaborators Jon Brundan and Alistair Savage.  A mathematician could not wish for finer colleagues.  
I would also like to thank Eric Vasserot for pointing out to me the
difficulties which arise from the Tits cone; Ben Elias for
some very helpful comments; Chris Leonard for
pointing out a very silly mistake; and all the people
(Wolfgang Soergel, Raphael Rouquier, and Catharina Stroppel among them)
who taught me the importance of deforming things.  A number of referees also gave extremely helpful comments on the paper, which have improved it greatly.

\section{Background}
\label{sec:background}

\notation{$I, \mathfrak{g}$}{The indexing set of simple roots in a Cartan datum and associated Kac-Moody algebra.}
Throughout, we fix a (possibly infinite) set $I$, and a Cartan datum on this set.  
That is, the free abelian group
generated by the simple roots $\al_i$ for $i\in I$ carries a symmetric bilinear form $\langle
-,-\rangle$ such that $\langle \al_i,\al_i\rangle \in 2\Z_{>0}$ and $C=\left(c_{ij}=2\frac{\langle\al_i,\al_j
  \rangle}{\langle\al_i,\al_i \rangle}\right)$ is a symmetrizable locally finite generalized Cartan matrix.     By {\it locally finite}, we mean that for $i\in I$, we have that $c_{ij}\neq 0$ for only finitely many $j$; that is, the associated Dynkin diagram is a locally finite graph. Note that
$d_i=\langle\al_i,\al_i \rangle/2$ are symmetrizing coefficients for
this Cartan matrix---we have an equality $d_ic_{ij}=d_jc_{ji}=\langle\al_i,\al_j\rangle$
for all $i,j$. Before we turn to a more detailed discussion of constructing Kac-Moody algebras from this data, let us reassure the reader: these details are not important and will not play a central role in our proofs.  Rather, we work in the greatest possible generality to show that our results do not depend on fussy details of the construction of Kac-Moody algebras, as we show below in \cref{lem:different-realizations}.  
\notation{$c_{ij},C$}{The entries $c_{ij}=\al_i^{\vee}(\al_j)=2\frac{\langle\al_i,\al_j
  \rangle}{\langle\al_i,\al_i \rangle}$ of the Cartan matrix $C$.}
 \notation{$d_i$}{The symmetrizing coefficients $d_i=\langle\al_i,\al_i \rangle/2$ satisfying  $d_ic_{ij}=d_jc_{ji}=\langle\al_i,\al_j\rangle$.}

\notation{$\mathbb{F}$}{The base field of characteristic 0 for the Lie algebra $\fg$.} Let
$\mathbb{F}$ be a field of characteristic 0. We will choose a
realization of the Cartan datum discussed above over $\mathbb{F}$.
That is, we choose:
\begin{enumerate}
  \item An
$\mathbb{F}$-vector space $\mathfrak{h}$ (possibly of infinite
dimension).

\item Elements $\al_i^{\vee}\in
\mathfrak{h}$.  Let $\mathfrak{w}=\mathfrak{h}^\star$ be the restricted dual
of $\mathfrak{h}$, the subspace of the dual of $\mathfrak{h}$
consisting of functions which vanish on all but finitely many $\al^{\vee}_i\in
I$.
\item Elements $\al_i\in \mathfrak{w}$ for $i\in I$ such that $\al_i^{\vee}(\al_j)=c_{ij}$.  \end{enumerate}
As usual, we can define a Kac-Moody Lie algebra $\fg$ with Cartan $\mathfrak{h}$ generated by formal symbols $E_i$ and $F_i$ satisfying $[E_i,F_i]=\al_i^{\vee}$ and the Serre relations for the Cartan matrix.

For a given Cartan matrix, there are 4 canonical choices of realization (up to isomorphism), which we can derive from the free vector spaces $\mathfrak{h}_0=\mathbb{F}^I$ and $\mathfrak{w}_0=\mathbb{F}^I$.  The Cartan matrix defines a natural pairing between these spaces $\mathfrak{c}(h,w)=\sum_{i,j\in I}h_ic_{ij}w_j$. Let $e^{\vee}_i,e_j$ be the coordinate unit vectors in $\mathfrak{h}_0,\mathfrak{w}_0$.

\begin{enumerate}[wide]
	\item We can take 
 \begin{align*}
 	\mathfrak{h}&=\mathfrak{h}_0/\{h\in \mathfrak{h}_0\mid \mathfrak{c}(h,w)=0\quad \forall w\in \mathfrak{w}_0\} & e_i^{\vee}&\mapsto \al_i^{\vee} \\
 	\mathfrak{w}&=\mathfrak{w}_0/ \{w\in \mathfrak{w}_0\mid \mathfrak{c}(h,w)=0 \quad\forall h\in \mathfrak{h}_0\} & e_i &\mapsto \al_i.
 \end{align*}
  This corresponds to a Kac-Moody algebra with no grading elements (so weight multiplicities may be infinite) and trivial center. The sets $\al_i^{\vee}$ and $\al_i$ span $\mathfrak{h}$ and $\mathfrak{w}$, but if $C$ is degenerate, they are not linearly independent.  An affine example is the loop algebra $\fg[t,t^{-1}]$ for finite-dimensional $\fg$.  
	\item  We can take $\mathfrak{h}=\mathfrak{h}_0$ with $\al_i^{\vee}=e_i^{\vee}$; we must take $\al_i$ to be the function defined by $h\mapsto \mathfrak{c}(h,e_i)$.    
   If $C$ is degenerate, the roots $\al_i$ fail both to span and to be linearly independent.  This is the algebra obtained by naively extending the Serre presentation of finite-dimensional simple Lie algebras to more general Cartan matrices.  
  An affine example is the universal central extension of $\fg[t,t^{-1}]$.  
	
  We will see below that this realization is often the most convenient; since it will be useful to refer to it later, we call it the {\bf universal derived} Lie algebra for this Cartan matrix. One useful property is that it is initial among all realizations over $\mathbb{F}$---if $\fg$ is the Kac-Moody algebra with this realization, and $\fg'$ is any other Kac-Moody algebra for the same Cartan matrix, then there is a canonical homomorphism $\fg\to \fg'$ whose image is $[\fg',\fg']$.  
	 \item We can reverse the roles of roots and coroots, and take $\mathfrak{h}=\mathfrak{w}_0^{\star}$ with $\al_i=e_i$, which similarly requires that $\al_i^{\vee}$ be the function $w\mapsto \mathfrak{c}(e_i^{\vee},w)$.  
  In this case, the coroots $\al_i^{\vee}$ might fail both to span and to be linearly independent.   An affine example is $\mathbb{F}^{\times}\ltimes \fg[t,t^{-1}]$, the loop algebra with a grading element added to the Cartan.
	\item We can require that both $\al_i$ and $\al_i^{\vee}$ are linearly independent at the cost that neither will span their respective spaces if $C$ is degenerate.  This is
 most often the version of the Kac-Moody algebra considered.  An affine example is the universal central extension of $\mathbb{F}^{\times}\ltimes \fg[t,t^{-1}]$.
\end{enumerate}
Note that when $C$ has finite rank and is non-degenerate, these realizations are all the same.  In particular, these differences matter only for infinite-dimensional Kac-Moody algebras.  
\begin{definition}
	The {\bf weight lattice} $\wela=\{\la\in \mathfrak{w}\mid \al_i^{\vee}(\la)\in \Z\}$ of $\fg$ is the subgroup of $\mathfrak{w}$ on which $\al_i^{\vee}$ has integer value.  Elements of $\wela$ are called {\bf weights}.\notation{$\wela$}{The weight lattice $\{\la\in\mathfrak{w}\mid \al_i^{\vee}(\la)\in\Z\}$ of $\fg$.}
\end{definition}
If the coroots $\al_i^{\vee}$ span $\mathfrak{h}$, then the map $\al^{\vee}\colon \wela\to \Z^I$ sending $\la\mapsto (\al_i^{\vee}(\la))$ is injective, and so $\wela$ is a free abelian group.  On the other hand, if this map is not injective, its kernel is a positive-dimensional $\mathbb{F}$-vector space.  As a prelude to letting the reader completely forget the discussion above, the abelian group $\wela$ together with the map $\al^{\vee}$ will be the only aspect of the Kac-Moody algebra above and the corresponding realization of the Cartan matrix that are relevant for the construction of the categorification below.  

\notation{$\K$}{The base field for KLR algebras, which is algebraically closed of characteristic coprime to
all $d_i$.}
Our construction of KLR algebras and categorified quantum groups depends on the Cartan matrix $C$ and weight lattice $\wela$ above, a field $\K$ which we assume is algebraically closed of characteristic coprime to
all $d_i$, and, for each $i\neq j\in I$, a polynomial $Q_{ij}(x,y)=Q_{ji}(y,x)\in \K[x,y]$ which is homogeneous
of degree $-2\langle\al_i,\al_j\rangle= -2d_ic_{ij}=-2d_jc_{ji}$ when
$x$ has degree $2d_i$ and $y$ has degree $2d_j$. Despite the assumption on $\K$ here, we can extend \cref{thm:main} to a general commutative ring in which the integers $d_i$ are invertible by standard commutative algebra arguments; see page \pageref{reduction} for details.
We assume
throughout that $Q_{ij}(1,0)$ is a unit for all $i\neq j$.
\notation{$Q_{ij}$}{The polynomials that appear in the definition of the KLR algebra.}

\notation{$P_{ij}$}{Polynomials
$P_{ij}(x,y)\in \K[x,y]$ such that $Q_{ij}(x,y)=P_{ij}(x,y)P_{ji}(y,x)$.}
Choose polynomials
$P_{ij}(x,y)\in \K[x,y]$ such that $Q_{ij}(x,y)=P_{ij}(x,y)P_{ji}(y,x)$.  
Let
$p_{ij}=P_{ij}(1,0)$ and $t_{ij}=Q_{ij}(0,1)^{-1}$.  

This homogeneity gives a sensible notion of the order of vanishing of $Q_{ij}(x,y)$ or $P_{ij}(x,y)$
  at $x=u,y=u'$ for $u,u'\in \K\setminus 0$.  This is the same as the order of vanishing of $Q_{ij}(u,y)$ at $y=u'$ or $Q_{ij}(x,u')$ at $x=u$.  

Fix a subset (finite or infinite) $U_i\subset \K\setminus \{0\}$ for each $i\in I$.  \begin{definition}\label{def:tilde-I}
  Let $\tilde{I}$ be the set of pairs $\{(i,u)\in I\times \K\mid u\in U_i\}$.  Define an oriented graph with vertex set $\tilde{I}$ by taking the number of edges oriented from $(i,u)$ to $(j,u')$ to be the order of vanishing of $P_{ij}(x,y)$ at $x=u,y=u'$, as defined above. This order of vanishing for $Q_{ij}(x,y)$ is the total number of edges, with both orientations, joining $(i,u)$ and $(j,u')$.
Let $\tilde{\fg}$ be the Kac-Moody algebra associated to this graph.  
\notation{$U_i$}{A subset of $\K$ attached to a vertex $i\in I$.  Typically the spectrum of a natural transformation of an associated functor (\cref{def:tilde-I}).}
\notation{$\tilde{I},\tilde{\fg}$}{The set of pairs $(i,u)$ with $i\in I$ and $u\in U_i\subset \K$, together with the induced graph structure and the attached Kac-Moody algebra.}
  
  We call a choice of the sets $U_i$ {\bf complete} if, whenever $Q_{ij}(u,u')=0$ for $u\in U_i$, then $u'\in U_j$. 
\end{definition}

\begin{example}\label{example2}
The first non-trivial case is when $I=\{1,2\}$ and the underlying Cartan matrix is $[\begin{smallmatrix}
	2 & -2 \\ -1 & 2
\end{smallmatrix}]$, so we have type $B_2$.  If $Q_{12}(x,y)=x^2-y$ and
$d_1=1,d_2=2$,  the number of edges joining $(1,x)$ to
$(2,y)$ is given by the number of solutions to $x^2=y$.
Thus, every component of $\tilde{I}$ is a subgraph of an $A_3$ formed
by $(1,x) \to (2,x^2) \leftarrow (1,-x)$ (since by assumption $1\neq -1$).  

More examples are covered in \cref{example1,example3}. 
\end{example}

When the choice of the sets $U_i$ is complete, we call $\tilde{I}$ (with this induced graph structure) an {\bf unfurling} of $I$. We will show that there is a natural homomorphism from $U(\fg)$ to a suitable completion $\hat{U}(\tilde{\fg})$ (\cref{prop:furling-homomorphism}) which sends:
\begin{equation}
F_i\mapsto \sum_{u\in U_i} F_{i,u}\qquad E_i\mapsto \sum_{u\in U_i} E_{i,u}\qquad H_i\mapsto \sum_{u\in U_i} H_{i,u}. \label{eq:furling-homomorphism}	
\end{equation}
We can think of \cref{th:B} as a categorification of this result, but first we must introduce the relevant categories.  

\subsection{The KLR algebra}
\label{sec:klr-algebra}

Let $\K$ and $Q_{ij}$ be as above.  

\begin{definition}
  Let $R_n$ denote  the KLR algebra with generators given by:
  \notation{$R_n$}{The KLR algebra attached to the Cartan datum $I$ and polynomials $Q_{ij}$ over $\K$.}
\begin{itemize}
\item The idempotent $e_{\Bi}$ given by straight lines labeled with
  $(i_1,\dots, i_n)\in I^n$.

\item The element $y_k^\Bi$ given by straight lines with a dot on
the $k$th strand.

\item The element $\psi_k^\Bi$ which is a crossing of the $k$th and $(k+1)$st strands.  
\end{itemize}

\begin{equation*}
    \tikz[<-]{
      \node[label=below:{$e_{\Bi}$}] at (-4.5,0){ 
        \tikz[very thick,xscale=1.2]{
          \draw (-.5,-.5)-- (-.5,.5) node[below,at start]{$i_1$};
          \draw (0,-.5)-- (0,.5) node[below,at start]{$i_2$};
          \draw (1.5,-.5)-- (1.5,.5) node[below,at start]{$i_n$};
          \node at (.75,0){$\cdots$};
        }
      };
      \node[label=below:{$y_k^\Bi$}] at (0,0){ 
        \tikz[very thick,xscale=1.2]{
          \draw (-.5,-.5)-- (-.5,.5) node[below,at start]{$i_1$};
          \draw (.5,-.5)-- (.5,.5) node [midway,fill=black,circle,inner
          sep=2pt]{} node[below,at start]{$i_j$};
          \draw (1.5,-.5)-- (1.5,.5) node[below,at start]{$i_n$};
          \node at (1,0){$\cdots$};
          \node at (0,0){$\cdots$};
        }
      };
      \node[label=below:{$\psi_k^\Bi$}] at (4.5,0){ 
        \tikz[very thick,xscale=1.2]{
          \draw (-.5,-.5)-- (-.5,.5) node[below,at start]{$i_1$};
          \draw (.1,-.5)-- (.9,.5) node[below,at start]{$i_j$};
          \draw (.9,-.5)-- (.1,.5) node[below,at start]{$i_{j+1}$};
          \draw (1.5,-.5)-- (1.5,.5) node[below,at start]{$i_n$};
          \node at (1,0){$\cdots$};
          \node at (0,0){$\cdots$};
        }
      };
    }
  \end{equation*}
  and  relations:  
  \newseq
  \begin{align}\subeqn\label{QHA1}
\mathord{
\begin{tikzpicture}[baseline = 1mm]
	\draw[->,thick] (0.25,.55) to (-0.25,-.15);
	\draw[<-,thick] (0.25,-.15) to (-0.25,.55);
  \node at (-0.25,-.3) {$\scriptstyle{j}$};
   \node at (0.25,-.3) {$\scriptstyle{i}$};
      \node at (-0.13,0.01) {$\bull$};
\end{tikzpicture}
}
-
\mathord{
\begin{tikzpicture}[baseline = 1mm]
	\draw[->,thick] (0.25,.55) to (-0.25,-.15);
	\draw[<-,thick] (0.25,-.15) to (-0.25,.55);
  \node at (-0.25,-.3) {$\scriptstyle{j}$};
   \node at (0.25,-.3) {$\scriptstyle{i}$};
      \node at (0.13,0.38) {$\bull$};
\end{tikzpicture}
}
&=
\mathord{
\begin{tikzpicture}[baseline = 1mm]
 	\draw[->,thick] (0.25,.55) to (-0.25,-.15);
	\draw[<-,thick] (0.25,-.15) to (-0.25,.55);
  \node at (-0.25,-.3) {$\scriptstyle{j}$};
   \node at (0.25,-.3) {$\scriptstyle{i}$};
      \node at (-0.13,0.38) {$\bull$};
\end{tikzpicture}
}
-
\mathord{
\begin{tikzpicture}[baseline = 1mm]
	\draw[->,thick] (0.25,.55) to (-0.25,-.15);
	\draw[<-,thick] (0.25,-.15) to (-0.25,.55);
  \node at (-0.25,-.3) {$\scriptstyle{j}$};
   \node at (0.25,-.3) {$\scriptstyle{i}$};
      \node at (0.13,0.01) {$\bull$};
\end{tikzpicture}
}
=
\delta_{i,j}
\mathord{
\begin{tikzpicture}[baseline = -.5mm]
 	\draw[<-,thick] (0.08,-.3) to (0.08,.4);
	\draw[<-,thick] (-0.28,-.3) to (-0.28,.4);
   \node at (-0.28,-.45) {$\scriptstyle{j}$};
   \node at (0.08,-.45) {$\scriptstyle{i}$};
\end{tikzpicture}
},\\\subeqn\label{QHA2}
\mathord{
\begin{tikzpicture}[baseline = 2.5mm]
	\draw[-,thick] (0.28,.4) to[out=90,in=-90] (-0.28,1.1);
	\draw[-,thick] (-0.28,.4) to[out=90,in=-90] (0.28,1.1);
	\draw[<-,thick] (0.28,-.3) to[out=90,in=-90] (-0.28,.4);
	\draw[<-,thick] (-0.28,-.3) to[out=90,in=-90] (0.28,.4);
  \node at (-0.28,-.45) {$\scriptstyle{j}$};
  \node at (0.28,-.45) {$\scriptstyle{i}$};
\end{tikzpicture}
}
&=
\left\{
\begin{array}{ll}
0\hspace{40mm}&\text{if $j=i$,}\\
Q_{ji}\left(\mathord{
\begin{tikzpicture}[baseline = -.5mm]
	\draw[<-,thick] (0.08,-.3) to (0.08,.4);
	\draw[<-,thick] (-0.28,-.3) to (-0.28,.4);
   \node at (-0.28,-.45) {$\scriptstyle{j}$};
   \node at (0.08,-.45) {$\scriptstyle{i}$};
      \node at (-0.28,0.05) {$\bull$};
\end{tikzpicture}
},\mathord{
\begin{tikzpicture}[baseline = -.5mm]
	\draw[<-,thick] (0.08,-.3) to (0.08,.4);
	\draw[<-,thick] (-0.28,-.3) to (-0.28,.4);
   \node at (-0.28,-.45) {$\scriptstyle{j}$};
   \node at (0.08,-.45) {$\scriptstyle{i}$};
     \node at (0.08,0.05) {$\bull$};
\end{tikzpicture}
}\right)
&\text{if $i\neq j$,}\\
\end{array}
\right. \\\subeqn\label{QHA3}
\mathord{
\begin{tikzpicture}[baseline = .5mm]
	\draw[->,thick] (0.45,.8) to (-0.45,-.4);
	\draw[<-,thick] (0.45,-.4) to (-0.45,.8);
        \draw[<-,thick] (0,-.4) to[out=90,in=-90] (-.45,0.2);
        \draw[-,thick] (-0.45,0.2) to[out=90,in=-90] (0,0.8);
   \node at (-0.45,-.6) {$\scriptstyle{k}$};
   \node at (0,-.6) {$\scriptstyle{j}$};
  \node at (0.45,-.6) {$\scriptstyle{i}$};
\end{tikzpicture}
}
\!\!-
\!\!\!
\mathord{
\begin{tikzpicture}[baseline = .5mm]
	\draw[->,thick] (0.45,.8) to (-0.45,-.4);
	\draw[<-,thick] (0.45,-.4) to (-0.45,.8);
        \draw[<-,thick] (0,-.4) to[out=90,in=-90] (.45,0.2);
        \draw[-,thick] (0.45,0.2) to[out=90,in=-90] (0,0.8);
   \node at (-0.45,-.6) {$\scriptstyle{k}$};
   \node at (0,-.6) {$\scriptstyle{j}$};
  \node at (0.45,-.6) {$\scriptstyle{i}$};
\end{tikzpicture}
}
&=
\left\{
\begin{array}{ll}
\tilde{Q}_{ij}\left(\mathord{\displaystyle
\begin{tikzpicture}[baseline = -.5mm]
	\draw[<-,thick] (0.44,-.3) to (0.44,.4);
	\draw[<-,thick] (0.08,-.3) to (0.08,.4);
	\draw[<-,thick] (-0.28,-.3) to (-0.28,.4);
   \node at (-0.28,-.5) {$\scriptstyle{i}$};
   \node at (0.08,-.5) {$\scriptstyle{j}$};
   \node at (0.44,-.5) {$\scriptstyle{i}$};
     \node at (-0.28,0.05) {$\bull$};
\end{tikzpicture}
},
\mathord{\displaystyle
\begin{tikzpicture}[baseline = -0.5mm]
	\draw[<-,thick] (0.44,-.3) to (0.44,.4);
	\draw[<-,thick] (0.08,-.3) to (0.08,.4);
	\draw[<-,thick] (-0.28,-.3) to (-0.28,.4);
   \node at (-0.28,-.5) {$\scriptstyle{i}$};
   \node at (0.08,-.5) {$\scriptstyle{j}$};
   \node at (0.44,-.5) {$\scriptstyle{i}$};
     \node at (0.08,0.05) {$\bull$};
\end{tikzpicture}
},
\mathord{\displaystyle
\begin{tikzpicture}[baseline = -.5mm]
	\draw[<-,thick] (0.44,-.3) to (0.44,.4);
	\draw[<-,thick] (0.08,-.3) to (0.08,.4);
	\draw[<-,thick] (-0.28,-.3) to (-0.28,.4);
   \node at (-0.28,-.5) {$\scriptstyle{i}$};
   \node at (0.08,-.5) {$\scriptstyle{j}$};
   \node at (0.44,-.5) {$\scriptstyle{i}$};
     \node at (0.44,0.05) {$\bull$};
\end{tikzpicture}
}
\right)&\text{if $i= k$,}\\
0&\text{otherwise,}
\end{array}\right.\end{align}
\text{where} \qquad \[\tilde{Q}_{ij}(u,v,w)=\frac{Q_{ij}(u,v)-Q_{ij}(w,v)}{u-w}.\]
 \end{definition} 

Following Rouquier \cite[Prop. 3.12]{Rou2KM}, we have a polynomial action of $R_n$ on a sum of polynomial rings $\K[Y_1,\dots, Y_n]$ in $n$ variables, one for each $\Bi\in I^n$.  
Since the idempotents $e_{\Bi}$ act by projection to the different summands, we can write this sum as \[\poly =\bigoplus_{\Bi\in I^n}\K[Y_1,\dots, Y_n]e_{\Bi},\]  with the action defined by 
\newseq
\begin{align*}
	\label{eq:poly-action1}\subeqn
		y_kf(Y_1,\dots, Y_n)e_{\Bi}&=Y_kf(Y_1,\dots, Y_n)e_{\Bi}\\ \subeqn\label{eq:poly-action2}\psi_kf(Y_1,\dots, Y_n)&=\begin{cases}
  P_{i_k i_{k+1}}(Y_{k+1},Y_k)f(Y_1,\dots, Y_{k+1},Y_k,\dots, Y_n)e_{s_k\Bi}&i_k\neq i_{k+1}\\
	\frac{f(Y_1,\dots, Y_{k+1},Y_k,\dots, Y_n)-f(Y_1,\dots, Y_k,Y_{k+1},\dots, Y_n)}{Y_{k+1}-Y_k}e_{\Bi}&i_k=i_{k+1}.
	\end{cases}
\end{align*}

\subsection{Categorical actions}
\label{sec:appl-categ-acti}

We have considered the KLR algebra $R$ with an eye toward studying
categorical actions of Lie algebras.  By ``a categorical action of a
Lie algebra'' we mean a representation of a specific 2-category $\tU$
defined by Khovanov-Lauda and Rouquier (the
equivalence of these 2-categories is proven in \cite{brundanDefinitionKac2016}).  
We follow the conventions of \cite{brundanDefinitionKac2016}: in the notation of that paper,
	\[Q_{ij}(u,v)=t_{ij}^{-1}t_{ji}^{-1}(t_{ij}u^{-c_{ij}}+t_{ji}v^{-c_{ji}}+\sum_{p,q} s^{pq}_{ji}u^pv^q).\] 
The {\em Kac-Moody 2-category}
 $\tU(\fg)$
is the strict $\K$-linear 2-category
whose:
\begin{enumerate}[label=(\roman*)]
\item  Objects are formal direct sums (with arbitrary indexing set) of the elements of $\wela$.
\item 
Generating
1-morphisms are 
$\eE_i 1_\lambda = \substack{\thickup \\ {\scriptscriptstyle i}}{\scriptstyle \color{gray}\lambda}:\lambda \rightarrow \lambda+\alpha_i$ and
$\eF_i 1_\lambda= \substack{{\scriptscriptstyle
    i}\\\thickdown}{\scriptstyle \color{gray}\lambda}:\lambda \rightarrow \lambda-\alpha_i$ for
$i\in I$ and $\lambda \in \wela$; at times, for simplicity of notation, we use $\eE_{-i}$ to denote $\eF_i$, and we use $\Bi$ to represent the monomial $\eE_{\Bi}=\eE_{i_1}\cdots \eE_{i_n}$.
\item 
Generating
2-morphisms
\notation{ $\tU(\fg)$}{The Kac-Moody 2-category categorifying $\dot{\mathbf{U}}(\fg)$.}
\begin{align}\label{QHgens}
\mathord{
\begin{tikzpicture}[baseline = -2]
	\draw[<-,thick] (0.08,-.15) to (0.08,.3);
      \node at (0.08,0.05) {$\bull$};
   \node at (0.08,-.3) {$\scriptstyle{i}$};
\end{tikzpicture}
}
{\color{gray}\scriptstyle\lambda}
&:\eF_i 1_\lambda \Rightarrow \eF_i 1_\lambda,
&
\mathord{
\begin{tikzpicture}[baseline = 0]
	\draw[<-,thick] (0.3,0.2) to[out=-90, in=0] (0.1,-0.1);
	\draw[-,thick] (0.1,-0.1) to[out = 180, in = -90] (-0.1,0.2);
    \node at (-0.1,.35) {$\scriptstyle{i}$};
\end{tikzpicture}
}\:\,
{\color{gray}\scriptstyle\lambda}
&:1_\lambda \Rightarrow \eF_i \eE_i 1_\lambda,
&
\mathord{
\begin{tikzpicture}[baseline = -2]
	\draw[<-,thick] (0.3,-0.1) to[out=90, in=0] (0.1,0.2);
	\draw[-,thick] (0.1,0.2) to[out = 180, in = 90] (-0.1,-0.1);
    \node at (-0.1,-.25) {$\scriptstyle{i}$};
\end{tikzpicture}
}
\:\,{\color{gray}\scriptstyle\lambda}
&:\eE_i \eF_i 1_\lambda \Rightarrow 1_\lambda,\\
\mathord{
\begin{tikzpicture}[baseline = -2]
	\draw[<-,thick] (0.18,-.15) to (-0.18,.3);
	\draw[<-,thick] (-0.18,-.15) to (0.18,.3);
   \node at (-0.18,-.3) {$\scriptstyle{j}$};
   \node at (0.18,-.3) {$\scriptstyle{i}$};
\end{tikzpicture}
}
{\color{gray}\scriptstyle\lambda}
&:\eF_j \eF_i 1_\lambda \Rightarrow \eF_i \eF_j 1_\lambda,&
\mathord{
\begin{tikzpicture}[baseline = 0]
	\draw[-,thick] (0.3,0.2) to[out=-90, in=0] (0.1,-0.1);
	\draw[->,thick] (0.1,-0.1) to[out = 180, in = -90] (-0.1,0.2);
    \node at (0.3,.35) {$\scriptstyle{i}$};
\end{tikzpicture}
}\!
{\color{gray}\scriptstyle\lambda}
&:1_\lambda \Rightarrow \eE_i \eF_i 1_\lambda,
&
\mathord{
\begin{tikzpicture}[baseline = -2]
	\draw[-,thick] (0.3,-0.1) to[out=90, in=0] (0.1,0.2);
	\draw[->,thick] (0.1,0.2) to[out = 180, in = 90] (-0.1,-0.1);
    \node at (0.3,-.25) {$\scriptstyle{i}$};
\end{tikzpicture}
}
\!{\color{gray}\scriptstyle\lambda}
&\eF_i \eE_i 1_\lambda \Rightarrow 1_\lambda.
\end{align}
\end{enumerate}
\begin{remark}\label{rem:0-object-sum}
	One would usually just take the set of 0-morphisms to be $\rola$ itself, but it will be more convenient later to consider direct sums of these objects.  As usual, a 1-morphism between direct sums is a column-finite matrix of 1-morphisms, and 2-morphisms between these are defined entrywise.
	
  Note that a 2-functor from our version of $\tU$ into a 2-category with direct sums of objects (such as $\K$-linear categories $\mathsf{Cat}_{\K}$) can be constructed canonically from a functor in the version where objects are just $\rola$, using the direct sum in the target 2-category.
\end{remark}

Taking $t_{ii}=1$ for all $i$, the sideways crossings and upward dots and crossings are defined by 
\begin{align}\label{sideways}
\mathord{
\begin{tikzpicture}[baseline = 0]
	\draw[<-,thick] (0.28,-.3) to (-0.28,.4);
	\draw[->,thick] (-0.28,-.3) to (0.28,.4);
   \node at (-.3,-.43) {$\scriptstyle{j}$};
   \node at (-.3,.55) {$\scriptstyle{i}$};
   \node at (0.5,0.05) {$\color{gray}\scriptstyle{\lambda}$};
\end{tikzpicture}
}
&:=t_{ij}\mathord{
\begin{tikzpicture}[baseline = 0]
	\draw[->,thick] (0.3,.5) to (-0.3,-.5);
	\draw[-,thick] (-0.2,.2) to (0.2,-.3);
        \draw[-,thick] (0.2,-.3) to[out=130,in=180] (0.5,-.5);
        \draw[->,thick] (0.5,-.5) to[out=0,in=270] (0.9,.5);
        \draw[-,thick] (-0.2,.2) to[out=130,in=0] (-0.6,.5);
        \draw[-,thick] (-0.6,.5) to[out=180,in=-270] (-0.9,-.5);
   \node at (-0.3,-.65) {$\scriptstyle{j}$};
   \node at (.9,.65) {$\scriptstyle{i}$};
   \node at (1.1,0) {$\color{gray}\scriptstyle{\lambda}$};
\end{tikzpicture}
}\:,&
\mathord{
\begin{tikzpicture}[baseline = 0]
	\draw[->,thick] (0.28,-.3) to (-0.28,.4);
	\draw[<-,thick] (-0.28,-.3) to (0.28,.4);
   \node at (.3,-.43) {$\scriptstyle{j}$};
   \node at (.3,.55) {$\scriptstyle{i}$};
   \node at (0.5,0.05) {$\color{gray}\scriptstyle{\lambda}$};
\end{tikzpicture}
}
&:=
\mathord{
\begin{tikzpicture}[baseline = 0]
	\draw[<-,thick] (0.3,-.5) to (-0.3,.5);
	\draw[-,thick] (-0.2,-.2) to (0.2,.3);
        \draw[-,thick] (0.2,.3) to[out=50,in=180] (0.5,.5);
        \draw[-,thick] (0.5,.5) to[out=0,in=90] (0.9,-.5);
        \draw[-,thick] (-0.2,-.2) to[out=230,in=0] (-0.6,-.5);
        \draw[->,thick] (-0.6,-.5) to[out=180,in=-90] (-0.9,.5);
   \node at (.3,-.65) {$\scriptstyle{j}$};
   \node at (-.9,.65) {$\scriptstyle{i}$};
   \node at (1.1,0) {$\color{gray}\scriptstyle{\lambda}$};
\end{tikzpicture}
}\:,
\end{align}
\begin{align}
\mathord{\begin{tikzpicture}[baseline=0mm]
  \draw[->,thick] (0,-0.4) to (0,.4);
   \node at (0,.55) {$\scriptstyle{i}$};
   \node at (0,0){$\bull$};
   \node at (0.2,0) {$\color{gray}\scriptstyle{\lambda}$};
\end{tikzpicture}
}=
\mathord{
\begin{tikzpicture}[baseline = 0mm]
  \draw[-,thick] (0.3,0) to (0.3,-.4);
	\draw[-,thick] (0.3,0) to[out=90, in=0] (0.1,0.4);
	\draw[-,thick] (0.1,0.4) to[out = 180, in = 90] (-0.1,0);
	\draw[-,thick] (-0.1,0) to[out=-90, in=0] (-0.3,-0.4);
	\draw[-,thick] (-0.3,-0.4) to[out = 180, in =-90] (-0.5,0);
  \draw[->,thick] (-0.5,0) to (-0.5,.4);
   \node at (-0.5,.55) {$\scriptstyle{i}$};
   \node at (0.5,0) {$\color{gray}\scriptstyle{\lambda}$};
   \node at (-0.1,0){$\bull$};
\end{tikzpicture}
}
\label{adjdot1}\end{align}
There are negatively dotted bubbles defined by 
\begin{align*}
\mathord{
\begin{tikzpicture}[baseline = 1.25mm]
  \draw[->,thick] (0.2,0.2) to[out=90,in=0] (0,.4);
  \draw[-,thick] (0,0.4) to[out=180,in=90] (-.2,0.2);
\draw[-,thick] (-.2,0.2) to[out=-90,in=180] (0,0);
  \draw[-,thick] (0,0) to[out=0,in=-90] (0.2,0.2);
   \node at (0.2,0.2) {$\bull$};
   \node at (1,0.2) {$\scriptstyle{n-\al_i^{\vee}(\lambda)-1}$};
   \node at (-0.05,-.15) {$\scriptstyle{i}$};
   \node at (-.38,0.2) {$\color{gray}\scriptstyle{\lambda}$};
\end{tikzpicture}}
&:=
\left\{
\begin{array}{ll}
(-1)^n 
\det\left(
\:\mathord{
\begin{tikzpicture}[baseline = 1.25mm]
  \draw[<-,thick] (0,0.4) to[out=180,in=90] (-.2,0.2);
  \draw[-,thick] (0.2,0.2) to[out=90,in=0] (0,.4);
 \draw[-,thick] (-.2,0.2) to[out=-90,in=180] (0,0);
  \draw[-,thick] (0,0) to[out=0,in=-90] (0.2,0.2);
   \node at (-0.2,0.2) {$\bull$};
   \node at (-.95,0.2) {$\scriptstyle{r-s+\al_i^{\vee}(\lambda)}$};
   \node at (0.05,-.15) {$\scriptstyle{i}$};
   \node at (.38,0.2) {$\color{gray}\scriptstyle{\lambda}$};
\end{tikzpicture}
}\:
\right)_{r,s=1,\dots,n}
&\text{if $\al_i^{\vee}(\lambda) \geq n > 0$},\\
1_{1_\lambda}&\text{if $\al_i^{\vee}(\lambda) \geq n  = 0$,}\\
0&\text{if $\al_i^{\vee}(\lambda) \geq n < 0$,}
\end{array}\right.
\\\mathord{
\begin{tikzpicture}[baseline = 1.25mm]
  \draw[<-,thick] (0,0.4) to[out=180,in=90] (-.2,0.2);
  \draw[-,thick] (0.2,0.2) to[out=90,in=0] (0,.4);
 \draw[-,thick] (-.2,0.2) to[out=-90,in=180] (0,0);
  \draw[-,thick] (0,0) to[out=0,in=-90] (0.2,0.2);
   \node at (-0.2,0.2) {$\bull$};
   \node at (-1,0.2) {$\scriptstyle{n+\al_i^{\vee}(\lambda)-1}$};
   \node at (0.05,-.15) {$\scriptstyle{i}$};
   \node at (.38,0.2) {$\color{gray}\scriptstyle{\lambda}$};
\end{tikzpicture}
}&:=
\left\{\begin{array}{ll}
(-1)^n 
\det\left(
\mathord{
\begin{tikzpicture}[baseline = 1.25mm]
  \draw[->,thick] (0.2,0.2) to[out=90,in=0] (0,.4);
  \draw[-,thick] (0,0.4) to[out=180,in=90] (-.2,.2);
\draw[-,thick] (-.2,0.2) to[out=-90,in=180] (0,0);
  \draw[-,thick] (0,0) to[out=0,in=-90] (0.2,0.2);
   \node at (0.2,0.2) {$\bull$};
   \node at (.95,0.2) {$\scriptstyle{r-s-\al_i^{\vee}(\lambda)}$};
   \node at (-0.05,-.15) {$\scriptstyle{i}$};
   \node at (-.38,0.2) {$\color{gray}\scriptstyle{\lambda}$};
\end{tikzpicture}
}\right)_{r,s=1,\dots,n}\:\:\,&\text{if $-\al_i^{\vee}(\lambda) \geq n > 0$,}\\1_{1_\lambda}&\text{if $-\al_i^{\vee}(\lambda) \geq n=0$,}\\
0&\text{if $-\al_i^{\vee}(\lambda) \geq n < 0$.}
\end{array}\right.
\end{align*}
The generating 2-morphisms are subject to the 
relations \crefrange{QHA1}{QHA3} (which agree with \cite[(2.6-8)]{brundanDefinitionKac2016}) and the additional relations:
\begin{align}
\mathord{
\begin{tikzpicture}[baseline = -1mm]
  \draw[->,thick] (0.3,0) to (0.3,.4);
	\draw[-,thick] (0.3,0) to[out=-90, in=0] (0.1,-0.4);
	\draw[-,thick] (0.1,-0.4) to[out = 180, in = -90] (-0.1,0);
	\draw[-,thick] (-0.1,0) to[out=90, in=0] (-0.3,0.4);
	\draw[-,thick] (-0.3,0.4) to[out = 180, in =90] (-0.5,0);
  \draw[-,thick] (-0.5,0) to (-0.5,-.4);
   \node at (-0.5,-.55) {$\scriptstyle{i}$};
   \node at (0.5,0) {$\color{gray}\scriptstyle{\lambda}$};
\end{tikzpicture}
}
&=
\mathord{\begin{tikzpicture}[baseline=-1mm]
  \draw[->,thick] (0,-0.4) to (0,.4);
   \node at (0,-.55) {$\scriptstyle{i}$};
   \node at (0.2,0) {$\color{gray}\scriptstyle{\lambda}$};
\end{tikzpicture}
},\qquad\quad\qquad
\mathord{
\begin{tikzpicture}[baseline = 0mm]
  \draw[->,thick] (0.3,0) to (0.3,-.4);
	\draw[-,thick] (0.3,0) to[out=90, in=0] (0.1,0.4);
	\draw[-,thick] (0.1,0.4) to[out = 180, in = 90] (-0.1,0);
	\draw[-,thick] (-0.1,0) to[out=-90, in=0] (-0.3,-0.4);
	\draw[-,thick] (-0.3,-0.4) to[out = 180, in =-90] (-0.5,0);
  \draw[-,thick] (-0.5,0) to (-0.5,.4);
   \node at (-0.5,.55) {$\scriptstyle{i}$};
   \node at (0.5,0) {$\color{gray}\scriptstyle{\lambda}$};
\end{tikzpicture}
}
=
\mathord{\begin{tikzpicture}[baseline=0mm]
  \draw[<-,thick] (0,-0.4) to (0,.4);
   \node at (0,.55) {$\scriptstyle{i}$};
   \node at (0.2,0) {$\color{gray}\scriptstyle{\lambda}$};
\end{tikzpicture}
},\label{rightadj}\\
\mathord{
\begin{tikzpicture}[baseline = 0mm]
  \draw[-,thick] (0.3,0) to (0.3,-.4);
	\draw[-,thick] (0.3,0) to[out=90, in=0] (0.1,0.4);
	\draw[-,thick] (0.1,0.4) to[out = 180, in = 90] (-0.1,0);
	\draw[-,thick] (-0.1,0) to[out=-90, in=0] (-0.3,-0.4);
	\draw[-,thick] (-0.3,-0.4) to[out = 180, in =-90] (-0.5,0);
  \draw[->,thick] (-0.5,0) to (-0.5,.4);
   \node at (0.3,-.55) {$\scriptstyle{i}$};
   \node at (0.5,0) {$\color{gray}\scriptstyle{\lambda}$};
\end{tikzpicture}
}
&=
\mathord{\begin{tikzpicture}[baseline=0mm]
  \draw[->,thick] (0,-0.4) to (0,.4);
   \node at (0,-.55) {$\scriptstyle{i}$};
   \node at (0.2,0) {$\color{gray}\scriptstyle{\lambda}$};
\end{tikzpicture}
},\qquad\quad\qquad
\mathord{
\begin{tikzpicture}[baseline = -1mm]
  \draw[-,thick] (0.3,0) to (0.3,.4);
	\draw[-,thick] (0.3,0) to[out=-90, in=0] (0.1,-0.4);
	\draw[-,thick] (0.1,-0.4) to[out = 180, in = -90] (-0.1,0);
	\draw[-,thick] (-0.1,0) to[out=90, in=0] (-0.3,0.4);
	\draw[-,thick] (-0.3,0.4) to[out = 180, in =90] (-0.5,0);
  \draw[->,thick] (-0.5,0) to (-0.5,-.4);
   \node at (0.3,.55) {$\scriptstyle{i}$};
   \node at (0.5,0) {$\color{gray}\scriptstyle{\lambda}$};
\end{tikzpicture}
}
=
\mathord{\begin{tikzpicture}[baseline=-1mm]
  \draw[<-,thick] (0,-0.4) to (0,.4);
   \node at (0,.55) {$\scriptstyle{i}$};
   \node at (0.2,0) {$\color{gray}\scriptstyle{\lambda}$};
\end{tikzpicture}
},\label{leftadj}\end{align}
\begin{equation}
\mathord{
\begin{tikzpicture}[baseline = 0mm]
  \draw[-,thick] (0.3,0) to (0.3,-.4);
	\draw[-,thick] (0.3,0) to[out=90, in=0] (0.1,0.4);
	\draw[-,thick] (0.1,0.4) to[out = 180, in = 90] (-0.1,0);
	\draw[-,thick] (-0.1,0) to[out=-90, in=0] (-0.3,-0.4);
	\draw[-,thick] (-0.3,-0.4) to[out = 180, in =-90] (-0.5,0);
  \draw[->,thick] (-0.5,0) to (-0.5,.4);
   \node at (0.5,-.55) {$\scriptstyle{i}$};
   \node at (0.5,0) {$\color{gray}\scriptstyle{\lambda}$};
   \node at (-0.1,0){$\bull$};
\end{tikzpicture}
}
=
\mathord{
\begin{tikzpicture}[baseline = -1mm]
  \draw[->,thick] (0.3,0) to (0.3,.4);
	\draw[-,thick] (0.3,0) to[out=-90, in=0] (0.1,-0.4);
	\draw[-,thick] (0.1,-0.4) to[out = 180, in = -90] (-0.1,0);
	\draw[-,thick] (-0.1,0) to[out=90, in=0] (-0.3,0.4);
	\draw[-,thick] (-0.3,0.4) to[out = 180, in =90] (-0.5,0);
  \draw[-,thick] (-0.5,0) to (-0.5,-.4);
   \node at (-0.5,-.55) {$\scriptstyle{i}$};
   \node at (0.5,0) {$\color{gray}\scriptstyle{\lambda}$};
   \node at (-0.1,0){$\bull$};
\end{tikzpicture}
}\label{adjdot2}\end{equation}
\begin{align}\label{loop1}
\mathord{
\begin{tikzpicture}[baseline = -0.5mm]
	\draw[-,thick] (0,0.6) to (0,0.3);
	\draw[-,thick] (0,0.3) to [out=-90,in=180] (.3,-0.2);
	\draw[-,thick] (0.3,-0.2) to [out=0,in=-90](.5,0);
	\draw[-,thick] (0.5,0) to [out=90,in=0](.3,0.2);
	\draw[-,thick] (0.3,.2) to [out=180,in=90](0,-0.3);
	\draw[->,thick] (0,-0.3) to (0,-0.6);
   \node at (0,-.75) {$\scriptstyle{i}$};
   \node at (0.4,-0.4) {$\color{gray}\scriptstyle{\lambda}$};
\end{tikzpicture}
}&=
-\delta_{\al_i^{\vee}(\lambda),0} \:
\mathord{
\begin{tikzpicture}[baseline = -0.5mm]
	\draw[->,thick] (0,0.6) to (0,-0.6);
   \node at (0,-.75) {$\scriptstyle{i}$};
   \node at (0.2,0) {$\color{gray}\scriptstyle{\lambda}$};
\end{tikzpicture}
}
\text{ if $\al_i^{\vee}(\lambda) \leq
   0$,}\\\label{loop2}
\mathord{
\begin{tikzpicture}[baseline = -0.5mm]
	\draw[-,thick] (0,0.6) to (0,0.3);
	\draw[-,thick] (0,0.3) to [out=-90,in=0] (-.3,-0.2);
	\draw[-,thick] (-0.3,-0.2) to [out=180,in=-90](-.5,0);
	\draw[-,thick] (-0.5,0) to [out=90,in=180](-.3,0.2);
	\draw[-,thick] (-0.3,.2) to [out=0,in=90](0,-0.3);
	\draw[->,thick] (0,-0.3) to (0,-0.6);
   \node at (0,-.75) {$\scriptstyle{i}$};
   \node at (-0.4,-0.4) {$\color{gray}\scriptstyle{\lambda}$};
\end{tikzpicture}
}&=
\delta_{\al_i^{\vee}(\lambda),0}
\: \:\mathord{
\begin{tikzpicture}[baseline = -0.5mm]
	\draw[->,thick] (0,0.6) to (0,-0.6);
   \node at (0,-.75) {$\scriptstyle{i}$};
   \node at (-0.2,0) {$\color{gray}\scriptstyle{\lambda}$};
\end{tikzpicture}
}
\phantom{-}\text{if $\al_i^{\vee}(\lambda) \geq 0$},\end{align}\begin{align}\label{bubble-normalize1}
\mathord{
\begin{tikzpicture}[baseline = 1.25mm]
  \draw[<-,thick] (0,0.4) to[out=180,in=90] (-.2,0.2);
  \draw[-,thick] (0.2,0.2) to[out=90,in=0] (0,.4);
 \draw[-,thick] (-.2,0.2) to[out=-90,in=180] (0,0);
  \draw[-,thick] (0,0) to[out=0,in=-90] (0.2,0.2);
   \node at (0,-.15) {$\scriptstyle{i}$};
   \node at (0.36,0.2) {$\color{gray}\scriptstyle{\lambda}$};
   \node at (-0.2,0.2) {$\bull$};
   \node at (-1,0.2) {$\scriptstyle{n +\al_i^{\vee}(\lambda)-1}$};
\end{tikzpicture}
}&=
\delta_{n,0}
\:
 \:1_{1_\lambda}
\text{ if $-\al_i^{\vee}(\lambda) < n\leq 0$,}\\ \label{bubble-normalize2}
\mathord{
\begin{tikzpicture}[baseline = 1.25mm]
  \draw[->,thick] (0.2,0.2) to[out=90,in=0] (0,.4);
  \draw[-,thick] (0,0.4) to[out=180,in=90] (-.2,0.2);
\draw[-,thick] (-.2,0.2) to[out=-90,in=180] (0,0);
  \draw[-,thick] (0,0) to[out=0,in=-90] (0.2,0.2);
   \node at (0.2,0.2) {$\bull$};
   \node at (1,0.2) {$\scriptstyle{n-\al_i^{\vee}(\lambda)-1}$};
   \node at (0,-.15) {$\scriptstyle{i}$};
   \node at (-0.4,0.2) {$\color{gray}\scriptstyle{\lambda}$};
\end{tikzpicture}
}&=
\delta_{n,0}
\: 1_{1_\lambda}
\text{ if $\al_i^{\vee}(\lambda) < n \leq 0$,}
\end{align}\begin{align}
\mathord{
\begin{tikzpicture}[baseline = 0]
	\draw[->,thick] (0.28,0) to[out=90,in=-90] (-0.28,.7);
	\draw[-,thick] (-0.28,0) to[out=90,in=-90] (0.28,.7);
	\draw[<-,thick] (0.28,-.7) to[out=90,in=-90] (-0.28,0);
	\draw[-,thick] (-0.28,-.7) to[out=90,in=-90] (0.28,0);
  \node at (-0.28,-.85) {$\scriptstyle{j}$};
  \node at (0.28,.85) {$\scriptstyle{i}$};
  \node at (.45,0) {$\color{gray}\scriptstyle{\lambda}$};
\end{tikzpicture}
}
&=
(-1)^{\delta_{i,j}}
\mathord{
\begin{tikzpicture}[baseline = 0]
	\draw[<-,thick] (0.08,-.3) to (0.08,.4);
	\draw[->,thick] (-0.28,-.3) to (-0.28,.4);
   \node at (-0.28,-.45) {$\scriptstyle{j}$};
   \node at (0.08,.55) {$\scriptstyle{i}$};
   \node at (.3,.05) {$\color{gray}\scriptstyle{\lambda}$};
\end{tikzpicture}
}+\delta_{i,j}\sum_{r,s\geq 0}
\mathord{
\begin{tikzpicture}[baseline = 0]
	\draw[-,thick] (0.3,0.7) to[out=-90, in=0] (0,0.3);
	\draw[->,thick] (0,0.3) to[out = 180, in = -90] (-0.3,0.7);
    \node at (0.3,.85) {$\scriptstyle{i}$};
    \node at (0.4,-0.32) {$\color{gray}\scriptstyle{\lambda}$};
  \draw[->,thick] (0.2,0) to[out=90,in=0] (0,0.2);
  \draw[-,thick] (0,0.2) to[out=180,in=90] (-.2,0);
\draw[-,thick] (-.2,0) to[out=-90,in=180] (0,-0.2);
  \draw[-,thick] (0,-0.2) to[out=0,in=-90] (0.2,0);
 \node at (-0.3,0) {$\scriptstyle{i}$};
   \node at (0.2,0) {$\bull$};
   \node at (0.85,0) {$\scriptstyle{-r-s-2}$};
   \node at (-0.25,0.45) {$\bull$};
   \node at (-0.42,0.45) {$\scriptstyle{r}$};
	\draw[<-,thick] (0.3,-.7) to[out=90, in=0] (0,-0.3);
	\draw[-,thick] (0,-0.3) to[out = 180, in = 90] (-0.3,-.7);
    \node at (-0.3,-.85) {$\scriptstyle{i}$};
   \node at (-0.25,-0.5) {$\bull$};
   \node at (-.42,-.5) {$\scriptstyle{s}$};
\end{tikzpicture}
},\label{flight1}\\
\mathord{
\begin{tikzpicture}[baseline = 0]
	\draw[-,thick] (0.28,0) to[out=90,in=-90] (-0.28,.7);
	\draw[->,thick] (-0.28,0) to[out=90,in=-90] (0.28,.7);
	\draw[-,thick] (0.28,-.7) to[out=90,in=-90] (-0.28,0);
	\draw[<-,thick] (-0.28,-.7) to[out=90,in=-90] (0.28,0);
  \node at (0.28,-.85) {$\scriptstyle{j}$};
  \node at (-0.28,.85) {$\scriptstyle{i}$};
  \node at (.45,0) {$\color{gray}\scriptstyle{\lambda}$};
\end{tikzpicture}
}
&=
(-1)^{\delta_{i,j}}\mathord{
\begin{tikzpicture}[baseline = 0]
	\draw[->,thick] (0.08,-.3) to (0.08,.4);
	\draw[<-,thick] (-0.28,-.3) to (-0.28,.4);
   \node at (-0.28,.55) {$\scriptstyle{i}$};
   \node at (0.08,-.45) {$\scriptstyle{j}$};
   \node at (.3,.05) {$\color{gray}\scriptstyle{\lambda}$};
\end{tikzpicture}
}+\delta_{i,j}\sum_{r,s\geq 0}
\mathord{
\begin{tikzpicture}[baseline=0]
	\draw[-,thick] (0.3,-0.7) to[out=90, in=0] (0,-0.3);
	\draw[->,thick] (0,-0.3) to[out = 180, in = 90] (-0.3,-.7);
    \node at (0.3,-.85) {$\scriptstyle{i}$};
   \node at (0.25,-0.5) {$\bull$};
   \node at (.42,-.5) {$\scriptstyle{r}$};
  \draw[<-,thick] (0,0.2) to[out=180,in=90] (-.2,0);
  \draw[-,thick] (0.2,0) to[out=90,in=0] (0,.2);
 \draw[-,thick] (-.2,0) to[out=-90,in=180] (0,-0.2);
  \draw[-,thick] (0,-0.2) to[out=0,in=-90] (0.2,0);
 \node at (0.3,0.1) {$\scriptstyle{i}$};
   \node at (0.4,-0.2) {$\color{gray}\scriptstyle{\lambda}$};
   \node at (-0.2,0) {$\bull$};
   \node at (-0.9,0) {$\scriptstyle{-r-s-2}$};
	\draw[<-,thick] (0.3,.7) to[out=-90, in=0] (0,0.3);
	\draw[-,thick] (0,0.3) to[out = -180, in = -90] (-0.3,.7);
   \node at (0.27,0.5) {$\bull$};
   \node at (0.45,0.5) {$\scriptstyle{s}$};
    \node at (-0.3,.85) {$\scriptstyle{i}$};
\end{tikzpicture}
}.\hspace{10mm}\label{flight2}
\end{align}

Frequently, when there is no danger of confusion, we abbreviate $\tU(\fg)$ to $\tU$.  The case $\fg=\mathfrak{sl}_2$ is particularly important, and we write $\tUdot=\tU(\mathfrak{sl}_2)$.  It is a standard and useful observation that the diagrams where all strands are labeled $i$ form a copy of $\tUdot$ inside $\tU$.  Just as the structure and representation theory of $\fg$ are largely determined by the $\mathfrak{sl}_2$-subalgebras corresponding to simple roots, the structure and representation theory of $\tU$ have an analogous relationship with $\tUdot$.

A brief comment is useful here on why this is the same as the 2-category defined in \cite{brundanDefinitionKac2016} (up to adding direct sums of objects), since our definition is not exactly either of the two given there.  
\begin{definition}
Let $\tU_R$ be the category defined in \cite[Def. 1.1]{brundanDefinitionKac2016}, called ``Rouquier's 2-category'' because its definition follows the approach of \cite{Rou2KM}.
For ease of comparison, we take the preimage of Brundan's generating set and relations under the ``Chevalley involution'' defined in \cite[Th. 2.3]{brundanDefinitionKac2016}.  This category is generated by 1-morphisms 
	\begin{align}\label{eq:BrundanRouquier-gens}
\mathord{
\begin{tikzpicture}[baseline = -2]
	\draw[<-,thick] (0.08,-.15) to (0.08,.3);
      \node at (0.08,0.05) {$\bull$};
   \node at (0.08,-.3) {$\scriptstyle{i}$};
\end{tikzpicture}
}
{\color{gray}\scriptstyle\lambda}
&:\eF_i 1_\lambda \Rightarrow \eF_i 1_\lambda,
&
\mathord{
\begin{tikzpicture}[baseline = 0]
	\draw[<-,thick] (0.3,0.2) to[out=-90, in=0] (0.1,-0.1);
	\draw[-,thick] (0.1,-0.1) to[out = 180, in = -90] (-0.1,0.2);
    \node at (-0.1,.35) {$\scriptstyle{i}$};
\end{tikzpicture}
}\:\,
{\color{gray}\scriptstyle\lambda}
&:1_\lambda \Rightarrow \eF_i \eE_i 1_\lambda,\\
\label{eq:BrundanRouquier-gens2}
\mathord{
\begin{tikzpicture}[baseline = -2]
	\draw[<-,thick] (0.18,-.15) to (-0.18,.3);
	\draw[<-,thick] (-0.18,-.15) to (0.18,.3);
   \node at (-0.18,-.3) {$\scriptstyle{j}$};
   \node at (0.18,-.3) {$\scriptstyle{i}$};
\end{tikzpicture}
}
{\color{gray}\scriptstyle\lambda}
&:\eF_j \eF_i 1_\lambda \Rightarrow \eF_i \eF_j 1_\lambda,
&
\mathord{
\begin{tikzpicture}[baseline = -2]
	\draw[<-,thick] (0.3,-0.1) to[out=90, in=0] (0.1,0.2);
	\draw[-,thick] (0.1,0.2) to[out = 180, in = 90] (-0.1,-0.1);
    \node at (-0.1,-.25) {$\scriptstyle{i}$};
\end{tikzpicture}
}
\:\,{\color{gray}\scriptstyle\lambda}
&:\eE_i \eF_i 1_\lambda \Rightarrow 1_\lambda.
\end{align}
and the relations 
\cref{rightadj} and 
\crefrange{QHA1}{QHA3} for which we will use the shorthands (KM1) and (KM2) respectively, as well as the relation (KM3):
\begin{enumerate}
	\item  If 
$k_i=\al_i^{\vee}(\mu)\geq 0$,
we adjoin an inverse to the 2-morphism $\eE_{i} \eF_{i}1_\lambda \Rightarrow
\eF_{i}\eE_{i}1_\lambda 
\oplus 1_\lambda^{\oplus k_i}$, induced by the column vector
\begin{equation}\label{EF1}
\left[
\mathord{
\begin{tikzpicture}[baseline = 0,thick]
	\draw[->] (-0.28,-.3) to (0.28,.3);
\draw[<-] (0.28,-.3) to (-0.28,.3);
      \node at (-0.33,-0.43) {$\scriptstyle{i}$};
      \node at (0.33,-0.43) {$\scriptstyle{i}$};
      \node at (-0.33,0.43) {$\scriptstyle{i}$};
      \node at (0.33,0.43) {$\scriptstyle{i}$};
   \end{tikzpicture}
}\quad 
\mathord{
\begin{tikzpicture}[baseline = 1mm,thick]
	\draw[<-] (0.4,0) to[out=90, in=0] (0.1,0.4);
      \node at (-0.15,0.45) {$\phantom\bullet$};
	\draw[-] (0.1,0.4) to[out = 180, in = 90] (-0.2,0);
      \node at (-0.2,-0.14) {$\scriptstyle{i}$};
\end{tikzpicture}
}\quad
\mathord{
\begin{tikzpicture}[baseline = 1mm,thick]
	\draw[<-] (0.4,0) to[out=90, in=0] (0.1,0.4);
	\draw[-] (0.1,0.4) to[out = 180, in = 90] (-0.2,0);
      \node at (-0.15,0.45) {$\phantom\bullet$};
      \node at (-0.15,0.2) {$\bullet$};
\node at (-0.37,.2) {$\scriptstyle x$};
      \node at (-0.2,-0.14) {$\scriptstyle{i}$};
\end{tikzpicture}
}\quad\cdots \quad 
\mathord{
\begin{tikzpicture}[baseline = 1mm,thick]
	\draw[<-] (0.4,0) to[out=90, in=0] (0.1,0.4);
	\draw[-] (0.1,0.4) to[out = 180, in = 90] (-0.2,0);
\node at (-0.6,.2) {$\scriptstyle x^{k_{i}-1}$};
      \node at (-0.15,0.42) {$\phantom\bullet$};
      \node at (-0.15,0.2) {$\bullet$};
      \node at (-0.2,-0.14) {$\scriptstyle{i}$};
\end{tikzpicture}
}\,.
\right]
\end{equation} \item   If $k_{i}\leq 0$, we adjoin an inverse to the 2-morphism  $\eE_{i} \eF_{i}1_\lambda 
\oplus 1_\lambda ^{\oplus -k_{i}}\Rightarrow
\eF_{i}\eE_{i}1_\lambda $ given by the row vector
\begin{equation}\label{EF2}
\left[
\begin{tikzpicture}[baseline = -1mm,thick]
	\draw[<-] (0.28,-.28) to (-0.28,.28);
\draw[->] (-0.28,-.28) to (0.28,.28);
      \node at (-0.33,-0.43) {$\scriptstyle{i}$};
      \node at (0.33,-0.43) {$\scriptstyle{i}$};
      \node at (-0.33,0.43) {$\scriptstyle{i}$};
      \node at (0.33,0.43) {$\scriptstyle{i}$};
\end{tikzpicture}
\quad
\mathord{
\begin{tikzpicture}[baseline = 1mm,thick]
	\draw[<-] (0.4,.4) to[out=-90, in=0] (0.1,0);
	\draw[-] (0.1,0) to[out = 180, in = -90] (-0.2,.4);
      \node at (-0.2,0.55) {$\scriptstyle{i}$};
\end{tikzpicture}
}
\quad
\mathord{
\begin{tikzpicture}[baseline = 1mm,thick]
	\draw[<-] (0.4,.4) to[out=-90, in=0] (0.1,0);
	\draw[-] (0.1,0) to[out = 180, in = -90] (-0.2,.4);
      \node at (-0.2,0.55) {$\scriptstyle{i}$};
\node at (0.37,.2) {$\bullet$};
\node at (0.57,.2) {$\scriptstyle x$};
\end{tikzpicture}
}
\quad
\cdots
\quad
\mathord{
\begin{tikzpicture}[baseline = 1mm,thick]
	\draw[<-] (0.4,.4) to[out=-90, in=0] (0.1,0);
	\draw[-] (0.1,0) to[out = 180, in = -90] (-0.2,.4);
      \node at (-0.2,0.55) {$\scriptstyle{i}$};
\node at (0.37,.2) {$\bullet$};
\node at (.9,.2) {$\scriptstyle x^{-k_{i}-1}$};
\end{tikzpicture}
}
\right]
\end{equation}
\end{enumerate}
\end{definition}
\begin{lemma}\label{lem:BrundanRouquier}
  There is an equivalence $c\colon\tU_R\to \tU$ sending Brundan's generators \crefrange{eq:BrundanRouquier-gens}{eq:BrundanRouquier-gens2} to the same diagrams in $\tU$.
\end{lemma}
\begin{proof}
	First we need to check that $c$ is a functor.  The relations (KM1,KM2) are explicitly given as relations of $\tU$, so we need only check (KM3).
	To check this we  note that \crefrange{flight1}{flight2} gives an explicit formula for the desired inverse of \crefrange{EF1}{EF2} in terms of oppositely oriented cups or caps and fake bubbles (for more detail, see the results leading up to \cite[Cor. 3.3]{brundanDefinitionKac2016}).

On the other hand, to define an inverse functor $d\colon \tU\to \tU_R$, we simply send the only additional generators in our set, the leftward cup and cap to those defined in \cite[(1.14-18)]{brundanDefinitionKac2016}. This is clearly inverse to $c$ if it is well-defined.  Checking that this functor is well-defined is effectively equivalent to the Main Theorem of \cite{brundanDefinitionKac2016}.  More explicitly, we simply need to check that our remaining relations \crefrange{leftadj}{flight2} hold in $\tU_R$, which is verified in {\it loc. cit.}:
	\begin{enumerate}
		\item For \cref{leftadj} in Theorem 4.3.
		\item For \cref{adjdot1} in Theorem 5.3.
		\item For \crefrange{loop1}{loop2} in Lemma 4.1.
		\item For \crefrange{bubble-normalize1}{bubble-normalize2} in Theorem 3.2.
		\item For \crefrange{flight1}{flight2} in Corollary 3.4.  \qedhere
	\end{enumerate}
\end{proof}

A representation of $\tU$ is a 2-functor from $\tU$ to the 2-category of exact categories with exact functors.  This sends $\mu$ to a category $\Ccat_{\mu}$ and a direct sum of 0-morphisms to the direct sum of the corresponding categories $\Ccat_{\mu}$.   We always assume that $\Ccat_{\mu}$ is locally finite abelian or Schurian over the field $\K$, as defined in \cite[\S 2.2]{brundanHeisenbergKacMoody2020}:

\begin{itemize}
  \item We call $\Ccat_{\mu}$ {\bf locally finite abelian} if every object has finite length and every morphism space is finitely dimensional over $\K$.  
	\item We call $\Ccat_{\mu}$ {\bf Schurian} if it is isomorphic to the locally finite-dimensional, locally unital modules over a locally finite-dimensional, locally unital algebra $A$.  
\end{itemize}
Since $I$ might potentially be infinite, we also add the restriction that:
\begin{itemize}
	\item[$(\dagger)$] on any given object $M$ we only have $\eE_iM\neq 0$ or $\eF_iM\neq 0$ for finitely many $i\in I$.  
\end{itemize} 

When constructing representations of $\tU$, it is more convenient to use the presentation $\tU_R$ defined in the remark above, writing out the relations (KM1-3).  That is:
\begin{theorem}\label{thm:brundandef}
Given categories $\Ccat_{\mu}$ for $\mu \in \rola$ and functors $\eE_i\colon \Ccat_{\mu}\to \Ccat_{\mu+\al_i},\eF_i\colon \Ccat_{\mu}\to \Ccat_{\mu-\al_i}$ define a representation of $\tU(\fg)$ if and only if:
	\begin{enumerate}
\item[(KM1)] there are prescribed adjunctions
$(\eE_{i}, \eF_{i})$ for all $i\in I$;
\item[(KM2)]
For $m \geq 0$ there is an action of the
KLR algebra  $R_m$ with $m$ strands for $I$ on the $m$th power of the functor $\eE := \bigoplus_{i\in I}\eE_{i}$;
\item[(KM3)]
If $k_i=\al_i^{\vee}(\mu)\geq 0$ then there is an isomorphism $\eE_{i} \eF_{i}|_{\Ccat_{\mu}}\Rightarrow
\eF_{i}\eE_{i}|_{\Ccat_{\mu}}
\oplus \operatorname{Id}_{\Ccat_{\mu}}^{\oplus k_i}$, induced by the column vector \cref{EF1}.
 If $k_{i}\leq 0$, there is an isomorphism $\eE_{i} \eF_{i}|_{\Ccat_\mu}
\oplus \operatorname{Id}_{\Ccat_\mu}^{\oplus -k_{i}}\Rightarrow
\eF_{i}\eE_{i}|_{\Ccat_\mu}$ induced by the row vector
\cref{EF2}.
\end{enumerate}
\end{theorem} 
As in \cite{brundanDegenerateHeisenberg2023,brundanHeisenbergKacMoody2020}, we will often use generating functions when working with elements of
an algebra $A$. This means that we will work with formal
Laurent series $f(z) \in A(\!(z^{-1})\!)$ in an indeterminate $z$ (or
$v$, $w$, \dots).
We write $\left[f(z)\right]_{z^r}$ for the $z^r$-coefficient of such a
series,
$\left[f(z)\right]_{z^{< 0}}$ for $\sum_{r < 0} \left[f(z)\right]_{z^r} z^r$,
$\left[f(z)\right]_{z^{\geq 0}}$ for $\sum_{r \geq 0} \left[f(z)\right]_{z^r} z^r$ (which is a polynomial),
and so on.
To give an example, suppose that $$f(z) = \sum_{r \geq 0} f_r
z^{k-r}= z^k+f_{1}z^{k-1}+\cdots \in z^k 1_A + z^{k-1} A[\![z^{-1}]\!]$$
for some $f_r \in A$. Then
we can define new elements $g_r \in A$ by declaring that
$$ 
g(z) = \sum_{r \geq 0} g_r z^{-k-r}= z^{-k}+g_{1}z^{-k-1}+\cdots \in z^{-k}1_A + z^{-k-1} A[\![z^{-1}]\!]
$$
is the inverse of the formal Laurent series $f(z)$.
In fact, setting $f_r := 0$ for $r < 0$, we have
\begin{equation}\label{psid}
g_r = \det\left(-f_{s-t+1}\right)_{s,t=1,\dots,r}.
\end{equation}
This identity is valid even if
$A$ is non-commutative providing
the determinant is interpreted
as a suitably ordered Laplace expansion.  We let
\begin{align}\label{chink}
{\color{gray}\scriptstyle\lambda}\:\anticlocki(z) &:= \sum_{r \in \Z}
\mathord{
\begin{tikzpicture}[baseline = 1.25mm]
  \draw[->,thick] (0.2,0.2) to[out=90,in=0] (0,.4);
  \draw[-,thick] (0,0.4) to[out=180,in=90] (-.2,0.2);
\draw[-,thick] (-.2,0.2) to[out=-90,in=180] (0,0);
  \draw[-,thick] (0,0) to[out=0,in=-90] (0.2,0.2);
   \node at (0.2,0.2) {$\bull$};
   \node at (.4,0.2) {$\scriptstyle{r}$};
   \node at (-0.05,-.15) {$\scriptstyle{i}$};
   \node at (-.38,0.2) {$\color{gray}\scriptstyle{\lambda}$};
\end{tikzpicture}}z^{-r-1}\in
z^{\al_i^{\vee}(\lambda)} 1_{1_\lambda}+
z^{\al_i^{\vee}(\lambda)-1} \End(1_\lambda)[\![z^{-1}]\!]
,
\\\label{chunk}
{\color{gray}\scriptstyle\lambda}\:\clocki(z)
&:=\sum_{r \in\Z}
\mathord{
\begin{tikzpicture}[baseline = 1.25mm]
  \draw[<-,thick] (0,0.4) to[out=180,in=90] (-.2,0.2);
  \draw[-,thick] (0.2,0.2) to[out=90,in=0] (0,.4);
 \draw[-,thick] (-.2,0.2) to[out=-90,in=180] (0,0);
  \draw[-,thick] (0,0) to[out=0,in=-90] (0.2,0.2);
   \node at (0,-.15) {$\scriptstyle{i}$};
   \node at (0.36,0.2) {$\color{gray}\scriptstyle{\lambda}$};
   \node at (-0.2,0.2) {$\bull$};
   \node at (-.4,0.2) {$\scriptstyle{r}$};
\end{tikzpicture}
}
z^{-r-1} \in
z^{-\al_i^{\vee}(\lambda)} 1_{1_\lambda}+
z^{-\al_i^{\vee}(\lambda)-1} \End(1_\lambda)[\![z^{-1}]\!].
\end{align}
\notation{$\anticlocki(z),\clocki(z)$}{The bubble power series defined in \crefrange{chink}{chunk}.}
Copying the notation of \cite{brundanHeisenbergKacMoody2020}, we interpret a dot with an adjacent power series $p(z,x)\in \K\llbracket z^{-1},x\rrbracket $ as the expression obtained by substituting in a dot on the adjacent strand for $x$, leaving $z$ as a formal variable.   
We can then rewrite the relations \crefrange{loop1}{bubble-normalize2} as 
\begin{align}\label{infgras}
\clocki(z) \:\anticlocki(z)\:{\scriptstyle\color{gray}\lambda} = 1_{1_\lambda},
\end{align}
\begin{align}\label{pop-curl}
\mathord{
\begin{tikzpicture}[baseline = -1mm]
	\draw[<-,thick] (0,0.6) to (0,0.3);
	\draw[-,thick] (0,0.3) to [out=-90,in=0] (-.3,-0.2);
	\draw[-,thick] (-0.3,-0.2) to [out=180,in=-90](-.5,0);
	\draw[-,thick] (-0.5,0) to [out=90,in=180](-.3,0.2);
	\draw[-,thick] (-0.3,.2) to [out=0,in=90](0,-0.3);
	\draw[-,thick] (0,-0.3) to (0,-0.5);
  \node at (-1.06,0) {$\scriptstyle{(z-x)^{-1}}$};
      \node at (-0.5,0) {$\bull$};
   \node at (0,-.6) {$\scriptstyle{i}$};
   \node at (-0.8,-0.3) {$\color{gray}\scriptstyle{\lambda}$};
\end{tikzpicture}
}&=
\left[\mathord{\begin{tikzpicture}[baseline = -1mm]
\node at (-.6,0.05) {$\anticlocki \scriptstyle (z)$};
	\draw[->,thick] (0.08,-.5) to (0.08,.6);
      \node at (.08,0.05) {$\bull$};
      \node at (.64,0.05) {$\scriptstyle{(z-x)^{-1}}$};
   \node at (0.08,-.6) {$\scriptstyle{i}$};
   \node at (-0.5,-0.3) {$\color{gray}\scriptstyle{\lambda}$};
\end{tikzpicture}
}\right]_{z^{< 0}}\!\!\!\!\!,
&\mathord{
\begin{tikzpicture}[baseline = -1mm]
	\draw[<-,thick] (0,0.6) to (0,0.3);
	\draw[-,thick] (0,0.3) to [out=-90,in=180] (.3,-0.2);
	\draw[-,thick] (0.3,-0.2) to [out=0,in=-90](.5,0);
	\draw[-,thick] (0.5,0) to [out=90,in=0](.3,0.2);
	\draw[-,thick] (0.3,.2) to [out=180,in=90](0,-0.3);
	\draw[-,thick] (0,-0.3) to (0,-0.5);
  \node at (1.06,0) {$\scriptstyle{(z-x)^{-1}}$};
      \node at (0.5,0) {$\bull$};
   \node at (0,-.6) {$\scriptstyle{i}$};
   \node at (0.5,-0.3) {$\color{gray}\scriptstyle{\lambda}$};
\end{tikzpicture}
}&=
-
\left[\mathord{
\begin{tikzpicture}[baseline = -1mm]
	\draw[->,thick] (0.08,-.5) to (0.08,.6);
  \node at (-0.46,0.05) {$\scriptstyle{(z-x)^{-1}}$};
      \node at (0.08,0.05) {$\bull$};
      \node at (.8,0.05) {$\clocki \scriptstyle (z)$};
   \node at (.08,-.6) {$\scriptstyle{i}$};
   \node at (0.5,-0.3) {$\color{gray}\scriptstyle{\lambda}$};
\end{tikzpicture}
}\right]_{z^{< 0}}\!\!\!\!\!\\
\mathord{\begin{tikzpicture}[baseline = -1mm]
	\draw[<-,thick] (0.08,-.4) to (0.08,.4);
      \node at (-0.6,0) {$\anticlocki\scriptstyle (z)$};
   \node at (.08,-.55) {$\scriptstyle{i}$};
   \node at (0.4,0) {$\color{gray}\scriptstyle{\lambda}$};
\end{tikzpicture}
}
&=
\mathord{
\begin{tikzpicture}[baseline = -1mm]
	\draw[<-,thick] (0.08,-.4) to (0.08,.4);
      \node at (0.8,0) {$\anticlocki\scriptstyle(z)$};
   \node at (.08,-.55) {$\scriptstyle{i}$};
   \node at (0.5,-0.3) {$\color{gray}\scriptstyle{\lambda}$};
      \node at (.08,0) {$\bull$};
      \node at (-.5,0.05) {$\scriptstyle{(z-x)^{-2}}$};
\end{tikzpicture}
},&
\mathord{
\begin{tikzpicture}[baseline = -1mm]
	\draw[<-,thick] (0.08,-.4) to (0.08,.4);
      \node at (0.8,0) {$\clocki\scriptstyle(z)$};
   \node at (.08,-.55) {$\scriptstyle{i}$};
   \node at (0.5,-0.3) {$\color{gray}\scriptstyle{\lambda}$};
\end{tikzpicture}
}
&=
\mathord{
\begin{tikzpicture}[baseline = -1mm]
	\draw[<-,thick] (0.08,-.4) to (0.08,.4);
      \node at (-0.6,0) {$\clocki\scriptstyle(z)$};
      \node at (.08,0) {$\bull$};
      \node at (.7,0.05) {$\scriptstyle{(z-x)^{-2}}$};
   \node at (.08,-.55) {$\scriptstyle{i}$};
   \node at (0.4,-.3) {$\color{gray}\scriptstyle{\lambda}$};
\end{tikzpicture}
}\label{bubslides1}\\
\shortintertext{Since we are working with arbitrary $Q_{ij}$, we must be careful when rewriting the bubble slides in the case $i\neq j$, but we can rewrite \cite[Prop. 2.8]{Webmerged} in the following form:} 
\mathord{\begin{tikzpicture}[baseline = -1mm]
	\draw[<-,thick] (0.08,-.4) to (0.08,.4);
      \node at (-0.6,0) {$\anticlockj\scriptstyle (z)$};
   \node at (.08,-.55) {$\scriptstyle{i}$};
   \node at (0.4,0) {$\color{gray}\scriptstyle{\lambda}$};
\end{tikzpicture}
}
&=
\mathord{
\begin{tikzpicture}[baseline = -1mm]
	\draw[<-,thick] (0.08,-.4) to (0.08,.4);
      \node at (0.8,0) {$\anticlockj\scriptstyle(z)$};
   \node at (.08,-.55) {$\scriptstyle{i}$};
   \node at (0.5,-0.3) {$\color{gray}\scriptstyle{\lambda}$};
      \node at (.08,0) {$\bull$};
      \node at (-.85,0.05) {$\scriptstyle{t_{ij}Q_{ji}(z,x)}$};
\end{tikzpicture}
},&
\mathord{
\begin{tikzpicture}[baseline = -1mm]
	\draw[<-,thick] (0.08,-.4) to (0.08,.4);
      \node at (0.8,0) {$\clockj\scriptstyle(z)$};
   \node at (.08,-.55) {$\scriptstyle{i}$};
   \node at (0.5,-0.3) {$\color{gray}\scriptstyle{\lambda}$};
\end{tikzpicture}
}
&=
\mathord{
\begin{tikzpicture}[baseline = -1mm]
	\draw[<-,thick] (0.08,-.4) to (0.08,.4);
      \node at (-0.6,0) {$\clockj\scriptstyle(z)$};
      \node at (.08,0) {$\bull$};
      \node at (1.05,0.05) {$\scriptstyle{t_{ij}Q_{ji}(z,x)}$};
   \node at (.08,-.55) {$\scriptstyle{i}$};
   \node at (0.4,-.3) {$\color{gray}\scriptstyle{\lambda}$};
\end{tikzpicture}
}.\label{bubslides2}
\end{align}

\begin{lemma}[\mbox{\cite[Lem. 3.5]{brundanHeisenbergKacMoody2020}}]\label{sure2lemma}
For a polynomial $p(z) \in \K[z]$, we have that:
\begin{align}\label{sure1a}
\mathord{
\begin{tikzpicture}[baseline = -1.5mm]
 	\draw[->,thick] (0.08,-.4) to (0.08,.4);
   \node at (.08,-.55) {$\scriptstyle{i}$};
     \node at (0.08,0) {$\bull$};
    \node at (-0.33,0) {$\scriptstyle p(x)$};
   \node at (0.35,0) {$\color{gray}\scriptstyle{\lambda}$};
\end{tikzpicture}
}
&=
\left[
\mathord{
\begin{tikzpicture}[baseline = -1.5mm]
 	\draw[->,thick] (0.08,-.4) to (0.08,.4);
   \node at (.08,-.55) {$\scriptstyle{i}$};
     \node at (0.08,0) {$\bull$};
    \node at (-0.52,0) {$\scriptstyle (z-x)^{-1}$};
   \node at (0.4,.3) {$\color{gray}\scriptstyle{\lambda}$};
\end{tikzpicture}
}
\!\!\!\!\!p(z)
\right]_{z^{-1}},
&
\mathord{
\begin{tikzpicture}[baseline = 0mm]
 	\draw[<-,thick] (0.08,-.4) to (0.08,.4);
   \node at (.08,.55) {$\scriptstyle{i}$};
     \node at (0.08,0.04) {$\bull$};
    \node at (-0.33,0.04) {$\scriptstyle p(x)$};
   \node at (0.35,0.05) {$\color{gray}\scriptstyle{\lambda}$};
\end{tikzpicture}
}
&=
\left[
\mathord{
\begin{tikzpicture}[baseline = 0mm]
 	\draw[<-,thick] (0.08,-.4) to (0.08,.4);
   \node at (.08,.55) {$\scriptstyle{i}$};
     \node at (0.08,0.04) {$\bull$};
    \node at (-0.5,0.04) {$\scriptstyle (z-x)^{-1}$};
   \node at (0.4,-.3) {$\color{gray}\scriptstyle{\lambda}$};
\end{tikzpicture}
}
\!\!\!\!\!p(z)
\right]_{z^{-1}},\\\label{sure2a}
\!\mathord{
\begin{tikzpicture}[baseline = 1.25mm]
  \draw[<-,thick] (0,0.4) to[out=180,in=90] (-.2,0.2);
  \draw[-,thick] (0.2,0.2) to[out=90,in=0] (0,.4);
 \draw[-,thick] (-.2,0.2) to[out=-90,in=180] (0,0);
  \draw[-,thick] (0,0) to[out=0,in=-90] (0.2,0.2);
   \node at (0.2,0.2) {$\bull$};
  \node at (0.55,0.2) {$\scriptstyle{p(x)}$};
   \node at (0,-.15) {$\scriptstyle{i}$};
   \node at (-0.4,.2) {$\color{gray}\scriptstyle{\lambda}$};
\end{tikzpicture}
}
&=
\left[{\scriptstyle{\color{gray}\lambda\:\:}\clocki(z)}\:p(z)\right]_{z^{-1}},&
\mathord{
\begin{tikzpicture}[baseline = 1.25mm]
  \draw[-,thick] (0,0.4) to[out=180,in=90] (-.2,0.2);
  \draw[->,thick] (0.2,0.2) to[out=90,in=0] (0,.4);
 \draw[-,thick] (-.2,0.2) to[out=-90,in=180] (0,0);
  \draw[-,thick] (0,0) to[out=0,in=-90] (0.2,0.2);
   \node at (0.2,0.2) {$\bull$};
  \node at (0.55,0.2) {$\scriptstyle{p(x)}$};
   \node at (0,-.15) {$\scriptstyle{i}$};
   \node at (-0.4,.2) {$\color{gray}\scriptstyle{\lambda}$};
\end{tikzpicture}
}
&=
\left[{\scriptstyle{\color{gray}\lambda\:\:}\anticlocki(z)}\:p(z)\right]_{z^{-1}},\\
\label{sure3a}
\mathord{
\begin{tikzpicture}[baseline = -1.5mm]
	\draw[<-,thick] (0,0.6) to (0,0.3);
	\draw[-,thick] (0,0.3) to [out=-90,in=180] (.3,-0.2);
	\draw[-,thick] (0.3,-0.2) to [out=0,in=-90](.5,0);
	\draw[-,thick] (0.5,0) to [out=90,in=0](.3,0.2);
	\draw[-,thick] (0.3,.2) to [out=180,in=90](0,-0.3);
	\draw[-,thick] (0,-0.3) to (0,-0.5);
   \node at (0,-.65) {$\scriptstyle{i}$};
   \node at (0.3,-.4) {$\color{gray}\scriptstyle{\lambda}$};
  \node at (.87,0) {$\scriptstyle{p(x)}$};
      \node at (0.5,0) {$\bull$};
\end{tikzpicture}
}&=
-\left[\!\mathord{
\begin{tikzpicture}[baseline = -1.5mm]
	\draw[->,thick] (0.08,-.5) to (0.08,.6);
   \node at (0.08,-.65) {$\scriptstyle{i}$};
   \node at (0.4,-.4) {$\color{gray}\scriptstyle{\lambda}$};
  \node at (-0.46,0.05) {$\scriptstyle{(z-x)^{-1}}$};
      \node at (0.08,0.05) {$\bull$};
      \node at (.7,0.05) {$\clocki \scriptstyle (z)$};
\end{tikzpicture}
}p(z)\right]_{z^{-1}}\!\!\!\!\!,
&
\mathord{
\begin{tikzpicture}[baseline = 0mm]
	\draw[-,thick] (0,0.5) to (0,0.3);
	\draw[-,thick] (0,0.3) to [out=-90,in=180] (.3,-0.2);
	\draw[-,thick] (0.3,-0.2) to [out=0,in=-90](.5,0);
	\draw[-,thick] (0.5,0) to [out=90,in=0](.3,0.2);
	\draw[-,thick] (0.3,.2) to [out=180,in=90](0,-0.3);
	\draw[->,thick] (0,-0.3) to (0,-0.6);
  \node at (.87,0) {$\scriptstyle{p(x)}$};
      \node at (0.5,0) {$\bull$};
   \node at (0,.65) {$\scriptstyle{i}$};
   \node at (0.4,-.4) {$\color{gray}\scriptstyle{\lambda}$};
\end{tikzpicture}
}&=
\left[\!\mathord{
\begin{tikzpicture}[baseline = 0mm]
	\draw[<-,thick] (0.08,-.6) to (0.08,.5);
  \node at (-0.46,0.05) {$\scriptstyle{(z-x)^{-1}}$};
      \node at (0.08,0.05) {$\bull$};
      \node at (.7,0.05) {$\anticlocki \scriptstyle (z)$};
   \node at (0.08,.65) {$\scriptstyle{i}$};
   \node at (0.4,-.4) {$\color{gray}\scriptstyle{\lambda}$};
\end{tikzpicture}
}p(z)\right]_{z^{-1}}\!\!\!\!\!\!.
\end{align}
\end{lemma}

Since the definition of $\tU$ uses the set of weights $\wela$, it depends on the choice of realization, not just the Cartan matrix $C$.  However, we can simplify this dependence by noting that the relations between 2-morphisms depend only on the scalars $\al_i^{\vee}$.  Thus, if $\la \in \mathfrak{h}^*$ is perpendicular to all $\al_i^{\vee}$, then there is an equivalence of 2-categories that acts on 0-morphisms by $\mu \mapsto \mu+\la$, and on 1-morphisms and 2-morphisms by the obvious identifications. 

Let $\mathfrak{g}'=[\mathfrak{g}, \mathfrak{g}]\subset \mathfrak{g}$ be the derived subalgebra.  The Cartan of $\mathfrak{g}'$ is the subspace $\mathfrak{h}'\subset \mathfrak{h}$ spanned by the elements $H_i=\al_i^{\vee}$.  We have a functor $\tU(\mathfrak{g})\to \tU(\mathfrak{g}')$ given by pulling back weights, together with the obvious identifications on 1-morphisms and 2-morphisms (note the slightly surprising direction of this functor).

This is not full on 1-morphisms: if the Cartan matrix is degenerate, then we could have a sum of roots which is perpendicular to all $\al_i^{\vee}$, but non-zero in $\mathfrak{h}^*$.  However, this functor is fully faithful on categories of 1-morphisms; in particular, two monomials $\eE_{\Bi}, \eE_{\Bj}$ which have different weights for $\mathfrak{h}$ might have the same weight under $\mathfrak{h}'$, but there are still no 2-morphisms between them.  In particular, the question of non-degeneracy of a categorification is unchanged when passing from $\mathfrak{g}$ to $\mathfrak{g}'$, and so we may assume $\mathfrak{g}=\mathfrak{g}'$.

Similarly, if there are linear relations between the $\al_i^{\vee}$ in $\mathfrak{g}$, then we can write $\mathfrak{g}$ as a quotient of its universal central extension $\mathfrak{g}''$.  In this case, we have a 2-functor $\tU(\mathfrak{g})\to \tU(\mathfrak{g}'')$ (again in a surprising direction) which acts on weights by pullback, together with the obvious identification of 1-morphisms and 2-morphisms.  

This functor is fully faithful on 1-morphisms and 2-morphisms, so if our categorification for $\mathfrak{g}''$ is non-degenerate, then the categorification for $\mathfrak{g}$ is also non-degenerate.

Some of our constructions will be a little easier if we use a realization of this type.  Thus it will be useful later to summarize this discussion:
\begin{lemma}\label{lem:different-realizations}
	If we check that non-degeneracy holds when $\fg$ is the universal derived Lie algebra for a Cartan matrix (where simple coroots are linearly independent and span $\mathfrak{h}$), then it holds for all other realizations of this Cartan matrix.
\end{lemma}
Thus, for \cref{sec:proof,sec:spectral}, we assume that $\fg$ is the universal derived Kac-Moody algebra attached to our chosen Cartan matrix.

\section{The proof of non-degeneracy}
\label{sec:proof}

\subsection{The \texorpdfstring{$\mathfrak{sl}_2$}{sl_2} coproduct functor}
\label{sec:coproduct}
In this section, we take $\fg=\mathfrak{sl}_2$ and work with the 2-category $\tUdot=\tU(\mathfrak{sl}_2)$.  We denote the simple root by $\alpha$ and the simple coroot by $\alpha^{\vee}$.  We also write $n$ for the weight $n\alpha/2$.
As in \cite{brundanDegenerateHeisenberg2023}, we can define a functor from $\tUdot$ to the outer tensor product of two copies of itself, with a morphism inverted.  We will use the convention from {\it loc. cit.} that the 1-morphisms $\eE,\eF$ in our first factor are colored blue and those in the second factor are colored red.  This might seem a peculiar idea if one is just now encountering it for the first time, but it is quite powerful when employed in non-degeneracy proofs.  The reasoning behind the definition of this map is explained in more detail below \cref{internal-through-red-blue}.

\newcommand{\tUrb}{\tU_{\blue{\bullet} \red{\bullet}}}
\newcommand{\bUrb}{\overline{\tU}_{\blue{\bullet} \red{\bullet}}}
For technical reasons, it is more convenient to avoid using the tensor product of 2-categories and instead consider the 2-category $\tU$ for the set $I=\{\blue{\bullet},\red{\bullet}\}$ with Cartan matrix $\big [\begin{smallmatrix}
		2 & 0\\
		0 & 2
	\end{smallmatrix}\big ]$.  We denote this 2-category by $\tUrb$.  

\notation{$\tU_{\blue{\bullet} \red{\bullet}}$}{The 2-category $\tU$ for the set $I=\{\blue{\bullet},\red{\bullet}\}$ with Cartan matrix $\big [\begin{smallmatrix}
		2 & 0\\
		0 & 2
	\end{smallmatrix}\big ]$.}

We will use blue and red to distinguish the two nodes of the Dynkin diagram, so for example we will write $(\blue{k},\red{n})$ for objects, and $\blue{\eE}=\eE_{\blue{\bullet}}$  and $\red{\eE}=\eE_{\red{\bullet}}$.

\begin{remark}\label{rem:biequivalence}
We let $\blue{\tUdot} \;{\otimes}\; \red{\tUdot}$ be the tensor product of additive 2-categories.  This is the strict 2-category whose objects are pairs of objects from $\blue{\tUdot}$ and $\red{\tUdot}$, where 1-morphisms $\Hom_{\blue{\tUdot} \;{\otimes}\; \red{\tUdot}}((\blue{k},\red{\ell}),(\blue{m},\red{n}))$ are given by the strict tensor product of the Hom categories $\Hom_{\blue{\tUdot}}(\blue{k},\blue{m})$ and $\Hom_{\red{\tUdot}}(\red{\ell},\red{n})$.  
There is a 2-functor $\tUrb\to \blue{\tUdot} \;{\otimes}\; \red{\tUdot}$ sending $(\blue{k},\red{n})\mapsto (\blue{k},\red{n})$ and 
 \[\blue{\eE}\mapsto \blue{\eE}\otimes \red{\Id}, \qquad \red{\eE}\mapsto \blue{\Id}\otimes \red{\eE},\] 
 \[\blue{\eF}\mapsto \blue{\eF}\otimes \red{\Id}, \qquad \red{\eF}\mapsto \blue{\Id}\otimes \red{\eF}.\] 
On 2-morphisms that are a single color, we send them to the corresponding 2-morphisms in the appropriate factor, tensored with the identity in the other factor.  The remaining generating 2-morphism is a crossing of a red and blue strand with downward orientation, which we send to the identity 2-morphism on $\blue{\eE}\otimes \red{\eE}$.  This manifestly preserves all the relations involving a single color and those involving both colors all follow from the interchange law in the tensor product of 2-categories.  

One can construct a pseudofunctor (that is, a functor compatible with composition of 1-morphisms only up to coherent isomorphism) $\blue{\tUdot} \;{\otimes}\; \red{\tUdot}\to \tUrb$ demonstrating that our original functor is a biequivalence (but not an equivalence of strict 2-categories) by sending $\blue{u}\otimes \red{v}$ to the composition $\blue{u}\red{v}$, and all 2-morphisms to the horizontal composition of the corresponding 2-morphisms in $\tUrb$.  The fact that we had to choose between $\blue{u}\red{v}$ and $\red{v}\blue{u}$ exactly reflects the failure to be a strict 2-functor---these objects are isomorphic but not equal in $\tUrb$.  Since we are working with strict 2-categories, it is preferable to work with $\tUrb$ rather than the tensor product, but the reader should keep this biequivalence in mind.
\end{remark}

Note that this is closely related to \cite{brundanDegenerateHeisenberg2023}.  There we consider monoidal categories, which can be thought of as 2-categories with a single object.  Thus the tensor product of 2-categories induces a product operation on monoidal categories, which is exactly the notion of {\bf symmetric product} discussed in \cite[Def. 2.1]{brundanDegenerateHeisenberg2023}.

Let $\bUrb$ be the localization of $\tUrb$ by the 2-morphism
$\mathord{
\begin{tikzpicture}[baseline = -1mm]
 	\draw[->,thin,red] (0.18,-.4) to (0.18,.4);
	\draw[->,thin,blue] (-0.38,-.4) to (-0.38,.4);
     \node at (0.18,0) {$\red{\bullet}$};
\end{tikzpicture}
}-
\mathord{
\begin{tikzpicture}[baseline = -1mm]
 	\draw[->,thin,red] (0.18,-.4) to (0.18,.4);
	\draw[->,thin,blue] (-0.38,-.4) to (-0.38,.4);
     \node at (-0.38,0) {$\blue{\bullet}$};
\end{tikzpicture}}
\,$.
This means that we adjoin a two-sided inverse to this 2-morphism, which
we denote as a dumbbell
\begin{equation}\label{dumb}
\mathord{
\begin{tikzpicture}[baseline = -1mm]
 	\draw[->,thin,red] (0.18,-.4) to (0.18,.4);
	\draw[->,thin,blue] (-0.38,-.4) to (-0.38,.4);
	\draw[-] (-0.38,0.01) to (0.18,0.01);
     \node at (0.18,0) {$\dt$};
     \node at (-0.38,0) {$\dt$};
\end{tikzpicture}
}:=
\left(\,\mathord{
\begin{tikzpicture}[baseline = -1mm]
 	\draw[->,thin,red] (0.18,-.4) to (0.18,.4);
	\draw[->,thin,blue] (-0.38,-.4) to (-0.38,.4);
     \node at (0.18,0) {$\red{\bullet}$};
\end{tikzpicture}
}-
\mathord{
\begin{tikzpicture}[baseline = -1mm]
 	\draw[->,thin,red] (0.18,-.4) to (0.18,.4);
	\draw[->,thin,blue] (-0.38,-.4) to (-0.38,.4);
     \node at (-0.38,0) {$\blue{\bullet}$};
\end{tikzpicture}}\,\right)^{-1}.
\end{equation}

\begin{theorem}\label{th:coproduct}
We can define a functor $\Delta$ from $\tUdot$ to $\bUrb$ sending:
\begin{equation}
	n \mapsto \bigoplus_{k+\ell=n}(\blue{k},\red{\ell})\qquad \qquad \eE\mapsto \blue{\eE}\oplus \red{\eE}\qquad \qquad \eF\mapsto \blue{\eF}\oplus \red{\eF}
\end{equation}
which we should interpret as sending $\eE\colon n\to n+2$ to a matrix of 1-morphisms with the entry in column $(k,\ell)$ and row $(k+2,\ell)$ equal to $\blue{\eE}$, and the entry in column $(k,\ell)$ and row $(k,\ell+2)$ equal to $\red{\eE}$.  On 2-morphisms, this acts by 
\begin{equation}
\label{com0old}
\mathord{
\begin{tikzpicture}[baseline = -.6mm]
	\draw[->,thin] (0.08,-.3) to (0.08,.3);
      \node at (0.08,0) {$\bullet$};
\end{tikzpicture}
}
\mapsto
\mathord{
\begin{tikzpicture}[baseline = -.6mm]
	\draw[->,thin,blue] (0.08,-.3) to (0.08,.3);
      \node at (0.08,0) {$\color{blue}\bullet$};
\end{tikzpicture}
}
+
\mathord{
\begin{tikzpicture}[baseline = -.6mm]
	\draw[->,thin,red] (0.08,-.3) to (0.08,.3);
      \node at (0.08,0) {$\color{red}\bullet$};
\end{tikzpicture}
}\:,
\qquad \qquad 
\mathord{
\begin{tikzpicture}[baseline = 1mm]
	\draw[<-,thin] (0.4,0) to[out=90, in=0] (0.1,0.4);
	\draw[-,thin] (0.1,0.4) to[out = 180, in = 90] (-0.2,0);
\end{tikzpicture}
}\mapsto
\mathord{
\begin{tikzpicture}[baseline = 1mm]
	\draw[<-,thin,blue] (0.4,0) to[out=90, in=0] (0.1,0.4);
	\draw[-,thin,blue] (0.1,0.4) to[out = 180, in = 90] (-0.2,0);
\end{tikzpicture}
}+
\mathord{
\begin{tikzpicture}[baseline = 1mm]
	\draw[<-,thin,red] (0.4,0) to[out=90, in=0] (0.1,0.4);
	\draw[-,thin,red] (0.1,0.4) to[out = 180, in = 90] (-0.2,0);
\end{tikzpicture}
}\:,
\qquad\qquad
\mathord{
\begin{tikzpicture}[baseline = 1mm]
	\draw[<-,thin] (0.4,0.4) to[out=-90, in=0] (0.1,0);
	\draw[-,thin] (0.1,0) to[out = 180, in = -90] (-0.2,0.4);
\end{tikzpicture}
}
\mapsto
\mathord{
\begin{tikzpicture}[baseline = 1mm]
	\draw[<-,thin,blue] (0.4,0.4) to[out=-90, in=0] (0.1,0);
	\draw[-,thin,blue] (0.1,0) to[out = 180, in = -90] (-0.2,0.4);
\end{tikzpicture}
}+\mathord{
\begin{tikzpicture}[baseline = 1mm]
	\draw[<-,thin,red] (0.4,0.4) to[out=-90, in=0] (0.1,0);
	\draw[-,thin,red] (0.1,0) to[out = 180, in = -90] (-0.2,0.4);
\end{tikzpicture}
}\:,
\end{equation}
\begin{equation}
\label{com0older}
\mathord{
\begin{tikzpicture}[baseline = -.6mm]
	\draw[->,thin] (0.28,-.3) to (-0.28,.3);
	\draw[thin,->] (-0.28,-.3) to (0.28,.3);
\end{tikzpicture}
}\mapsto
\mathord{
\begin{tikzpicture}[baseline = -.6mm]
	\draw[->,thin,blue] (0.28,-.3) to (-0.28,.3);
	\draw[thin,->,blue] (-0.28,-.3) to (0.28,.3);
\end{tikzpicture}
}+
\mathord{
\begin{tikzpicture}[baseline = -.6mm]
	\draw[->,thin,red] (0.28,-.3) to (-0.28,.3);
	\draw[thin,->,red] (-0.28,-.3) to (0.28,.3);
\end{tikzpicture}
}
-
\mathord{
\begin{tikzpicture}[baseline = -.6mm]
	\draw[->,thin,red] (0.28,-.3) to (-0.28,.3);
	\draw[thin,->,blue] (-0.28,-.3) to (0.28,.3);
	\draw[-] (-0.15,-.16) to (0.15,-.16);
     \node at (0.15,-.17) {$\dt$};
     \node at (-0.15,-.17) {$\dt$};
\end{tikzpicture}
}
+
\mathord{
\begin{tikzpicture}[baseline = -.6mm]
	\draw[->,thin,blue] (0.28,-.3) to (-0.28,.3);
	\draw[thin,->,red] (-0.28,-.3) to (0.28,.3);
	\draw[-] (-0.15,-.16) to (0.15,-.16);
     \node at (0.15,-.17) {$\dt$};
     \node at (-0.15,-.17) {$\dt$};
\end{tikzpicture}
}
-\mathord{
\begin{tikzpicture}[baseline =-.6mm]
 	\draw[->,thin,red] (0.2,-.3) to (0.2,.3);
	\draw[->,thin,blue] (-0.2,-.3) to (-0.2,.3);
	\draw[-] (-0.2,0.01) to (0.2,0.01);
     \node at (0.2,0.0) {$\dt$};
     \node at (-0.2,0.0) {$\dt$};
\end{tikzpicture}
}+
\mathord{
\begin{tikzpicture}[baseline = -0.6mm]
 	\draw[->,thin,blue] (0.2,-.3) to (0.2,.3);
	\draw[->,thin,red] (-0.2,-.3) to (-0.2,.3);
	\draw[-] (-0.2,0.01) to (0.2,0.01);
     \node at (0.2,0.0) {$\dt$};
     \node at (-0.2,0.0) {$\dt$};
\end{tikzpicture}
}\:.
\end{equation}
\end{theorem}
This theorem is manifestly analogous to \cite[Th. 5.4]{brundanDegenerateHeisenberg2023};  the proofs are also very similar, to the point that one could use a verbatim copy of the proof in {\it loc. cit.}  with only minor changes in some of the formulas, which follow the natural template for transposing results for degenerate affine Hecke algebras to nilHecke algebras.  This is one of the points where results appearing after the original draft of this paper have given us a template to simplify the approach taken here.

Even having this template to follow, this theorem is still a delicate direct computation.  Due to its length, we give the proof of this result on page \pageref{proof-th:coproduct}.  Note that we have not described the image of the left cups and caps, but these are determined by the functor $\tUdot\to (\tUdot)_R$ given in \cref{lem:BrundanRouquier}.  These images are explicitly given in \cref{com4}.

\subsection{Completions}
In \cref{sec:coproduct}, we worked with localizations of our categories, but for certain constructions we wish to do below, we require something a bit stronger.  Since the same considerations will apply to all Kac-Moody types, we return to the general setting of $\tU=\tU(\fg)$ for an arbitrary Kac-Moody algebra $\fg$.

First, we consider representations of $\tU$ where, on any object, only finitely many $\eE_i$ are non-zero.  Furthermore, we assume that the dots are topologically nilpotent, that is, the limit $\lim_{n\to\infty}  (\mathord{
\begin{tikzpicture}[baseline = -.6mm]
	\draw[->] (0.08,-.3) to (0.08,.3);
      \node at (0.08,0) {$\bullet$};
\end{tikzpicture}
})^n=0$; in the resulting completion, any power series in the dots gives a well-defined morphism.

These properties can be captured on the level of 2-categories by using completion. 

\begin{definition}\label{def:grading-topology}
  The graded vector space $\Hom_{\tU}(u,v)$ has a natural {\bf grading topology} where the set of elements $\Hom^{\geq k}_{\tU}(u,v)$ of degree $\geq k$ is an open neighborhood of zero for all $k$, and these form a basis of such neighborhoods.   
\end{definition}

  Of course, any 2-morphism of positive degree in $\Hom_{\tU}(u,u)$ is topologically nilpotent in this topology.  We need to modify this topology slightly to account for the possibility that $I$ is infinite.  
  
\begin{definition}\label{def:cofinite-topology}
    For a finite set $I_0\subset I$, we define a subspace $\Hom_{\tU}^{I_0,r}(u,v)$ by the 2-morphisms that factor through a direct sum of monomials of length $\leq r$ which contain one of the symbols $\eE_{\pm i}$ with $i\notin I_0$.  
    
    The {\bf cofinite topology} on the vector space $\Hom_{\tU}(u,v)$ is the coarsest topology invariant under translation such that the sets $\Hom_{\tU}^{I_0,r}(u,v)$ are neighborhoods of the zero morphism.
  \end{definition}
   Note that $\Hom_{\tU}^{I_0,r}(u,v)$ is a linear subspace, since sums of two morphisms that factor through such monomials also factor through a similar direct sum.  If $I_0\subset I_1$, then $\Hom_{\tU}^{I_0,r}(u,v)\subset \Hom_{\tU}^{I_1,r}(u,v)$, and if $r\leq s$, then $\Hom_{\tU}^{I_0,r}(u,v)\subset \Hom_{\tU}^{I_0,s}(u,v)$.  In particular, for any $I_0,I_0'$ and $r,r'$, we have \[\Hom_{\tU}^{I_0,r}(u,v)\cap \Hom_{\tU}^{I_0',r'}(u,v)\supset \Hom_{\tU}^{I_0\cup I_0',\min(r,r')}(u,v).\]  Therefore, the sets $\Hom_{\tU}^{I_0,r}(u,v)$ form a basis of neighborhoods of zero in the cofinite topology.  Note that if $I$ is finite, then the cofinite topology is just the discrete topology.

Given a representation $\mathcal{C}$ of $\tU$ satisfying $(\dagger)$, we equip the morphism spaces $\Hom(uM,vM)$ for $M\in \Ob(\mathcal{C})$ with the discrete topology.  We are interested in continuity of the map $\Hom_{\tU}(u,v)\to \Hom(uM,vM)$.  Recall that by Lagrange's theorem, a coset of a subgroup of a topological group is open if and only if it contains an open set. Thus, a homomorphism of topological groups with discrete target is continuous if and only if its kernel contains a neighborhood of zero.  

Recall that for any object $X\in \Ob(\mathcal{C})$ in a representation $\mathcal{C}$, we assume that $(\dagger)$ holds, that is, the set of vertices such that $\eE_{\pm i}X$ are not both zero is finite.

  Then we have the following:
\begin{lemma}\label{lem:continuous}
For a representation $\mathcal{C}$ of $\tU$ satisfying $(\dagger)$, the map $\Hom_{\tU}(u,v)\to \Hom(uX,vX)$ is continuous for all objects $X$ when $\Hom_{\tU}(u,v)$ is given the cofinite topology and $\Hom(uX,vX)$ is given the discrete topology.
\end{lemma}
\begin{proof}
  As discussed above, we need to show that the kernel of this map contains a neighborhood of zero in the cofinite topology. Equivalently, we need some $\Hom_{\tU}^{I_0,r}(u,v)$ such that all morphisms in this set act trivially on $X$. 

	 Applying $(\dagger)$ repeatedly, we find that there are only finitely many monomials of length $\leq r$ such that $\eE_{\Bi}X\neq 0$, and thus the set $I_0$ of indices appearing in these monomials is finite.  Accordingly, any monomial of length $\leq r$ which contains one of the symbols $\eE_{i}$ with $i\notin I_0$ acts trivially.
\end{proof}

Now, consider the join of the grading and cofinite topologies, which we call the {\bf GC topology}; this topology is generated by the spans $\Hom^{\geq k}_{\tU}(u,v)+ \Hom_{\tU}^{I_0,r}(u,v)$ for all $k,r\in \Z_{\geq 0}$ and finite sets $I_0\subset I$.
\notation{GC topology}{The topology generated by spans $\Hom^{\geq k}_{\tU}(u,v)+ \Hom_{\tU}^{I_0,r}(u,v)$, the 2-morphisms of degree $\geq k$ and factoring through monomials of length $\leq r$ containing $\eE_{\pm i}$ with $i\notin I_0$.}
\notation{$\htU$}{The completion of the 2-morphisms $\tU$ by the GC topology.}
\begin{definition}
	Let $\htU$ be the 2-category enriched on the level of 2-morphisms by topological vector spaces, where: 
\begin{itemize}
\item The object sets are unchanged from $\tU$.
	\item The 1-morphism sets are sums of (potentially infinite) formal Cartesian products
of monomials $\eE_{\Bi}$ in 1-morphisms where:
 \begin{itemize}
     \item there is an upper bound $r$ on the length of all monomials appearing, and
\item for any finite subset $I_0\subset I$, there are only finitely many monomials $\eE_{\Bi}$ in which only functors $\eE_{\pm i}$ for $i\in I_0$ appear.
\end{itemize} 
	\item The 2-morphism sets $\Hom_{\htU}(\prod u_\alpha,\prod v_{\beta})$ are the completion of the Cartesian product $\prod_{\alpha, \beta}\Hom_{\tU}(u_\alpha,v_{\beta})$ with respect to the GC topology (i.e., they are matrices of 2-morphisms with no finiteness conditions on rows or columns).
\end{itemize}
\end{definition} 

It is not manifest that the composition in this category is well-defined, but we confirm it below:
\begin{enumerate}
	\item For 1-morphisms, note that for each finite set $I_0$, there are only finitely many monomials in which $\eE_{\pm i}$ for $i\in I_0$ appears in both $u$ and $v$, so the same is true of their composition.

	\item For 2-morphisms, we have to be concerned about the fact that we are composing matrices where there are potentially infinitely many non-zero entries in both rows and columns.  That is, when we consider the composition \[\Hom_{\htU}(\prod u_\alpha,\prod v_{\beta}) \times \Hom_{\htU}(\prod v_{\beta}, \prod w_\gamma)\to  \Hom_{\htU}(\prod u_\alpha,\prod w_\gamma),\] the matrix entry for $u_\alpha \to w_\gamma $ is potentially an infinite sum of morphisms factoring through $v_{\beta}$.  However, for any $I_0$, all but finitely many of these terms lie in $ \Hom_{\tU}^{I_0,r}(u,v) $ where $r$ is the upper bound on the length of monomials in $\prod v_{\beta}$, so this composition is well-defined in the completion.  
\end{enumerate}

This might seem a little strange, but the simplest way to think about this completion is that it acts on any graded $\tU$-module category satisfying $(\dagger)$ where the grading on $\Hom(uM,vM)$ for an object $M$ is bounded above. In fact, this action is continuous when the morphisms in this module category are given the discrete topology. In this setting, we can make sense of $\pwr_a(\mathord{
\begin{tikzpicture}[baseline = -2]
	\draw[->,thick] (0.08,-.15) to (0.08,.3);
      \node at (0.08,0.05) {$\bull$};
   \node at (0.08,-.3) {$\scriptstyle{i}$};
\end{tikzpicture}
})$ for any power series $\pwr_a(z)=a+a^{(1)}z+a^{(2)}z^2+\cdots$ with $a\neq 0$ and $a^{(1)}\neq 0$.  Note that the power series 
\notation{$\pwr_a$}{A power series $\pwr_a(z)=a+a^{(1)}z+a^{(2)}z^2+\cdots$ with $a\neq 0$ and $a^{(1)}\neq 0$ used in the definition of the functor of \cref{lem:many-sl2}.}
\[q_a\Big(\begin{tikzpicture}[baseline = -.6mm]
	\draw[->,thin] (0.08,-.3) to (0.08,.3);
	\draw[->,thin] (-0.08,-.3) to (-0.08,.3);
      \node at (-0.08,0) {$\bullet$};
\end{tikzpicture},\begin{tikzpicture}[baseline = -.6mm]
	\draw[->,thin] (-0.08,-.3) to (-0.08,.3);
	\draw[->,thin] (0.08,-.3) to (0.08,.3);
      \node at (0.08,0) {$\bullet$};
\end{tikzpicture}\Big)=\frac{\begin{tikzpicture}[baseline = -.6mm]
	\draw[->,thin] (0.08,-.3) to (0.08,.3);
	\draw[->,thin] (-0.08,-.3) to (-0.08,.3);
      \node at (-0.08,0) {$\bullet$};
\end{tikzpicture}-\begin{tikzpicture}[baseline = -.6mm]
	\draw[->,thin] (-0.08,-.3) to (-0.08,.3);
	\draw[->,thin] (0.08,-.3) to (0.08,.3);
      \node at (0.08,0) {$\bullet$};
\end{tikzpicture}}{\pwr_a\Big(\!\!\begin{tikzpicture}[baseline = -.6mm]
	\draw[->,thin] (0.08,-.3) to (0.08,.3);
	\draw[->,thin] (-0.08,-.3) to (-0.08,.3);
      \node at (-0.08,0) {$\bullet$};
\end{tikzpicture}\Big)-\pwr_a\Big(\begin{tikzpicture}[baseline = -.6mm]
	\draw[->,thin] (-0.08,-.3) to (-0.08,.3);
	\draw[->,thin] (0.08,-.3) to (0.08,.3);
      \node at (0.08,0) {$\bullet$};
\end{tikzpicture}\!\!\Big)}\]
is well-defined and a unit in the power series ring, since its reciprocal has constant term $a^{(1)}$. 

\subsection{Iterated coproducts} 
Now we return to the case $\fg=\mathfrak{sl}_2$.  We expand to an arbitrary set $A$, viewed as a Dynkin diagram with no edges, so $C=2I$.  We have an associated 2-category $\tU_A$ corresponding to the Kac-Moody algebra $\mathfrak{sl}_2^{\oplus A}$, with objects given by $A$-tuples $(k_a)_{a\in A}$ of integers, and 1-morphisms generated by $\eE_a$ and $\eF_a$ for $a\in A$.\notation{$\tU_A$}{The Kac-Moody 2-category for $\mathfrak{sl}_2^{\oplus A}$,  $A$ a Dynkin diagram with no edges.}  As with $\tUrb$, we could interpret this, up to biequivalence, as a tensor product of copies of $\tUdot$ indexed by $A$ (discussed in more detail in \cref{lem:htUA-to-tensor-htUbullet}), but it is more convenient to view it as a Kac-Moody 2-category.  We also consider the completion $\htU_A$ of this category, as described in the previous section. 

\newcommand{\bUA}{\ensuremath{\overline{\tU}_A}}

As in the previous section, we can also consider the localization $\bUA$ of $\tU_A$ by the 2-morphisms
\begin{equation}\label{dumbA}
\mathord{\begin{tikzpicture}[baseline = -1mm]
 	\draw[->,thin] (0.18,-.4) to node [below,at start]{$a\phantom{b}$} (0.18,.4);
  \draw[->,thin] (-0.38,-.4) to node [below,at start]{$b$} (-0.38,.4);
     \node at (0.18,0) {$\bullet$};
     \node at (-0.38,0) {$\bullet$};
\end{tikzpicture}}-
\mathord{\begin{tikzpicture}[baseline = -1mm] 	\draw[->,thin] (0.18,-.4) to node [below,at start]{$a\phantom{b}$} (0.18,.4);
  \draw[->,thin] (-0.38,-.4) to node [below,at start]{$b$} (-0.38,.4);
     \node at (0.18,0) {$\bullet$};
     \node at (-0.38,0) {$\bullet$};
\end{tikzpicture}} \qquad \text{for all $a\neq b\in A$}.
\end{equation}\notation{$\overline{\tU}_A$}{The localization of $\tU_A$ inverting differences of dots on distinct strands.}

We now want to iterate the coproduct of \cref{sec:coproduct} to obtain a functor $\tUdot\to \bUA$.  This phrasing is slightly imprecise, since $\bUA$ is not a tensor product on the nose.  However, whenever we have a simple root $\al_i$ orthogonal to all other simple roots, we can define a functor $\tU \to \tU'$, where $\tU'$ is the Kac-Moody 2-category obtained by replacing $i$ with red and blue copies $\red{i}$ and $\blue{i}$, while leaving all other simple roots unchanged.  This is done by applying formulas \crefrange{com0old}{com0older} to strands labeled $i$, and leaving all other strands unchanged.  Since these formulas are symmetric in red and blue strands, we can apply this operation any number of times, and the result is independent of the order of application.  

Given any finite set $A$, this iterated operation gives a functor $\Delta^A\colon \tUdot\to \bUA$ which acts on 2-morphisms by the formulas \crefrange{com0newp}{com0newerp}:
\begin{equation}
\label{com0newp}
\mathord{
\begin{tikzpicture}[baseline = -.6mm]
	\draw[->,thin] (0.08,-.3) to (0.08,.3);
      \node at (0.08,0) {$\bullet$};
\end{tikzpicture}
}
\mapsto \sum_{a\in A} \mathord{
\begin{tikzpicture}[baseline = -.6mm]
	\draw[->,thin] (0.08,-.3) to  node [below,at start]{$a$} (0.08,.3);
      \node at (0.08,0) {$\bullet$};
\end{tikzpicture}
} 
\:,
\qquad \qquad 
\mathord{
\begin{tikzpicture}[baseline = 1mm]
	\draw[<-,thin] (0.4,0) to[out=90, in=0] (0.1,0.4);
	\draw[-,thin] (0.1,0.4) to[out = 180, in = 90] (-0.2,0);
\end{tikzpicture}
}\mapsto
\sum_{a\in A} \mathord{
\begin{tikzpicture}[baseline = 1mm]
	\draw[<-,thin] (0.4,0) to[out=90, in=0] node [below,at start]{$a$} (0.1,0.4);
	\draw[-,thin] (0.1,0.4) to[out = 180, in = 90] (-0.2,0);
\end{tikzpicture}
}\:,
\qquad\qquad
\mathord{
\begin{tikzpicture}[baseline = 1mm]
	\draw[<-,thin] (0.4,0.4) to[out=-90, in=0] (0.1,0);
	\draw[-,thin] (0.1,0) to[out = 180, in = -90] (-0.2,0.4);
\end{tikzpicture}
}
\mapsto\sum_{a\in A} 
\mathord{
\begin{tikzpicture}[baseline = 1mm]
	\draw[<-,thin] (0.4,0.4) to[out=-90, in=0] node [above,at start]{$a$}(0.1,0);
	\draw[-,thin,] (0.1,0) to[out = 180, in = -90] (-0.2,0.4);
\end{tikzpicture}
}\:,
\end{equation}
\begin{equation}
\label{com0newerp}
\mathord{
\begin{tikzpicture}[baseline = -.6mm,xscale=2.6,yscale=2]
	\draw[->,thin] (0.28,-.3) to (-0.28,.3);
	\draw[thin,->] (-0.28,-.3) to (0.28,.3);
\end{tikzpicture}
}\mapsto \sum_{a\in A} 
\mathord{
\begin{tikzpicture}[baseline = -.6mm,xscale=2.6,yscale=2]
	\draw[->,thin] (0.28,-.3) to node [below,at start]{$a\phantom{b}$}(-0.28,.3);
	\draw[thin,->] (-0.28,-.3) to node [below,at start]{$a\phantom{b}$}(0.28,.3);
\end{tikzpicture}
}+\sum_{a\neq b\in A}
\mathord{
\begin{tikzpicture}[baseline = -.6mm,xscale=2.6,yscale=2]
	\draw[->,thin] (0.28,-.3) to node [below,at start]{$b$}(-0.28,.3);
	\draw[thin,->] (-0.28,-.3) to node [below,at start]{$a\phantom{b}$}(0.28,.3);
	\draw[densely dotted,->] (-0.15,-.16) to  (0.11,-.16);
 \node[scale=.6] at (0,-.3){$(x_1-x_2)^{-1}$};
     \node at (0.15,-.17) {$\bull$};
     \node at (-0.15,-.17) {$\bull$};
\end{tikzpicture}
}+\sum_{a\neq b\in A}
\mathord{
\begin{tikzpicture}[baseline = -.6mm,xscale=2.6,yscale=2]
	\draw[->,thin] (0.28,-.3) to node [below,at start]{$b$}(0.28,.3);
	\draw[thin,->] (-0.28,-.3) to node [below,at start]{$a\phantom{b}$}(-0.28,.3);
	\draw[densely dotted,->] (-0.28,-.16) to  (0.24,-.16);
     \node at (0.28,-.17) {$\bull$};
      \node[scale=.6] at (0,-.3){$(x_1-x_2)^{-1}$};
     \node at (-0.28,-.17) {$\bull$};
\end{tikzpicture}
}.
\end{equation}
Here we use the convention, following \cite[\S 4]{brundanHeisenbergKacMoody2020}, that:
\begin{itemize}
	\item[$(\ddagger)$] 
 a dotted arrow labeled with an expression in $f(x_1,x_2)$ means that we substitute the dot at the start of the arrow for $x_1$ and the dot at the end of the arrow for $x_2$.
\end{itemize}  In \crefrange{com0newp}{com0newerp}, this is well-defined because we have localized the difference between the two dots.  Most uses in the rest of the paper will use a power series in these variables, which will be well-defined in a completion of the category.

We ultimately want to connect this map to understanding actions where we have an action of $\tUdot$ where the dot has multiple different eigenvalues, with the set $A$ corresponding to the spectrum in $\K$.  To do this, we want to map the dot, not just to a sum of dots, but to the desired scalar plus a power series in the dots, so that matching the desired spectrum becomes a matter of making dots nilpotent.


\begin{lemma}\label{lem:many-sl2}
  There is a functor $\sigma\colon \bUA\to \htU_A$ which acts on objects and 1-morphisms by the identity and on 2-morphisms by the formulas
\begin{equation}
 \label{sigma-1} \mathord{
\begin{tikzpicture}[baseline = -.6mm]
	\draw[->,thin] (0.08,-.3) to node [below,at start]{$a$} (0.08,.3);
      \node at (0.08,0) {$\bullet$};
\end{tikzpicture}
}
\mapsto
\pwr_{a} \big(\mathord{
\begin{tikzpicture}[baseline = -.6mm]
	\draw[->,thin] (0.08,-.3) to node [below,at start]{$a$} (0.08,.3);
      \node at (0.08,0) {$\bullet$};
\end{tikzpicture}
} \big)\:,
\qquad \qquad 
\mathord{
\begin{tikzpicture}[baseline = 1mm]
	\draw[<-,thin] (0.4,0) to[out=90, in=0] node [below,at start]{$a$}(0.1,0.4);
	\draw[-,thin] (0.1,0.4) to[out = 180, in = 90] node [below,at end]{$a$}(-0.2,0);
\end{tikzpicture}
}\mapsto
\mathord{
\begin{tikzpicture}[baseline = 1mm]
	\draw[<-,thin] (0.4,0) to[out=90, in=0] node [below,at start]{$a$}(0.1,0.4);
	\draw[-,thin] (0.1,0.4) to[out = 180, in = 90] node [below,at end]{$a$}(-0.2,0);
\end{tikzpicture}
}\:\:,
\end{equation}
\begin{equation} \label{sigma-2}
\mathord{
\begin{tikzpicture}[baseline = -.6mm,xscale=2.6,yscale=2]
	\draw[->,thin] (0.28,-.3) to node [below,at start]{$a$}(-0.28,.3);
	\draw[thin,->] (-0.28,-.3) to node [below,at start]{$a$}(0.28,.3);
\end{tikzpicture}
}\mapsto
\begin{tikzpicture}[baseline = -.6mm,xscale=2.6,yscale=2]
	\draw[->,thin] (0.28,-.3) to node [below,at start]{$a$}(-0.28,.3);
	\draw[thin,->] (-0.28,-.3) to node [below,at start]{$a$}(0.28,.3);
	\draw[densely dotted,->] (-0.15,-.16) to  (0.11,-.16);
 \node[scale=.6] at (0,-.3){$q_a(x_1,x_2)$};
     \node at (0.15,-.17) {$\bull$};
     \node at (-0.15,-.17) {$\bull$};
\end{tikzpicture}\:, \qquad\qquad 
\mathord{\begin{tikzpicture}[baseline = -.6mm,xscale=2.6,yscale=2]
  \draw[->,thin] (0.28,-.3) to node [below,at start]{$b$}(-0.28,.3);
  \draw[thin,->] (-0.28,-.3) to node [below,at start]{$a$}(0.28,.3);
\end{tikzpicture}} \mapsto \mathord{\begin{tikzpicture}[baseline = -.6mm,xscale=2.6,yscale=2]
  \draw[->,thin] (0.28,-.3) to node [below,at start]{$b$}(-0.28,.3);
  \draw[thin,->] (-0.28,-.3) to node [below,at start]{$a$}(0.28,.3);
\end{tikzpicture}} \:.
\end{equation}
\end{lemma}
\begin{proof}
  \newcommand{\BY}{\mathbf{Y}}
  First, we note that by \cref{thm:brundandef}, we need to check the relations (KM1-3).
  \begin{enumerate}
  \item[(KM1)] This is immediate since the images of the left cups and caps are unchanged.
  \item[(KM2)] We just need to check that the formulas above define an automorphism $\sigma_-$ of the KLR algebra.  We can see this immediately by considering the automorphism $\sigma_{\Ba}$ of the completed polynomial ring $\K[[Y_1,\dots, Y_n]]$ defined by $\sigma_{\Ba}(Y_k)=\pwr_{a_k}(Y_k)$, which is an automorphism since the power series $\pwr_{a_k}$ has non-zero linear term. These induce an automorphism $\sigma_{\mathcal{P}}$ of the grading completion $\widetilde{\mathcal{P}}$ of the polynomial representation of the KLR algebra.  Writing $\BY=(Y_1,\dots, Y_n), s_k\BY=(Y_1,\dots, Y_k,Y_{k+1},\dots, Y_n)$,  we have the following formulas for the action of the images of the generators on $\sigma_{\mathcal{P}}$: 
  \newseq 
  \begin{align*}
	\label{eq:poly-action1-check}\subeqn
		\sigma_-(y_k)\sigma_{\Ba}(f(\BY))e_{\Ba}&=\pwr_{a_k}(Y_k)\sigma_{\Ba}(f(\BY))e_{\Ba}=\sigma_{\Ba}(Y_kf(\BY))\\ \subeqn\label{eq:poly-action2-check}
    \sigma_-(\psi_k)\sigma_{\Ba}(f(\BY))&=
	\sigma_{s_k\Ba}( f(s_k\BY))e_{s_k\Ba}&a_k\neq a_{k+1}
\end{align*}
On the other hand, if $a_k=a_{k+1}$, we have
\begin{align*}
  \sigma_-(\psi_k)\sigma_{\Ba}(f(\BY))e_{\Ba}& =
	\subeqn \label{eq:poly-action3-check} q_{a_k}(Y_k,Y_{k+1})\frac{\sigma_{\Ba}(f(s_k\BY)-f(\BY))}{Y_{k+1}-Y_k}e_{\Ba}\\&=\frac{\sigma_{\Ba}(f(s_k\BY)-f(\BY))}{\pwr_{a_k}(Y_{k+1})-\pwr_{a_k}(Y_k)}e_{\Ba}\\
  &=\sigma_{\Ba}(\psi_k f(\BY))e_{\Ba}.
\end{align*}
which completes the proof.
\item[(KM3)] First, we note that if $\gamma_a(z)=a+z$, then we have $q_a=1$ and so the maps \crefrange{EF1}{EF2} are sent to the same maps with the dots on cups shifted by $a$;  since the span of $1,z+a,\dots, (z+a)^{r-1}$ is the same as the span of $1,z,\dots, z^{r-1}$, the resulting maps differ from the original maps by an automorphism of the trivial summand corresponding to this change of basis, so the relations still hold.  This in fact shows that the formulas of $\sigma_a$ define an automorphism of $\tUdot$.  

Now, we turn to the general case.   We can assume that $a=0$ by precomposing with the automorphism for $-a+z$.    We can apply the formulas \crefrange{sigma-1}{sigma-2} to the maps \crefrange{EF1}{EF2}, and we obtain maps which we hope to show are an isomorphism in $\htU_A$.  Since $a=0$, the resulting morphisms are sums of 2-morphisms of degree $\geq 0$ when we choose grading shifts so that \crefrange{EF1}{EF2}.  The leading term of these are the original maps, so they are invertible in the completion.  This shows that the images satisfy (KM3).
  \end{enumerate}
Having checked the relations, we have a well-defined functor $\sigma\colon \bUA\to \htU_A$. 
\end{proof}

Taking the composition $\rho=\sigma\circ \Delta^A \colon \tUdot\to \htU_A$, we obtain a 2-functor which maps:
\begin{equation}\label{eq:1-morphism}
		n \mapsto \bigoplus_{\sum_{a\in A}\lambda^a=n}\lambda\qquad \qquad \eE\mapsto \bigoplus_{a\in A} \eE_a \end{equation}
Note that this maps each integer $n$, thought of as a weight, to an infinite direct sum of weights, which we can think of as functions $A\to \Z$ whose values sum to $n$.  On 2-morphisms, this functor can be found by applying $\sigma$ to the formulas \crefrange{com0new}{com0newer}:
\begin{equation}
\label{com0new}
\mathord{
\begin{tikzpicture}[baseline = -.6mm]
	\draw[->,thin] (0.08,-.3) to (0.08,.3);
      \node at (0.08,0) {$\bullet$};
\end{tikzpicture}
}
\mapsto \sum_{a\in A} \pwr_a \big(\mathord{
\begin{tikzpicture}[baseline = -.6mm]
	\draw[->,thin] (0.08,-.3) to  node [below,at start]{$a$} (0.08,.3);
      \node at (0.08,0) {$\bullet$};
\end{tikzpicture}
} \big) 
\:,
\qquad \qquad 
\mathord{
\begin{tikzpicture}[baseline = 1mm]
	\draw[<-,thin] (0.4,0) to[out=90, in=0] (0.1,0.4);
	\draw[-,thin] (0.1,0.4) to[out = 180, in = 90] (-0.2,0);
\end{tikzpicture}
}\mapsto
\sum_{a\in A} \mathord{
\begin{tikzpicture}[baseline = 1mm]
	\draw[<-,thin] (0.4,0) to[out=90, in=0] node [below,at start]{$a$} (0.1,0.4);
	\draw[-,thin] (0.1,0.4) to[out = 180, in = 90] (-0.2,0);
\end{tikzpicture}
}\:,
\qquad\qquad
\mathord{
\begin{tikzpicture}[baseline = 1mm]
	\draw[<-,thin] (0.4,0.4) to[out=-90, in=0] (0.1,0);
	\draw[-,thin] (0.1,0) to[out = 180, in = -90] (-0.2,0.4);
\end{tikzpicture}
}
\mapsto\sum_{a\in A} 
\mathord{
\begin{tikzpicture}[baseline = 1mm]
	\draw[<-,thin] (0.4,0.4) to[out=-90, in=0] node [above,at start]{$a$}(0.1,0);
	\draw[-,thin,] (0.1,0) to[out = 180, in = -90] (-0.2,0.4);
\end{tikzpicture}
}\:,
\end{equation}
\begin{equation}
\label{com0newer}
\mathord{
\begin{tikzpicture}[baseline = -.6mm,xscale=2.6,yscale=2]
	\draw[->,thin] (0.28,-.3) to (-0.28,.3);
	\draw[thin,->] (-0.28,-.3) to (0.28,.3);
\end{tikzpicture}
}\mapsto \sum_{a\in A} 
\mathord{
\begin{tikzpicture}[baseline = -.6mm,xscale=2.6,yscale=2]
	\draw[->,thin] (0.28,-.3) to node [below,at start]{$a\phantom{b}$}(-0.28,.3);
	\draw[thin,->] (-0.28,-.3) to node [below,at start]{$a\phantom{b}$}(0.28,.3); \node[scale=.6] at (0,-.3){$q_a(x_1,x_2)$};
	\draw[densely dotted,->] (-0.15,-.16) to  (0.11,-.16);
     \node at (0.15,-.17) {$\bull$};
     \node at (-0.15,-.17) {$\bull$};
\end{tikzpicture}
}+\sum_{a\neq b\in A}
\mathord{
\begin{tikzpicture}[baseline = -.6mm,xscale=2.6,yscale=2]
	\draw[->,thin] (0.28,-.3) to node [below,at start]{$b$}(-0.28,.3);
	\draw[thin,->] (-0.28,-.3) to node [below,at start]{$a\phantom{b}$}(0.28,.3);
	\draw[densely dotted,->] (-0.15,-.16) to  (0.11,-.16);
 \node[scale=.6] at (0,-.3){$f_{a,b}(x_1,x_2)$};
     \node at (0.15,-.17) {$\bull$};
     \node at (-0.15,-.17) {$\bull$};
\end{tikzpicture}
}+\sum_{a\neq b\in A}
\mathord{
\begin{tikzpicture}[baseline = -.6mm,xscale=2.6,yscale=2]
	\draw[->,thin] (0.28,-.3) to node [below,at start]{$b$}(0.28,.3);
	\draw[thin,->] (-0.28,-.3) to node [below,at start]{$a\phantom{b}$}(-0.28,.3);
	\draw[densely dotted,->] (-0.28,-.16) to  (0.24,-.16);
     \node at (0.28,-.17) {$\bull$};
      \node[scale=.6] at (0,-.3){$f_{a,b}(x_1,x_2)$};
     \node at (-0.28,-.17) {$\bull$};
\end{tikzpicture}
}\:
\end{equation}
where $f_{a,b}=(\pwr_a(x_1)-\pwr_b(x_2))^{-1}$, which is a well-defined power series in $x_1$ and $x_2$ since the constant term of $\pwr_a(x_1)-\pwr_b(x_2)$ is $a-b\neq 0$. 

These formulas make sense in the completion $\htU_A$, since modulo any power of the dot, each formula becomes a finite sum.  Note that we have not written the images of the left cups and caps here.  Implicitly they are fixed by the formulas above, and can be determined from iterating the formula \cref{com4} and then applying the automorphism  $\sigma_a$ to the results.  However, we have not given a formula for the right cup under $\sigma_a$ and the result is quite unwieldy, so we have chosen not to write it out explicitly.

Let $\htU_{\bullet}$ denote the completion of $\tUdot$ in the GC topology.  
\begin{lemma}\label{lem:htUA-to-tensor-htUbullet}
For a finite subset $A\subset \K$, there is a natural 2-functor
\[
\Phi_A\colon \htU_A\to \otimes_{a\in A}\htU_{\bullet}
\]
sending an object $(k_a)_{a\in A}$ to $\otimes_{a\in A}k_a$, and sending generators by
\[
\eE_a\mapsto 1\otimes\cdots\otimes \eE\otimes\cdots\otimes 1,
\qquad
\eF_a\mapsto 1\otimes\cdots\otimes \eF\otimes\cdots\otimes 1,
\]
with $\eE,\eF$ in the $a$th factor.  This sends dots, cups, and caps and single color crossings to the corresponding maps in the $a$th factor, and sends crossings between different colors to the identity map.  
\end{lemma}
\begin{proof}
The one-color relations (KM1) and (KM3) hold factorwise, since each factor is $\htU_{\bullet}$.  For labels $a\neq b$, there are no edges in the Dynkin diagram for $A$, so (KM2) is exactly the statement that different factors commute; this is the interchange law in the tensor product of 2-categories.  The finiteness/completion conditions are naturally preserved.
\end{proof}

Consider a collection of non-zero representations $\mathcal{X}_a$ for $a\in A$ of $  \tUdot$ in which $\mathord{
\begin{tikzpicture}[baseline = -.6mm]
	\draw[->] (0.08,-.3) to (0.08,.3);
      \node at (0.08,0) {$\bullet$};
\end{tikzpicture}}
$ acts nilpotently on $\eE M$ for all $M$.   Since the dot acts nilpotently on $\eE M$ for every object $M$ of $\mathcal{X}_a$, any power series in the dot applied to $\eE M$ reduces to a finite sum; this means that the action of $\tUdot$ on $\mathcal{X}_a$ extends to an action of the GC-completion $\htU_{\bullet}$.  

Once we have an action of $\htU_{\bullet}$ on each factor $\mathcal{X}_a$, the tensor product $\otimes_{a\in A}\htU_{\bullet}$ naturally acts on the Deligne tensor product $\bigboxtimes_{a\in A}\mathcal{X}_a$: by the universal property of the Deligne tensor product, giving an action of $\otimes_{a\in A}\htU_{\bullet}$ on $\bigboxtimes_{a\in A}\mathcal{X}_a$ is equivalent to giving a multilinear functor $\prod_{a\in A}\mathcal{X}_a\to \bigboxtimes_{a\in A}\mathcal{X}_a$ for each $A$-tuple $(u_a)$ of 1-morphisms in $\htU_{\bullet}$.  We just precompose the tautological multilinear functor with $u_a$ acting on the $a$-factor.  Composing with the functor $\Phi_A$ of \cref{lem:htUA-to-tensor-htUbullet}, we obtain an action of $\htU_A$ on $\bigboxtimes_{a\in A}\mathcal{X}_a$, and precomposing with the functor $\rho\colon \tUdot\to \htU_A$, we arrive at an action of $\tUdot$ on $\bigboxtimes_{a\in A}\mathcal{X}_a$.

Note that the collections of locally finite abelian and Schurian abelian categories are both closed under Deligne tensor product.  This is shown in the locally finite abelian case in \cite[Prop. 22]{lopezfrancoTensorProducts2013}, and in the Schurian case, we simply take the tensor product of the corresponding locally finite-dimensional locally unital algebras.   

\begin{remark}
The reader might naturally ask why we need the additional step of acting by $\sigma$, instead of simply constructing an action of $\bUA$ on the Deligne product.  This is slightly subtler, since $\bUA$ is not a tensor product, it is a localization of one.  Such an action only exists if the polynomials in dots that we invert in the definition of $\bUA$ act invertibly on the Deligne product.  This will often be the case, but not always.  

 We could apply this construction to obtain the same categorical modules by taking $\mathcal{X}_a'$ to be the same underlying category as $\mathcal{X}_a$, but with the action of $\tUdot$ twisted by $\sigma_a$. Unlike $\mathcal{X}_a$, the categorical module $\mathcal{X}_a'$ does not have the property that the dot acts nilpotently on $\eE M$ for all $M$.  Rather, it acts by the scalar $a$ plus a nilpotent operator.  

While this other phrasing is equivalent, our first version of the construction plays more naturally with extension to higher rank, so we have chosen to emphasize this perspective.
\end{remark}

The most natural example is to take $\mathcal{X}_a$ to be the category of projective modules over the cyclotomic quotient for some highest weight $\la_a$ with $\la_a=0$ for all but finitely many $a\in A$; this is the unique categorification generated by a single object $\mathbb{V}_a$ of weight $\la$ where $\mathord{
\begin{tikzpicture}[baseline = -.6mm]
\draw[<-, thick] (0.08,-.3) to (0.08,.3);
      \node at (0.08,0) {$\bullet$};
\end{tikzpicture}}
\colon \eF\mathbb{V}_a\to \eF\mathbb{V}_a$ satisfies \begin{equation}\label{eq:deformed-cyclotomic-quotient}\mathord{
\begin{tikzpicture}[baseline = -1.5mm]
 	\draw[<-,thick] (0.08,-.4) to (0.08,.4);
   \node at (.08,-.55) {$\scriptstyle{i}$};
     \node at (0.08,0) {$\bullet$};
    \node at (-0.68,.05) {$\scriptstyle \prod_{a}(x-a)^{\la_a}$};
   \node at (0.35,0) {$\color{gray}\scriptstyle{\lambda}$};
\end{tikzpicture}
}=0.
 \end{equation}
 We will show in \cref{equivalence}
that, in fact, the tensor product has an induced equivalence to the deformed cyclotomic quotient with \cref{eq:deformed-cyclotomic-quotient}.  Of course, for $\tUdot$, it is easy to construct this cyclotomic quotient by hand, and so the discussion above seems like an unnecessarily roundabout method, but for higher rank, it will pay significant dividends.  
The formulas \crefrange{eq:1-morphism}{com0newer} define a 2-functor $\rho\colon \tUdot\to \htU_A$ for an arbitrary subset $A\subset \K$.  
Above, we have assumed that $A$ is finite, but in fact we can relax this assumption:
\begin{lemma} \label{lem:many-sl2-general}
The formulas \crefrange{eq:1-morphism}{com0newer} define a 2-functor $\rho\colon \tUdot\to \htU_A$ for an arbitrary subset $A\subset \K$.  
\end{lemma}
\begin{proof}
	Checking each relation involves only finitely many indices, so assuming these maps make sense, the relations follow automatically from the finite case.   On the 1-morphisms $\eE,\eF$, these are sent to a sum of length 1 monomials, where each $i$ appears once, so in fact only finitely many terms use $i\in I_0$ for any finite $I_0$.  The result for all other 1-morphisms follows by composition.
	
Now consider 2-morphisms. For each $I_0$, in \crefrange{com0new}{com0newer} only finitely many terms involve only indices from $I_0$ so all but finitely many lie in $\Hom_{\tU}^{I_0,2}(u,v) $.  Thus, the functor is well-defined.
\end{proof}

\subsection{Extension to higher rank}
Now, we generalize and consider $\tU(\mathfrak{g})$ for a general Cartan datum.  As mentioned previously, we assume in this section that $\fg$ is the universal derived algebra for this Cartan matrix.  By \cref{lem:different-realizations}, we will lose no generality by doing this.  Furthermore, since now, we may have $d_i>1$ for some $i$, we recall the assumption that $\K$ is an algebraically closed field of characteristic coprime to $d_i$ for all $i$.  

As before, we choose a spectrum $U_i$ for each $i\in I$ which we assume to be complete (in the sense of \cref{def:tilde-I}).  As defined in \cref{def:tilde-I}, we have a set $\tilde{I}$ naturally mapping to $I$, together with an induced graph structure whose vertices are pairs $(i,u)$ for $u\in U_i$.  Our previous section discussed an example of this construction when $I$ is a single point with no edges.  In this case, $\tilde{I}$ was an arbitrary number of points, identified with the set $U_{\bullet}$.  As we show in \cref{prop:furling-homomorphism}, there is a natural Lie algebra homomorphism $\tilde{\fg}\to \fg$ which sends
\begin{equation}\label{eq:map-on-generators}
  E_i \mapsto \sum_{u\in U_i} E_{i,u}, \qquad F_i \mapsto \sum_{u\in U_i} F_{i,u}, \qquad H_i\mapsto \sum_{u\in U_i} H_{i,u}
\end{equation}
In particular, we have a homomorphism of the Cartan subalgebras $\mathfrak{h}\to \tilde{\mathfrak{h}}$, along which we can consider the pullback of weights---we use $\mu|_{\mathfrak{h}}$ to denote the pullback of a weight $\mu\in \tilde{\mathfrak{h}}^*$ to $\mathfrak{h}^*$.  We will show in \cref{lem:g-action} that this homomorphism extends to a 2-functor $\tU(\fg)\to \htU(\tilde{\fg})$.  

In order to understand the relationship between these 2-categories, let us give a more detailed analysis of the polynomials $P_{ij}$. While these are two variable polynomials, their homogeneity properties make them behave more like one variable polynomials.  
In particular, if $P(x,y)$ is a polynomial homogeneous in the usual grading, then we have a factorization $P(x,y)=P(1,0)\cdot (x-a_1y)\cdots (x-a_ny)$ where $a_i$ are the roots of $P(x,1)$; at $x=u, y=u'$ for $u,u'\in \K\setminus\{0\}$, the number of these factors which vanish is the order of vanishing of $P(u,v)$ at the divisor $x/u=y/u'$.   

Now, we turn to generalizing this notion to the case of non-symmetric Cartan data.  Let $g_{ij}=\gcd(-c_{ij},-c_{ji})$ and $h_{ij}=-c_{ij}/g_{ij}$.\notation{$g_{ij}, h_{ij}$}{The greatest common divisor of the Cartan matrix entries  $g_{ij}=\gcd(-c_{ij},-c_{ji})$ and the ratio $h_{ij}=-c_{ij}/g_{ij}$.}
The natural interpretation of the vanishing order of $P_{ij}$ at $x=u,y=u'$ is the vanishing order of $P_{ij}$ at the irreducible divisor $(x/u)^{h_{ij}}=(y/u')^{h_{ji}}$. By homogeneity and since $\K$ is algebraically closed, divisors of this form are the only ones where $P_{ij}$ can vanish. 

Put differently, if we let $x^{1/d_i}, y^{1/d_j}$ be formal roots of these variables, then we have a resulting factorization 
\begin{equation}\label{eq:factorize}
P_{ij}(x,y)=p_{ij}\prod_{a^{(k)}_{ij} \in A_{ij}}
(x^{1/d_i}-a^{(k)}_{ij}y^{1/d_j})
\end{equation} where
$A_{ij}=\{a^{(k)}_{ij}\}$ is the multiset of roots of
$P_{ij}(x^{d_i},1)$, considered with multiplicity.\notation{$A_{ij},B_{ij}$}{The multiset $A_{ij}=\{a^{(k)}_{ij}\}$ of roots of
$P_{ij}(x^{d_i},1)$ and its reciprocals $B_{ij}=\{b^{(k)}_{ij}=(a^{(k)}_{ji})^{-1}\}$.}
  Let
$B_{ij}=\{b^{(k)}_{ij}=(a^{(k)}_{ji})^{-1}\}$ be the reciprocals of these
numbers.  Note that
\[Q_{ij}(x,y)=t_{ji}^{-1}\prod_{a^{(k)}_{ij} \in A_{ij}}
(x^{1/d_i}-a^{(k)}_{ij}y^{1/d_j}) \prod_{b^{(k)}_{ij} \in B_{ij}}
(x^{1/d_i}-b^{(k)}_{ij}y^{1/d_j}).\] 

Let $\#[(i,u)\to (j,u')]$ be the number of edges from $(i,u)$ to $(j,u')$ in the graph with vertex set $\tilde{I}$. Associated to this graph structure, we have polynomials
\begin{equation*}
\mathsf{P}_{(i,u),(j,u')}(x,y) =(x-y)^{\#[(i,u)\to (j,u')]}\in \K[x,y].
\end{equation*} 
\notation{$\mathsf{P}_{(i,u),(j,u')}$}{The polynomials corresponding to ``geometric coefficients'' for the graph structure on $\tilde{I}$.}
These are the ``geometric coefficients'' for a KLR algebra of this graph.
\begin{definition}
Consider the KLR algebra $\mathsf{R}$ of $\tilde{I}$ with the symmetric Cartan
datum associated to its graph structure and the polynomials $\mathsf{P}_{(i,u),(j,u')}$.  To avoid confusion, we denote the elements $\mathsf{\psi}_k , \mathsf{y}_k,\mathsf{e}_{\Bj}$ of this algebra with sans serif letters.  

Let $\tilde{\fg}$ be the derived universal symmetric Kac-Moody algebra attached to the Cartan datum on $\tilde{I}$. Let $\tU(\tilde{\fg})$ be the 2-category categorifying $\tilde{\fg}$ with the polynomials $\mathsf{P}_{(i,u),(j,u')}$.
\end{definition}

We want to extend \cref{lem:many-sl2-general} to give a 2-functor $\tU(\fg)\to \htU(\tilde{\fg})$. First, we must be careful about how we set up these functors on the individual $\mathfrak{sl}_2$'s.  In order to define these functors on strands with label $i$ via the formulas \crefrange{eq:1-morphism}{com0newer}, we use the power series \[\pwr_{\aa,i}(z)=\aa(z+1)^{d_i}.\]  

In order to accomplish this extension, we only need to describe the image of the crossing of differently labeled strands.  To understand the formulas for this image, note that by \cref{eq:factorize}, we have that 
\begin{equation}
P_{ij}(\pwr_{u,i}(x),\pwr_{u',j}(y))=p_{ij}\prod_{a^{(k)}_{ij} \in A_{ij}}
(u^{1/d_i}(x+1)-a^{(k)}_{ij}(u')^{1/d_{j}}(y+1))
\end{equation}
for $u^{1/d_i}$ and $(u')^{1/d_j}$ any roots of $u,u'\in \K$.  Note that such roots exist by our assumption that $\K$ is algebraically closed.  The product above is independent of this choice.  Note that each of the factors in the product above is either:
\begin{enumerate}
	\item invertible if $u^{1/d_i}\neq a^{(k)}_{ij}(u')^{1/d_{j}}$ or
	\item a multiple of $x-y$ if $u^{1/d_i}=a^{(k)}_{ij}(u')^{1/d_{j}}$.
\end{enumerate}
and that the number of the latter factors is exactly the number of arrows $(i,u)\to (j,u')$.  
Thus, we can factor \begin{equation}\label{eq:P-factorization}
P_{ij}(\pwr_{u,i}(x),\pwr_{u',j}(y))=\mathsf{P}_{(i,u),(j,u')}(x,y) P_{ij}^{\circ,uu'}(x,y)
\end{equation}
 for an invertible power series $P_{ij}^{\circ,uu'}(x,y)\in \K\llbracket x,y\rrbracket$.  
We find that:
\notation{$P_{ij}^{\circ,uu'}$}{The invertible power series such that $P_{ij}(\pwr_{u,i}(x),\pwr_{u',j}(y))=\mathsf{P}_{(i,u),(j,u')}(x,y) P_{ij}^{\circ,uu'}(x,y)$.}
\begin{lemma}\label{lem:g-action} Given a complete choice of $U_i$,  
there is a 2-functor $\rho\colon \tU(\fg)\to \htU(\tilde\fg)$ acting on 0-morphisms and 1-morphisms by:
\begin{equation}
	\la \mapsto \bigoplus_{\mu|_{\mathfrak{h}}=\la}\mu,\qquad \eE_i\mapsto \bigoplus_{u\in U_i} \eE_{i,u}, \qquad \eF_i\mapsto \bigoplus_{u\in U_i} \eF_{i,u}. 
\end{equation}

The action on 2-morphisms is given by 
\begin{equation}
\label{com0new2}
\mathord{
\begin{tikzpicture}[baseline = -.6mm]
	\draw[->,thin] (0.08,-.3) to  node [below,at start]{$i$} (0.08,.3);
      \node at (0.08,0) {$\bullet$};
\end{tikzpicture}
}
\mapsto \sum_{a\in U_i} \pwr_{a,i} \big(\mathord{
\begin{tikzpicture}[baseline = -.6mm]
	\draw[->,thin] (0.08,-.3) to  node [below,at start]{$(i,a)$} (0.08,.3);
      \node at (0.08,0) {$\bullet$};
\end{tikzpicture}
} \big) 
\:,
\qquad \qquad 
\mathord{
\begin{tikzpicture}[baseline = 1mm]
	\draw[<-,thin] (0.4,0) to[out=90, in=0] node [below,at start]{$i$}(0.1,0.4);
	\draw[-,thin] (0.1,0.4) to[out = 180, in = 90]  (-0.2,0);
\end{tikzpicture}
}\mapsto
\sum_{a\in U_i} \mathord{
\begin{tikzpicture}[baseline = 1mm]
	\draw[<-,thin] (0.4,0) to[out=90, in=0] node [below,at start]{$(i,a)$} (0.1,0.4);
	\draw[-,thin] (0.1,0.4) to[out = 180, in = 90] (-0.2,0);
\end{tikzpicture}
}\:,
\qquad\qquad
\mathord{
\begin{tikzpicture}[baseline = 1mm]
	\draw[<-,thin] (0.4,0.4) to[out=-90, in=0] node [above,at start]{$i$}(0.1,0);
	\draw[-,thin] (0.1,0) to[out = 180, in = -90] (-0.2,0.4);
\end{tikzpicture}
}
\mapsto\sum_{a\in U_i} 
\mathord{
\begin{tikzpicture}[baseline = 1mm]
	\draw[<-,thin] (0.4,0.4) to[out=-90, in=0] node [above,at start]{$(i,a)$}(0.1,0);
	\draw[-,thin,] (0.1,0) to[out = 180, in = -90] (-0.2,0.4);
\end{tikzpicture}
}\:,
\end{equation}
\begin{equation}
\label{com0newer2}
\mathord{
\begin{tikzpicture}[baseline = -.6mm,xscale=2.6,yscale=2]
	\draw[->,thin] (0.28,-.3) to node [below,at start]{$i\phantom{b}$}(-0.28,.3);
	\draw[thin,->] (-0.28,-.3) to node [below,at start]{$i\phantom{b}$}(0.28,.3);
\end{tikzpicture}
}\mapsto \sum_{a\in U_i} 
\mathord{
\begin{tikzpicture}[baseline = -.6mm,xscale=2.6,yscale=2]
	\draw[->,thin] (0.28,-.3) to node [below,at start]{$(i,a)\phantom{b}$}(-0.28,.3);
	\draw[thin,->] (-0.28,-.3) to node [below,at start]{$(i,a)\phantom{b}$}(0.28,.3);	\draw[thin,->] (-0.28,-.3) to node [below,at start]{$(i,a)\phantom{b}$}(0.28,.3);
	\draw[densely dotted,->] (-0.15,-.16) to  (0.11,-.16);
 \node[scale=.6] at (0,-.3){$q_{a,i}(x_1,x_2)$};
     \node at (0.15,-.17) {$\bull$};
     \node at (-0.15,-.17) {$\bull$};
\end{tikzpicture}
}+\sum_{a\neq b\in U_i}
\mathord{
\begin{tikzpicture}[baseline = -.6mm,xscale=2.6,yscale=2]
	\draw[->,thin] (0.28,-.3) to node [below,at start]{$(i,b)$}(-0.28,.3);
	\draw[thin,->] (-0.28,-.3) to node [below,at start]{$(i,a)\phantom{b}$}(0.28,.3);
	\draw[densely dotted,->] (-0.15,-.16) to  (0.11,-.16);
 \node[scale=.6] at (0,-.3){$f_{a,b,i}(x_1,x_2)$};
     \node at (0.15,-.17) {$\bull$};
     \node at (-0.15,-.17) {$\bull$};
\end{tikzpicture}
}+\sum_{a\neq b\in U_i}
\mathord{
\begin{tikzpicture}[baseline = -.6mm,xscale=2.6,yscale=2]
	\draw[->,thin] (0.28,-.3) to node [below,at start]{$(i,b)$}(0.28,.3);
	\draw[thin,->] (-0.28,-.3) to node [below,at start]{$(i,a)\phantom{b}$}(-0.28,.3);
	\draw[densely dotted,->] (-0.28,-.16) to  (0.24,-.16);
     \node at (0.28,-.17) {$\bull$};
      \node[scale=.6] at (0,-.3){$f_{a,b,i}(x_1,x_2)$};
     \node at (-0.28,-.17) {$\bull$};
\end{tikzpicture}
}\:
\end{equation}
	\begin{equation}
\label{com0newerer}
\mathord{
\begin{tikzpicture}[baseline = -.6mm,xscale=3,yscale=2]
	\draw[->,thin] (0.28,-.3) to node [below,at start]{$i$}(-0.28,.3);
	\draw[thin,->] (-0.28,-.3) to node [below,at start]{$j$} (0.28,.3);
\end{tikzpicture}
}\mapsto \sum_{\substack{a\in U_i\\ b\in U_j}}
\mathord{
\begin{tikzpicture}[baseline = -.6mm,xscale=3,yscale=2]
	\draw[->,thin] (0.28,-.3) to node [below,at start]{$\phantom{j}(i,a)$}(-0.28,.3);
	\draw[thin,->] (-0.28,-.3) to node [below,at start]{$(j,b)\phantom{b}$}(0.28,.3);
	\draw[densely dotted,->] (-0.15,.16) to  (0.11,.16);
 \node[scale=.6] at (0,.3){$P_{ij}^{\circ,ab}(x_1,x_2)$};
     \node at (0.15,.17) {$\bull$};
     \node at (-0.15,.17) {$\bull$};
\end{tikzpicture}
}
\end{equation}
\[ q_{a,i}(x_1,x_2)=\frac{x_1-x_2}{\gamma_{a,i}(x_1)-\gamma_{a,i}(x_2)},\qquad 
f_{a,b,i}(x_1,x_2)=(\gamma_{a,i}(x_1)-\gamma_{b,i}(x_2))^{-1}.\]
\end{lemma}
Note that \crefrange{com0new2}{com0newer2} are effectively the same as \crefrange{com0new}{com0newer}.  We give the proof of this result on page \pageref{proof-lem:g-action}.

\subsection{Non-degeneracy}
\label{sec:non-degenerate}
\notation{$B_{\Bi,\Bi',\la}$}{The basis of  $\Hom_{\tU}
(\Bi,\Bi')$ constructed in \cref{sec:non-degenerate}.}

This provides us with the tools we need to prove non-degeneracy for categorified Kac-Moody algebras $\tU=\tU(\fg)$.  For simplicity, from now on, we will use the notation $\Hom_{\tU}(\Bi,\Bi')$ for the Hom-space $\Hom(\eE_{\Bi},\eE_{\Bi'})$ between two 1-morphisms $\eE_{\Bi},\eE_{\Bi'}\colon \la\to \mu$ in $\tU$.  In \cite[\S
3.2.3]{khovanovCategorificationQuantum2010}, Khovanov and Lauda define a spanning set $B_{\Bi,\Bi',\la}$ of the Hom space  $\Hom_{\tU}(\Bi,\Bi')$  that they hypothesize are a basis.  
This set is not truly canonical, but it is well-defined up to strictly lower triangular change of basis.  

This set is constructed as follows:
      \begin{enumerate}
      \item for each way of pairing up the combined entries of $\Bi$ and $\Bi'$ in a way compatible with orientation (an $(\Bi,\Bi')$-pairing, in the
terminology of \cite{khovanovCategorificationQuantum2010}), choose a diagram which connects the terminals for the two elements of each
pair, with a minimal number of crossings and no bubbles or
dots.  Note that this requires us to match an $i$ in $\Bi$ with an $i$ in $\Bi'$ or a $-i$ in $\Bi$, etc.
\item for each diagram of the above, and each arc joining two
terminals, we fix a position on that arc (avoiding the
crossings with any others).  The full list of diagrams in $B_{\Bi,\Bi',\la}$ are
indexed by 
(i) a diagram as above, (ii) a choice of non-negative
integer for each arc joining two terminals, and (iii) a monomial in
the unnested bubbles (a basis vector of the ring $\Pi_{\la}$ in the
notation of \cite[\S 3.2.1]{khovanovCategorificationQuantum2010}).  We construct the
corresponding diagram by adding the corresponding number of
dots on each arc and putting the bubbles to the left of the diagram. 
      \end{enumerate}
      	Thus far in the paper, we have assumed that $\K$ is an algebraically closed field. In fact, we can work over an arbitrary commutative ring $\K'$ where the integers $d_i$ are invertible, together with a choice of $Q_{ij}(x,y)$ such that $t_{ij}$ and $t_{ji}$ are multiplicative inverses of the scalars $Q_{ij}(0,1),Q_{ij}(1,0)$.  

      \begin{theorem}\label{nondegenerate}
	In the category $\tU$ defined over $\K'$, Khovanov and Lauda's spanning set $B_{\Bi,\Bi',\la}$ is a basis of the Hom space $\Hom_{\tU}
      (\Bi,\Bi')$ as a free $\K'$-module.
\end{theorem}
We prove this by exploiting \cref{lem:g-action} to construct examples of representations where any hypothetical relation must act trivially.  We give the proof of this result on page \pageref{proof-nondegenerate}.
Note that while our proof will use the universal derived Kac-Moody algebra of our Cartan matrix, \cref{lem:different-realizations} shows that this will imply the result for all different choices of realization.  

\section{Spectral analysis of actions}
\label{sec:spectral}
\subsection{Eigenfunctors}
\label{sec:eigenfunctors}
In this section, we reverse the construction of $\tU(\fg)$-representations from $\tU(\tilde{\fg})$-representations.  As mentioned in the introduction, our approach is very close to that of \cite[\S 4]{brundanHeisenbergKacMoody2020}, requiring only occasional changes due to the different nature of Heisenberg and Kac-Moody categorifications.  We begin with a discussion of the spectral analysis of $\tU(\fg)$-representations, so in \crefrange{sec:eigenfunctors}{swd}, we write $\tU=\tU(\fg)$ without ambiguity. 

In much of the literature, the representations of $\tU$ considered
have had
the action of the dots and bubbles be nilpotent; this is necessary in
a graded 2-representation which is locally finite abelian or Schurian.  However,
it can be a very powerful technique to deform these representations in
such a way as to break this assumption.  

Consider a representation of the 2-category
$\tU$, sending $\la\mapsto \mathcal{C}_\la$ such that
$\mathcal{C}_\la$ is $\K$-linear locally finite abelian or Schurian. Since we assume that $\K$ is algebraically closed, the endomorphisms of any simple object in such a category are always the base field $\K$.  

For simplicity, we start with the case $\fg=\mathfrak{sl}_2$.  By the usual identification of the weight lattice of $\mathfrak{sl}_2$ with $\Z$, we write $\Ccat_k$ for the weight categories of a categorical $\fg$-action. 

\notation{$\eE_{i,\aa}, \eF_{i,\aa}$}{The summands of $\eE_{i}$ and $\eF_{i}$ given by generalized eigenspaces of dots.}
\begin{definition}\label{def:Ea}
	For $\aa \in  \K$, let $\eE_{\aa}$ and $\eF_{\aa}$ be the subfunctors of $\eE$ and $\eF$
defined on $V \in \Ccat_k$
by declaring that
$\eE_{\aa} V$ and $\eF_{\aa} V$ are the generalized $\aa$-eigenspaces of
the endomorphisms
$\mathord{\begin{tikzpicture}[baseline = -1mm]
 	\draw[->] (0.08,-.2) to (0.08,.2);
     \node at (0.08,0) {$\bullet$};
 	\draw[-,darkg,thick] (0.38,.2) to (0.38,-.2);
     \node at (0.55,0) {$\darkg\scriptstyle{V}$};
\end{tikzpicture}
}$ and
$\mathord{\begin{tikzpicture}[baseline = -1mm]
 	\draw[<-] (0.08,-.2) to (0.08,.2);
     \node at (0.08,0.02) {$\bullet$};
 	\draw[-,darkg,thick] (0.38,.2) to (0.38,-.2);
     \node at (0.55,0) {$\darkg\scriptstyle{V}$};
\end{tikzpicture}
}$, respectively.
\end{definition}
 
As discussed in \cite[\S 4.1]{brundanHeisenbergKacMoody2020}, it suffices to define these functors on finitely generated objects.  For a finitely generated object, we have minimal polynomials $m_V(z), n_V(z) \in \K[z]$ of
the endomorphisms
$\mathord{\begin{tikzpicture}[baseline = -1mm]
 	\draw[->] (0.08,-.2) to (0.08,.2);
     \node at (0.08,0) {$\bullet$};
 	\draw[-,darkg,thick] (0.38,.2) to (0.38,-.2);
     \node at (0.55,0) {$\darkg\scriptstyle{V}$};
\end{tikzpicture}
}$ and
$\mathord{\begin{tikzpicture}[baseline = -1mm]
 	\draw[<-] (0.08,-.2) to (0.08,.2);
     \node at (0.08,0.02) {$\bullet$};
 	\draw[-,darkg,thick] (0.38,.2) to (0.38,-.2);
     \node at (0.55,0) {$\darkg\scriptstyle{V}$};
\end{tikzpicture}
}$, respectively.
Then
there are injective homomorphisms
\begin{align}\label{CRT1}
\K[z] / (m_V(z)) &\hookrightarrow \End_{\Ccat_k} (\eE V),
&
\K[z] / (n_V(z)) &\hookrightarrow \End_{\Ccat_k} ( \eF V),\\
p(z) &
\mapsto \mathord{
\begin{tikzpicture}[baseline = -1mm]
 	\draw[->] (0.08,-.3) to (0.08,.4);
     \node at (0.08,0.05) {$\bullet$};
    \node at (-0.3,0.05) {$\scriptstyle p(x)$};
 	\draw[-,darkg,thick] (0.45,.4) to (0.45,-.3);
     \node at (0.45,-.45) {$\darkg\scriptstyle{V}$};
\end{tikzpicture}
},
&p(z) &\mapsto \mathord{
\begin{tikzpicture}[baseline = -1mm]
 	\draw[<-] (0.08,-.3) to (0.08,.4);
     \node at (0.08,0.1) {$\bullet$};
    \node at (-0.3,0.1) {$\scriptstyle p(x)$};
 	\draw[-,darkg,thick] (0.45,.4) to (0.45,-.3);
     \node at (0.45,-.45) {$\darkg\scriptstyle{V}$};
\end{tikzpicture}
}.\notag
\end{align}
Also, let $\eps_{\aa}(V)$ and $\phi_{\aa}(V)$ denote the multiplicities
of $\aa \in  \K$ as a root of the polynomials $m_V(z)$ and $n_V(z)$,
respectively.
By the Chinese remainder theorem, we have that
\begin{align}\label{CRTdef}
\K[z] / (m_V(z)) &\cong \bigoplus_{\aa \in  \K} \K[z] \big/ \big((z-\aa)^{\eps_{\aa}(V)}\big),&
\K[z] / (n_V(z)) &\cong \bigoplus_{\aa \in  \K} \K[z] \big/
                   \big((z-\aa)^{\phi_{\aa}(V)}\big).
\end{align}
There are corresponding decompositions
$1_{\K[z] / (m_V(z))} = \sum_{\aa \in  \K} e_{\aa, V}$ and
$1_{\K[z] / (n_V(z))} = \sum_{\aa \in  \K} f_{\aa,V}$
of the identity elements of these algebras as sums of mutually orthogonal
idempotents projecting to the generalized eigenspaces of $z$.  
The functors $\eE_{\aa} V$ and $\eF_{\aa} V$ can be defined as the images of the idempotents $e_{\aa,V}$ and $f_{\aa,V}$ acting on $\eE V$ and $\eF V$.  

We will represent the identity endomorphisms of the functors $\eE_{\aa}$ and $\eF_{\aa}$ by
vertical strings colored by $\aa$ as in the diagrams
below.
The inclusions $\eE_{\aa} \hookrightarrow \eE$ and $\eF_{\aa} \hookrightarrow \eF$ are depicted
by the second pair of diagrams below. The projections $\eE
\twoheadrightarrow\eE_{\aa}$ and $\eF \twoheadrightarrow \eF_{\aa}$ are the final pair.
\begin{align*}
&\mathord{
\begin{tikzpicture}[baseline = 0]
	\draw[->] (0.08,-.3) to (0.08,.4);
      \node at (0.08,-0.45) {$\scriptstyle{\aa}$};
\end{tikzpicture}
}
\!: \eE_{\aa} \Rightarrow\eE_{\aa},
&&
\mathord{
\begin{tikzpicture}[baseline = 0]
	\draw[<-] (0.08,-.3) to (0.08,.4);
\node at (0.08,0.55) {$\scriptstyle{\aa}$};
\end{tikzpicture}
} \!: \eF_{\aa} \Rightarrow \eF_{\aa},
&&
\mathord{
\begin{tikzpicture}[baseline = 0]
	\draw[->] (0.08,-.3) to (0.08,.4);
	\draw[-] (0.04,.05) to (0.12,.05);
      \node at (0.08,-0.45) {$\scriptstyle{\aa}$};
\end{tikzpicture}
} \!: \eE_{\aa} \Rightarrow \eE,
&&
\mathord{
\begin{tikzpicture}[baseline = 0]
	\draw[<-] (0.08,-.3) to (0.08,.4);
	\draw[-] (0.04,.05) to (0.12,.05);
      \node at (0.08,-0.45) {$\scriptstyle{\aa}$};
\end{tikzpicture}
} \!: \eF_{\aa} \Rightarrow \eF,
&&
\mathord{
\begin{tikzpicture}[baseline = 0mm]
	\draw[->] (0.08,-.3) to (0.08,.4);
	\draw[-] (0.04,.05) to (0.12,.05);
      \node at (0.08,0.55) {$\scriptstyle{\aa}$};
\end{tikzpicture}
} \!: \eE \Rightarrow\eE_{\aa},
&&
\mathord{
\begin{tikzpicture}[baseline = 0mm]
  \draw[<-] (0.08,-.3) to (0.08,.4);
	\draw[-] (0.04,.05) to (0.12,.05);
      \node at (0.08,0.55) {$\scriptstyle{\aa}$};
\end{tikzpicture}
} \!: \eF \Rightarrow \eF_{\aa}.
\end{align*}
To illustrate the notation, the natural transformation
$\:\mathord{
\begin{tikzpicture}[baseline = -0.8mm]
	\draw[->] (0.08,-.25) to (0.08,.35);
	\draw[-] (0.04,.15) to (0.12,.15);
	\draw[-] (0.04,-.05) to (0.12,-.05);
      \node at (0.2,0.05) {$\scriptstyle{\aa}$};
\end{tikzpicture}
}:\eE \Rightarrow \eE
$
is the projection of $\eE$ onto its summand $\eE_{\aa}$, while
\begin{equation}\label{othogonality}
\mathord{
\begin{tikzpicture}[baseline = 0]
	\draw[->] (0.08,-.3) to (0.08,.4);
	\draw[-] (0.04,.15) to (0.12,.15);
	\draw[-] (0.04,-0.05) to (0.12,-0.05);
      \node at (0.08,0.55) {$\scriptstyle{\bb}$};
      \node at (0.08,-.45) {$\scriptstyle{\aa}$};
\end{tikzpicture}
}
= \delta_{\aa,\bb}
\mathord{
\begin{tikzpicture}[baseline = 0]
	\draw[->] (0.08,-.3) to (0.08,.4);
        \node at (0.08,-.45) {$\scriptstyle{\aa}$};
\end{tikzpicture}
}\:\:.
\end{equation}
It is also clear from the definition that the endomorphisms of
$\eE$ and $\eF$
defined by the dots restrict to endomorphisms of the
summands $\eE_{\aa}$
and $\eF_{\aa}$.
Representing these restrictions simply by drawing the dots on a string
colored by $\aa$, we have that
\begin{align}\label{sizzle}
\mathord{
\begin{tikzpicture}[baseline = -1mm]
	\draw[->] (0.08,-.3) to (0.08,.4);
      \node at (0.08,-0.05) {$\bullet$};
	\draw[-] (0.04,.15) to (0.12,.15);
      \node at (0.08,-0.45) {$\scriptstyle{\aa}$};
\end{tikzpicture}
}&=
\mathord{
\begin{tikzpicture}[baseline = -1mm]
	\draw[->] (0.08,-.3) to (0.08,.4);
      \node at (0.08,0.15) {$\bullet$};
	\draw[-] (0.04,-.05) to (0.12,-.05);
      \node at (0.08,-0.45) {$\scriptstyle{\aa}$};
\end{tikzpicture}
}\:\:,
&
\mathord{
\begin{tikzpicture}[baseline = -1mm]
	\draw[<-] (0.08,-.3) to (0.08,.4);
      \node at (0.08,-0.05) {$\bullet$};
	\draw[-] (0.04,.15) to (0.12,.15);
      \node at (0.08,-0.45) {$\scriptstyle{\aa}$};
\end{tikzpicture}
}&=
\mathord{
\begin{tikzpicture}[baseline = -1mm]
	\draw[<-] (0.08,-.3) to (0.08,.4);
      \node at (0.08,0.15) {$\bullet$};
	\draw[-] (0.04,-.05) to (0.12,-.05);
      \node at (0.08,-0.45) {$\scriptstyle{\aa}$};
\end{tikzpicture}
}\:\:,&
\mathord{
\begin{tikzpicture}[baseline = 1mm]
	\draw[->] (0.08,-.3) to (0.08,.4);
      \node at (0.08,-0.05) {$\bullet$};
	\draw[-] (0.04,.15) to (0.12,.15);
      \node at (0.08,0.55) {$\scriptstyle{\aa}$};
\end{tikzpicture}
}&=
\mathord{
\begin{tikzpicture}[baseline = 1mm]
	\draw[->] (0.08,-.3) to (0.08,.4);
      \node at (0.08,0.15) {$\bullet$};
	\draw[-] (0.04,-.05) to (0.12,-.05);
      \node at (0.08,0.55) {$\scriptstyle{\aa}$};
\end{tikzpicture}
}\:\:,
&
\mathord{
\begin{tikzpicture}[baseline = 1mm]
	\draw[<-] (0.08,-.3) to (0.08,.4);
      \node at (0.08,-0.05) {$\bullet$};
	\draw[-] (0.04,.15) to (0.12,.15);
      \node at (0.08,0.55) {$\scriptstyle{\aa}$};
\end{tikzpicture}
}&=
\mathord{
\begin{tikzpicture}[baseline = 1mm]
	\draw[<-] (0.08,-.3) to (0.08,.4);
      \node at (0.08,0.15) {$\bullet$};
	\draw[-] (0.04,-.05) to (0.12,-.05);
      \node at (0.08,0.55) {$\scriptstyle{\aa}$};
\end{tikzpicture}
}
\:\:.
\end{align}

Since the downwards dot is both the left and right mate of the upwards dot,
the adjunctions $(\eE,\eF)$ and $(\eF,\eE)$ induce adjunctions $(\eE_{\aa}, \eF_{\aa})$
and $(\eF_{\aa},\eE_{\aa})$ for all $\aa \in U$.
We draw the units and counits of these adjunctions using cups and caps
colored by $\aa$.
Again, the various inclusions and projections commute with these morphisms:
\begin{align}
\begin{array}{llll}
\mathord{
\begin{tikzpicture}[baseline=1mm]
	\draw[<-] (0.4,0.4) to[out=-90, in=0] (0.1,0);
	\draw[-] (-0.2,.188) to (-0.123,.22);
	\draw[-] (0.1,0) to[out = 180, in = -90] (-0.2,0.4);
      \node at (-0.2,0.55) {$\scriptstyle{\aa}$};
\end{tikzpicture}
}
\:=
\mathord{
\begin{tikzpicture}[baseline = 1mm]
	\draw[<-] (0.4,0.4) to[out=-90, in=0] (0.1,0);
	\draw[-] (.402,.188) to (.325,.22);
	\draw[-] (0.1,0) to[out = 180, in = -90] (-0.2,0.4);
      \node at (-0.2,0.55) {$\scriptstyle{\aa}$};
\end{tikzpicture}
}\:
,
\quad&
\mathord{
\begin{tikzpicture}[baseline = 1mm]
	\draw[-] (0.4,0.4) to[out=-90, in=0] (0.1,0);
	\draw[-] (-0.2,.188) to (-0.123,.22);
	\draw[->] (0.1,0) to[out = 180, in = -90] (-0.2,0.4);
      \node at (-0.2,0.55) {$\scriptstyle{\aa}$};
\end{tikzpicture}
}
\;=\;
\mathord{
\begin{tikzpicture}[baseline = 1mm]
	\draw[-] (0.4,0.4) to[out=-90, in=0] (0.1,0);
	\draw[-] (.402,.188) to (.325,.22);
	\draw[->] (0.1,0) to[out = 180, in = -90] (-0.2,0.4);
      \node at (0.4,0.55) {$\scriptstyle{\aa}$};
\end{tikzpicture}
}\:
,
\quad&
\mathord{
\begin{tikzpicture}[baseline = 1mm]
	\draw[<-] (0.4,0.4) to[out=-90, in=0] (0.1,0);
	\draw[-] (-0.2,.188) to (-0.123,.22);
	\draw[-] (0.1,0) to[out = 180, in = -90] (-0.2,0.4);
      \node at (0.4,0.55) {$\scriptstyle{\aa}$};
\end{tikzpicture}
}
=
\:\mathord{
\begin{tikzpicture}[baseline = 1mm]
	\draw[<-] (0.4,0.4) to[out=-90, in=0] (0.1,0);
	\draw[-] (.402,.188) to (.325,.22);
	\draw[-] (0.1,0) to[out = 180, in = -90] (-0.2,0.4);
      \node at (0.4,0.55) {$\scriptstyle{\aa}$};
\end{tikzpicture}
}\:
,
\quad&
\mathord{
\begin{tikzpicture}[baseline = 1mm]
	\draw[-] (0.4,0.4) to[out=-90, in=0] (0.1,0);
	\draw[-] (-0.2,.188) to (-0.123,.22);
	\draw[->] (0.1,0) to[out = 180, in = -90] (-0.2,0.4);
      \node at (0.4,0.55) {$\scriptstyle{\aa}$};
\end{tikzpicture}
}
=\:
\mathord{
\begin{tikzpicture}[baseline = 1mm]
	\draw[-] (0.4,0.4) to[out=-90, in=0] (0.1,0);
	\draw[-] (.402,.188) to (.325,.22);
	\draw[->] (0.1,0) to[out = 180, in = -90] (-0.2,0.4);
      \node at (0.4,0.55) {$\scriptstyle{\aa}$};
\end{tikzpicture}
}\:
,
\\\\
\mathord{
\begin{tikzpicture}[baseline = 1mm]
	\draw[<-] (0.4,0) to[out=90, in=0] (0.1,0.4);
	\draw[-] (-0.2,.22) to (-0.123,.188);
	\draw[-] (0.1,0.4) to[out = 180, in = 90] (-0.2,0);
      \node at (-0.2,-0.15) {$\scriptstyle{\aa}$};
\end{tikzpicture}
}\:=\:
\mathord{
\begin{tikzpicture}[baseline = 1mm]
	\draw[<-] (0.4,0) to[out=90, in=0] (0.1,0.4);
	\draw[-] (.402,.22) to (.325,.188);
	\draw[-] (0.1,0.4) to[out = 180, in = 90] (-0.2,0);
      \node at (-0.2,-0.15) {$\scriptstyle{\aa}$};
\end{tikzpicture}
}\:,
\quad&
\mathord{
\begin{tikzpicture}[baseline = 1mm]
	\draw[-] (0.4,0) to[out=90, in=0] (0.1,0.4);
	\draw[-] (-0.2,.22) to (-0.123,.188);
	\draw[->] (0.1,0.4) to[out = 180, in = 90] (-0.2,0);
      \node at (-0.2,-0.15) {$\scriptstyle{\aa}$};
\end{tikzpicture}
}\;=\;
\mathord{
\begin{tikzpicture}[baseline = 1mm]
	\draw[-] (0.4,0) to[out=90, in=0] (0.1,0.4);
	\draw[-] (.402,.22) to (.325,.188);
	\draw[->] (0.1,0.4) to[out = 180, in = 90] (-0.2,0);
      \node at (-0.2,-0.15) {$\scriptstyle{\aa}$};
\end{tikzpicture}
}\:,
\quad&
\mathord{
\begin{tikzpicture}[baseline = 1mm]
	\draw[<-] (0.4,0) to[out=90, in=0] (0.1,0.4);
	\draw[-] (-0.2,.22) to (-0.123,.188);
	\draw[-] (0.1,0.4) to[out = 180, in = 90] (-0.2,0);
      \node at (0.4,-0.15) {$\scriptstyle{\aa}$};
\end{tikzpicture}
}=\:
\mathord{
\begin{tikzpicture}[baseline = 1mm]
	\draw[<-] (0.4,0) to[out=90, in=0] (0.1,0.4);
	\draw[-] (.402,.22) to (.325,.188);
	\draw[-] (0.1,0.4) to[out = 180, in = 90] (-0.2,0);
      \node at (0.4,-0.15) {$\scriptstyle{\aa}$};
\end{tikzpicture}
}\:,
\quad&
\mathord{
\begin{tikzpicture}[baseline = 1mm]
	\draw[-] (0.4,0) to[out=90, in=0] (0.1,0.4);
	\draw[-] (-0.2,.22) to (-0.123,.188);
	\draw[->] (0.1,0.4) to[out = 180, in = 90] (-0.2,0);
      \node at (0.4,-0.15) {$\scriptstyle{\aa}$};
\end{tikzpicture}
}=\:
\mathord{
\begin{tikzpicture}[baseline = 1mm]
	\draw[-] (0.4,0) to[out=90, in=0] (0.1,0.4);
	\draw[-] (.402,.22) to (.325,.188);
	\draw[->] (0.1,0.4) to[out = 180, in = 90] (-0.2,0);
      \node at (0.4,-0.15) {$\scriptstyle{\aa}$};
\end{tikzpicture}
}\:.
\end{array}\label{incpro}
\end{align}

The situation with crossings is more interesting.
For $\aa,\bb,\aa',\bb' \in \K$, define
\begin{equation}
\begin{tikzpicture}[baseline = -1mm]
	\draw[->] (0.28,-.28) to (-0.28,.28);
	\draw[->] (-0.28,-.28) to (0.28,.28);
      \node at (-0.33,-0.43) {$\scriptstyle{\bb}$};
      \node at (0.33,-0.43) {$\scriptstyle{\aa}$};
      \node at (-0.33,0.43) {$\scriptstyle{\bb'}$};
      \node at (0.33,0.43) {$\scriptstyle{\aa'}$};
\node at (0,-0.01) {$\diamond$};
\end{tikzpicture}
:=
\begin{tikzpicture}[baseline = -1mm]
	\draw[->] (0.28,-.28) to (-0.28,.28);
	\draw[->] (-0.28,-.28) to (0.28,.28);
      \node at (-0.33,-0.43) {$\scriptstyle{\bb}$};
      \node at (0.33,-0.43) {$\scriptstyle{\aa}$};
      \node at (-0.33,0.43) {$\scriptstyle{\bb'}$};
      \node at (0.33,0.43) {$\scriptstyle{\aa'}$};
	\draw[-] (.14,.22) to (.22,.14);
	\draw[-] (.14,-.22) to (.22,-.14);
	\draw[-] (-.14,.22) to (-.22,.14);
	\draw[-] (-.14,-.22) to (-.22,-.14);
\end{tikzpicture}
\label{interesting}\end{equation}
Thus, these natural transformations are defined by first including
the summand $\eE_{\bb}\eE_{\aa}$ into $\eE \eE$, then
applying natural transformation $\eE \eE \Rightarrow \eE \eE$ defined by
the usual crossing, then
projecting
$\eE \eE$ onto the summand $\eE_{\bb'}\eE_{\aa'}$.
The defining relations plus \cref{sizzle} imply
that
\begin{equation}\label{newdotslide}
\begin{tikzpicture}[baseline = -1mm]
	\draw[->] (0.28,-.28) to (-0.28,.28);
	\draw[->] (-0.28,-.28) to (0.28,.28);
      \node at (-0.33,-0.43) {$\scriptstyle{\bb}$};
      \node at (0.33,-0.43) {$\scriptstyle{\aa}$};
      \node at (-0.33,0.43) {$\scriptstyle{\bb'}$};
      \node at (0.33,0.43) {$\scriptstyle{\aa'}$};
\node at (0,-.01) {$\diamond$};
\node at (-.17,-.19) {$\bullet$};
\end{tikzpicture}
=
\begin{tikzpicture}[baseline = -1mm]
	\draw[->] (0.28,-.28) to (-0.28,.28);
	\draw[->] (-0.28,-.28) to (0.28,.28);
      \node at (-0.33,-0.43) {$\scriptstyle{\bb}$};
      \node at (0.33,-0.43) {$\scriptstyle{\aa}$};
      \node at (-0.33,0.43) {$\scriptstyle{\bb'}$};
      \node at (0.33,0.43) {$\scriptstyle{\aa'}$};
\node at (0,-.01) {$\diamond$};
\node at (.17,.15) {$\bullet$};
\end{tikzpicture}
+
\delta_{\aa,\aa'} \delta_{\bb,\bb'}
\begin{tikzpicture}[baseline = -1mm]
	\draw[->] (0.18,-.28) to (0.18,.28);
	\draw[->] (-0.18,-.28) to (-0.18,.28);
      \node at (-0.18,-0.43) {$\scriptstyle{\bb}$};
      \node at (0.18,-0.43) {$\scriptstyle{\aa}$};
\end{tikzpicture},\qquad
\begin{tikzpicture}[baseline = -1mm]
	\draw[->] (0.28,-.28) to (-0.28,.28);
	\draw[->] (-0.28,-.28) to (0.28,.28);
      \node at (-0.33,-0.43) {$\scriptstyle{\bb}$};
      \node at (0.33,-0.43) {$\scriptstyle{\aa}$};
      \node at (-0.33,0.43) {$\scriptstyle{\bb'}$};
      \node at (0.33,0.43) {$\scriptstyle{\aa'}$};
\node at (0,-.01) {$\diamond$};
\node at (-.17,.15) {$\bullet$};
\end{tikzpicture}
=
\begin{tikzpicture}[baseline = -1mm]
	\draw[->] (0.28,-.28) to (-0.28,.28);
	\draw[->] (-0.28,-.28) to (0.28,.28);
      \node at (-0.33,-0.43) {$\scriptstyle{\bb}$};
      \node at (0.33,-0.43) {$\scriptstyle{\aa}$};
      \node at (-0.33,0.43) {$\scriptstyle{\bb'}$};
      \node at (0.33,0.43) {$\scriptstyle{\aa'}$};
\node at (0,-.01) {$\diamond$};
\node at (.18,-.19) {$\bullet$};
\end{tikzpicture}
+
\delta_{\aa,\aa'} \delta_{\bb,\bb'}
\begin{tikzpicture}[baseline = -1mm]
	\draw[->] (0.18,-.28) to (0.18,.28);
	\draw[->] (-0.18,-.28) to (-0.18,.28);
      \node at (-0.18,-0.43) {$\scriptstyle{\bb}$};
      \node at (0.18,-0.43) {$\scriptstyle{\aa}$};
\end{tikzpicture}
\end{equation}
There are also sideways and downwards versions of the new crossings
that may be defined in a similar way, or equivalently by ``rotating'' the upwards
ones using \cref{incpro}.
The same proofs as in \cite[Lem. 4.1--2]{brundanHeisenbergKacMoody2020} show that:
\begin{lemma}\label{essential}
If $\{\aa,\bb\}\neq\{\aa',\bb'\}$ then the natural transformation \cref{interesting}
is zero.
The same
holds for the rotated versions of these crossings.

If  $\bb=\bb' \neq \aa=\aa'$, we have that
$$
\begin{tikzpicture}[baseline = -1mm]
	\draw[->] (0.28,-.28) to (-0.28,.28);
	\draw[->] (-0.28,-.28) to (0.28,.28);
      \node at (-0.33,-0.43) {$\scriptstyle{\bb}$};
      \node at (0.33,-0.43) {$\scriptstyle{\aa}$};
      \node at (-0.33,0.43) {$\scriptstyle{\bb}$};
      \node at (0.33,0.43) {$\scriptstyle{\aa}$};
\node at (0,-.01) {$\diamond$};
\end{tikzpicture}
=
\begin{tikzpicture}[baseline = -0.8mm]
	\draw[->,densely dotted] (0.55,.03) to (0.15,.03);
	\draw[->] (0.58,-.25) to (0.58,.3);
	\draw[->] (0.08,-.25) to (0.08,.3);
      \node at (0.08,-0.4) {$\scriptstyle{\bb}$};
      \node at (0.08,0.02) {$\bullet$};
      \node at (0.58,-0.4) {$\scriptstyle{\aa}$};
      \node at (0.58,0.02) {$\bullet$};
      \node at (1.3,0.04) {$\scriptstyle{(x_2-x_1)^{-1}}$};
\end{tikzpicture}
$$
\end{lemma}

\subsection{Bubbles and central characters}\label{swd} Any dotted bubble 
defines an endomorphism of the identity functor
$\operatorname{Id}_{k}$, i.e., an element of the center of
the category $\Ccat_{k}$.
In particular, for $V \in \Ccat_{k}$, dotted bubbles evaluate to
elements of the center $Z_V$ of the endomorphism
algebra $\End_{\Ccat_{k}}(V)$.
It is convenient to
work with all of these endomorphisms at once in terms of the
generating function
\notation{$\h_V(z)$}{The action of the power series $\anticlock(z)$ on the object $V$.}
\begin{align}\label{jjj}
\h_V(z)&:=
\mathord{
\begin{tikzpicture}[baseline = -1mm]
     \node at (0.08,0) {$\scriptstyle\anticlock(z)$};
 	\draw[-,darkg,thick] (0.68,.2) to (0.68,-.22);
     \node at (0.68,-.37) {$\darkg\scriptstyle{V}$};
\end{tikzpicture}
}=
\left(\mathord{
\begin{tikzpicture}[baseline = -1mm]
     \node at (0.08,0) {$\scriptstyle\clock(z)$};
 	\draw[-,darkg,thick] (0.68,.2) to (0.68,-.22);
     \node at (0.68,-.37) {$\darkg\scriptstyle{V}$};
\end{tikzpicture}
}\right)^{-1}.
\end{align}
Recalling \cref{chink} and
\cref{chunk},
we have $\h_V(z) \in z^{k} + z^{k-1} Z_V[\![z^{-1}]\!]$ when $V$ has weight $k$.
In the following lemma, 
given a polynomial
$p(z) = \sum_{s=0}^r z_s z^{r-s} \in Z_V[z]$,
we let
\begin{align*}
\mathord{
\begin{tikzpicture}[baseline = -1mm]
 	\draw[->] (0.08,-.3) to (0.08,.3);
     \node at (0.08,0.01) {$\bullet$};
    \node at (-0.3,0.01) {$\scriptstyle p(x)$};
 	\draw[-,darkg,thick] (0.45,.3) to (0.45,-.3);
     \node at (0.45,-0.45) {$\darkg\scriptstyle{V}$};
\end{tikzpicture}
}
&:=
\sum_{s=0}^r 
\mathord{
\begin{tikzpicture}[baseline = -1mm]
 	\draw[->] (0.08,-.3) to (0.08,.3);
     \node at (0.08,0.01) {$\bullet$};
    \node at (-0.27,0.01) {$\scriptstyle x^{r-s}$};
 	\draw[-,darkg,thick] (0.45,.15) to (0.45,.3);
 	\draw[-,darkg,thick] (0.45,-.15) to (0.45,-.3);
     \node at (0.45,-.45) {$\darkg\scriptstyle{V}$};
      \draw[darkg,thick] (0.45,0.0) circle (4.5pt);
   \node at (0.45,0) {$\darkg\scriptstyle{z_s}$};
\end{tikzpicture}
},
&
\mathord{
\begin{tikzpicture}[baseline = -1mm]
 	\draw[<-] (0.08,-.3) to (0.08,.3);
     \node at (0.08,0.01) {$\bullet$};
    \node at (-0.3,0.01) {$\scriptstyle p(x)$};
 	\draw[-,darkg,thick] (0.45,.3) to (0.45,-.3);
     \node at (0.45,-.45) {$\darkg\scriptstyle{V}$};
\end{tikzpicture}
}
&:=
\sum_{s=0}^r
\mathord{
\begin{tikzpicture}[baseline = -1mm]
 	\draw[<-] (0.08,-.3) to (0.08,.3);
     \node at (0.08,0.01) {$\bullet$};
    \node at (-0.27,0.01) {$\scriptstyle x^{r-s}$};
 	\draw[-,darkg,thick] (0.45,.15) to (0.45,.3);
 	\draw[-,darkg,thick] (0.45,-.15) to (0.45,-.3);
      \draw[darkg,thick] (0.45,0.0) circle (4.5pt);
   \node at (0.45,0) {$\darkg\scriptstyle{z_s}$};
     \node at (0.45,-.45) {$\darkg\scriptstyle{V}$};
\end{tikzpicture}
}.
\end{align*}
\cref{sure2lemma} obviously extends to the setting of coefficients in $Z_V$.

Most categorical $\mathfrak{sl}_2$-actions with which readers are familiar will assume that $\mathord{
\begin{tikzpicture}[baseline = -1mm]
 	\draw[<-] (0.08,-.3) to (0.08,.3);
     \node at (0.08,0.01) {$\bullet$};
\end{tikzpicture}
}$ is nilpotent, but this is by no means necessary.  On any object $V$, the endomorphism $\mathord{
\begin{tikzpicture}[baseline = -1mm]
 	\draw[<-] (0.08,-.3) to (0.08,.3);
     \node at (0.08,0.01) {$\bullet$};
 	\draw[-,darkg,thick] (0.45,.3) to (0.45,-.3);
     \node at (0.45,-.45) {$\darkg\scriptstyle{V}$};
\end{tikzpicture}
}$ must satisfy a polynomial relation with coefficients in $\K$, and thus has a finite spectrum \[U_V=\{a\in \K \mid \mathord{
\begin{tikzpicture}[baseline = -1mm]
 	\draw[<-] (0.08,-.3) to (0.08,.3);
    \node at (-0.27,0.01) {$\scriptstyle x-a$};
     \node at (0.08,0.01) {$\bullet$};
 	\draw[-,darkg,thick] (0.45,.3) to (0.45,-.3);
     \node at (0.45,-.45) {$\darkg\scriptstyle{V}$};
\end{tikzpicture}}\text{ not invertible}\}.\] Let $U=\cup_{V}U_V$ be the union of these spectra for all objects.  

Repeating the proof of \cite[Lem. 4.3--4]{brundanHeisenbergKacMoody2020} with minor modifications shows that:
\begin{lemma}\label{preimpy}
Let $V \in \Ccat_{k}$ be any object of weight $k$:
\begin{itemize}
\item[(1)]
If $f(z) \in Z_V[z]$ is a monic polynomial such that
$\mathord{
\begin{tikzpicture}[baseline = -1mm]
 	\draw[->] (0.08,-.2) to (0.08,.2);
     \node at (0.08,0.01) {$\bullet$};
    \node at (-0.3,0.01) {$\scriptstyle f(x)$};
 	\draw[-,darkg,thick] (0.32,.2) to (0.32,-.2);
     \node at (0.49,0) {$\darkg\scriptstyle{V}$};
\end{tikzpicture}
}=0$,
then $g(z) := \h_V(z) f(z)$ is a monic polynomial
in $Z_V[z]$
of degree $\deg f(z)+k$ 
such that
${\begin{tikzpicture}[baseline = -1mm]
 	\draw[<-] (0.08,-.2) to (0.08,.2);
    \node at (-0.3,0.01) {$\scriptstyle g(x)$};
     \node at (0.08,0.01) {$\bullet$};
 	\draw[-,darkg,thick] (0.32,.2) to (0.32,-.2);
     \node at (0.49,0) {$\darkg\scriptstyle{V}$};
\end{tikzpicture}
}=0$.
\item[(2)]
If $g(z) \in Z_V[z]$ is a monic polynomial such that
$\mathord{
\begin{tikzpicture}[baseline = -1mm]
 	\draw[<-] (0.08,-.2) to (0.08,.2);
    \node at (-0.3,0.01) {$\scriptstyle g(x)$};
     \node at (0.08,0.01) {$\bullet$};
 	\draw[-,darkg,thick] (0.32,.2) to (0.32,-.2);
     \node at (0.49,0) {$\darkg\scriptstyle{V}$};
\end{tikzpicture}
}=0$, then $f(z) := \h_V(z)^{-1} g(z)$ is a monic polynomial
in $Z_V[z]$
of degree $\deg g(z)-k$ 
such that
$\mathord{
\begin{tikzpicture}[baseline = -1mm]
 	\draw[->] (0.08,-.2) to (0.08,.2);
     \node at (0.08,0.01) {$\bullet$};
    \node at (-0.3,0.01) {$\scriptstyle f(x)$};
 	\draw[-,darkg,thick] (0.32,.2) to (0.32,-.2);
     \node at (0.49,0) {$\darkg\scriptstyle{V}$};
\end{tikzpicture}
}=0$.
\end{itemize}

For any simple module $L \in \Ccat_{k}$,  we can think of $\h_L(z)$ as a power series in $\K((z^{-1}))$, and 
the only power series satisfying properties (1) and (2) is thus the rational function:
$$
\h_L(z) = n_L(z)/m_L(z).
$$
\end{lemma}

Now we will finally encounter a point where there is a substantive difference with \cite{brundanHeisenbergKacMoody2020} (though a strong analogy remains):
\begin{lemma}\label{choose}
Suppose that $L \in \Ccat_{k}$ is an irreducible object and
let $K$ be an irreducible subquotient of $\eE_{\aa} L$ for some $\aa \in  \K$ and $K'$ an irreducible subquotient of $\eF_{\aa} L$.
Then
\begin{equation} \label{choose1}
    \h_K(z)
    = \h_L(z) (z-\aa)^2 \qquad     \h_{K'}(z)
    = \h_L(z) (z-\aa)^{-2}.
\end{equation}
\end{lemma}

\begin{proof}
This follows from the bubble slides \cref{bubslides1}.  For $\eE_{\aa}$, we note that:
$$
\mathord{
\begin{tikzpicture}[baseline = -1mm]
     \node at (0.08,0) {$\scriptstyle\anticlock(z)$};
 	\draw[-,darkg,thick] (1.1,.3) to (1.1,-.3);
 	\draw[<-] (0.68,.3) to (0.68,-.32);
     \node at (.68,-.45) {$\scriptstyle{\aa}$};
     \node at (1.1,-.45) {$\darkg\scriptstyle{L}$};
\end{tikzpicture}
}=
\mathord{
\begin{tikzpicture}[baseline = -1mm]
     \node at (0.6,0) {$\scriptstyle\anticlock(z)$};
\draw[-,darkg,thick] (1.1,.3) to (1.1,-.3);
 	\draw[<-] (0.08,.3) to (0.08,-.3);
    \node at (-.33,0.03) {$\scriptstyle (z-x)^2\:\bullet$};
     \node at (1.1,-.45) {$\darkg\scriptstyle{L}$};
     \node at (.08,-.45) {$\scriptstyle{\aa}$};
\end{tikzpicture}
}
=
\mathord{
\begin{tikzpicture}[baseline = -1mm]
 	\draw[-,darkg,thick] (.5,.3) to (.5,-.3);
 	\draw[<-] (0.08,.3) to (0.08,-.3);
     \node at (-.665,0.03) {$
\scriptstyle \h_L(z) (z-x)^2\:\bullet$};
     \node at (.5,-.45) {$\darkg\scriptstyle{L}$};
     \node at (.08,-.45) {$\scriptstyle{\aa}$};
\end{tikzpicture}
}.
$$
When we pass to the irreducible subquotient $K$ of $\eE_{\aa} L$, 
we can replace the occurrences of $x$ in the expression on the right-hand side
here with $\aa$, and the lemma follows. The proof for $\eF_{\aa}$ is identical with orientations reversed.
\end{proof}

Just as in \cite[\S 4.2]{brundanHeisenbergKacMoody2020}, we can decompose the weight categories $\Ccat_{k}$ in a way that is compatible with the functors $\eE_{\aa}$ and $\eF_{\aa}$:  for any simple $L$, let $k_{\aa}(L)=\phi_{\aa}(L)-\eps_{\aa}(L)$ for each $\aa\in U$.  In other words,
$k_{\aa}(L) \in \Z$ is the multiplicity of $z=\aa$ as a zero
or pole of the rational function $\h_L(z) \in \K(z)$ for each $\aa \in U$.  
\begin{definition}\label{def:Cm}
	 For a fixed $\Bm=(m_{\aa})\in \Z^U$, we let $\Ccat_{\Bm}$ be the Serre
subcategory of $\Ccat_m$ for $m=\sum_{\aa \in U} m_{\aa}$ consisting of the objects $V$
such that every irreducible subquotient $L$ of $V$
satisfies $k_{\aa}(L)=m_{\aa}$. 
\end{definition} 
\begin{lemma}\label{lem:block}
	Any indecomposable object $M$ in $\Ccat_m$ lies in one of these categories.
\end{lemma}
\begin{proof}
  Since $\Ccat_m$ is Schurian or locally finite abelian over $\K$, it is Krull-Schmidt and 
  the endomorphism algebra of any indecomposable finitely generated object in $\Ccat_m$ is local.  Furthermore, since $\K$ is algebraically closed, the residue field of this local ring is $\K$.   The image of $\h_V(z)$ under the unique $\K$-algebra map $\End(V)\to \K$ gives a power series in $\K((z^{-1}))$ which must coincide with $\h_L(z)$ for any simple subquotient $L$ of $M$. 
\end{proof}


In terms of categorifications, we think of $U$ as a Dynkin diagram with no edges, thus corresponding to the Lie algebra $\tilde{\fg}=\mathfrak{sl}_2^U$.  It is worth noting the contrast with the similar analysis in \cite[\S 4.2]{brundanHeisenbergKacMoody2020};  in that case we naturally ended up with a non-trivial Dynkin diagram.  This is most easily seen by comparing \cref{choose1} with \cite[(4.11)]{brundanHeisenbergKacMoody2020}; in both cases, the entries of the Cartan matrix can be read off from the poles and zeros of the ratio of the bubble power series acting on the two simples.  
Thus, we can think of $\Bm$ as a weight over $\tilde{\fg}=\mathfrak{sl}_2^U$, and in particular, we have the simple roots $\al_{\aa}=(0,\dotsc, 2,\dotsc, 0)$.  
By \cref{choose}, for $\Bm\in \Z^{U}$ and $\aa \in U$, the restrictions of $\eE_{\aa}$ and $\eF_{\aa}$ to $\Ccat_{\Bm}$ give functors
\begin{align*}
E_{\aa}|_{\Ccat_{\Bm}}&:\Ccat_{\Bm} \rightarrow \Ccat_{\Bm+\alpha_{\aa}},&
F_{\aa}|_{\Ccat_{\Bm}}&:\Ccat_{\Bm} \rightarrow \Ccat_{\Bm-\alpha_{\aa}},
\end{align*}
\begin{proposition} \label{prop:KM3}
The functors $\eE_{\aa}$ give a categorical action of $\mathfrak{sl}_2^U$.
\end{proposition}
We give the proof of this result on page \pageref{proof-prop:KM3}.  

\subsection{Higher rank}

Now we turn to constructing a $\tU(\tilde \fg)$-action.  Given that most of the difficult relations in $\tU(\tilde \fg)$ only concern strands with a single label, it is simple to extend this result to higher rank.  So, as in \cref{sec:eigenfunctors}, we consider an action of $\tU(\fg)$ with weight categories $\Ccat_{\lambda}$ which are Schurian or locally finite abelian.  

We apply the constructions of the previous two sections for each $i\in I$.  In particular, we consider the generalized eigenspaces as in  \cref{def:Ea}:
\begin{definition}\label{def:Eiu}
Let $\eE_{i,u}$ be the $u$-generalized eigenspace of
$y$ acting on $\eE_i$, and similarly, $\eF_{i,u}$ the analogous
eigenspace for $\eF_i$. 

Let $U_i=\{u\in \K\mid \eE_{i,u}\neq 0\}\subset \K$; note that since
  the corresponding functors are adjoint, $\eE_{i,u}$ or $\eF_{i,u}$
can be used symmetrically in this definition. 
\end{definition}

Consider the locally finite graph with vertex set $\tilde{I}$, where $\tilde{I}$ is the set of pairs $(i,u)$ with $u\in U_i$ constructed from the polynomials $P_{ij}$ as in \cref{def:tilde-I}.  For each $i$ and simple $L\in \Ccat_{\nu}$ for some $\nu \in \rola$, we have the statistic $k_{i,u}(L)$ defined as the vanishing order of $\h_L(z)$ at $z=u$.

  We assign to this simple $L$ the unique weight $\mu_L$ 
 in the weight lattice $\tilde{\rola}$ for
$\tilde{I}$ such that $\al_{i,u}^\vee(\mu_L)=k_{i,u}(L)$.  
\begin{definition}
	\label{def:subcat} Let\; $\Ccat_{(\mu)}$ be the Serre subcategory of objects whose simple sub\-quo\-tients $L$ all satisfy $\mu_L=\mu$.\notation{$\Ccat_{(\mu)}$}{The Serre subcategory of objects with all simple subquotients of weight $\mu$.}
\end{definition}
By the same logic as \cref{lem:block}, every indecomposable module lies in one of these subcategories, we can define a block decomposition of our category $\Ccat_\nu=\oplus_{\mu|_{\mathfrak{h}}=\nu}\Ccat_{(\mu)}$.  
  \begin{lemma}\label{lem:weight-act}
  	  The functors $\eE_{i,u} ,\eF_{i,u}$ send objects in $\Ccat_{(\mu)}$ to $\Ccat_{(\mu\pm \al_{i,u})}$.
  \end{lemma}
\begin{proof}
	This follows immediately from the bubble slides.  It follows immediately from \cref{bubslides1} that $k_{i,u'}(\eE_{i,u}L)=2\delta_{u,u'}$, and $k_{i,u'}(\eF_{i,u}L)=-2\delta_{u,u'}$ as expected.  On the other hand, \cref{bubslides2} shows that $k_{j,u'}(\eE_{i,u}L)$ decreases by exactly the vanishing order of $Q_{ji}(x,y)$ at $x=u'$ and $y=u$, which is the same as the number of edges joining $(i, u)$ and $(j,u')$ in $\tilde{I}$.  
\end{proof}

This result allows us to prove an analogue of \cite[Lem. 4.6]{brundanHeisenbergKacMoody2020}:
\begin{lemma}\label{lem:complete}
Given any categorical $\fg$-action with weight categories $\Ccat_\la$, the sets $U_i$ defined in \cref{def:Eiu} are complete in the sense of \cref{def:tilde-I}.
\end{lemma}
\begin{proof}
  First, note that by \cref{preimpy}, if the power series $\h_L(z)$ has a zero or pole at $z=u$ for some simple $L$, then the functors $\eE_{i,u}$ and $\eF_{i,u}$ are non-zero, so $u\in U_i$. Thus, by \cref{lem:weight-act}, 
  if $\eE_{j,u'}\neq 0$, then there exists $\mu$ such that the categories $\Ccat_{(\mu)}$ and $\Ccat_{(\mu+\al_{j,u'})}$ are both non-zero.  If the vanishing order of $Q_{ij}(x,y)$ at $(x,y)=(u,u')$ is positive, then by definition $\al_{i,u}^\vee(\al_{j,u'})\neq 0$, so $\al_{i,u}^\vee$ cannot vanish on both $\mu$ and $\mu+\al_{j,u'}$.  Hence there must be some simple $L$ in one of these categories such that $k_{i,u}(L)\neq 0$.  Thus, $\eE_{i,u}\neq 0$ and $u\in U_i$.  This shows that the sets $U_i$ are complete.
\end{proof}

Thus, we have weight categories $\Ccat_{(\mu)}$ indexed by weights of $\tilde{\fg}$.  
The eigenspace functors $\eE_{i,u}, \eF_{i,u},$ defined as in
\cref{def:Eiu}, act as expected on weights.  
\begin{theorem}\label{thm:deform-action}
The functors $\eE_{i,u}$ and $\eF_{i,u}$ and the weight space categories
$\Ccat_{(\mu)}$ define a categorical action of $\tilde{\fg}$.
\end{theorem}

\begin{proof}
We will use \cref{thm:brundandef} to show that we have a categorical action.  Checking the conditions given in that Theorem:
\begin{itemize}
    \item [(KM1)] The adjunction of  $\eE_{i,u}$ and $\eF_{i,u}$ follows immediately from the adjunction of
  $\eE_i$ and $\eF_i$ and the equation \cref{adjdot1}. 
  \item[(KM2)] As part of the structure of a categorical
action, $R_n$ acts on the $n$th power $\eE^n$. Since the action of any
dot on $\eE_i$ satisfies a polynomial relation with roots in $U_i$, this
extends to an action of $\widehat{R}_n$, the completion defined in \cref{def:widehat-R}. 
     By transport of structure using the isomorphism $\gamma$ of
    \cref{prop:nu-iso}, we have an induced action of
     ${\mathsf{R}}_n$ such that $\mathsf{y}$ is nilpotent.
    \item[(KM3)] This is simply \cref{prop:KM3}.
\end{itemize}
This completes the proof.  
  \end{proof}

\section{Deformed tensor product algebras}
\label{sec:deform-tens-prod}

\subsection{The definition}
\label{sec:definition}

In \cite[\S 5]{WebCB}, we introduced natural categorifications $\mathcal{X}^\bla$ for tensor
products of highest and lowest weight representations.  These
categorifications have natural deformations, which we now study in
the context of the previous section.  As mentioned in the introduction, we require certain non-degeneracy results for these algebras, analogous to \cref{nondegenerate}, in order to show that $\mathcal{X}^\bla$ has the expected Grothendieck group and Hom spaces.  In an earlier version of this article, we used the results of this section to prove \cref{nondegenerate}.  We still believe that this gives the more conceptual picture of the proof, although the version in \cref{sec:proof} requires less technical overhead.  

In \cite[\S 4]{WebCB}, we introduced {\bf tricolore diagrams}, which we briefly recall here (see \cite[Def. 4.2]{WebCB} for the full definition). 
\begin{definition}
  A {\bf tricolore diagram} is a collection of finitely many oriented curves in $\R\times [0,1]$ that satisfy the usual genericity relations. Each curve is either
  \begin{itemize}
  \item colored red and labeled with a dominant weight of $\fg$, or
  \item colored blue\footnote{The reader might wonder if these red and blue lines have anything to do with the red and blue copies of $\tU(\fg)$ used in \cref{sec:coproduct}.  It would make a lot of sense if that were the case, but there is no connection.  We are just very unoriginal about choosing colors.} and labeled with an anti-dominant weight of
    $\fg$, or 
  \item colored black and labeled with $i\in I$ and decorated with finitely many dots.
\end{itemize}
We will consistently use $n$ for the number of black strands and $\ell$ for the number of red and blue.
We also include a labeling of regions
by weights, consistent with the rules 
\begin{equation*}
  \tikz[baseline,very thick]{
\draw[wei,postaction={decorate,decoration={markings,
    mark=at position .5 with {\arrow[scale=1.3]{<}}}}] (0,-.5) -- node[below,at start]{$\la$}  (0,.5);
\node at (-1,0) {$\mu$};
\node at (1,.05) {$\mu+\la$};
}\qquad \qquad   \tikz[baseline,very thick]{
\draw[awei,postaction={decorate,decoration={markings,
    mark=at position .7 with {\arrow[scale=1.3]{>}}}}] (0,-.5) -- node[below,at start]{$-\la$}  (0,.5);
\node at (-1,0) {$\mu$};
\node at (1,.05) {$\mu-\la$};
}\qquad \qquad 
  \tikz[baseline,very thick]{
\draw[postaction={decorate,decoration={markings,
    mark=at position .5 with {\arrow[scale=1.3]{<}}}}] (0,-.5) -- node[below,at start]{$i$}  (0,.5);
\node at (-1,0) {$\mu$};
\node at (1,.05) {$\mu-\al_i$};
}
\end{equation*}
Since this labeling is fixed as soon as one region is labeled, we
typically do not draw weights in all regions, in order to
simplify pictures.

The red and blue strands must map diffeomorphically to $[0,1]$ after forgetting the $x$-value and are forbidden to intersect with any other red or blue strand.  The
black strands are allowed to close into circles, self-intersect,
intersect red and blue strands, etc.

As usual, we record the horizontal slices at $y=0$ and
$y=1$, the {\bf bottom} and {\bf top} of the diagram.  This will be
encoded as a {\bf tricolore quadruple}, consisting of 
\begin{itemize}
\item A sequence $\Bi\in (\pm I)^n$ of simple roots and their
  negatives on black strands, read from the left;
\item A sequence $\bla\in (\wela^{\pm})^\ell$ of dominant or
  anti-dominant weights on red and blue strands, read from the left;
	\item A weakly increasing function $\kappa\colon [1,\ell]\to [0,n]$
  such that $\kappa(m)$ is the number of black strands to the left of the $m$th red or blue strand (both counted
  from the left). By
convention, we write $\kappa(i)=0$ if the $i$th red or blue strand is left of all
black strands.  
\item Weights $\EuScript{L}$ and $\EuScript{R}$ at the far left and far right of the diagram, respectively.  These are related by
\[\EuScript{L}+\sum_{k=1}^\ell\la_k+\sum_{m=1}^n\al_{i_m}=\EuScript{R}.\]
\end{itemize}
\end{definition}  

Tricolore diagrams are endowed with horizontal and vertical
composition operations, just like KL and DS diagrams; similarly,
tricolore quadruples are endowed with horizontal composition.  
These naturally make tricolore quadruples the 1-morphisms and tricolore diagrams the 2-morphisms of a 2-category
$\doubletilde{\mathcal{T}}$ whose objects are the integral weights $\wela$.  

The categorifications $\mathcal{X}^\bla$ are natural
subquotients of this category, and our deformed categorifications
arise from a straightforward deformation of the relations from
\cite[\S 4]{WebCB}, which we present below:

\begin{definition}\label{def:hl}
  Let $\EuScript{T}$ be the quotient of $\doubletilde{\mathcal{T}}\otimes \K[\zw_1,\dots, \zw_\ell]$
 by the relations \crefrange{QHA1}{QHA3},\crefrange{rightadj}{flight2} on black strands
 and \crefrange{color-opp-cancel}{red-triple} below
 relating red and blue strands to black.  Note that the relations \crefrange{color-opp-cancel}{red-triple}
 are deformations of the relations of $\cT$ in \cite[4.3]{WebCB}; thus we will recover the category $\cT$ if we specialize $\zw_i=0$.
\newseq

    \begin{equation*}\label{dcost2}\subeqn
    \begin{tikzpicture}[very thick,baseline]
      \draw[postaction={decorate,decoration={markings, mark=at
          position .51 with {\arrow[scale=1.3]{<}}}}] (-2.8,0) +(0,-1)
      .. controls (-1.2,0) ..  +(0,1) node[below,at
      start]{$i$}; \draw[wei] (-1.2,0) +(0,-1) .. controls (-2.8,0) ..
      +(0,1) node[below,at
      start]{$\la_k$}; \end{tikzpicture}\hspace{-1mm}=\hspace{-1mm}
    \begin{tikzpicture}[very thick,baseline] \draw[wei] (2,0)
      +(0,-1) -- +(0,1) node[below,at
      start]{$\la_k$};
            \draw[postaction={decorate,decoration={markings, mark=at
          position .2 with {\arrow[scale=1.3]{<}}}}] (1,0) +(0,-1) --
      node[midway,draw, fill=black, scale=.7, inner sep=2pt,
      circle,label=left:{$\scriptstyle{ (x-\zw_k)^{\la^i_k}}$}]{} +(0,1) node[below,at
      start]{$i$}; 
    \end{tikzpicture}\qquad\qquad\qquad
    \begin{tikzpicture}[very thick,baseline]
      \draw[postaction={decorate,decoration={markings, mark=at
          position .5 with {\arrow[scale=1.3]{>}}}}] (-2.8,-1)
      .. controls (-1.2,0) ..  (-2.8,1) node[below,at
      start]{$i$}; \draw[awei] (-1.2,-1) .. controls (-2.8,0)
      ..  (-1.2,1) node[below,at
      start]{$-\la_k$}; \end{tikzpicture}\hspace{-1mm}=\hspace{-1mm}
    \begin{tikzpicture}[very thick,baseline] \draw[awei] (2.4,0)
      +(0,-1) -- +(0,1) node[below,at
      start]{$-\la_k$};
      \draw[postaction={decorate,decoration={markings, mark=at
          position .8 with {\arrow[scale=1.3]{>}}}}] (1.2,0) +(0,-1)
      -- node[midway,draw, fill=black, scale=.7, inner sep=2pt,
      circle,label=left:{$\scriptstyle{ (x-\zw_k)^{\la^i_k}}$}]{} +(0,1) node[below,at
      start]{$i$}; 
    \end{tikzpicture}
  \end{equation*}
  \begin{equation*}\subeqn
\begin{tikzpicture} [scale=.9,baseline]
\node at (9,0){\begin{tikzpicture} [scale=1.1]
\draw[postaction={decorate,decoration={markings,
    mark=at position .5 with {\arrow[scale=1.3]{<}}}},very thick] (0,0) to[out=90,in=-90] node[at start,below]{$i$} (1,1) to[out=90,in=-90] (0,2) ;  
\draw[awei] (1,0) to[out=90,in=-90]  (0,1) to[out=90,in=-90] (1,2);
\end{tikzpicture}};
\node at (10.7,0) {$=$};
\node at (12.4,0){\begin{tikzpicture} [scale=1.3]

\draw[postaction={decorate,decoration={markings,
    mark=at position .5 with {\arrow[scale=1.3]{<}}}},very thick] (1,0)--  (1,2) node[at start,below]{$i$};  
\draw[awei] (1.7,0) --  (1.7,2);
\end{tikzpicture}};
\node at (0,0){\begin{tikzpicture} [scale=1.3]
\draw[postaction={decorate,decoration={markings,
    mark=at position .5 with {\arrow[scale=1.3]{>}}}},very thick] (0,0) to[out=90,in=-90] node[at start,below]{$i$} (1,1) to[out=90,in=-90] (0,2) ;  
\draw[wei] (1,0) to[out=90,in=-90] (0,1) to[out=90,in=-90] (1,2);
\end{tikzpicture}};

\node at (1.7,0) {$=$};
\node at (3.4,0){\begin{tikzpicture} [scale=1.3]

\draw[postaction={decorate,decoration={markings,
    mark=at position .5 with {\arrow[scale=1.3]{>}}}},very thick] (1,0) -- (1,2) node[at start,below]{$i$};  
\draw[wei] (1.7,0) --  (1.7,2) ;
\end{tikzpicture}};
\end{tikzpicture}
\label{color-opp-cancel}
\end{equation*} 
\begin{equation*}\subeqn \label{fig:pass-through1}
   \begin{tikzpicture}[yscale=.9,baseline]
      \node at (4,0){    \begin{tikzpicture}[very thick]
          \draw[postaction={decorate,decoration={markings,
    mark=at position .2 with {\arrow[scale=1.3]{>}}}}] (-3,0) +(1,-1) -- +(-1,1);
          \draw[postaction={decorate,decoration={markings,
    mark=at position .8 with {\arrow[scale=1.3]{<}}}}] (-3,0) +(-1,-1) -- +(1,1);
          \draw[awei] (-3,0) +(0,-1) .. controls (-4,0) ..  +(0,1);
          \node at (-1.5,0) {=}; \draw[postaction={decorate,decoration={markings,
    mark=at position .8 with {\arrow[scale=1.3]{>}}}}] (0,0) +(1,-1) -- +(-1,1); \draw[postaction={decorate,decoration={markings,
    mark=at position .2 with {\arrow[scale=1.3]{<}}}}] (0,0) +(-1,-1) -- +(1,1); \draw[awei] (0,0) +(0,-1) .. controls
          (1,0) ..  +(0,1); 
        \end{tikzpicture}};
     \node at (-4,0){    \begin{tikzpicture}[very thick]
       \draw (-3,0)[postaction={decorate,decoration={markings,
    mark=at position .2 with {\arrow[scale=1.3]{<}}}}] +(1,-1) -- +(-1,1);
          \draw[postaction={decorate,decoration={markings,
    mark=at position .8 with {\arrow[scale=1.3]{>}}}}] (-3,0) +(-1,-1) -- +(1,1);
          \draw[wei] (-3,0) +(0,-1) .. controls (-4,0) ..  +(0,1);
          \node at (-1.5,0) {=}; \draw[postaction={decorate,decoration={markings,
    mark=at position .8 with {\arrow[scale=1.3]{<}}}}] (0,0) +(1,-1) -- +(-1,1); \draw[postaction={decorate,decoration={markings,
    mark=at position .2 with {\arrow[scale=1.3]{>}}}}] (0,0) +(-1,-1) -- +(1,1); \draw[wei] (0,0) +(0,-1) .. controls
          (1,0) ..  +(0,1); 
        \end{tikzpicture}};
      \end{tikzpicture}
    \end{equation*}
\begin{equation*}\subeqn \label{fig:pass-through}
   \begin{tikzpicture}[yscale=.9,baseline]
      \node at (4,1.5){    \begin{tikzpicture}[very thick]
          \draw[postaction={decorate,decoration={markings,
    mark=at position .2 with {\arrow[scale=1.3]{<}}}}] (-3,0) +(1,-1) -- +(-1,1);
          \draw[postaction={decorate,decoration={markings,
    mark=at position .8 with {\arrow[scale=1.3]{<}}}}] (-3,0) +(-1,-1) -- +(1,1);
          \draw[awei] (-3,0) +(0,-1) .. controls (-4,0) ..  +(0,1);
          \node at (-1.5,0) {=}; \draw[postaction={decorate,decoration={markings,
    mark=at position .8 with {\arrow[scale=1.3]{<}}}}] (0,0) +(1,-1) -- +(-1,1); \draw[postaction={decorate,decoration={markings,
    mark=at position .2 with {\arrow[scale=1.3]{<}}}}] (0,0) +(-1,-1) -- +(1,1); \draw[awei] (0,0) +(0,-1) .. controls
          (1,0) ..  +(0,1); 
        \end{tikzpicture}};
     \node at (-4,1.5){    \begin{tikzpicture}[very thick]
       \draw (-3,0)[postaction={decorate,decoration={markings,
    mark=at position .2 with {\arrow[scale=1.3]{>}}}}] +(1,-1) -- +(-1,1);
          \draw[postaction={decorate,decoration={markings,
    mark=at position .8 with {\arrow[scale=1.3]{>}}}}] (-3,0) +(-1,-1) -- +(1,1);
          \draw[wei] (-3,0) +(0,-1) .. controls (-4,0) ..  +(0,1);
          \node at (-1.5,0) {=}; \draw[postaction={decorate,decoration={markings,
    mark=at position .8 with {\arrow[scale=1.3]{>}}}}] (0,0) +(1,-1) -- +(-1,1); \draw[postaction={decorate,decoration={markings,
    mark=at position .2 with {\arrow[scale=1.3]{>}}}}] (0,0) +(-1,-1) -- +(1,1); \draw[wei] (0,0) +(0,-1) .. controls
          (1,0) ..  +(0,1); 
        \end{tikzpicture}};
    \node at (-4,-1.5){    \begin{tikzpicture}[very thick]
          \draw[postaction={decorate,decoration={markings,
    mark=at position .8 with {\arrow[scale=1.3]{<}}}}](-3,0) +(-1,-1) -- +(1,1); \draw[wei](-3,0) +(1,-1) --
          +(-1,1); \fill (-3.5,-.5) circle (3pt); \node at (-1.5,0) {=};
          \draw[postaction={decorate,decoration={markings,
    mark=at position .2 with {\arrow[scale=1.3]{<}}}}](0,0) +(-1,-1) -- +(1,1); \draw[wei](0,0) +(1,-1) --
          +(-1,1); \fill (.5,0.5) circle (3pt);
        \end{tikzpicture}};
    \node at (4,-1.5){    \begin{tikzpicture}[very thick]
          \draw[postaction={decorate,decoration={markings,
    mark=at position .8 with {\arrow[scale=1.3]{<}}}}](-3,0) +(-1,-1) -- +(1,1); \draw[awei](-3,0) +(1,-1) --
          +(-1,1); \fill (-3.5,-.5) circle (3pt); \node at (-1.5,0) {=};
          \draw[postaction={decorate,decoration={markings,
    mark=at position .2 with {\arrow[scale=1.3]{<}}}}](0,0) +(-1,-1) -- +(1,1); \draw[awei](0,0) +(1,-1) --
          +(-1,1); \fill (.5,0.5) circle (3pt);
        \end{tikzpicture}};
      \end{tikzpicture}
    \end{equation*}
 \begin{equation*}\subeqn
 \begin{tikzpicture}[very thick,baseline]\label{pitch}
      \draw[postaction={decorate,decoration={markings,
    mark=at position .9 with {\arrow[scale=1.3]{>}}}}] (-3,0)  +(1,-1)
to[out=90,in=0] +(0,0) to[out=180,in=90] +(-1,-1);
      \draw[wei] (-3,0)  +(0,-1) .. controls (-4,0) ..  +(0,1);
      \node at (-1.5,0) {=};
      \draw[postaction={decorate,decoration={markings,
    mark=at position .9 with {\arrow[scale=1.3]{>}}}}] (0,0)  +(1,-1)
to[out=90,in=0] +(0,0) to[out=180,in=90] +(-1,-1);
      \draw[wei] (0,0) +(0,-1) .. controls (1,0) ..  +(0,1);   
 \end{tikzpicture}\qquad \qquad \qquad    \begin{tikzpicture}[very thick,baseline]
      \draw[postaction={decorate,decoration={markings,
    mark=at position .9 with {\arrow[scale=1.3]{>}}}}] (-3,0)  +(1,-1)
to[out=90,in=0] +(0,0) to[out=180,in=90] +(-1,-1);
      \draw[awei] (-3,0)  +(0,-1) .. controls (-4,0) ..  +(0,1);
      \node at (-1.5,0) {=};
      \draw[postaction={decorate,decoration={markings,
    mark=at position .9 with {\arrow[scale=1.3]{>}}}}] (0,0)  +(1,-1)
to[out=90,in=0] +(0,0) to[out=180,in=90] +(-1,-1);
      \draw[awei] (0,0) +(0,-1) .. controls (1,0) ..  +(0,1);   
 \end{tikzpicture}
  \end{equation*}
We also include the reflection through a vertical line of all the relations above this point.
  \begin{equation*}\subeqn\label{blue-triple}
    \begin{tikzpicture}[very thick,baseline]
      \draw[postaction={decorate,decoration={markings, mark=at
          position .2 with {\arrow[scale=1.3]{>}}}}] (-3,0) +(1,-1) --
      +(-1,1) node[at
      start,below]{$i$};
      \draw[postaction={decorate,decoration={markings, mark=at
          position .8 with {\arrow[scale=1.3]{>}}}}] (-3,0) +(-1,-1)
      -- +(1,1)node [at
      start,below]{$j$}; \draw[awei] (-3,0) +(0,-1) .. controls (-4,0)
      ..  +(0,1) node [at start,
      below]{$-\la_k$}; \node at (-1,0) {=};
      \draw[postaction={decorate,decoration={markings, mark=at
          position .8 with {\arrow[scale=1.3]{>}}}}] (1,0) +(1,-1) --
      +(-1,1) node[at
      start,below]{$i$};
      \draw[postaction={decorate,decoration={markings, mark=at
          position .2 with {\arrow[scale=1.3]{>}}}}] (1,0) +(-1,-1) --
      +(1,1) node [at
      start,below]{$j$}; \draw[awei] (1,0) +(0,-1) .. controls (2,0)
      ..  +(0,1) node [at start,
      below]{$-\la_k$}; 
      \end{tikzpicture}-
       \begin{tikzpicture}[very thick,baseline]
       \draw[postaction={decorate,decoration={markings, mark=at
          position .2 with {\arrow[scale=1.3]{>}}}}] (6.5,0) +(1,-1)
      -- +(1,1)  node [at
      start,below]{$i$};
      \draw[postaction={decorate,decoration={markings, mark=at
          position .2 with {\arrow[scale=1.3]{>}}}}] (6.5,0) +(-1,-1)
      -- +(-1,1)  node [at
      start,below]{$j$}; 
      \draw[densely dotted,->] (5.5,0)--(7.5,0); \node[scale=.8] at (8.7,0) {$\frac{(x_1-\zw_k)^{\la_k^i}-(x_2-\zw_k)^{\la_k^i}}{x_1-x_2}$};
      \draw[awei] (6.5,0) +(0,-1) -- +(0,1) node
      [at start, below]{$-\la_k$};     \end{tikzpicture}
  \end{equation*}
  \begin{equation*}\subeqn\label{red-triple}
    \begin{tikzpicture}[very thick,baseline]
      \draw (-3,0)[postaction={decorate,decoration={markings, mark=at
          position .2 with {\arrow[scale=1.3]{<}}}}] +(1,-1) --
      +(-1,1) node[at
      start,below]{$i$};
      \draw[postaction={decorate,decoration={markings, mark=at
          position .8 with {\arrow[scale=1.3]{<}}}}] (-3,0) +(-1,-1)
      -- +(1,1)node [at
      start,below]{$j$}; \draw[wei] (-3,0) +(0,-1) .. controls (-4,0)
      ..  +(0,1) node [at start,
      below]{$\la_k$}; \node at (-1,0) {=};
      \draw[postaction={decorate,decoration={markings, mark=at
          position .8 with {\arrow[scale=1.3]{<}}}}] (1,0) +(1,-1) --
      +(-1,1) node[at
      start,below]{$i$};
      \draw[postaction={decorate,decoration={markings, mark=at
          position .2 with {\arrow[scale=1.3]{<}}}}] (1,0) +(-1,-1) --
      +(1,1) node [at
      start,below]{$j$}; \draw[wei] (1,0) +(0,-1) .. controls (2,0) ..
      +(0,1) node [at start, below]{$\la_k$};    \end{tikzpicture}+
       \begin{tikzpicture}[very thick,baseline] \draw[postaction={decorate,decoration={markings, mark=at
          position .8 with {\arrow[scale=1.3]{<}}}}] (6.5,0) +(1,-1)
      -- +(1,1)  node[at
      start,below]{$i$};
      \draw[postaction={decorate,decoration={markings, mark=at
          position .8 with {\arrow[scale=1.3]{<}}}}] (6.5,0) +(-1,-1)
      -- +(-1,1)  node [at
      start,below]{$j$}; \draw[wei] (6.5,0) +(0,-1) -- +(0,1) node [at
      start, below]{$\la_k$}; 
      \draw[densely dotted,->] (5.5,0)--(7.5,0); \node[scale=.8] at (8.7,0) {$\frac{(x_1-\zw_k)^{\la_k^i}-(x_2-\zw_k)^{\la_k^i}}{x_1-x_2}$};
    \end{tikzpicture}
  \end{equation*}
The reader should read the label $\la_k$ in this diagram to indicate
that the strand shown is the $k$th of the red and blue strands from
the left.  In particular, $\zw_k$ is connected to this $k$th strand, and
could be thought of as a new endomorphism of the tricolore triple with
a single red or blue strand and $\Bi=\emptyset$.

\notation{$\hld^\bla$}{The deformed categorification of a tensor product of highest and lowest weight representations.}
We let $\hld^\bla$ be the idempotent completion of the quotient of the
category of tricolore quadruples $(\bla,\Bi,\kappa,\EuScript{L})$
in $\EuScript{T}$ by the tricolore quadruples where
$\kappa(\ell)<n$. That is, we consider 1-morphisms with  the region at the far right labelled with the weight $0$,
where we fix the labels of
the red and blue strands as well as their order to match $\bla$, but allow arbitrary
black strands.  We then take the quotient of this category of
1-morphisms  by killing the diagrams with a black line at the
far right.  
\end{definition}
\begin{remark}\label{left-right}
  Note that we have switched the role of left and right from \cite[Def. 4.3]{WebCB}, where we killed diagrams at the far {\it left}.  These two approaches give equivalent categories.  We can see this by extending the functor $\tilde{\sigma}$ on $\tU$ defined in \cite[\S 3.3.3]{khovanovCategorificationQuantum2010}, which reflects diagrams in the $y$-axis while negating the label on each region and multiplying by $-1$ raised to the number of crossings of pairs of strands with the same labels.  This extension must replace a red or blue strand labeled $\lambda$ by a strand of the opposite color labeled $-\lambda$.  The only interesting relation check is that the image of \cref{blue-triple} is the negation of \cref{red-triple}, and {\it vice versa}.   
\end{remark}

The definition of $\hld^\bla$ has precisely the same form as that of
$\hl^\bla$, with the only difference being the relations
\crefrange{color-opp-cancel}{red-triple} in place of the relations in \cite[Def. 4.3]{WebCB} and the left/right reversal noted in \cref{left-right}.  

From the definition, it is clear that there is a 2-functor $\tU\to
\EuScript{T}$, since \crefrange{QHA1}{flight2} are simply the
relations of $\tU$.  Thus, horizontal composition on the left induces a $\tU$ action
on $\hld^\bla$.  

Note that we will need to know the relations for bubble slides in $\EuScript{T}$.  
\begin{lemma}
In  $\EuScript{T}$, we have the bubble slides:
\begin{align}
	\mathord{\begin{tikzpicture}[baseline = -1mm]
	\draw[wei] (0.08,-.4) to node [at start, below]{$\scriptstyle\la_k$} (0.08,.4) ;
      \node at (-0.6,0) {$\anticlockj\scriptstyle (z)$};
\end{tikzpicture}
}
&=
\mathord{
\begin{tikzpicture}[baseline = -1mm]
	\draw[wei] (0.08,-.4) tonode [at start, below]{$\scriptstyle\la_k$} (0.08,.4) ;
      \node at (0.8,0) {$\anticlockj\scriptstyle(z)$};
\node at (-.85,0.05) {$\scriptstyle{(z-\beta_k)^{\la_k}}$};
\end{tikzpicture}
},&
\mathord{
\begin{tikzpicture}[baseline = -1mm]
	\draw[wei] (0.08,-.4) to node [at start, below]{$\scriptstyle\la_k$}(0.08,.4) ;
      \node at (0.8,0) {$\clockj\scriptstyle(z)$};
\end{tikzpicture}
}
&=
\mathord{
\begin{tikzpicture}[baseline = -1mm]
	\draw[wei] (0.08,-.4) to  node [at start, below]{$\scriptstyle\la_k$}(0.08,.4);
      \node at (-0.6,0) {$\clockj\scriptstyle(z)$};
\node at (1.05,0.05) {$\scriptstyle{(z-\beta_k)^{\la_k}}$};
\end{tikzpicture}
}.\label{bubslides3}
\\
\mathord{\begin{tikzpicture}[baseline = -1mm]
	\draw[awei] (0.08,-.4) to node [at start, below]{$\scriptstyle-\la_k$} (0.08,.4) ;
      \node at (-0.6,0) {$\anticlockj\scriptstyle (z)$};
\end{tikzpicture}
}
&=
\mathord{
\begin{tikzpicture}[baseline = -1mm]
	\draw[awei] (0.08,-.4) tonode [at start, below]{$\scriptstyle-\la_k$} (0.08,.4) ;
      \node at (0.8,0) {$\anticlockj\scriptstyle(z)$};
\node at (-.85,0.05) {$\scriptstyle{(z-\beta_k)^{-\la_k}}$};
\end{tikzpicture}
},&
\mathord{
\begin{tikzpicture}[baseline = -1mm]
	\draw[awei] (0.08,-.4) to node [at start, below]{$\scriptstyle-\la_k$}(0.08,.4) ;
      \node at (0.8,0) {$\clockj\scriptstyle(z)$};
\end{tikzpicture}
}
&=
\mathord{
\begin{tikzpicture}[baseline = -1mm]
	\draw[awei] (0.08,-.4) to  node [at start, below]{$\scriptstyle-\la_k$}(0.08,.4);
      \node at (-0.6,0) {$\clockj\scriptstyle(z)$};
\node at (1.05,0.05) {$\scriptstyle{(z-\beta_k)^{-\la_k}}$};
\end{tikzpicture}
}.\label{bubslides4}
\end{align}
\end{lemma}

\begin{definition}
Given a $\K[\zw_1,\cdots \zw_\ell]$-algebra $K$, we let $\hld^\bla_K$ be
the idempotent completion of the extension of scalars
$\hld^\bla\otimes_{\K[\zw_1,\cdots \zw_\ell]}K$.  

\notation{$\mathbb{K},\bar{\mathbb{K}}$}{The function field $\mathbb{K}=\K(\zw_1,\dots, \zw_\ell)$  and its algebraic closure.} 
\notation{$\hld^\bla_K$}{The idempotent completion of the extension of scalars
$\hld^\bla\otimes_{\K[\zw_1,\cdots \zw_\ell]}K$. }
The main examples we will want to consider are $\mathbb{K}=\K(\zw_1,\dots, \zw_\ell)$ and the
  algebraic closure $\bar{\mathbb{K}}$.
\end{definition}

\subsection{Spectral analysis}
\label{sec:spectral-analysis}

Throughout the rest of \cref{sec:deform-tens-prod}, we fix an $\ell$-tuple $\bla$ of highest and lowest weights.  Given this choice of $\bla$, we define sets $\underline{U}_i$ that we will ultimately identify with the spectrum of the dot operators $y_k$ acting on objects in $\hldbK$. 

\begin{definition}
  Define sets $\underline{U}_i\subset \bar{\mathbb{K}}$ as follows:  let \[\underline U_i^{(0)}=\{\zw_k\mid  \al_i^\vee(\la_k)\neq 0 \text{ for some
  $k$}\}.\]
  Now we inductively define
  $\underline U_i^{(N)}$ to be the union of $\underline U_i^{(N-1)}$ with the elements $u$
  of $\bar{\mathbb{K}}$ that satisfy $Q_{ij}(u,u')=0$ for $u'\in
  \underline U_j^{(N-1)}$, and $\underline U_i=\bigcup_{N\in \Z} \underline U_i^{(N)}$.

 Let $U_i'$ be the union of $\underline U_i$ with the set $U_i$ of elements in $\bar{\mathbb{K}}$ that appear in the spectrum of the elements $y_ke_{\Bi}$ with $i_k=i$
  acting on objects in the
 category $\hldbK$;
  that is, the eigenvalues that appear when dots on strands labeled $i$ act.
\end{definition}
It might seem strange that we add the elements of $\underline U_i$ to $U_i'$ by
definition, but this simplifies matters for us, since we have not yet
established that $\hldbK$ is non-zero. Thus, we have not yet
established that there are any elements of the spectrum of $y$ acting
on objects in this category.  We
ultimately see that $U_i=U_i'=\underline U_i$.

\begin{definition}\label{tla}
We let $\tilde{I}',\tilde{I}$ be the sets constructed from $U_i'$ and $U_i$ as before, with their induced graph structures and Cartan data, and $\tilde{\fg}'$ and $\tilde{\fg}$ the corresponding Kac-Moody algebras.  We let $\tilde{\la}$ be the weight for $\tilde{\fg}'$ such that $\al_{(i,u)}^\vee(\tilde{\la})=\sum_m\delta_{u,\zw_m}\al_i^\vee(\la_m)$.
  \notation{$\tilde{\la}$}{The weight for $\tilde{\fg}'$ defined in \cref{tla}.}
\end{definition}

\notation{$m(u)$}{The unique index $m$ such that $u$ and $\zw_m$ are algebraically dependent.  }
Note that by construction, any element of $\underline{U}_i^{(N)}$ is algebraically dependent on some element of $\underline{U}_j^{(N-1)}$ for some $j$ connected to $i$ in the Dynkin diagram of $\fg$.  Thus, any element $u\in \underline{U}_i$ is algebraically dependent on some element of $\underline{U}_j^{(0)}$, which is a subset of $\{\zw_m\}$. Also note that since the coefficients of $Q_{ij}$ are in $\K$, this relation is symmetric and $\zw_m$ is algebraically dependent on $u$ as well. 
If there were two such elements $\zw_{k},\zw_{k'}$, then we would have that $\K(\zw_k,\zw_{k'})$ is contained in an algebraic extension of $\K(u)$, which is impossible since $\K(\zw_k,\zw_{k'})$ has transcendence degree 2 over $\K$ while $\K(u)$ has transcendence degree at most 1.

Thus, every element $u\in \underline{U}_i$ is algebraically dependent on exactly one $\zw_m$.  We
denote this index $m(u)$.  In many cases that interest us, there is
exactly one component of $\tilde{I}$ for each of these indices, but if
$\al_i^\vee(\la_m)\neq 0$ for several elements $i$, the pairs
$(i,\zw_m)$ can lie in different components for different $i$.  

\excise{Generically, this
deformation ``pulls apart'' the tensor factors of $\hl^\bla$.  More
precisely, consider the Lie algebra  $\tilde{\fg}^{\oplus\ell}$;  for simplicity
  of notation, we use $(i,m)$ to denote the copy of $i\in I$ corresponding
to the $m$th copy of $\fg$ in the sum $\fg^{\oplus\ell}$.    
Let $K$ be the fraction field of $\K[\Bz]$.   For $(\Bi,\Bm)\in
\tilde{I}^n$, let
$\epsilon_{\Bi,\Bm}\colon \eF_{\Bi}\to \eF_{\Bi}$ be the projection to the
space where $y_k$ acts with generalized eigenvalue $\zw_{m_k}$.  As
before, let
$\eF_{(i,m)}$ denote the image of $\epsilon_{i,m}$ in $\eF_i$.}
 \excise{One should think about these formulas in terms of the polynomial
representation of the KLR algebra given in \cite[3.12]{Rou2KM}.  The
$p\neq m$ cases of \cref{eq:6} are simply the formulas for $s_i$ in
terms of $\psi$. }

We can define formal power series valued in the center of the
category $\hld^{\bla}$ which act on the object $(\bla,\Bi,\kappa)$ by \[\mathbf{y}_i(z):=\prod_{i_k=\pm i}(z-y_k)^{\pm 1}\qquad
  \mathbf{Q}_{ji}(z)=\prod_{i_r=\pm j} \Big(t_{ji}Q_{ij}(z,y_r)\Big)^{\pm 1},\] where $y_r$ is the dot acting on the $r$th strand from the left.  

We can view this as a formal power series in $z$ with coefficients in the polynomial ring generated by dots on strands labeled $j$.  These polynomials are symmetric when we swap two dots on strands with the same orientation.  That is, we can regard them as symmetric polynomials in a pair of alphabets, one for each orientation.  In fact, since the factors corresponding to upward- and downward-oriented strands are inverses of each other, the resulting polynomials are supersymmetric in the sense of \cite{stembridgeCharacterizationSupersymmetric1985}; that is, they become independent of $t$ after substituting $t$ simultaneously for one variable from an upward strand and one from a downward strand.  As discussed below, we need this supersymmetric property to show that these power series are central 2-morphisms in $\tU$.

Any such polynomial commutes with all upward- or downward-oriented diagrams by \cite[2.9]{KLI}, since each coefficient is symmetric in the corresponding variables. 
It commutes with a cup or cap joining the $k$th strand to the $k+1$st since multiplying by $(z-y_{k})^{\pm 1}$ at one end of the cup or cap cancels with $(z-y_{k+1})^{\mp 1}$ at the other (this is a restatement of the supersymmetric property).  Note that the bubble slides \crefrange{bubslides1}{bubslides2},\crefrange{bubslides3}{bubslides4} and the triviality of bubbles at the far right
show that at the far left of the diagram, we have:
  \begin{align}\label{bubble-slide}
  \clocki(z) &=\mathbf{y}_i(z)^{-2}\prod_{j\neq i}
      \mathbf{Q}_{ji}(z)\prod_{m=1}^\ell(z-\zw_m)^{-\la_m^i}\\   \anticlocki(z) &=\mathbf{y}_i(z)^{2}\prod_{j\neq i}
\mathbf{Q}_{ji}(z)^{-1}\prod_{m=1}^\ell(z-\zw_m)^{\la_m^i}.
  \end{align}

 Let $\tilde{\mu}=\tilde{\la}-\sum a_{i,u}\al_{i,u}$ be a weight of $\tilde{I}'$.  We can define subcategories 
  $\mathcal{V}_{(\tilde{\mu})}\subset \hldbK$ by the vanishing orders of $\clocki(z)$ as in \cref{def:subcat}.  
By \cref{thm:deform-action},
  the functors $\eF_{i,u}$ and their adjoints $\eE_{i,u}$ induce a
  categorical action of $\tilde{\fg}$ on $\hldbK$, with weight decomposition given by $\hldbK\cong \oplus_{\tilde{\mu}}\mathcal{V}_{\tilde{\mu}}$.  It will be convenient to extend this action to an action of $\tilde{\fg}'$ by setting $\eF_{i,u}=\eE_{i,u}=0$ for $u\notin U_i$;  we will prove in a moment that there are no such $u$'s, but it is helpful to allow them as a possibility while we prove this.

Given a triple $(\bla,\Bi,\kappa)$ with $\Bi=(i_1,\dots,
i_n)$ considered as an object in $\hldbK$ (recall that we will
often exclude $\bla$ from the notation when it is unlikely to be
confused), we can thus decompose it according to the spectrum of the
dots $y_k$.  For a sequence $j_k=(i_k,u_k)\in \tilde{I}'$ for $k=1,\dots, n$,
we let $(\Bi,\kappa)_{\Bu}$ be the simultaneous generalized eigenspace of $y_k$ with eigenvalue $u_k$, that is, the simultaneous stable kernel of
$y_k-u_k$, for all $1\leq k\leq n$.  

\begin{lemma}\label{lem:cyclic}
We have that
\begin{equation}
(\Bi,\kappa)_{\Bu}\cong \eE_{i_1,u_1}\cdots
\eE_{i_n,u_n}(\emptyset,0)\label{eq:subu}
\end{equation}
if $u_k\in \underline{U}_{i_k}$ and
$k\leq \kappa(m(u_k))$ for each $k$, and $(\Bi,\kappa)_{\Bu}=0$
otherwise.  

In particular, we have $U_i= U_i'$ for all $i$ and
$\tilde{\fg}= \tilde{\fg}'$, and the category $\hldbK$ is generated by the tricolore
triple $(\emptyset,0)$ as a categorical module over
$\tilde{\fg}$.    
\end{lemma}
\begin{proof}
First, we note that if $u\in \underline U_i$ and $u'\notin \underline U_j$, then  $Q_{ij}(y_1,y_2)$ acts on
$\eE_{i,u}\eE_{j,u'}$ with its only eigenvalue $Q_{ij}(u,u')\neq 0$
(by the definition of $\underline U_i$).  Thus the crossing $\psi$ induces an isomorphism
$\eE_{i,u}\eE_{j,u'}\cong \eE_{j,u'}\eE_{i,u}$.  Similarly, still assuming that $u'\notin \underline U_j$, the crossing of  $\eE_{j,u'}$ over 
red and blue strands is invertible since $y-\zw_k$ is invertible for all $k$, with its only
eigenvalue $u'-\zw_k$; in fact, this crossing is still invertible for the $k$th red/blue
strand if we simply assume that $u'\neq \zw_k$ (in particular, if $k\neq m(u')$).  

We establish the result by induction on $\ell$ and $n$ (that is, on the total
number of strands).  If $n=0$, the result is tautological.  Otherwise, the leftmost strand in the idempotent for the
object $(\Bi,\kappa)$ is one of black, blue, or red.  If it is
black, then the image of $\kappa$ lies in $[1,n-1]$ and we can use the same symbol to mean the same map to this codomain.  Thus, we have $(\Bi,\kappa)=\eE_{i_1} (\Bi^-,\kappa-1)$ for some $i\in \pm I$,
where  $\Bi^-=(i_2,\dots, i_n)$ and $\kappa-1$ denote the usual difference of functions $(\kappa-1)(k)=\kappa(k)-1$.  
Decomposing with respect to the eigenvalues of $y$, we have $\eE_i (\Bi^-,\kappa)\cong \bigoplus_u \eE_{i,u}
(\Bi^-,\kappa)$ where $u$ ranges over the roots of the minimal
polynomial of $y$ acting on $\eE_i (\Bi^-,\kappa)$.  By induction
$(\Bi^-,\kappa)$ is a sum of modules obtained from
$(\emptyset,0)$ by the functors $\eE_{j,u'}$ and $\eF_{j,u'}$ for
$u'\in \underline U_j$.  If $u$
is not in $\underline U_i$, then all
these functors commute with $\eE_{i,u}$ (as argued above), and
$\eE_{i,u}(\emptyset,0)=0$, so $\eE_{i,u}
(\Bi^-,\kappa)=0$, and $\eF_i (\Bi^-,\kappa)\cong \bigoplus_{u\in \underline U_i} \eF_{i,u}
(\Bi^-,\kappa)$.  By induction, this establishes the result for the case where the leftmost strand is black.

On the other hand, consider the case where the leftmost strand is blue or red, so $\kappa(1)=0$.
We apply induction with the tricolore triple $(\bla^-,\Bi,\kappa^-)$ with
this strand removed, where $\kappa^-$ is the composition of $\kappa$ with the map $[1,\ell-1] \to [1,\ell]$ of addition by 1.  By induction,
$(\bla^-,\Bi,\kappa^-)\cong \bigoplus(\bla^-,\Bi,\kappa^-)_{\Bu}$ with
$k\leq \kappa(m(u_k))$;  in particular, no eigenvalue with $m(u)=1$ appears.  Since adding in the $\ell$th blue or
red strand does not change the eigenvalues of the dots, we also have
$(\bla,\Bi,\kappa)\cong \bigoplus (\bla,\Bi,\kappa)_{\Bu}$ with $\Bu$
ranging over the same set; this rules out $u_k$ satisfying $\kappa(m(u_k))<k$. This shows \cref{eq:subu}, and
that $U_i'= U_i$.  
\end{proof}

Thus $\hldbK$ is generated by a single object, which is
highest weight for the components of $\tilde{I}$ with $\la_{m(u)}$ dominant and lowest weight 
for those with $\la_{m(u)}$ anti-dominant.   Alternatively, we can
easily choose a Borel for
which this representation is straightforwardly highest weight.   To
distinguish objects which are highest weight for this Borel, we call
them {\bf signed highest weight}.  We can write each weight $\tilde{\la}$
uniquely as a sum $\tilde{\la}=\tilde{\la}_1+\cdots +\tilde{\la}_\ell$
where $\tilde{\la}_m$ is supported on components with $m(u)=m$.

Let $\hl^{\tilde\la}$ be the category over the base field $\bar{\mathbb{K}}$ defined in \cite[\S 5]{WebCB} for the
singleton $(\tilde\la)$ and the Dynkin diagram $\tilde{I}$, and 
$\hl^{\tilde\bla}$ this category for the sequence
$\tilde{\bla}=(\tilde\la_1,\tilde\la_2,\dots,\tilde\la_\ell )$. Both
of these are defined using the signed highest weight Borel, so for
example, $\hl^{\tilde\la}$ is defined using a single colored strand at
right (whether it is red, blue or purple is a matter of taste) that
satisfies the red version of the relations
\crefrange{color-opp-cancel}{red-triple} for a black strand with
label $(i,u)$ such that $\la_{m(u)}$ is dominant, and the blue version
if $\la_{m(u)}$ is anti-dominant.  
\begin{lemma}\label{lem:obvious-functor}
The obvious functor $\hl^{\tilde\la}\to
\hl^{\tilde\bla}$ is an equivalence of categories.
\end{lemma}
\begin{proof}
    This is equivalent to the statement that the tensor product algebra $T^{\tilde\bla}$ is Morita equivalent to the cyclotomic quotient $T^{\tilde{\la}}$ categorifying the fact that if highest weights are concentrated on different components of the Dynkin diagram, then their tensor product is irreducible.

    This Morita equivalence is induced by the idempotent given by the sum of vertical diagrams where all the black strands are at the far left (note the reverse of left and right from \cref{left-right}).  If an idempotent in $T^{\tilde\bla}$ is non-zero, then all black and red strands from the same component of the Dynkin diagram must be ordered with the black to the left.  Thus, for any non-zero idempotent, we can take any diagram and push all the black strands to the far left in the middle, since all red and black crossings created involve strands from different components of the Dynkin diagram.  This shows the desired Morita equivalence.  
\end{proof}
 Both categories can be defined over $\K$, but we
will be more often interested in their base extensions to
$\bar{\mathbb{K}}$.  

The important fact we will need about the category $\hl^{\tilde\la}$
is that unlike the general case, we already know the non-degeneracy
results we need for this category, based on \cite{Webmerged}, since
we have deformed to a signed highest-weight representation.  Thus our
strategy for proving non-degeneracy for $\hld^{\bla}$ is to compare it
with this category whose Grothendieck group and Hom spaces are known
by \cite[Cor. 3.20 \& 3.22]{Webmerged}.  By applying the automorphism
$\tilde{\omega}$ from \cite[\S 3.3.2]{khovanovCategorificationQuantum2010} which swaps the functors
$\eE_i$ and $\eF_i$ to all components where $\tilde\la_i$ is
anti-dominant, we can reduce to the case where all $\tilde\la_i$ are
dominant.  In this case, the endomorphism algebra of the sum of all objects
in $\hl^{\tilde\la}$ is precisely the algebra $DR^{\tilde\la}$ defined
in \cite[Def. 3.1]{Webmerged}. By the Morita equivalence of this ring
to the cyclotomic quotient $R^{\tilde\la}$
(\cite[Def. 3.3]{Webmerged}) shown in \cite[Cor. 3.20]{Webmerged}, we
can compute the dimension of morphism spaces in $\hl^{\tilde\la}$.  In
fact, we will only need the basic observation that in this
category \begin{equation}\label{eq:cyc-quotient}
  \Hom_{\hl^{\tilde\la}}(((i,u),0),((i,u),0))\cong
  \K[y]/(y^{\la_{m(u)}^i}),
\end{equation} that is, if we have a single strand, then the dots satisfy the cyclotomic relation and nothing else; this is obvious in the cyclotomic quotient  $R^{\tilde\la}$.  Of course, this is readily extended to the case where there is only one strand with label in any single component; up to Morita equivalence, one just obtains a tensor product of these endomorphism algebras.
\excise{ Consider the deformed cyclotomic quotient $\check{R}^{\tilde{\la}}$
for $\tilde{\fg}$ of signed highest weight $\tilde{\la}$.  For any
weight $\mu$, the ring $\check{R}^{\tilde{\la}}_{\mu}$ is free of
finite rank over the central subalgebra
$\check{R}^{\tilde{\la}}_{\tilde \la}$, the polynomial
ring on fake bubbles at the weight $\tilde\la$.  

As is true whenever the underlying Dynkin diagram is disconnected, the
ring $\check{R}^{\tilde{\la}}$ is Morita equivalent to the tensor
product $\check{R}^{\tilde\la_1}\otimes \cdots\otimes \check
R^{\tilde\la_\ell}$.  The fact that we use the signed Borel means
that it is more convenient to think of the deformed cyclotomic quotient $\check{R}^{\tilde\la_k}$ as
written with a red strand and downward black strands if $\la_k$ is dominant and with a blue
strand and upward black strands if $\la_k$ is anti-dominant.  However,
if the reader prefers, applying the Cartan involution in the factors where $\la_{m(u)}$ is anti-dominant,
which relabels $\eF_{i,u}$ as $\eE_{i,u}$, allows us to write all of
these cyclotomic quotients with downward strands.

The classification of highest weight
representations in \cite[Prop. 3.25]{Webmerged} shows that:
\begin{lemma}
  The category $\hl^{\tilde\la}$ is
  equivalent to the category of projective modules over  $\check{R}^{\tilde\la_1}\otimes \cdots\otimes \check
R^{\tilde\la_\ell}\otimes_{\check{R}^{\tilde{\la}}_{\tilde
    \la}}\End(\emptyset,0)$, the base change via the natural map
$\check{R}^{\tilde{\la}}_{\tilde{\la}}\to \End(\emptyset,0)$ to the endomorphisms of the signed highest weight object. 
\end{lemma}}

\begin{lemma}\label{equivalence}
There is a strongly equivariant functor
  \begin{equation}
\Phi\colon \hl^{\tilde\la}\cong \hl^{\tilde\bla}\to \hldbK\label{eq:Phi}
\end{equation} sending $(\emptyset,0) \mapsto (\emptyset,0)$.
  The functor $\Phi$  is an  equivalence.
\end{lemma}

We give the proof of this result on page \pageref{proof-equivalence}.

\subsection{Non-degeneracy}
\label{sec:applications}

Given two tricolore triples $(\bla,\Bi,\kappa)$ and $(\bla,\Bi',\kappa')$, we define a set $B$ of diagrams in $\Hom_{\EuScript{T}}\big((\bla,\Bi,\kappa),(\bla,\Bi',\kappa')\big)$ which generalizes the set $B_{\Bi,\Bi',\la}$ defined in \cite[\S 3.2.3]{khovanovCategorificationQuantum2010}.  In fact, we can define $B$ to be the set of diagrams where the black strands trace out an element of Khovanov and Lauda's spanning set $B_{\Bi,\Bi',\la}$ and red and blue strands are added as required  by the functions $\kappa$ and $\kappa'$.  That is, it is a set satisfying the following conditions:
\begin{enumerate}
     \item for each way of dividing the set $\Bi\cup -\Bi'$ into pairs of matching elements (an $(\Bi,\Bi')$-pairing, in the terminology of \cite{khovanovCategorificationQuantum2010}), there is a unique diagram in $B$ which connects the terminals for the two elements of each pair, with a minimal number of crossings and no bubbles or dots.
     \item for each diagram of the above, and each arc joining two terminals, we fix a position on that arc (avoiding the crossings with any others).  The full list of diagrams in $B$ are indexed by diagrams as described in part (1), a choice of non-negative  integer for each arc joining two terminals, and a monomial in the unnested bubbles (a basis vector of the ring $\Pi_{\la}$ in the notation of \cite[\S 3.2.1]{khovanovCategorificationQuantum2010}).  We construct the corresponding diagram by adding the corresponding number of
 dots on each arc and putting the bubbles to the left of the diagram. 
     \end{enumerate}
   \begin{theorem}\label{tricolore-nondegenerate}
For two tricolore triples $(\bla,\Bi,\kappa)$ and $(\bla,\Bi',\kappa')$, the set $B$ is a basis over $\C[\Bz]$  for the morphism space $\Hom_{\EuScript{T}}  \big((\bla,\Bi,\kappa),(\bla,\Bi',\kappa')\big)$.
   \end{theorem}
   
We give the proof of this result on page \pageref{proof-tricolore-nondegenerate}.

\section{Proofs}
\label{sec:proofs}

\subsection{Proof of \cref{th:coproduct}}

\begin{proof}[Proof of \cref{th:coproduct}]
  \refstepcounter{dummy}\label{proof-th:coproduct}
While the morphisms given in the statement of \cref{th:coproduct} fix the 2-functor, to prove that we have a 2-functor, it is useful to also describe the image of the functor on some additional morphisms:
\begin{align}\label{com4}
\mathord{
\begin{tikzpicture}[baseline = 1mm]
	\draw[-,thin] (0.4,0) to[out=90, in=0] (0.1,0.4);
	\draw[->,thin] (0.1,0.4) to[out = 180, in = 90] (-0.2,0);
\end{tikzpicture}
}&\mapsto
\mathord{
\begin{tikzpicture}[baseline = 1mm]
	\draw[-,thin,blue] (0.4,-.2) to[out=90, in=0] (0.1,0.6);
	\draw[->,thin,blue] (0.1,0.6) to[out = 180, in = 90] (-0.2,-.2);
     \node at (0.4,.3) {$\color{red}\anticlockleft$};
\end{tikzpicture}
}+
\mathord{
\begin{tikzpicture}[baseline = 1mm]
	\draw[-,thin,red] (0.4,-.2) to[out=90, in=0] (0.1,0.6);
	\draw[->,thin,red] (0.1,0.6) to[out = 180, in = 90] (-0.2,-.2);
     \node at (0.4,.3) {$\color{blue}\anticlockleft$};
\end{tikzpicture}
}\:,
&\mathord{
\begin{tikzpicture}[baseline = 1mm]
	\draw[-,thin] (0.4,0.4) to[out=-90, in=0] (0.1,0);
	\draw[->,thin] (0.1,0) to[out = 180, in = -90] (-0.2,0.4);
\end{tikzpicture}
}&\mapsto
\mathord{
\begin{tikzpicture}[baseline = 1mm]
	\draw[-,thin,blue] (0.4,0.6) to[out=-90, in=0] (0.1,-0.2);
	\draw[->,thin,blue] (0.1,-0.2) to[out = 180, in = -90] (-0.2,0.6);
     \node at (-0.2,.2) {$\color{red}\clockright$};
\end{tikzpicture}
}+\mathord{
\begin{tikzpicture}[baseline = 1mm]
	\draw[-,thin,red] (0.4,0.6) to[out=-90, in=0] (0.1,-0.2);
	\draw[->,thin,red] (0.1,-0.2) to[out = 180, in = -90] (-0.2,0.6);
     \node at (-0.2,.2) {$\color{blue}\clockright$};
\end{tikzpicture}
}\:,\\\label{com5}
\mathord{\begin{tikzpicture}[baseline=-1mm]
\node at (0.04,0.05) {$\anticlockleft\,{\scriptstyle a}$};
\node at (.1,0) {$\bullet$};
\end{tikzpicture}}
&\mapsto
\sum_{b \in \Z}
\begin{array}{l}
\mathord{
\begin{tikzpicture}[baseline = 1.25mm]
  \draw[-,thin,blue] (0,0.4) to[out=180,in=90] (-.2,0.2);
  \draw[->,thin,blue] (0.2,0.2) to[out=90,in=0] (0,.4);
 \draw[-,thin,blue] (-.2,0.2) to[out=-90,in=180] (0,0);
  \draw[-,thin,blue] (0,0) to[out=0,in=-90] (0.2,0.2);
   \node at (-0.2,0.2) {$\color{blue}\bullet$};
   \node at (-.38,0.2) {$\color{blue}\scriptstyle{b}$};
\end{tikzpicture}
}\\ \mathord{
\begin{tikzpicture}[baseline = 1.25mm]
  \draw[->,thin,red] (0.2,0.2) to[out=90,in=0] (0,.4);
  \draw[-,thin,red] (0,0.4) to[out=180,in=90] (-.2,0.2);
\draw[-,thin,red] (-.2,0.2) to[out=-90,in=180] (0,0);
  \draw[-,thin,red] (0,0) to[out=0,in=-90] (0.2,0.2);
   \node at (0.2,0.2) {$\color{red}\bullet$};
   \node at (.7,0.2) {$\color{red}\scriptstyle{a-b-1}$};
\end{tikzpicture}
}\end{array},&
\mathord{\begin{tikzpicture}[baseline=-1mm]
\node at (-0.04,.05) {${\scriptstyle a}\,\clockright$};
\node at (-.1,0) {$\bullet$};
\end{tikzpicture}}
&\mapsto
\sum_{b \in \Z}
\begin{array}{l}
\mathord{
\begin{tikzpicture}[baseline = 1.25mm]
  \draw[<-,thin,blue] (0,0.4) to[out=180,in=90] (-.2,0.2);
  \draw[-,thin,blue] (0.2,0.2) to[out=90,in=0] (0,.4);
 \draw[-,thin,blue] (-.2,0.2) to[out=-90,in=180] (0,0);
  \draw[-,thin,blue] (0,0) to[out=0,in=-90] (0.2,0.2);
   \node at (-0.2,0.2) {$\color{blue}\bullet$};
   \node at (-.38,0.2) {$\color{blue}\scriptstyle{b}$};
\end{tikzpicture}
}\\
\mathord{
\begin{tikzpicture}[baseline = 1.25mm]
  \draw[-,thin,red] (0.2,0.2) to[out=90,in=0] (0,.4);
  \draw[<-,thin,red] (0,0.4) to[out=180,in=90] (-.2,0.2);
\draw[-,thin,red] (-.2,0.2) to[out=-90,in=180] (0,0);
  \draw[-,thin,red] (0,0) to[out=0,in=-90] (0.2,0.2);
   \node at (0.2,0.2) {$\color{red}\bullet$};
   \node at (.7,0.2) {$\color{red}\scriptstyle{a-b-1}$};
\end{tikzpicture}
}\end{array}.
\end{align}
The sum in \cref{com5} only has finitely non-zero terms since on any object, only finitely many fake bubbles can be non-zero.

In \cref{com4}, we use the notation of internal bubbles:
\begin{align}
\mathord{
\begin{tikzpicture}[baseline=2mm]
	\draw[->,thin,red] (0,-.2) to[out=90,in=-90] (0,.8);
     \node at (0,.3) {$\color{blue}\anticlockleft$};
\end{tikzpicture}
}&:=
\sum_{a \geq 0}
\mathord{
\begin{tikzpicture}[baseline=2mm]
	\draw[->,thin,red] (0,-.2) to[out=90,in=-90] (0,.8);
     \node at (-1.1,.3) {$\blue{\anticlockleft}$};
     \node at (0,0.3) {$\color{red}\bullet$};
     \node at (0.2,0.3) {$\color{red}\scriptstyle a$};
     \node at (-0.9,0.3) {$\color{blue}\bullet$};
     \node at (-0.5,0.3) {$\color{blue}\scriptstyle -a-1$};
\end{tikzpicture}
}+
\mathord{
\begin{tikzpicture}[baseline=2mm]
	\draw[->,thin,red] (0,-.2) to[out=90,in=-90] (0,.8);
     \node at (-.6,.3) {$\color{blue}\anticlockleft$};
	\draw[-] (0,0.31) to (-0.4,0.31);
     \node at (0,0.3) {$\dt$};
     \node at (-0.4,0.3) {$\dt$};
\end{tikzpicture}
}
\:,
&
\mathord{
\begin{tikzpicture}[baseline=2mm]
	\draw[->,thin,red] (0,-.2) to[out=90,in=-90] (0,.8);
     \node at (0,.3) {$\color{blue}\clockright$};
\end{tikzpicture}
}
&:=
\sum_{a \geq 0}
\mathord{
\begin{tikzpicture}[baseline=2mm]
	\draw[->,thin,red] (0,-.2) to[out=90,in=-90] (0,.8);
     \node at (1.1,.3) {$\blue{\clockright}$};
     \node at (0,0.3) {$\color{red}\bullet$};
     \node at (-0.2,0.3) {$\color{red}\scriptstyle a$};
     \node at (0.9,0.3) {$\color{blue}\bullet$};
     \node at (0.5,0.3) {$\color{blue}\scriptstyle -a-1$};
\end{tikzpicture}
}+
\mathord{
\begin{tikzpicture}[baseline=2mm]
	\draw[->,thin,red] (0,-.2) to[out=90,in=-90] (0,.8);
     \node at (.6,.3) {$\color{blue}\clockright$};
	\draw[-] (0,0.31) to (0.4,0.31);
     \node at (0,0.3) {$\dt$};
     \node at (0.4,0.3) {$\dt$};
\end{tikzpicture}
}\:,\label{odd1}\\
\mathord{
\begin{tikzpicture}[baseline=2mm]
	\draw[->,thin,blue] (0,-.2) to[out=90,in=-90] (0,.8);
     \node at (0,.3) {$\color{red}\anticlockleft$};
\end{tikzpicture}
}&:=
\sum_{a \geq 0}
\mathord{
\begin{tikzpicture}[baseline=2mm]
	\draw[->,thin,blue] (0,-.2) to[out=90,in=-90] (0,.8);
     \node at (-1.1,.3) {$\red{\anticlockleft}$};
     \node at (0,0.3) {$\color{blue}\bullet$};
     \node at (0.2,0.3) {$\color{blue}\scriptstyle a$};
     \node at (-0.9,0.3) {$\color{red}\bullet$};
     \node at (-0.5,0.3) {$\color{red}\scriptstyle -a-1$};
\end{tikzpicture}
}-\mathord{
\begin{tikzpicture}[baseline=2mm]
	\draw[->,thin,blue] (0,-.2) to[out=90,in=-90] (0,.8);
     \node at (-.6,.3) {$\color{red}\anticlockleft$};
	\draw[-] (0,0.31) to (-0.4,0.31);
     \node at (0,0.3) {$\dt$};
     \node at (-0.4,0.3) {$\dt$};
\end{tikzpicture}
}\:,
&
\mathord{
\begin{tikzpicture}[baseline=2mm]
	\draw[->,thin,blue] (0,-.2) to[out=90,in=-90] (0,.8);
     \node at (0,.3) {$\color{red}\clockright$};
\end{tikzpicture}
}
&:=
\sum_{a \geq 0}
\mathord{
\begin{tikzpicture}[baseline=2mm]
	\draw[->,thin,blue] (0,-.2) to[out=90,in=-90] (0,.8);
     \node at (1.1,.3) {$\red{\clockright}$};
     \node at (0,0.3) {$\color{blue}\bullet$};
     \node at (-0.2,0.3) {$\color{blue}\scriptstyle a$};
     \node at (0.9,0.3) {$\color{red}\bullet$};
     \node at (0.5,0.3) {$\color{red}\scriptstyle -a-1$};
\end{tikzpicture}
}-\mathord{
\begin{tikzpicture}[baseline=2mm]
	\draw[->,thin,blue] (0,-.2) to[out=90,in=-90] (0,.8);
     \node at (.6,.3) {$\color{red}\clockright$};
	\draw[-] (0,0.31) to (0.4,0.31);
     \node at (0,0.3) {$\dt$};
     \node at (0.4,0.3) {$\dt$};
\end{tikzpicture}
}
\:.\label{odd2}
\end{align}

Note that since red and blue 1-morphisms commute, it does not matter which side the bubbles above are drawn on.

While the Heisenberg and $\mathfrak{sl}_2$ Kac-Moody categories are quite different, because the relations \cref{QHA1,infgras,pop-curl} are almost identical to corresponding relations in the Heisenberg category \cite[(3.14--3.17)]{brundanHeisenbergKacMoody2020}, many of the same relations hold. 
One unfortunate proviso here: the signs are off in several of these relations, so our proofs will be nearly identical but with sign differences. 
That is, one can check that:
\begin{equation}\label{internal-inverse}
	\mathord{
\begin{tikzpicture}[baseline=-0.8mm]
	\draw[->,thin,blue] (0,-.5) to[out=90,in=-90] (0,.5);
     \node at (0,.0) {$\color{red}\anticlockleft$};
\end{tikzpicture}
}
=
\left(
\mathord{
\begin{tikzpicture}[baseline=-0.8mm]
	\draw[->,thin,blue] (0,-.5) to[out=90,in=-90] (0,.5);
     \node at (0,.0) {$\color{red}\clockright$};
\end{tikzpicture}
}
\right)^{-1} \qquad \qquad \displaystyle\mathord{
\begin{tikzpicture}[baseline=-1mm]
\draw[-,thin,blue] (0,-0.38) to[out=180,in=-90] (-.38,0);
\draw[-,thin,blue] (-0.38,0) to[out=90,in=180] (0,0.38);
\draw[->,thin,blue] (0,0.38) to[out=0,in=90] (0.38,0);
\draw[-,thin,blue] (0.378,0.02) to[out=-86,in=0] (0,-0.38);
     \node at (-.29,-.2) {$\color{red}\clockright$};
     \node at (-.27,.27) {$\color{blue}\bullet$};
     \node at (-.44,.28) {$\color{blue}\scriptstyle{a}$};
\end{tikzpicture}}
\:+
\mathord{\begin{tikzpicture}[baseline=-1mm]
\draw[-,thin,red] (0,-0.38) to[out=180,in=-90] (-.38,0);
\draw[-,thin,red] (-0.38,0) to[out=90,in=180] (0,0.38);
\draw[->,thin,red] (0,0.38) to[out=0,in=90] (0.38,0);
\draw[-,thin,red] (0.378,0.02) to[out=-86,in=0] (0,-0.38);
     \node at (-.29,-.2) {$\color{blue}\clockright$};
     \node at (-.27,.27) {$\color{red}\bullet$};
     \node at (-.44,.28) {$\color{red}\scriptstyle{a}$};
\end{tikzpicture}}
=
\sum_{b \in \Z}
\begin{array}{l}
\mathord{
\begin{tikzpicture}[baseline = 1.25mm]
  \draw[<-,thin,blue] (0,0.4) to[out=180,in=90] (-.2,0.2);
  \draw[-,thin,blue] (0.2,0.2) to[out=90,in=0] (0,.4);
 \draw[-,thin,blue] (-.2,0.2) to[out=-90,in=180] (0,0);
  \draw[-,thin,blue] (0,0) to[out=0,in=-90] (0.2,0.2);
   \node at (-0.2,0.2) {$\color{blue}\bullet$};
   \node at (-.38,0.2) {$\color{blue}\scriptstyle{b}$};
\end{tikzpicture}
}\\
\mathord{
\begin{tikzpicture}[baseline = 1.25mm]
  \draw[-,thin,red] (0.2,0.2) to[out=90,in=0] (0,.4);
  \draw[<-,thin,red] (0,0.4) to[out=180,in=90] (-.2,0.2);
\draw[-,thin,red] (-.2,0.2) to[out=-90,in=180] (0,0);
  \draw[-,thin,red] (0,0) to[out=0,in=-90] (0.2,0.2);
   \node at (0.2,0.2) {$\color{red}\bullet$};
   \node at (.65,0.2) {$\color{red}\scriptstyle{a-b-1}$};
\end{tikzpicture}
}\end{array} \quad \text{for all }a\geq 0.
\end{equation}
\begin{equation}
	\displaystyle\:\mathord{
\begin{tikzpicture}[baseline = -1mm]
	\draw[<-,thin,blue] (0.45,.6) to (-0.45,-.6);
	\draw[->,thin,blue] (0.45,-.6) to (-0.45,.6);
   \node at (0.25,-0.3) {$\color{red}\clockright$};
\end{tikzpicture}
}=
\mathord{
\begin{tikzpicture}[baseline = -1mm]
	\draw[<-,thin,blue] (0.45,.6) to (-0.45,-.6);
	\draw[->,thin,blue] (0.45,-.6) to (-0.45,.6);
   \node at (-0.25,0.3) {$\color{red}\clockright$};
\end{tikzpicture}
}
+\mathord{
\begin{tikzpicture}[baseline = -1mm]
	\draw[<-,thin,blue] (0.5,.6) to (0.5,-.6);
	\draw[->,thin,blue] (-0.5,-.6) to (-0.5,.6);
   \node at (0,0) {$\color{red}\clocktop$};
	\draw[-] (.2,0.01) to (0.5,0.01);
	\draw[-] (-.2,0.01) to (-0.5,0.01);
     \node at (0.2,0) {$\dt$};
     \node at (-0.2,0) {$\dt$};
     \node at (0.5,0) {$\dt$};
     \node at (-0.5,0) {$\dt$};
\end{tikzpicture}
}
-\sum_{a,b \geq 0}
\mathord{
\begin{tikzpicture}[baseline = -1mm]
	\draw[<-,thin,blue] (0.5,.6) to (0.5,-.6);
	\draw[->,thin,blue] (-0.5,-.6) to (-0.5,.6);
   \node at (0,0) {$\color{red}\clocktop$};
   \node at (0,-0.2) {$\color{red}\bullet$};
\node at (0,-.35) {$\color{red}\scriptstyle{-a-b-2}$};
   \node at (.5,0) {$\color{blue}\bullet$};
\node at (0.65,0) {$\color{blue}\scriptstyle{b}$};
   \node at (-.5,0) {$\color{blue}\bullet$};
\node at (-0.65,0) {$\color{blue}\scriptstyle{a}$};
\end{tikzpicture}
}
\end{equation}
\begin{equation}
\mathord{
\begin{tikzpicture}[baseline = -1mm]
  \draw[->,thin,blue] (-0.45,-.6) to (-0.45,-.5)
        to[in=180,out=90] (.1,.4) to[out=0,in=90] (.45,0)
        to[out=-90,in=0] (.1,-.4) to [out=180,in=-90] (-0.45,.45) to (-0.45,.6);
   \node at (.45,0) {$\color{red}\clockright$};
\end{tikzpicture}
}
= \mathord{\begin{tikzpicture}[baseline=-1mm]
\draw[->,thin,blue] (-0.8,-0.6) to (-.8,0.6);
\draw[-,thin,red] (0,-0.38) to[out=180,in=-90] (-.38,0);
\draw[->,thin,red] (-0.38,0) to[out=90,in=180] (0,0.38);
\draw[-,thin,red] (0,0.38) to[out=0,in=90] (0.38,0);
\draw[-,thin,red] (0.378,0.02) to[out=-86,in=0] (0,-0.38);
     \node at (.38,0) {$\color{blue}\clockright$};
	\draw[-] (-.38,0.01) to (-0.8,0.01);
     \node at (-0.38,0) {$\dt$};
     \node at (-0.8,0) {$\dt$};
\end{tikzpicture}}
-\sum_{\substack{a\geq 0\\b \in \Z}}
\mathord{\begin{tikzpicture}[baseline=-1mm]
\draw[->,thin,blue] (0,-0.6) to (0,0.6);
   \node at (0,0) {$\color{blue}\bullet$};
   \node at (-.2,0) {$\color{blue}\scriptstyle{a}$};
\end{tikzpicture}}
\begin{array}{l}
\mathord{
\begin{tikzpicture}[baseline = 1.25mm]
  \draw[<-,thin,blue] (0,0.4) to[out=180,in=90] (-.2,0.2);
  \draw[-,thin,blue] (0.2,0.2) to[out=90,in=0] (0,.4);
 \draw[-,thin,blue] (-.2,0.2) to[out=-90,in=180] (0,0);
  \draw[-,thin,blue] (0,0) to[out=0,in=-90] (0.2,0.2);
   \node at (-0.2,0.2) {$\color{blue}\bullet$};
   \node at (-.38,0.2) {$\color{blue}\scriptstyle{b}$};
\end{tikzpicture}
}\\
\mathord{
\begin{tikzpicture}[baseline = 1.25mm]
  \draw[-,thin,red] (0.2,0.2) to[out=90,in=0] (0,.4);
  \draw[<-,thin,red] (0,0.4) to[out=180,in=90] (-.2,0.2);
\draw[-,thin,red] (-.2,0.2) to[out=-90,in=180] (0,0);
  \draw[-,thin,red] (0,0) to[out=0,in=-90] (0.2,0.2);
   \node at (0.2,0.2) {$\color{red}\bullet$};
   \node at (.7,0.2) {$\color{red}\scriptstyle{-a-b-2}$};
\end{tikzpicture}
}\end{array}\label{curl-red-blue}
\end{equation}

\begin{equation}\label{bigon-red-blue}
-\mathord{
\begin{tikzpicture}[baseline=-.5mm]
	\draw[->,thin,blue] (0,0.6) to[out=-90, in=90] (.5,0) to
        [out=-90,in=90] (0,-.6);
	\draw[<-,thin,blue] (.5,0.6) to[out=-90, in=90] (0,0) to
        [out=-90,in=90] (0.5,-.6);
     \node at (.5,0) {$\color{red}\clockright$};
\end{tikzpicture}
}\:-
\mathord{
\begin{tikzpicture}[baseline=-.5mm]
	\draw[<-,thin,blue] (.3,0.6) to[out=-90, in=0] (0,.1);
	\draw[-,thin,blue] (0,.1) to[out = 180, in = -90] (-.3,0.6);
	\draw[-,thin,blue] (.3,-.6) to[out=90, in=0] (0,-0.1);
	\draw[->,thin,blue] (0,-0.1) to[out = 180, in = 90] (-0.3,-.6);
\draw[-,thin,red] (1,-0.38) to[out=0,in=-90] (1.38,0);
\draw[-,thin,red] (1.38,0) to[out=90,in=0] (1,0.38);
\draw[<-,thin,red] (1,0.38) to[out=180,in=90] (.62,0);
\draw[-,thin,red] (.622,0.02) to[out=-86,in=180] (1,-0.38);
     \node at (1.38,0) {$\color{blue}\clockright$};
\draw[-] (.71,-.24) to (.22,-.24);
\draw[-] (.71,.26) to (.22,.26);
      \node at (.71,-0.25) {$\dt$};
   \node at (.71,0.25) {$\dt$};
      \node at (.22,-0.25) {$\dt$};
   \node at (0.22,0.25) {$\dt$};
\end{tikzpicture}}=
\mathord{
\begin{tikzpicture}[baseline = -1mm]
  \draw[->,thin,blue] (0,-.6) to (0,.6);
  \draw[<-,thin,blue] (-.5,-.6) to (-.5,.6);
   \node at (-.5,0) {$\color{red}\clockright$};
\end{tikzpicture}} 
-\displaystyle
\sum_{\substack{a,b \geq 0 \\ c \in \Z}}
\!\!\!\!\!\begin{array}{l}
\;\quad\mathord{
\begin{tikzpicture}[baseline = 1.25mm]
  \draw[<-,thin,blue] (0,0.4) to[out=180,in=90] (-.2,0.2);
  \draw[-,thin,blue] (0.2,0.2) to[out=90,in=0] (0,.4);
 \draw[-,thin,blue] (-.2,0.2) to[out=-90,in=180] (0,0);
  \draw[-,thin,blue] (0,0) to[out=0,in=-90] (0.2,0.2);
   \node at (-0.2,0.2) {$\color{blue}\bullet$};
   \node at (-.38,0.2) {$\color{blue}\scriptstyle{c}$};
\end{tikzpicture}
}\\
\mathord{
\begin{tikzpicture}[baseline = 1.25mm]
  \draw[-,thin,red] (0.2,0.2) to[out=90,in=0] (0,.4);
  \draw[<-,thin,red] (0,0.4) to[out=180,in=90] (-.2,0.2);
\draw[-,thin,red] (-.2,0.2) to[out=-90,in=180] (0,0);
  \draw[-,thin,red] (0,0) to[out=0,in=-90] (0.2,0.2);
   \node at (-0.2,0.2) {$\color{red}\bullet$};
   \node at (-.85,0.2) {$\color{red}\scriptstyle{-a-b-c-3}$};
\end{tikzpicture}
}\end{array}
\mathord{
\begin{tikzpicture}[baseline=-.5mm]
	\draw[<-,thin,blue] (0.3,0.6) to[out=-90, in=0] (0,.1);
	\draw[-,thin,blue] (0,.1) to[out = 180, in = -90] (-0.3,0.6);
      \node at (0.44,-0.3) {$\color{blue}\scriptstyle{b}$};
	\draw[-,thin,blue] (0.3,-.6) to[out=90, in=0] (0,-0.1);
	\draw[->,thin,blue] (0,-0.1) to[out = 180, in = 90] (-0.3,-.6);
      \node at (0.27,-0.3) {$\color{blue}\bullet$};
   \node at (0.27,0.3) {$\color{blue}\bullet$};
   \node at (.43,.3) {$\color{blue}\scriptstyle{a}$};
\end{tikzpicture}}
\end{equation}

\begin{equation}
\displaystyle\mathord{
\begin{tikzpicture}[baseline = -1mm]
  \draw[<-,thin,red] (1,-.6) to (1,.6);
  \draw[->,thin,blue] (-0.45,-.6) to (-0.45,-.5)
        to[in=180,out=90] (.1,.4) to[out=0,in=90] (.45,0)
        to[out=-90,in=0] (.1,-.4) to [out=180,in=-90] (-0.45,.45) to (-0.45,.6);
   \node at (.4,-0.2) {$\color{red}\clockright$};
	\draw[-] (.42,0.19) to (1,0.19);
     \node at (0.42,0.18) {$\dt$};
     \node at (1,0.18) {$\dt$};
\end{tikzpicture}
}\:-\:\mathord{
\begin{tikzpicture}[baseline = -1mm]
  \draw[->,thin,blue] (-1,-.6) to (-1,.6);
  \draw[<-,thin,red] (0.45,-.6) to (0.45,-.5)
        to[in=0,out=90] (-.1,.4) to[out=180,in=90] (-.45,0)
        to[out=-90,in=180] (-.1,-.4) to [out=0,in=-90] (0.45,.45) to (0.45,.6);
   \node at (-.4,-0.2) {$\color{blue}\clockright$};
	\draw[-] (-.42,0.19) to (-1,0.19);
     \node at (-0.42,0.18) {$\dt$};
     \node at (-1,0.18) {$\dt$};
\end{tikzpicture}
}
=
\sum_{\substack{a,b \geq 0 \\ c \in \Z}}
\mathord{
\begin{tikzpicture}[baseline = -1mm]
  \draw[->,thin,blue] (0,-.6) to (0,.6);
  \node at (0,0) {$\color{blue}\bullet$};
   \node at (-0.2,0) {$\color{blue}\scriptstyle{a}$};
\end{tikzpicture}}
\mathord{
\begin{tikzpicture}[baseline=2.5mm]
 \draw[-,thin,blue] (0,0.75) to[out=180,in=90] (-.2,0.55);
  \draw[-,thin,blue] (0.2,0.55) to[out=90,in=0] (0,.75);
 \draw[-,thin,blue] (-.2,0.55) to[out=-90,in=180] (0,.35);
  \draw[<-,thin,blue] (0,0.35) to[out=0,in=-90] (0.2,0.55);
   \node at (0,0.75) {$\color{blue}\bullet$};
   \node at (0.03,.93) {$\color{blue}\scriptstyle{c}$};
  \draw[-,thin,red] (0.2,0) to[out=90,in=0] (0,.2);
  \draw[<-,thin,red] (0,0.2) to[out=180,in=90] (-.2,0);
\draw[-,thin,red] (-.2,0) to[out=-90,in=180] (0,-0.2);
  \draw[-,thin,red] (0,-.2) to[out=0,in=-90] (0.2,0);
   \node at (0,-.2) {$\color{red}\bullet$};
   \node at (0,-0.38) {$\color{red}\scriptstyle{-a-b-c-3}$};
\end{tikzpicture}
}
\mathord{
\begin{tikzpicture}[baseline = -1mm]
  \draw[<-,thin,red] (0,-.6) to (0,.6);
  \node at (0,0) {$\color{red}\bullet$};
   \node at (0.2,0) {$\color{red}\scriptstyle{b}$};
\end{tikzpicture}}\label{teleporter-on-loop}
\end{equation}

\begin{equation}
\mathord{
\begin{tikzpicture}[baseline = -1mm]
	\draw[<-,thin,red] (0.45,.6) to (-0.45,-.6);
	\draw[->,thin,blue] (0.45,-.6) to (-0.45,.6);
   \node at (-0.25,0.3) {$\color{red}\clockright$};
\end{tikzpicture}
}=\mathord{
\begin{tikzpicture}[baseline = -1mm]
	\draw[<-,thin,red] (0.45,.6) to (-0.45,-.6);
	\draw[->,thin,blue] (0.45,-.6) to (-0.45,.6);
   \node at (0.25,-0.3) {$\color{red}\clockright$};
\end{tikzpicture}
}-
\mathord{
\begin{tikzpicture}[baseline = -1mm]
	\draw[<-,thin,red] (0.45,.6) to (-0.45,-.6);
	\draw[->,thin,blue] (0.45,-.6) to (-0.45,.6);
   \node at (0.25,-0.3) {$\color{red}\clockright$};
	\draw[-] (.2,0.26) to (-0.2,0.26);
     \node at (0.2,0.25) {$\dt$};
     \node at (-0.2,0.25) {$\dt$};
	\draw[-] (.3,0.4) to (-0.3,0.4);
     \node at (0.3,0.39) {$\dt$};
     \node at (-0.3,0.39) {$\dt$};
\end{tikzpicture}
}\label{internal-through-red-blue} 
\end{equation}
Since the proofs are identical to those of \cite[Lem. 5.6--12]{brundanHeisenbergKacMoody2020} except for minor sign differences, we omit the proofs.

First, we check the relations of the nilHecke algebra \crefrange{QHA1}{QHA3} in the case where all labels are the same.    This is not difficult to do ``by hand,'' but it might be clearer to explain more conceptually why this works.  Fix an integer $n$ and consider the direct sum of one copy of $\K[y_1,\dots,y_n]$ for each sequence $\mathbf{c}\in \{r,b\}^n$, and localize this ring at the multiplicative set $\{y_i-y_j\}$ for $i,j$ ranging over all pairs where $c_i\neq c_j$.  Let $S_n$ act on this direct sum by permuting variables and entries of $\mathbf{c}$ simultaneously.  We have an action of the nilHecke algebra on this direct sum which sends $\psi_k$ to $\frac{1}{y_k-y_{k+1}}(1-s_k)$; if $c_k=c_{k+1}$, this is just the usual divided difference operator on the corresponding polynomial ring, whereas if $c_k\neq c_{k+1}$, this operator corresponds to the last 4 terms in \cref{com0older}, and in particular, relies on using the fact that we have inverted $y_k-y_{k+1}$.  Thus, checking the relations of the nilHecke algebra \crefrange{QHA1}{QHA3} simply consists of translating the proof that these divided difference operators satisfy the nilHecke relations into diagrams.  

For instance, \cref{QHA1} reduces to the equations below and their mirror images:
\begin{align*}
\mathord{
\begin{tikzpicture}[baseline = -1mm]
	\draw[<-,thin,blue] (0.25,.3) to (-0.25,-.3);
	\draw[->,thin,blue] (0.25,-.3) to (-0.25,.3);
     \node at (-0.12,-0.145) {$\blue{\bullet}$};
\end{tikzpicture}
}
=
\mathord{
\begin{tikzpicture}[baseline = -1mm]
	\draw[<-,blue,thin] (0.25,.3) to (-0.25,-.3);
	\draw[->,blue,thin] (0.25,-.3) to (-0.25,.3);
     \node at (0.12,0.135) {$\blue{\bullet}$};
\end{tikzpicture}}
+\:\mathord{
\begin{tikzpicture}[baseline = -1mm]
 	\draw[->,thin,blue] (0.08,-.3) to (0.08,.3);
	\draw[->,thin,blue] (-0.28,-.3) to (-0.28,.3);
\end{tikzpicture}
}\:,\qquad
\mathord{
\begin{tikzpicture}[baseline = -1mm]
	\draw[<-,thin,blue] (0.25,.3) to (-0.25,-.3);
	\draw[->,thin,red] (0.25,-.3) to (-0.25,.3);
	\draw[-] (-0.12,-.135) to (0.12,-.135);
     \node at (-0.12,-0.145) {$\dt$};
     \node at (0.12,-0.145) {$\dt$};
     \node at (0.12,0.135) {$\blue{\bullet}$};
\end{tikzpicture}
}
=
\mathord{
\begin{tikzpicture}[baseline = -1mm]
	\draw[<-,thin,blue] (0.25,.3) to (-0.25,-.3);
	\draw[->,thin,red] (0.25,-.3) to (-0.25,.3);
	\draw[-] (-0.12,.155) to (0.12,.155);
     \node at (-0.12,0.145) {$\dt$};
     \node at (0.12,0.145) {$\dt$};
     \node at (-0.12,-0.145) {$\blue{\bullet}$};
\end{tikzpicture}
}
\:,
\qquad
\mathord{
\begin{tikzpicture}[baseline = -1mm]
 	\draw[->,thin,blue] (0.08,-.3) to (0.08,.3);
	\draw[->,thin,red] (-0.28,-.3) to (-0.28,.3);
	\draw[-] (-0.28,.06) to (0.08,.06);
     \node at (-0.28,0.05) {$\dt$};
     \node at (0.08,0.05) {$\dt$};
     \node at (-0.28,-0.16) {$\red{\bullet}$};
\end{tikzpicture}
}
=\:
\mathord{
\begin{tikzpicture}[baseline = -1mm]
 	\draw[->,thin,blue] (0.08,-.3) to (0.08,.3);
	\draw[->,thin,red] (-0.28,-.3) to (-0.28,.3);
	\draw[-] (-0.28,-.15) to (0.08,-.15);
     \node at (-0.28,-0.16) {$\dt$};
     \node at (0.08,-0.16) {$\dt$};
     \node at (0.08,0.05) {$\blue{\bullet}$};
\end{tikzpicture}
}
+\:
\mathord{
\begin{tikzpicture}[baseline = -1mm]
 	\draw[->,thin,blue] (0.08,-.3) to (0.08,.3);
	\draw[->,thin,red] (-0.28,-.3) to (-0.28,.3);
\end{tikzpicture}
}
\:.
\end{align*}  Only the first relation of \cref{QHA2} will appear, and this quadratic relation follows by:  
\begin{equation*}
\mathord{
\begin{tikzpicture}[baseline = -1mm]
	\draw[->,thin,blue] (-.2,-.4) to [out=90,in=-90] (0.2,0) to[out=90,in=-90] (-0.2,.4);
	\draw[->,thin,blue] (.2,-.4) to [out=90,in=-90] (-0.2,0) to[out=90,in=-90] (0.2,.4);
\end{tikzpicture}
}
-
\mathord{
\begin{tikzpicture}[baseline = -1mm]
	\draw[->,thin,red] (-.2,-.4) to [out=90,in=-90] (0.2,0) to[out=90,in=-90] (-0.2,.4);
	\draw[->,thin,blue] (.2,-.4) to [out=90,in=-90] (-0.2,0) to[out=90,in=-90] (0.2,.4);
\draw[-] (.17,-.31) to (-.17,-.31);
      \node at (-.17,-0.32) {$\dt$};
   \node at (.17,-0.32) {$\dt$};
\draw[-] (.17,.09) to (-.17,.09);
      \node at (-.17,0.08) {$\dt$};
   \node at (.17,0.08) {$\dt$};
\end{tikzpicture}
}
+
\mathord{
\begin{tikzpicture}[baseline = -1mm]
	\draw[->,thin,red] (0.2,-.4) to (-0.2,.4);
	\draw[->,thin,blue] (-0.2,-.4) to (0.2,.4);
\draw[-] (.15,-.31) to (-.15,-.31);
      \node at (-.15,-0.32) {$\dt$};
   \node at (.15,-0.32) {$\dt$};
\draw[-] (.09,-.18) to (-.09,-.18);
      \node at (-.09,-0.19) {$\dt$};
   \node at (.09,-0.19) {$\dt$};
\end{tikzpicture}}
-
\mathord{
\begin{tikzpicture}[baseline = -1mm]
	\draw[->,thin,red] (0.2,-.4) to (-0.2,.4);
	\draw[->,thin,blue] (-0.2,-.4) to (0.2,.4);
\draw[-] (.09,.18) to (-.09,.18);
      \node at (-.09,0.17) {$\dt$};
   \node at (.09,0.17) {$\dt$};
\draw[-] (.09,-.18) to (-.09,-.18);
      \node at (-.09,-0.19) {$\dt$};
   \node at (.09,-0.19) {$\dt$};
\end{tikzpicture}}
+\mathord{
\begin{tikzpicture}[baseline = -1mm]
	\draw[->,thin,blue] (0.2,-.4) to (0.2,.4);
	\draw[->,thin,red] (-0.2,-.4) to (-0.2,.4);
\draw[-] (.2,.18) to (-.2,.18);
      \node at (-.2,0.17) {$\dt$};
   \node at (.2,0.17) {$\dt$};
\draw[-] (.2,-.18) to (-.2,-.18);
      \node at (-.2,-0.19) {$\dt$};
   \node at (.2,-0.19) {$\dt$};
\end{tikzpicture}
}+\:\mathord{
\begin{tikzpicture}[baseline = -1mm]
	\draw[->,thin,red] (-.2,-.4) to [out=90,in=-90] (0.2,0) to[out=90,in=-90] (-0.2,.4);
	\draw[->,thin,red] (.2,-.4) to [out=90,in=-90] (-0.2,0) to[out=90,in=-90] (0.2,.4);
\end{tikzpicture}
}
-
\mathord{
\begin{tikzpicture}[baseline = -1mm]
	\draw[->,thin,blue] (-.2,-.4) to [out=90,in=-90] (0.2,0) to[out=90,in=-90] (-0.2,.4);
	\draw[->,thin,red] (.2,-.4) to [out=90,in=-90] (-0.2,0) to[out=90,in=-90] (0.2,.4);
\draw[-] (.17,-.31) to (-.17,-.31);
      \node at (-.17,-0.32) {$\dt$};
   \node at (.17,-0.32) {$\dt$};
\draw[-] (.17,.09) to (-.17,.09);
      \node at (-.17,0.08) {$\dt$};
   \node at (.17,0.08) {$\dt$};
\end{tikzpicture}
}
+
\mathord{
\begin{tikzpicture}[baseline = -1mm]
	\draw[->,thin,blue] (0.2,-.4) to (-0.2,.4);
	\draw[->,thin,red] (-0.2,-.4) to (0.2,.4);
\draw[-] (.15,-.31) to (-.15,-.31);
      \node at (-.15,-0.32) {$\dt$};
   \node at (.15,-0.32) {$\dt$};
\draw[-] (.09,-.18) to (-.09,-.18);
      \node at (-.09,-0.19) {$\dt$};
   \node at (.09,-0.19) {$\dt$};
\end{tikzpicture}}
-
\mathord{
\begin{tikzpicture}[baseline = -1mm]
	\draw[->,thin,blue] (0.2,-.4) to (-0.2,.4);
	\draw[->,thin,red] (-0.2,-.4) to (0.2,.4);
\draw[-] (.09,.18) to (-.09,.18);
      \node at (-.09,0.17) {$\dt$};
   \node at (.09,0.17) {$\dt$};
\draw[-] (.09,-.18) to (-.09,-.18);
      \node at (-.09,-0.19) {$\dt$};
   \node at (.09,-0.19) {$\dt$};
\end{tikzpicture}}
+\mathord{
\begin{tikzpicture}[baseline = -1mm]
	\draw[->,thin,red] (0.2,-.4) to (0.2,.4);
	\draw[->,thin,blue] (-0.2,-.4) to (-0.2,.4);
\draw[-] (.2,.18) to (-.2,.18);
\node at (-.2,0.17) {$\dt$};
   \node at (.2,0.17) {$\dt$};
\draw[-] (.2,-.18) to (-.2,-.18);
      \node at (-.2,-0.19) {$\dt$};
   \node at (.2,-0.19) {$\dt$};
\end{tikzpicture}
}=0\:.
\end{equation*}
The braid type relation \cref{QHA3} is a similar calculation; we will spare the reader the diagrammatic proof of it, but note that it corresponds to the proof that $\psi_k\psi_{k+1}\psi_k=\psi_{k+1}\psi_k\psi_{k+1}$ in the polynomial representation discussed at the start of the proof.
	
Now, we turn to the relations \crefrange{rightadj}{flight2}.
The relations \cref{rightadj} are immediate and \cref{leftadj} follow from \cref{internal-inverse}.  The proof of \crefrange{loop1}{loop2} is exactly as in \cite[Th. 5.4]{brundanDegenerateHeisenberg2023}:  the LHS is sent to $$
-
\mathord{
\begin{tikzpicture}[baseline = -1mm]
  \draw[->,thin,blue] (-0.45,-.4) to (-0.45,-.3)
        to[in=180,out=90] (-.2,.2) to[out=0,in=90] (0,0.05)
        to[out=-90,in=0] (-.2,-.2) to [out=180,in=-90] (-0.45,.3) to (-0.45,.4);
   \node at (0,0) {$\color{red}\smallclock$};
\end{tikzpicture}
}
+
\mathord{\begin{tikzpicture}[baseline=-1mm]
\draw[->,thin,blue] (-0.6,-0.4) to (-.6,0.4);
\draw[-,thin,red] (0,-0.2) to[out=180,in=-90] (-.2,0);
\draw[->,thin,red] (-0.2,0) to[out=90,in=180] (0,0.2);
\draw[-,thin,red] (0,0.2) to[out=0,in=90] (0.2,0);
\draw[-,thin,red] (0.2,0) to[out=-90,in=0] (0,-0.2);
     \node at (.2,0) {$\color{blue}\smallclock$};
	\draw[-] (-.2,0.01) to (-0.6,0.01);
     \node at (-0.2,0) {$\dt$};
     \node at (-0.6,0) {$\dt$};
\end{tikzpicture}}
-
\mathord{
\begin{tikzpicture}[baseline = -1mm]
  \draw[->,thin,red] (-0.45,-.4) to (-0.45,-.3)
        to[in=180,out=90] (-.2,.2) to[out=0,in=90] (0,0.05)
        to[out=-90,in=0] (-.2,-.2) to [out=180,in=-90] (-0.45,.3) to (-0.45,.4);
   \node at (0,0) {$\color{blue}\smallclock$};
\end{tikzpicture}
}
-
\mathord{\begin{tikzpicture}[baseline=-1mm]
\draw[->,thin,red] (-0.6,-0.4) to (-.6,0.4);
\draw[-,thin,blue] (0,-0.2) to[out=180,in=-90] (-.2,0);
\draw[->,thin,blue] (-0.2,0) to[out=90,in=180] (0,0.2);
\draw[-,thin,blue] (0,0.2) to[out=0,in=90] (0.2,0);
\draw[-,thin,blue] (0.2,0) to[out=-90,in=0] (0,-0.2);
     \node at (.2,0) {$\color{red}\smallclock$};
	\draw[-] (-.2,0.01) to (-0.6,0.01);
     \node at (-0.2,0) {$\dt$};
     \node at (-0.6,0) {$\dt$};
\end{tikzpicture}}
= \delta_{k,0}\:
\mathord{\begin{tikzpicture}[baseline=-1mm]
\draw[->,thin,blue] (0,-0.4) to (0,0.4);
\end{tikzpicture}}
\:+\: \delta_{k,0}\:
\mathord{\begin{tikzpicture}[baseline=-1mm]
\draw[->,thin,red] (0,-0.4) to (0,0.4);
\end{tikzpicture}}
\:.
$$
where the equality follows from \cref{curl-red-blue}.  

Finally, we turn to the proof of \crefrange{flight1}{flight2}. 
Under our functor, the sideways crossings will be sent to:
	\begin{align*}
\mathord{
\begin{tikzpicture}[baseline = 0]
	\draw[<-,thin] (0.28,-.3) to (-0.28,.4);
	\draw[->,thin] (-0.28,-.3) to (0.28,.4);
\end{tikzpicture}
}
&\mapsto
\mathord{
\begin{tikzpicture}[baseline = 0]
	\draw[<-,thin,blue] (0.28,-.3) to (-0.28,.4);
	\draw[->,thin,blue] (-0.28,-.3) to (0.28,.4);
\end{tikzpicture}
}
+
\mathord{
\begin{tikzpicture}[baseline = 0]
	\draw[<-,thin,red] (0.28,-.3) to (-0.28,.4);
	\draw[->,thin,red] (-0.28,-.3) to (0.28,.4);
\end{tikzpicture}
}
-\mathord{
\begin{tikzpicture}[baseline = 0]
	\draw[<-,thin,blue] (0.28,-.3) to (-0.28,.4);
	\draw[->,thin,red] (-0.28,-.3) to (0.28,.4);
	\draw[-] (-0.15,-.14) to (0.15,-.14);
     \node at (0.15,-.15) {$\dt$};
     \node at (-0.15,-.15) {$\dt$};
\end{tikzpicture}
}
+\mathord{
\begin{tikzpicture}[baseline = 0]
	\draw[<-,thin,red] (0.28,-.3) to (-0.28,.4);
	\draw[->,thin,blue] (-0.28,-.3) to (0.28,.4);
	\draw[-] (-0.15,-.14) to (0.15,-.14);
     \node at (0.15,-.15) {$\dt$};
     \node at (-0.15,-.15) {$\dt$};
\end{tikzpicture}
}
-
\mathord{
\begin{tikzpicture}[baseline = 0mm]
	\draw[<-,thin,blue] (0.4,0.4) to[out=-90, in=0] (0.1,-.1);
	\draw[-,thin,blue] (0.1,-.1) to[out = 180, in = -90] (-0.2,0.4);
	\draw[<-,thin,red] (1.4,-.3) to[out=90, in=0] (1.1,0.2);
	\draw[-,thin,red] (1.1,0.2) to[out = 180, in = 90] (0.8,-.3);
	\draw[-] (.34,0.06) to (0.88,0.06);
     \node at (0.34,0.05) {$\dt$};
     \node at (0.88,0.05) {$\dt$};
\end{tikzpicture}
}+
\mathord{
\begin{tikzpicture}[baseline = 0mm]
	\draw[<-,thin,red] (0.4,0.4) to[out=-90, in=0] (0.1,-.1);
	\draw[-,thin,red] (0.1,-.1) to[out = 180, in = -90] (-0.2,0.4);
	\draw[<-,thin,blue] (1.4,-.3) to[out=90, in=0] (1.1,0.2);
	\draw[-,thin,blue] (1.1,0.2) to[out = 180, in = 90] (0.8,-.3);
	\draw[-] (.34,0.06) to (0.88,0.06);
     \node at (0.34,0.05) {$\dt$};
     \node at (0.88,0.05) {$\dt$};
\end{tikzpicture}
}\:,\\
\mathord{
\begin{tikzpicture}[baseline = 0]
	\draw[->,thin] (0.28,-.3) to (-0.28,.4);
	\draw[<-,thin] (-0.28,-.3) to (0.28,.4);
\end{tikzpicture}
}&\mapsto
-\mathord{
\begin{tikzpicture}[baseline = 0]
	\draw[->,thin,blue] (0.28,-.3) to (-0.28,.4);
	\draw[<-,thin,blue] (-0.28,-.3) to (0.28,.4);
     \node at (.13,.21) {$\color{red}\smallclock$};
\node at (-.14,-.09) {$\color{red}\smallanticlock$};
\end{tikzpicture}
}
-\mathord{
\begin{tikzpicture}[baseline = 0]
	\draw[->,thin,red] (0.28,-.3) to (-0.28,.4);
	\draw[<-,thin,red] (-0.28,-.3) to (0.28,.4);
     \node at (.13,.21) {$\color{blue}\smallclock$};
\node at (-.14,-.09) {$\color{blue}\smallanticlock$};
\end{tikzpicture}
}
+\mathord{
\begin{tikzpicture}[baseline = 0]
	\draw[->,thin,blue] (0.35,-.3) to (-0.35,.4);
	\draw[<-,thin,red] (-0.35,-.3) to (0.35,.4);
	\draw[-] (-0.25,-.21) to (0.25,-.21);
     \node at (0.25,-.22) {$\dt$};
     \node at (-0.25,-.22) {$\dt$};
     \node at (.16,.23) {$\color{blue}\smallclock$};
\node at (-.11,-.06) {$\color{blue}\smallanticlock$};
\end{tikzpicture}
}
-\mathord{
\begin{tikzpicture}[baseline = 0]
	\draw[->,thin,red] (0.35,-.3) to (-0.35,.4);
	\draw[<-,thin,blue] (-0.35,-.3) to (0.35,.4);
	\draw[-] (-0.25,-.21) to (0.25,-.21);
     \node at (0.25,-.22) {$\dt$};
     \node at (-0.25,-.22) {$\dt$};
     \node at (.16,.23) {$\color{red}\smallclock$};
\node at (-.11,-.06) {$\color{red}\smallanticlock$};
\end{tikzpicture}
}
-
\mathord{
\begin{tikzpicture}[baseline = 0mm]
	\draw[-,thin,red] (1.4,0.4) to[out=-90, in=0] (1.1,-.1);
	\draw[->,thin,red] (1.1,-.1) to[out = 180, in = -90] (0.8,0.4);
	\draw[-,thin,blue] (.4,-.3) to[out=90, in=0] (.1,0.2);
	\draw[->,thin,blue] (.1,0.2) to[out = 180, in = 90] (-0.2,-.3);
	\draw[-] (.34,0.06) to (0.88,0.06);
\node at (1.35,.15) {$\color{blue}\smallclock$};
\node at (-.15,-.04) {$\color{red}\smallanticlock$};
     \node at (0.34,0.05) {$\dt$};
     \node at (0.88,0.05) {$\dt$};
\end{tikzpicture}
}+
\mathord{
\begin{tikzpicture}[baseline = 0mm]
	\draw[-,thin,blue] (1.4,0.4) to[out=-90, in=0] (1.1,-.1);
	\draw[->,thin,blue] (1.1,-.1) to[out = 180, in = -90] (0.8,0.4);
	\draw[-,thin,red] (.4,-.3) to[out=90, in=0] (.1,0.2);
	\draw[->,thin,red] (.1,0.2) to[out = 180, in = 90] (-0.2,-.3);
	\draw[-] (.34,0.06) to (0.88,0.06);
     \node at (0.34,0.05) {$\dt$};
     \node at (0.88,0.05) {$\dt$};
\node at (1.35,.15) {$\color{red}\smallclock$};
\node at (-.15,-.04) {$\color{blue}\smallanticlock$};
\end{tikzpicture}
}\:.
\end{align*}
Thus, the image of the LHS of \cref{flight1} is:
\begin{align}
	\mathord{
\begin{tikzpicture}[baseline = 0mm]
	\draw[->,thin] (0.28,0) to[out=90,in=-90] (-0.28,.6);
	\draw[-,thin] (-0.28,0) to[out=90,in=-90] (0.28,.6);
	\draw[<-,thin] (0.28,-.6) to[out=90,in=-90] (-0.28,0);
	\draw[-,thin] (-0.28,-.6) to[out=90,in=-90] (0.28,0);
\end{tikzpicture}
}&=A+B+C & A&=-\mathord{
\begin{tikzpicture}[baseline=-.5mm]
	\draw[->,thin,blue] (0,0.6) to[out=-90, in=90] (.5,0) to
        [out=-90,in=90] (0,-.6);
	\draw[<-,thin,blue] (.5,0.6) to[out=-90, in=90] (0,0) to
        [out=-90,in=90] (0.5,-.6);
     \node at (.39,-.15) {$\color{red}\smallclock$};
\node at (0.11,-.39) {$\color{red}\smallanticlock$};
\end{tikzpicture}
}-
\mathord{
\begin{tikzpicture}[baseline=-.5mm]
	\draw[<-,thin,blue] (.3,0.6) to[out=-90, in=0] (0,.1);
	\draw[-,thin,blue] (0,.1) to[out = 180, in = -90] (-.3,0.6);
	\draw[-,thin,blue] (.3,-.6) to[out=90, in=0] (0,-0.1);
	\draw[->,thin,blue] (0,-0.1) to[out = 180, in = 90] (-0.3,-.6);
\draw[-,thin,red] (.9,-0.38) to[out=0,in=-90] (1.18,0);
\draw[-,thin,red] (1.18,0) to[out=90,in=0] (.9,0.38);
\draw[<-,thin,red] (.9,0.38) to[out=180,in=90] (.62,0);
\draw[-,thin,red] (.622,0.02) to[out=-86,in=180] (.9,-0.38);
     \node at (1.18,0) {$\color{blue}\smallclock$};
\draw[-] (.71,-.24) to (.22,-.24);
\draw[-] (.71,.26) to (.22,.26);
      \node at (.71,-0.25) {$\dt$};
   \node at (.71,0.25) {$\dt$};
      \node at (.22,-0.25) {$\dt$};
   \node at (0.22,0.25) {$\dt$};
\node at (-0.25,-.33) {$\color{red}\smallanticlock$};
\end{tikzpicture}}\\
B&=-\mathord{
\begin{tikzpicture}[baseline=-.5mm]
	\draw[->,thin,blue] (0,0.6) to[out=-90, in=90] (.5,0) to
        [out=-90,in=90] (0,-.6);
	\draw[<-,thin,red] (.5,0.6) to[out=-90, in=90] (0,0) to
        [out=-90,in=90] (0.5,-.6);
     \node at (.39,-.15) {$\color{red}\smallclock$};
\node at (0.11,-.39) {$\color{red}\smallanticlock$};
\end{tikzpicture}
}+
\mathord{
\begin{tikzpicture}[baseline=-.5mm]
	\draw[->,thin,blue] (0,0.6) to[out=-90, in=90] (.5,0) to
        [out=-90,in=90] (0,-.6);
	\draw[<-,thin,red] (.5,0.6) to[out=-90, in=90] (0,0) to
        [out=-90,in=90] (0.5,-.6);
     \node at (.39,-.15) {$\color{red}\smallclock$};
\node at (0.11,-.39) {$\color{red}\smallanticlock$};
\draw[-] (.44,.42) to (.07,.42);
      \node at (.44,0.41) {$\dt$};
   \node at (.07,0.41) {$\dt$};
\draw[-] (.44,.16) to (.07,.16);
      \node at (.44,0.15) {$\dt$};
   \node at (.07,0.15) {$\dt$};
\end{tikzpicture}
}& C&=\mathord{
\begin{tikzpicture}[baseline = -3.8mm]
	\draw[<-,thin,blue] (0.4,0.4) to[out=-90, in=0] (0.1,-.1);
	\draw[-,thin,blue] (0.1,-.1) to[out = 180, in = -90] (-0.2,0.4);
	\draw[<-,thin,red] (.8,-.7) to [out=90,in=-90] (1.4,-.1) to[out=90, in=0] (1.1,0.2);
	\draw[-,thin,red] (1.1,0.2) to[out = 180, in = 90] (0.8,-.1)
        to[out=-90,in=90] (1.4,-.7);
	\draw[-] (.34,0.06) to (0.85,0.06);
     \node at (0.34,0.05) {$\dt$};
     \node at (0.85,0.05) {$\dt$};
\node at (1.38,-.22) {$\color{blue}\smallclock$};
\node at (.9,-.62) {$\color{blue}\smallanticlock$};
\end{tikzpicture}
}-\mathord{
\begin{tikzpicture}[baseline = 0mm]
	\draw[-,thin,blue] (.8,.8) to [out=-90,in=90] (1.4,0.2) to[out=-90, in=0] (1.1,-.1);
	\draw[->,thin,blue] (1.1,-.1) to[out = 180, in = -90]
        (0.8,0.2) to [out=90,in=-90] (1.4,.8);
	\draw[-,thin,red] (.4,-.3) to[out=90, in=0] (.1,0.2);
	\draw[->,thin,red] (.1,0.2) to[out = 180, in = 90] (-0.2,-.3);
	\draw[-] (.34,0.06) to (0.85,0.06);
     \node at (0.34,0.05) {$\dt$};
     \node at (0.85,0.05) {$\dt$};
\node at (1.35,.15) {$\color{red}\smallclock$};
\node at (-.15,-.04) {$\color{blue}\smallanticlock$};
\end{tikzpicture}
}
\end{align}
We simplify these expressions as follows:
\begin{itemize}
\item[(A)] For $A$, we multiply \cref{bigon-red-blue} on the bottom left strand by $\mathord{
\begin{tikzpicture}[baseline=2mm]
	\draw[<-,thin,blue] (0,-.2) to[out=90,in=-90] (0,.8);
     \node at (0,.3) {$\color{red}\anticlockleft$};
\end{tikzpicture}
}$, so the lefthand side of the equation matches $A$, using \cref{internal-inverse} to cancel pairs internal bubbles.
\item[(B)] For $B$,  we apply a 90 degree counter-clockwise rotation to \cref{internal-through-red-blue}, and multiply it on the bottom left by $\mathord{
\begin{tikzpicture}[baseline=2mm]
	\draw[<-,thin,blue] (0,-.2) to[out=90,in=-90] (0,.8);
     \node at (0,.3) {$\color{red}\anticlockleft$};
\end{tikzpicture}
}$ and on the top by $\mathord{
\begin{tikzpicture}[baseline = 0]
	\draw[<-,thin,blue] (0.28,-.3) to (-0.28,.4);
	\draw[->,thin,red] (-0.28,-.3) to (0.28,.4);
\end{tikzpicture}
}$, so the right-hand side of the equation matches $B$.  Using the cancellation of internal bubbles \cref{internal-inverse}, the fact that internal bubbles and teleporters commute, we arrive at a  red-blue bigon $\mathord{
\begin{tikzpicture}[baseline=-.5mm]
	\draw[->,thin,blue] (0,0.6) to[out=-90, in=90] (.5,0) to
        [out=-90,in=90] (0,-.6);
	\draw[<-,thin,red] (.5,0.6) to[out=-90, in=90] (0,0) to
        [out=-90,in=90] (0.5,-.6);
\end{tikzpicture}
}$.  We can straighten this by \cref{flight2}.
 \item[(C)] For $C$, we multiply the outgoing red strand of \cref{teleporter-on-loop} by $\mathord{
\begin{tikzpicture}[baseline=2mm]
	\draw[<-,thin,red] (0,-.2) to[out=90,in=-90] (0,.8);
     \node at (0,.3) {$\color{blue}\anticlockleft$};
\end{tikzpicture}
}$.  After isotopy, left-hand side of the equation matches $C$.
\end{itemize}
Taking the other side of the equation in each of the cases above, we arrive at:
\begin{equation}
A=\mathord{
\begin{tikzpicture}[baseline = -1mm]
  \draw[->,thin,blue] (0,-.6) to (0,.6);
  \draw[<-,thin,blue] (-.5,-.6) to (-.5,.6);
\end{tikzpicture}}
\:-\;\displaystyle
\sum_{\substack{a,b \geq 0 \\ c \in \Z}}
\!\!\!\!\!\begin{array}{l}
\;\quad\mathord{
\begin{tikzpicture}[baseline = 1.25mm]
  \draw[<-,thin,blue] (0,0.4) to[out=180,in=90] (-.2,0.2);
  \draw[-,thin,blue] (0.2,0.2) to[out=90,in=0] (0,.4);
 \draw[-,thin,blue] (-.2,0.2) to[out=-90,in=180] (0,0);
  \draw[-,thin,blue] (0,0) to[out=0,in=-90] (0.2,0.2);
   \node at (-0.2,0.2) {$\color{blue}\bullet$};
   \node at (-.38,0.2) {$\color{blue}\scriptstyle{c}$};
\end{tikzpicture}
}\\
\mathord{
\begin{tikzpicture}[baseline = 1.25mm]
  \draw[-,thin,red] (0.2,0.2) to[out=90,in=0] (0,.4);
  \draw[<-,thin,red] (0,0.4) to[out=180,in=90] (-.2,0.2);
\draw[-,thin,red] (-.2,0.2) to[out=-90,in=180] (0,0);
  \draw[-,thin,red] (0,0) to[out=0,in=-90] (0.2,0.2);
   \node at (-0.2,0.2) {$\color{red}\bullet$};
   \node at (-.85,0.2) {$\color{red}\scriptstyle{-a-b-c-3}$};
\end{tikzpicture}
}\end{array}
\mathord{
\begin{tikzpicture}[baseline=-.5mm]
	\draw[<-,thin,blue] (0.3,0.6) to[out=-90, in=0] (0,.1);
	\draw[-,thin,blue] (0,.1) to[out = 180, in = -90] (-0.3,0.6);
      \node at (0.44,-0.3) {$\color{blue}\scriptstyle{b}$};
	\draw[-,thin,blue] (0.3,-.6) to[out=90, in=0] (0,-0.1);
	\draw[->,thin,blue] (0,-0.1) to[out = 180, in = 90] (-0.3,-.6);
      \node at (0.27,-0.3) {$\color{blue}\bullet$};
   \node at (0.27,0.3) {$\color{blue}\bullet$};
   \node at (.43,.3) {$\color{blue}\scriptstyle{a}$};
\node at (-0.25,-.33) {$\color{red}\smallanticlock$};
\end{tikzpicture}}
\qquad B=\mathord{
\begin{tikzpicture}[baseline = -1mm]
  \draw[->,thin,red] (0,-.6) to (0,.6);
  \draw[<-,thin,blue] (-.5,-.6) to (-.5,.6);
\end{tikzpicture}}\qquad 
C=-\sum_{\substack{a,b \geq 0 \\ c \in \Z}}
\!\!\!\!\!\begin{array}{l}
\;\quad\mathord{
\begin{tikzpicture}[baseline = 1.25mm]
  \draw[<-,thin,blue] (0,0.4) to[out=180,in=90] (-.2,0.2);
  \draw[-,thin,blue] (0.2,0.2) to[out=90,in=0] (0,.4);
 \draw[-,thin,blue] (-.2,0.2) to[out=-90,in=180] (0,0);
  \draw[-,thin,blue] (0,0) to[out=0,in=-90] (0.2,0.2);
   \node at (-0.2,0.2) {$\color{blue}\bullet$};
   \node at (-.38,0.2) {$\color{blue}\scriptstyle{c}$};
\end{tikzpicture}
}\\
\mathord{
\begin{tikzpicture}[baseline = 1.25mm]
  \draw[-,thin,red] (0.2,0.2) to[out=90,in=0] (0,.4);
  \draw[<-,thin,red] (0,0.4) to[out=180,in=90] (-.2,0.2);
\draw[-,thin,red] (-.2,0.2) to[out=-90,in=180] (0,0);
  \draw[-,thin,red] (0,0) to[out=0,in=-90] (0.2,0.2);
   \node at (-0.2,0.2) {$\color{red}\bullet$};
   \node at (-.85,0.2) {$\color{red}\scriptstyle{-a-b-c-3}$};
\end{tikzpicture}
}\end{array}
\mathord{
\begin{tikzpicture}[baseline=-.5mm]
	\draw[<-,thin,blue] (0.3,0.6) to[out=-90, in=0] (0,.1);
	\draw[-,thin,blue] (0,.1) to[out = 180, in = -90] (-0.3,0.6);
      \node at (0.44,-0.3) {$\color{red}\scriptstyle{b}$};
	\draw[-,thin,red] (0.3,-.6) to[out=90, in=0] (0,-0.1);
	\draw[->,thin,red] (0,-0.1) to[out = 180, in = 90] (-0.3,-.6);
      \node at (0.27,-0.3) {$\color{red}\bullet$};
   \node at (0.27,0.3) {$\color{blue}\bullet$};
   \node at (.43,.3) {$\color{blue}\scriptstyle{a}$};
\node at (-.25,-.3) {$\color{blue}\smallanticlock$};
\end{tikzpicture}}
\end{equation}
This immediately completes the proof of \cref{flight1}, and the proof of \cref{flight2} is identical.  
\end{proof}

\subsection{Proof of \cref{lem:g-action}}

Now we turn to the proof of \cref{lem:g-action}.  By \cref{lem:BrundanRouquier}, we only need to give the image of the generators \crefrange{eq:BrundanRouquier-gens}{eq:BrundanRouquier-gens2} and verify the relations (KM1-3).  These images are given in the statement of \cref{lem:g-action}.

To verify this result, we need to check the relations (KM1-3).  As we note in the proof below, the conditions (KM1) and (KM3) only involve a single node in the Dynkin diagram and thus follow from the $\mathfrak{sl}_2$-case established in \cref{lem:many-sl2-general}.  Thus, we need to focus on the condition (KM2) and consequently on the relation between the KLR algebras $R$ for $\fg$ and $\mathsf{R}$ for $\tilde{\fg}$. These are related by some results which we think will be of some independent interest for those who are interested in KLR algebras.  While completeness of the spectrum $U_i$ (in the sense of \cref{def:tilde-I}) is needed for the comparison of the full Kac-Moody actions, for the results \cref{lem:gamma-homo,prop:nu-iso}, it is not needed.  
\begin{definition}
Let $\mathsf{\widehat{R}}$ be completion of $\mathsf{{R}}$ which is the inverse limit of the quotients $\mathsf{{R}}/\mathsf{I}_n$ with respect to the two-sided ideals $\mathsf{I}_n$ generated by all elements of degree $\geq n$.  By \cite[Lemma 2.3]{WebBKnote}, this is the same as the completion with respect to the topology where the elements of degree $\geq n$ form neighborhoods of 0.  

  Let $\mathsf{{P}}$ be the polynomial representation for $\mathsf{{R}}$ defined by the equations \crefrange{eq:poly-action1}{eq:poly-action2}, using the polynomials $\mathsf{{P}}_{*,*}$ and $\mathsf{\widehat{P}}$ its completion with respect to the grading topology, which is faithful by  \cite[Lem. 2.5]{WebBKnote}.
\end{definition}

\begin{lemma}\label{lem:gamma-homo}
The formulas \crefrange{com0new2}{com0newerer} define an injective homomorphism $\gamma\colon R\to \mathsf{\widehat{R}}$.
\end{lemma}
\begin{proof}
  The algebras $R$ and $\mathsf{R}$ are equipped with the polynomial representations $P$ and $\mathsf{P}$ as in \crefrange{eq:poly-action1}{eq:poly-action2}. The homomorphism $\gamma_P\colon P\to \widehat{\mathsf{P}}$ on the underlying polynomial representations is defined in \cref{com0new}.  

Now, we will confirm that for each generator of $R$, its action on the polynomial representation is intertwined with its image in $\mathsf{\widehat R}$ by this map between polynomials, that is, we have $\gamma(r)\gamma_P(p)=\gamma_P(rp)$ for all generators $r$ and all $p\in P$.  
For $r=y_k$ a dot, we have 
\[\gamma(y_k)\gamma_P(p)=\gamma_{i_k,\aa_k}(Y_k) \gamma_P(p)=\gamma_P(Y_k p).\]
For $r=\psi_{k}$ a crossing with $i_k\neq i_{k+1}$, the factorization \cref{eq:P-factorization} shows that 
\begin{multline*}
  \gamma(\psi_k)\gamma_P(p)=P_{i_ki_{k+1}}^{\circ,\aa_k \aa_{k+1}}(y_k,y_{k+1}) \psi_k\gamma_P(p)=P_{i_ki_{k+1}}^{\circ,\aa_k \aa_{k+1}}(y_k,y_{k+1}) \psi_k\gamma_P(p)\\=P_{i_ki_{k+1}}^{\circ,\aa_k \aa_{k+1}}(y_k,y_{k+1}) \mathsf{P}_{i_k,i_{k+1}}(y_k,y_{k+1})s_k\gamma_P(p)\\=P_{i_k,i_{k+1}}(\pwr_{a,i_k}(y_k),\pwr_{b,i_{k+1}}(y_k))s_k\gamma_P( p)=\gamma_P(\psi_k p).
\end{multline*}
For $i_k=i_{k+1}$ and $\aa_k=\aa_{k+1}$, we have the same calculation as \cref{eq:poly-action3-check}, that is
\[ \gamma(\psi_k)\gamma_P(p)= q_{a,i}(y_k,y_{k+1})\frac{\gamma_P(s_kp-p)}{y_{k+1}-y_k}=\gamma_P(\psi_k p).\]
Finally for $i_k=i_{k+1}$ and $\aa_k\neq \aa_{k+1}$, we have 
\[ \gamma(\psi_k)\gamma_P(p)= \frac{\gamma_P(s_kp)}{\pwr_{\aa_{k},i_k}(y_{k+1})-\pwr_{\aa_{k+1},i_k}(y_{k})}-\frac{\gamma_P(p)}{\pwr_{\aa_k,i_{k+1}}(y_{k+1})-\pwr_{\aa_{k},i_k}(y_{k})}=\gamma_P(\psi_k p).\]
 This verifies that we have $\gamma(r)\gamma_P(p)=\gamma_P(rp)$ for all $r\in R$ and all $p\in P$.
Since the action of $R$ on $P$, the action of $\mathsf{R}$ and $\mathsf{P}$, and the action of the completion $\widehat{\mathsf{R}}$ on the completion $\widehat{\mathsf{P}}$ are faithful (the last of these by \cite[Lem. 2.5]{WebBKnote}), this means that $\gamma$ must extend to an injective homomorphism.  
\end{proof}

In fact, one can show relatively easily that this map becomes an isomorphism if we replace $R$ by an appropriate completion.

For any commutative subalgebra $S\subset R$ and any finite-dimensional quotient $R/J$, we can decompose $R/J$ over the maximal spectrum of $S$.  That is, we can diagonalize the semi-simple part of each element of $S$.  For each maximal ideal $\mathfrak{m}$, there is a unique idempotent $e$ in $R/J$ whose image is the vectors killed by $\mathfrak{m}^N$ for some $N$. 

For any finite subset $T\subset \MaxSpec(S)$, we can define the completion of $R$ at $T$ as the inverse limit of the directed system of all quotients which are supported in the spectral decomposition above on $T$.  Equivalently, as coherent sheaves on $\MaxSpec(S)$, they are supported on $T$.  

In $R_n$, we have a commutative subring $S$ generated by the idempotents $e(\Bi)$ for $\Bi\in I^n$ and the dots $y_1,\dots, y_n$; this is a direct sum of polynomial rings, so its $\MaxSpec$ is a disjoint union of copies of $\mathbb{A}^n$.  Let $\mathbb{A}^n_{\Bi}$ be the component of $\MaxSpec(S)$ on which $e(\Bi)$ is not zero.
\begin{definition}\label{def:widehat-R}
	Let $\widehat{R}_n$ be the completion of $R_n$ with respect to the subset $T$ defined by the union of the subsets \[T=\bigcup_{\Bi\in I^n}\{\mathbf{a}\in \mathbb{A}^n_{\Bi} \mid a_k\in U_{i_k}\}.\] That is, this is the completion at the directed system of all quotients where the spectrum of $y_ke(\Bi)$ lies in $U_{i_k}$.  Let $\widehat{R}=\bigoplus_{n} \widehat{R}_n$.
\end{definition}
\begin{definition}
 For a given  $\Bj=(j_1=(i_1,u_1),\dots, j_n=(i_n,u_n))$,  let $\epsilon_{\Bj}$ be the idempotent corresponding to the point $\mathbf{u}$ in $\mathbb{A}^n_{\Bi}$, that is, the maximal ideal generated by $e_{\Bi'}$ for $\Bi'\neq \Bi$ and by $y_k-u_k$. 
 In particular, $e_{\Bi}\epsilon_{\Bj}=\epsilon_{\Bj}
e_{\Bi}=\epsilon_{\Bj}$ and $(y_j-u_j)\epsilon_{\Bj}$ is
topologically nilpotent.  
\end{definition}

\begin{proposition}\label{prop:nu-iso}
	The map $\gamma$ induces an isomorphism $\widehat{R}\cong \widehat{\mathsf{R}}$ sending $\epsilon_{\Bj}$ to $e_{\Bj}$.
\end{proposition}
\begin{proof}
First, note that $\gamma_P$ induces an isomorphism of the corresponding completion of polynomial rings.  Since the action of $\widehat{R}$ on $\widehat{P}$ is faithful, we have an injective map $\widehat{R}\to \widehat{\mathsf{R}}$.  By uniqueness, $\gamma(\epsilon_{\Bj})=e_{\Bj}$.  Furthermore, we can use these formulas to solve for the inverse map to $\gamma$.  On dots, we use the inverse power series 
\[\gamma_{u_k,i_k}^{-1}(x)=\sqrt[d_i]{x/a}-1=\frac{1}{d}\big(\frac{x}{a}-1\big)+\binom{1/d}{2}\big(\frac{x}{a}-1\big)^2+\cdots \] using Taylor expansion at $x/a=1$.  Comparing with \cref{com0newer} and \cref{com0newerer}, we find that the inverse is given by \begin{align}\label{eq:nu-inv1}
	\gamma^{-1}(y_ke_{\Bj})&=\gamma_{u_k,i_k}^{-1}(y_k)\epsilon_{\Bj}\\
	\gamma^{-1}(e_{\Bj}\psi_k)&=\begin{cases}\Big(P_{i_ki_{k+1}}^{\circ,u_ku_{k+1}}\big(\gamma_{u_k,i_k}^{-1}(y_k),\gamma_{u_k,i_{k+1}}^{-1}(y_{k+1}) \big )\Big)^{-1}\epsilon_{\Bj}{\psi}_k & i_k\neq i_{k+1}\\
    ((y_{k+1}-y_k) \psi_k +1)\epsilon_{\Bj} & i_k=i_{k+1}, u_k\neq u_{k+1}\\
    \frac{y_{k+1}-y_k}{\gamma_{u_{k+1},i_{k+1}}^{-1}(y_{k+1})-\gamma_{u_k,i_k}^{-1}(y_k)}\psi_k\epsilon_{\Bj} & i_k=i_{k+1}, u_k= u_{k+1}\\
  \end{cases}\label{eq:nu-inv2}
\end{align}
This completes the proof.
\end{proof}
While all our preliminary results about the quiver Hecke algebra have not required completeness, we will need it to complete the proof.  

\begin{proof}[Proof of \cref{lem:g-action}]  \refstepcounter{dummy}\label{proof-lem:g-action}
The proof that the finiteness conditions of $\htU(\tilde \fg)$ hold is the same as in \cref{lem:many-sl2-general}.  

As usual, we use \cref{lem:BrundanRouquier}. Note that completeness is necessary to ensure that the images of the 1-morphisms $\eE_i$ and $\eF_i$ have the correct image and target.  We need to confirm that for any $(i,\aa)\in \tilde{I}$, we have that $\sum_{\aa'\in U_j}\al_{i,\aa}^{\vee}(\al_{j,\aa'})=\al_i^{\vee}(\al_j)$.  This holds if and only if the spectra $U_i$ are complete.  Note the similarity to \cref{prop:furling-homomorphism}, which shows that we need completeness to ensure that we have a homomorphism of Lie algebras $\fg\to\tilde{\fg}$. 

 The conditions (KM1) and (KM3) involve only a single color on strands, and so follow by 
\cref{lem:many-sl2-general}.  Thus, it suffices to check (KM2) which follows from \cref{lem:gamma-homo}.
\end{proof}

\subsection{The proof of \cref{nondegenerate}}

Before starting the proof, we give a simple lemma that we will need in the proof.  Assume that $I$ is the disjoint union of two subgraphs $I_+,I_-$.  Then $\tilde{\fg}=\tilde{\fg}_+\oplus\tilde{\fg}_-$ is the direct sum of two subalgebras corresponding to these subgraphs.  By an argument essentially identical to \cref{lem:htUA-to-tensor-htUbullet}, we obtain:
\begin{lemma}\label{lem:Ipm}
There is a natural 2-functor $\tU(\tilde{\fg})\to \tU(\tilde{\fg}_+)\times \tU(\tilde{\fg}_-)$ sending
\[(\mu_+,\mu_-)\mapsto \mu_+\otimes\mu_-, \qquad \eE_i \mapsto \begin{cases} \eE_i\otimes 1 &\text{if }i\in \tilde{I}_+\\ 1\otimes \eE_i &\text{if }i\in \tilde{I}_-\end{cases}\]
with the obvious action on 2-morphisms with all labels in $I_+$ or $I_-$, and sending the crossing of strands with labels in $i\in I_+$ and $j\in I_-$ to the identity on $\eE_i\otimes \eE_j$.

Accordingly, if $\Ccat^\pm_{\bullet}$ is a categorical $\tilde{\fg}_{\pm}$-module, then the Deligne product $\Ccat_{(\mu_+,\mu_-)}=\Ccat^+_{\mu_+}\otimes\Ccat^-_{\mu_-}$ is a categorical $\tilde{\fg}$-module by pullback.
\end{lemma}
 As in \cref{rem:biequivalence}, one can relatively easily show that this is a biequivalence of 2-categories, but we will not need this fact.

\begin{proof}[Proof of \cref{nondegenerate}]\refstepcounter{dummy}\label{proof-nondegenerate}

The strategy is simple at heart: if we produce a large class of $\tU(\tilde{\fg})$-representations that are ``big enough,'' we can use the categorical action of \cref{lem:g-action} to deduce the required linear independence.  

\noindent \mybox{Reduction to algebraically closed fields:} \label{reduction} First, let us show that we can reduce to the case where the coefficient field is an algebraically closed field.  Let $L$ be the ring given by polynomials over $\Z[1/\prod_i d_i]$ with $t_{ij}^{\pm 1}$ and the other coefficients of $Q_{ij}$ adjoined as formal variables.  If we prove that $B_{\Bi,\Bi',\la}$ is a basis of the Hom space $\Hom_{\tU}
(\Bi,\Bi')$ as a free $L$-module when $\K'=L$, it will follow for every other choice of $\K'$ by base change.  
      
      Furthermore, note that $\Hom_{\tU}
      (\Bi,\Bi')$ is filtered by finitely generated submodules consisting of elements of degree $\leq k$ for each $k\in \Z$, and it is equivalent to show that the elements of $B_{\Bi,\Bi',\la}$ lying in this space form a basis for it as a free module.  As usual, it is enough to prove this after localization, so we may assume that $\K'$ is a regular local domain.  Since $d_i$ is invertible in $L$, it is coprime to the characteristic of the residue field of $\K'$.  By Nakayama's lemma, a finite set of elements is a free basis for a module over a local domain if and only if it is a basis after base change to the residue field and to the fraction field.  Thus, we may assume that $\K'$ is a field. The result is unchanged by passing to a field extension, so we may assume without loss of generality that $\K'=\bar{\K}'$ is algebraically closed.  Therefore, we can apply the other results of the paper with $\K'=\K$ an algebraically closed field of characteristic coprime to $d_i$, the case considered throughout the rest of the paper.  For the rest of the proof, we use $\K$ for our base ring.

\noindent \mybox{Construction of representations:} By \cref{lem:different-realizations}, we can assume that $\fg$ is universal derived.  For a given $m$, consider the field $\mathbb{K}$ of rational functions on an alphabet with $2m$ elements $\zw_{i,\pm 1},\dots, \zw_{i,\pm m}$ for each $i\in I$ and $\bar{\mathbb{K}}$ its algebraic closure.  Let
$U_i$ be the complete choice of spectra in the algebraic closure of the field $\K(\zw_{i,\pm  k})_{i\in I,k\in [1,m]}$ produced by starting with $\zw_{i,\pm k}\in U_i$, and then completing as necessary.   Note that all elements connected to $\zw_{i,k}$ in $\tilde{I}$ are in the algebraic closure of $\K(\zw_{i,k})$, and thus are algebraically independent from all other $\zw_{j,\ell}$.  In particular, $\tilde{I}$ divides into two disjoint subgraphs $\tilde{I}_{\pm}$ given by the components that contain $\zw_{i,\pm k}$ for $i\in I, k\in [1,m]$.  

Consider a weight $\lambda$, which we can think of as just the function $I\to \Z$ sending $i$ to $\al_i^{\vee}(\lambda)$.  Consider weights $\mu_{\pm}$ that satisfy, for all $i\in I$:
\[
\al^{\vee}_{i,{\zw_{i,1}}}(\mu_+)=m+\max(0, \al_i^{\vee}(\lambda))\qquad \al^{\vee}_{i,{\zw_{i,-1}}}(\mu_-)=-m+\min(0, \al_i^{\vee}(\lambda))\]
\[ \al^{\vee}_{i,{\zw_{i,\pm k}}}(\mu_{\pm})=\pm m \text{ for }k\in [2,m].\]  

The 2-category $\tU(\tilde{\fg}_+)$ has a categorical representation on the category of projective modules over the cyclotomic quotient $R^{\mu_+}$ over $\bar{\mathbb{K}}$, and $\tU(\tilde{\fg}_-)$ has a categorical representation on the category of projective modules over the cyclotomic quotient $R^{-\mu_-}$.  In both cases, the $\eF$'s act by induction functors and the $\eE$'s by restriction functors. These categorify the highest weight representations $V(\mu_+)$ and $V(-\mu_-)$, respectively, and are generated by highest weight objects $\mathbb{V}_{\pm}$ corresponding to the cyclotomic quotient with $0$ strands as a module over itself.   

Let $\tilde{\omega}\colon \tU(\tilde{\fg}_-)\to \tU(\tilde{\fg}_-)$ be Khovanov and Lauda's automorphism reversing the orientation of strands; see \cite[\S 3.3.3]{khovanovCategorificationQuantum2010}. 
Twisting the action of $\tU(\tilde{\fg}_-)$ by $\tilde{\omega}$, we obtain a categorical representation of $\tU(\tilde{\fg}_-)$ on the category of projective modules over the cyclotomic quotient $R^{-\mu_-}$, which categorifies the lowest weight representation $\Lambda(\mu_-)$.  Now, $\eF$'s act by restriction functors and $\eE$'s by induction functors, since $\tilde{\omega}$ reverses the orientation of strands.  We can then use \cref{lem:Ipm} to obtain a categorical representation of $\tU(\tilde{\fg})$ on the Deligne product of these categories, which categorifies the tensor product $V(\mu_+)\otimes \Lambda(\mu_-)$.  Note that we can think of this tensor product as the category of projective modules over the cyclotomic quotient $R^{\mu_+-\mu_-}$, but with the categorical action twisted by an autoequivalence.  Call this category with this action $\Ccat^{\mu_+,\mu_-}$.

This means that by \cite[Th. 6.2]{kangCategorificationHighest2012} or \cite[Cor. 3.22]{Webmerged}, we know the dimensions of Hom spaces in this category in terms of the Shapovalov form on the corresponding representation, and in particular, that it is generated by the object $\mathbb{V}=\mathbb{V}_{+}\otimes \mathbb{V}_{-}$, which is no longer highest or lowest weight.

By \cref{lem:g-action}, we have a $\tU(\fg)$-action on $\Ccat^{\mu_+,\mu_-}$.  This category is generated (as a Karoubian additive category) by the objects $\eE_{\Bi}\mathbb{V}$, which we view as projective modules over the tensor product of the cyclotomic quotients. The action of $\eE_i$ for $i\in I$ is the sum of restriction functors in the factor $R^{\mu_+}$ and induction functors in the factor $R^{-\mu_-}$.  The action of $\eF_i$ is analogous, with induction and restriction switched.  Thus, each 2-morphism is sent to a natural transformation between iterated functors, built from the quiver Hecke action and the two adjunctions between induction and restriction.  

By construction, these have a decomposition $\eE_{\Bi}\mathbb{V}\cong \bigoplus_{\Bu\in \prod U_{|i_k|}}\tilde{\eE}_{(\Bi,\Bu)}\mathbb{V}$, where to avoid confusion, we let $\tilde{\eE}_*$ denote the categorical action of $\tU(\tilde{\fg})$.  The projections in this direct sum decomposition can be constructed purely in terms of the action of $\tU(\fg)$ as projections to the eigenspaces of dots---by Jordan--Chevalley decomposition, the projection 
to $\tilde{\eE}_{(\Bi,\Bu)}\mathbb{V}$ can be written as a power series in the dots.

Note that while we have not explicitly given the image of the bubbles in this action, it is relatively easy to determine their action on $\mathbb{V}$ by \cref{preimpy} (the reader might be concerned that we are using a result from \cref{sec:spectral} here, but no other results from this paper are used in the proof of \cref{preimpy}, so there is no danger of circularity).  Since $\mathbb{V}$ is simple, we only need to determine the minimal polynomials $m_{\mathbb{V}}(z),n_{\mathbb{V}}(z)$.  These are given by:
\begin{equation}
	m_{\mathbb{V}}(z)=\prod_{k=1}^m (z-\zw_{i,-k})^{\al^{\vee}_{i,{\zw_{i,-k}}}(\mu_-)}, \qquad n_{\mathbb{V}}(z)=\prod_{k=1}^m (z-\zw_{i,k})^{\al^{\vee}_{i,{\zw_{i,k}}}(\mu_+)}
\end{equation}
since in $\bar{\mathbb{K}}[z]/(z^r)$, the image of the power series $\pwr_{\aa,i}(z)$ has minimal polynomial $(z-\aa)^{r}$.

Thus, we have that the action of $\anticlocki(z)$ on $\mathbb{V}$ is given by the power series
\begin{equation}\label{eq:anticlock-action}
\anticlocki(z)=\frac{n_{\mathbb{V}}(z)}{m_{\mathbb{V}}(z)}=\prod_{k=1}^m\frac{(z-\zw_{i,k})^{\al^{\vee}_{i,{\zw_{i,k}}}(\mu_+)}}{(z-\zw_{i,-k})^{\al^{\vee}_{i,{\zw_{i,-k}}}(\mu_-)}}.
\end{equation}

 We claim that the first 2m non-constant coefficients of the image of $\anticlocki(z)$---the coefficients of $z^{\al_i^{\vee}(\lambda)-1}, \cdots,  z^{\al_i^{\vee}(\lambda)-2m}$---are algebraically independent.  

This is easiest to see by taking the logarithmic derivative $\eta(z)$ of $\anticlocki(z)$; by \cref{eq:anticlock-action}, we have 
\[\eta(z)=\sum_{k=1}^m\frac{\al^{\vee}_{i,{\zw_{i,k}}}(\mu_+)}{z-\zw_{i,k}}- \frac{\al^{\vee}_{i,{\zw_{i,-k}}}(\mu_-)}{z-\zw_{i,-k}}. \]
The coefficients $\eta_0,\dots, \eta_{2m-1}$ of the Laurent expansion $\eta(z)=\eta_0z^{-1}+\eta_1z^{-2}+\cdots$ at $z=\infty$ are the power sums
\[\eta_r=\sum_{k=1}^m \al^{\vee}_{i,{\zw_{i,k}}}(\mu_+)\zw_{i,k}^r- \sum_{k=1}^m \al^{\vee}_{i,{\zw_{i,-k}}}(\mu_-)\zw_{i,-k}^r\]
for $r=0,\dots, 2m-1$.  These form a transcendence basis for the field $\mathbb{K}$, since the Jacobian of the induced map is non-zero, and therefore they are algebraically independent.    

\noindent \mybox{Setup of generic spectrum:} Consider a relation in this spanning set $B_{\Bi,\Bi',\la}$ for $1$-morphisms $\Bi,\Bi'\colon \mu \to \mu'$.  
	Amongst the $(\Bi,\Bi')$ pairings appearing, fix a pairing $\pi$ which contributes at least one diagram with non-trivial coefficient and has a maximal number of crossings with respect to this property.  

	Let $m$ be $1$ plus the maximum of:
	\begin{enumerate}
		\item The degrees of all bubbles appearing in diagrams for $\pi$.
		\item The number of arcs appearing in one of these diagrams, that is, $(|\Bi|+|\Bi'|)/2$.  
		\item The number of dots that appear on any of these arcs.
	\end{enumerate}  
We wish to choose a generic enough point in the spectrum of the dots that our ostensible relation must have non-zero image, and it is easy to check that this image is not zero.  That is, we want to choose a $\Bj=(\Bi,\Bu)$ and $\Bj'=(\Bi',\Bu')$ such that the image of the relation in $\Hom(\tilde{\eE}_{\Bj}\mathbb{V}, \tilde{\eE}_{\Bj'}\mathbb{V})$ is non-zero. 

For every arc $A$ in $\pi$, we have a choice of $\aa_A$ such that both endpoints are of the form $(\pm i,\aa_A)$.  We wish to choose these so that they will be algebraically independent from each other.   
By assumption, we can choose an injection $\Theta$ from the set of arcs of $\pi$ to $[2,m]$, and given this function, we can choose $\aa_A=\zw_{i,\pm \Theta(A)}$.   This means that each second index appears at at most two terminals in $\Bj$ and $\Bj'$, namely at opposite ends of an arc, when it appears at all.  We choose $\aa_A$ to be $\zw_{i,\Theta(A)}$ if the strand at the rightmost endpoint of $A$ is oriented down as $-i$, and $\zw_{i,-\Theta(A)}$ if it is oriented up.  That is, if $A$ has one endpoint at the top of the diagram and one at the bottom, we follow the orientation of the strand; if it is a cup or cap, we look at its right endpoint.  
This convention might seem a little strange, but it is natural given how we set up $\Ccat^{\mu_+,\mu_-}$ as a Deligne product of a highest weight and lowest weight representation---we choose this sign so that the right endpoint acts on our tensor product of cyclotomic quotients by an induction functor, which is always non-zero, while a restriction functor is always $0$ on $\mathbb{V}$, since it already corresponds to having $0$ strands.

Note that once we have fixed this spectrum, $\pi$ is the only matching where the two ends of each arc have the correct labels; there are no other elements of $\Hom_{\Ccat^{\mu_+,\mu_-}}(\tilde{\eE}_{\Bj}\mathbb{V},\tilde{\eE}_{\Bj'}\mathbb{V})$.

We thus have a map $\Hom_{\tU}(\eE_{\Bi}, \eE_{\Bi'})\to \Hom_{\Ccat^{\mu_+,\mu_-}}(\eE_{\Bi}\mathbb{V}, \eE_{\Bi'}\mathbb{V})\to \Hom_{\Ccat^{\mu_+,\mu_-}}(\eE_{\Bj}\mathbb{V}, \eE_{\Bj'}\mathbb{V})$.   

In fact, this last Hom space is easy to describe---let us first discuss the case where $|\Bi|+|\Bi'|=2$.  All such cases can be reduced to the case where $\Bi=\Bi'=(\pm i)$ by the adjunction:
\begin{equation}\label{eq:adjunction}
\Hom_{\Ccat^{\mu_+,\mu_-}}(\tilde{\eE}_{j}\tilde{\eF}_{j}\mathbb{V},\mathbb{V})\cong \Hom_{\Ccat^{\mu_+,\mu_-}}(\tilde{\eF}_{j}\mathbb{V},\tilde{\eF}_{j}\mathbb{V})
\cong \Hom_{\Ccat^{\mu_+,\mu_-}}(\mathbb{V}, \tilde{\eE}_{j}\tilde{\eF}_{j}\mathbb{V}),
\end{equation}
and its analogue with $\tilde{\eE}$ and $\tilde{\eF}$ swapped.  

Note that $\aa_A=\zw_{i, \Theta(A)}$ in the case of \cref{eq:adjunction}, so the space $\Hom_{\Ccat^{\mu_+,\mu_-}}(\tilde{\eF}_{j}\mathbb{V},\tilde{\eF}_{j}\mathbb{V})$ is simply the summand of the cyclotomic quotient $R^{\mu_+}$ with one strand labeled $j$, which is a quotient of the polynomial ring in one variable $y$ by the ideal generated by $y^m$, and hence has basis $1,y,\dots, y^{m-1}$.  

If we swap $\eE$ and $\eF$, we get the same result, since we now take $\aa_A=\zw_{i, -\Theta(A)}$, and now $\eE$ acts by an induction functor in the factor $R^{-\mu_-}$.

As we add more strands, the calculation is essentially unchanged, since we never use two labels from the same component of $\tilde{I}$.  The space $\Hom_{\Ccat^{\mu_+,\mu_-}}(\tilde{\eF}_{\Bj}\mathbb{V},\tilde{\eF}_{\Bj}\mathbb{V})$ is again a cyclotomic quotient, where we have one strand labeled with each of the $j_k$.  This is a quotient of the polynomial ring in the dots, obtained by applying the cyclotomic relation independently to each strand.  All other cases are obtained from this one by adjunction and the fact that $\tilde{\eE}_j$ and $\tilde{\eE}_{j'}$ from different components of $\tilde{I}$ commute.  

Put differently, we can forget all nodes of $\tilde{I}$ except those that appear in $\Bj$ and $\Bj'$, and the result is a graph with no edges, so we only need the $\mathfrak{sl}_2$ calculation for a single strand, giving this truncation of the polynomial ring.  Geometrically, we can describe this as the subset of Khovanov and Lauda's basis where we constrain the number of dots to be $<m$, due to the cyclotomic relation.   

\noindent \mybox{Triviality of a relation:}
Now return to consideration of our relation.  Since all bubbles appearing in the diagrams have $\K$-algebraically independent images in $\mathbb{K}$, this relation must give a non-trivial relation with coefficients in $\mathbb{K}$ between diagrams in $B_{\Bi,\Bi',\la}$ with no bubbles.

By the assumption that we started with a relation, the resulting morphism must be 0. Note also that the functor defined by \crefrange{com0new}{com0newerer} sends the diagram for any pairing to a sum of diagrams for that pairing and for pairings with strictly fewer crossings. Since $\pi$ gives a maximal number of crossings amongst pairings
that appear in the relation, any other element $D'\in B_{\Bi,\Bi',\la}$ with pairing $\pi'$ has image in $\Hom_{\Ccat^{\mu_+,\mu_-}}(\eE_{\Bj}\mathbb{V}, \eE_{\Bj'}\mathbb{V})$ that is a sum of diagrams where the pairing $\pi$ never appears.  Since $\pi$ is the only pairing that matches the spectrum we have chosen, all such diagrams are sent to 0, since the eigenvalue $\aa_*$ at the two ends of an arc are different.

Thus, when we consider the image of our relation, it is the same as the image of only the elements of $B_{\Bi,\Bi',\la}$ appearing that match according to $\pi$.  Let $D$ be such an element of $B_{\Bi,\Bi',\la}$ with no bubbles that appears in our relation, and has a minimal number of dots amongst such diagrams.  By construction, under \crefrange{com0new}{com0newerer}, every diagram is sent to a multiple of that diagram by a power series in dots with non-zero constant term.  Thus, the diagram $D$ is sent to $cD+\dots$ where all other terms have strictly more dots than $D$.  The diagram $D$ will be one of the elements of our basis of diagrams with $<m$ dots on each arc, by assumption.  Thus, when expanded in this basis, the diagram $D$ cannot appear in the image of any other diagram, since this would require the existence of $D'$ in the relation with strictly fewer dots than $D$.  Thus, $cD$ cannot be cancelled and this contradicts the assumption that our relation was non-zero.  This shows that there are no relations between the elements of $B_{\Bi,\Bi',\la}$, completing the proof of \cref{nondegenerate}.

 \end{proof}
 
 \subsection{The proof of \cref{prop:KM3}}

\begin{proof}[Proof of \cref{prop:KM3}]
\refstepcounter{dummy}\label{proof-prop:KM3}

We need only verify the conditions of \cref{thm:brundandef}.

The axiom (KM1) is already established by the left and right adjunctions that we have constructed.  

 Since there are no edges in the graph $U$, the polynomials $Q_{\aa \bb}$ are simply $1$ for all $\aa\neq \bb$.  One can easily check that $R_m$ is Morita equivalent to the sum of tensor products of nilHecke algebras \[\bigoplus_{\sum_{\aa}m_{\aa}=m}\bigotimes_{\aa\in U} NH_{m_{\aa}}\] with each summand corresponding to the subalgebra where $m_{\aa}$ strands have the label $\aa$.  Taking $\gamma_\aa(x)=x-\aa$ and applying \cref{prop:nu-iso}, we find that the axiom (KM2) holds as well.  That is, there is a homomorphism $R_m\to \End(\eF^m)$ which sends \begin{equation}
		\begin{tikzpicture}[baseline = -1mm]
	\draw[thick, <-] (0.28,-.28) to (-0.28,.28);
	\draw[thick, <-] (-0.28,-.28) to (0.28,.28);
      \node at (-0.33,-0.43) {$\scriptstyle{\aa}$};
      \node at (0.33,-0.43) {$\scriptstyle{\bb}$};
      \node at (-0.33,0.43) {$\scriptstyle{\bb}$};
      \node at (0.33,0.43) {$\scriptstyle{\aa}$};
\end{tikzpicture}\mapsto \begin{cases}
		\begin{tikzpicture}[baseline = -1mm]
	\draw[<-] (0.28,-.28) to (-0.28,.28);
	\draw[<-] (-0.28,-.28) to (0.28,.28);
      \node at (-0.33,-0.43) {$\scriptstyle{\aa}$};
      \node at (0.33,-0.43) {$\scriptstyle{\bb}$};
      \node at (-0.33,0.43) {$\scriptstyle{\bb}$};
      \node at (0.33,0.43) {$\scriptstyle{\aa}$};
\node at (0,-.01) {$\diamond$};
\end{tikzpicture} & \aa=\bb\\
\begin{tikzpicture}[baseline = 0]
	\draw[<-] (0.38,-.4) to (-0.38,.5);
      \node at (0,0.05) {$\diamond$};
	\draw[<-] (-0.38,-.4) to (0.38,.5);
	\draw[->,densely dotted] (0.26,-.24) to (-0.19,-.24);
   \node at (-.4,-.55) {$\scriptstyle{\aa}$};
   \node at (.4,-.55) {$\scriptstyle{\bb}$};
   \node at (.4,.65) {$\scriptstyle{\aa}$};
   \node at (-.4,.65) {$\scriptstyle{\bb}$};
      \node at (-0.26,-0.25) {$\bull$};
      \node at (0.26,-0.25) {$\bull$};
      \node at (1,-0.2) {$\scriptstyle{x_2-x_1}$};
\end{tikzpicture}& \aa\neq \bb
\end{cases}\qquad 		\begin{tikzpicture}[baseline = -1mm]
	\draw[thick, <-] (0,-.28) to (-0,.28);
      \node at (0,-0.43) {$\scriptstyle{\aa}$};
      \node at (0,0.43) {$\scriptstyle{\aa}$};
         \node at (0,0) {$\bullet$};
\end{tikzpicture}\mapsto \begin{tikzpicture}[baseline = -1mm]
	\draw[ <-] (0,-.28) to (-0,.28);
      \node at (0,-0.43) {$\scriptstyle{\aa}$};
      \node at (0,0.43) {$\scriptstyle{\aa}$};
         \node at (0,0) {$\bullet$};
\end{tikzpicture}-\aa \begin{tikzpicture}[baseline = -1mm]
	\draw[ <-] (0,-.28) to (-0,.28);
      \node at (0,-0.43) {$\scriptstyle{\aa}$};
      \node at (0,0.43) {$\scriptstyle{\aa}$};
\end{tikzpicture}
	\end{equation}

Thus, we need only establish (KM3).  This is identical to the proof of \cite[Lem. 4.9]{brundanHeisenbergKacMoody2020}.
\end{proof}

\subsection{The proof of \cref{equivalence}}

In order to prove \cref{equivalence}, we need to construct a functor $\Phi$ which gives the equivalence.  It is in fact easier to begin by constructing a functor $\Xi$, which we will ultimately show is effectively inverse to $\Phi$, once it has been constructed.

\begin{lemma}
	There is a strongly equivariant functor
\[\Xi \colon \hldbK\to \hl^{\tilde{\bla}}\]
which sends $(\Bi,\kappa)$ 
to the sum $\bigoplus (\Bj,\kappa)$ where $\Bj$ ranges over sequences with
$j_k=(i_k,u_k)$.  
\end{lemma}
\begin{proof}
	\mybox{$\Xi$ on morphisms:} 
We have described how $\Xi$ behaves on objects, so we must only define how $\Xi$ acts on morphisms.  First, note that by \cref{thm:deform-action}, we have a categorical action
of $\fg$ on $\hl^{\tilde{\bla}}$.  On purely black
diagrams, $\Xi$ simply employs this action; that is, on
upward-oriented diagrams, it follows from the formulas \crefrange{eq:nu-inv1}{eq:nu-inv2}.
Since left (or right) adjunctions are unique up to isomorphism, we can
send the leftward cup and cap in $\tU_{\fg}$ to any adjunction we
choose.  For simplicity, we
simply match leftward oriented cups and caps as below:
\begin{equation}
  \tikz[baseline,very thick,scale=3]{\draw[->] (.2,.1)
    to[out=-120,in=-60] node[at end,above,scale=.8]{$i$} node[at start,above,scale=.8]{$i$}  (-.2,.1)
    ;} \mapsto \sum_{u\in U_i}   \tikz[baseline,very thick,scale=3]{\draw[->] (.2,.1)
    to[out=-120,in=-60] node[at end,above,scale=.8]{$(i,u)$} node[at start,above,scale=.8]{$(i,u)$} (-.2,.1)
    ;} 
 \qquad \tikz[baseline,very
  thick,scale=3]{\draw[->] (.2,.1) to[out=120,in=60] node[at
    end,below,scale=.8]{$i$}node[at
    start,below,scale=.8]{$i$} (-.2,.1);}\mapsto  \sum_{u\in U_i} \tikz[baseline,very
  thick,scale=3]{\draw[->] (.2,.1) to[out=120,in=60] node[at
end,below,scale=.8]{$(i,u)$} node[at
start,below,scale=.8]{$(i,u)$}  (-.2,.1);}\label{eq:cup/cap}
\end{equation}
For rightward oriented cups, the formula is quite complicated, but it is
fixed by the choices we have made thus far, and the
existence of a consistent choice follows from the existence of the $\fg$-action.
Thus, we need only define this action on diagrams with red/blue strands,
with formulas given below
\crefrange{eq:3}{eq:4}.  

\excise{In reading these formulas, we always use
$y_1,y_2$ to denote the dots on the lefthand and righthand of two
black strands in the diagram, and $y$ to denote the dot when a single
black strand appears, and a fraction $\frac ab$ should be interpreted
as $b^{-1}a$.   
\begin{equation}
\begin{tikzpicture}[very thick,baseline,yscale=.7]\draw[<-]  (1,0) +(0,-1) -- node[midway,draw, fill=black, scale=.7,
        inner sep=2pt, circle]{} +(0,1)
        node[below,at start]{$i$};  \end{tikzpicture}\mapsto \sum_m \begin{tikzpicture}[very thick,baseline,yscale=.7]\draw[<-]  (1,0) +(0,-1) -- node[midway,draw, fill=white, scale=.7,
        inner sep=.5pt, circle]{$+\zw_m$} +(0,1)
        node[below,at start]{$(i,m)$};  \end{tikzpicture}\label{eq:1}
\end{equation}
\begin{equation}
\tikz[baseline,very thick,scale=2]{\draw[->] (.2,.3) --
    (-.2,-.1) node[at end,below, scale=.8]{$i$}; \draw[<-] (.2,-.1) --
    (-.2,.3) node[at start,below,scale=.8]{$j$};}\mapsto
\begin{cases}\displaystyle
\sum_m\tikz[baseline,very thick,scale=2]{\draw[->] (.2,.3) --
    (-.2,-.1) node[at end,below, scale=.8]{$(i,m)$}; \draw[<-] (.2,-.1) --
    (-.2,.3) node[at start,below,scale=.8]{$(i,m)$};}+\sum_{m\neq p}  \frac{\tikz[baseline,very thick,scale=2]{\draw[->] (.2,.3) --
    (-.2,-.1) node[at end,below, scale=.8]{$(i,m)$}; \draw[<-] (.2,-.1) --
    (-.2,.3) node[at start,below,scale=.8]{$(i,p)$};}-\tikz[baseline,very thick,scale=2]{\draw[->] (-.2,.3) --
    (-.2,-.1) node[at end,below, scale=.8]{$(i,p)$}; \draw[<-] (.2,-.1) --
    (.2,.3) node[at start,below,scale=.8]{$(i,m)$};}} {y_1-y_2 +\zw_p-\zw_m}& i=j\\
\displaystyle\sum_m\tikz[baseline,very thick,scale=2]{\draw[->] (.2,.3) --
    (-.2,-.1) node[at end,below, scale=.8]{$(i,m)$}; \draw[<-] (.2,-.1) --
    (-.2,.3) node[at start,below,scale=.8]{$(j,m)$};} + \sum_{m\neq p}  R_{ji}(y_1+\zw_p,y_2+\zw_m)\Big(\tikz[baseline,very thick,scale=2]{\draw[->] (.2,.3) --
    (-.2,-.1) node[at end,below, scale=.8]{$(i,m)$}; \draw[<-] (.2,-.1) --
    (-.2,.3) node[at start,below,scale=.8]{$(j,p)$};}\Big) & i\neq j
\end{cases}\label{eq:2}
\end{equation}}
\begin{equation}
  \label{eq:3}
    \tikz[baseline,very thick,scale=2]{\draw[->] (.2,.3) --
    (-.2,-.1) node[at end,below, scale=.8]{$i$}; \draw[wei] (.2,-.1) --
    (-.2,.3) node[at start,below,scale=.8]{$\la_m$};}
  \mapsto \sum_{u\in U_i} \tikz[baseline,very thick,scale=2]{\draw[->] (.2,.3) --
    (-.2,-.1) node[at end,below, scale=.8]{$(i,u)$}; \draw[wei] (.2,-.1) --
    (-.2,.3) node[at start,below,scale=.8]{$\la_m$};}\qquad \tikz[baseline,very thick,scale=2]{\draw[->] (-.2,.3) --
    (.2,-.1) node[at end,below, scale=.8]{$i$}; \draw[wei] (-.2,-.1) --
    (.2,.3) node[at start,below,scale=.8]{$\la_m$};}
  \mapsto \tikz[baseline,very thick,scale=2]{\draw[->] (-.2,.3) --
    (.2,-.1) node[at end,below, scale=.8]{$(i,\zw_m)$}; \draw[wei] (-.2,-.1) --
    (.2,.3) node[at start,below,scale=.8]{$\la_m$};}+\sum_{u\in U_i\setminus\{\zw_m\}} (y-\zw_m+u)^{\la_m^i} \tikz[baseline,very thick,scale=2]{\draw[->] (-.2,.3) --
    (.2,-.1) node[at end,below, scale=.8]{$(i,u)$}; \draw[wei] (-.2,-.1) --
    (.2,.3) node[at start,below,scale=.8]{$\la_m$};}
  \end{equation}\begin{equation}
  \label{eq:4}
    \tikz[baseline,very thick,scale=2]{\draw[<-] (.2,.3) --
    (-.2,-.1) node[at end,below, scale=.8]{$i$}; \draw[awei] (.2,-.1) --
    (-.2,.3) node[at start,below,scale=.8]{$\la_m$};}
  \mapsto \sum_{u\in U_i}  \tikz[baseline,very thick,scale=2]{\draw[<-] (.2,.3) --
    (-.2,-.1) node[at end,below, scale=.8]{$(i,u)$}; \draw[awei] (.2,-.1) --
    (-.2,.3) node[at start,below,scale=.8]{$\la_m$};}\qquad \tikz[baseline,very thick,scale=2]{\draw[<-] (-.2,.3) --
    (.2,-.1) node[at end,below, scale=.8]{$i$}; \draw[awei] (-.2,-.1) --
    (.2,.3) node[at start,below,scale=.8]{$\la_m$};}
  \mapsto \tikz[baseline,very thick,scale=2]{\draw[<-] (-.2,.3) --
    (.2,-.1) node[at end,below, scale=.8]{$(i,\zw_m)$}; \draw[awei] (-.2,-.1) --
    (.2,.3) node[at start,below,scale=.8]{$\la_m$};}+\sum_{u\in U_i\setminus\{\zw_m\}} (y-\zw_m+u)^{\la_m^i} \tikz[baseline,very thick,scale=2]{\draw[<-] (-.2,.3) --
    (.2,-.1) node[at end,below, scale=.8]{$(i,u)$}; \draw[awei] (-.2,-.1) --
    (.2,.3) node[at start,below,scale=.8]{$\la_m$};}
\end{equation}

\mybox{Checking relations:} Now, we need to prove that this assignment defines a functor, that is, it is compatible with all the relations on morphisms.  
The equations that can be stated purely using upward or downward diagrams, that is,
\crefrange{QHA1}{QHA3} and
\cref{dcost2}, \cref{fig:pass-through}, \cref{red-triple}, and \cref{blue-triple} all follow by straightforward calculations as in the proof of \cref{prop:nu-iso}.  Similarly, once we know \cref{pitch}, the relation \cref{fig:pass-through1} can be rewritten this way:
\begin{equation*}\subeqn \label{fig:pass-throughup}
   \begin{tikzpicture}[yscale=.9,baseline]
      \node at (4,0){    \begin{tikzpicture}[very thick]
          \draw[postaction={decorate,decoration={markings,
    mark=at position .2 with {\arrow[scale=1.3]{>}}}}] (-3,0) +(.5,-1) to[out=90,in=-90] +(-.5,1);
          \draw[postaction={decorate,decoration={markings,
    mark=at position .8 with {\arrow[scale=1.3]{<}}}}] (-3,0) +(-1,1) to[out=-35,in=100] +(1,-1);
          \draw[awei] (-3,0) +(0,-1) -- +(0,1);
          \node at (-1.5,0) {=}; \draw[postaction={decorate,decoration={markings,
    mark=at position .2 with {\arrow[scale=1.3]{<}}}}] (0.2,0) +(-.5,1) to[out=270,in=90] +(.5,-1); \draw[postaction={decorate,decoration={markings,
    mark=at position .8 with {\arrow[scale=1.3]{>}}}}] (0.2,0) +(1,-1) to[out=145,in=280] +(-1,1); \draw[awei] (0.2,0) +(0,1) -- +(0,-1); 
        \end{tikzpicture}};
     \node at (-4,0){    \begin{tikzpicture}[very thick]
       \draw (-3,0)[postaction={decorate,decoration={markings,
    mark=at position .2 with {\arrow[scale=1.3]{<}}}}] +(.5,-1) to[out=90,in=-90] +(-.5,1);
          \draw[postaction={decorate,decoration={markings,
    mark=at position .8 with {\arrow[scale=1.3]{>}}}}] (-3,0) +(-1,1) to[out=-35,in=100] +(1,-1);
          \draw[wei] (-3,0) +(0,-1) -- +(0,1);
          \node at (-1.5,0) {=}; \draw[postaction={decorate,decoration={markings,
    mark=at position .2 with {\arrow[scale=1.3]{>}}}}] (0.2,0) +(-.5,1) to[out=270,in=90] +(.5,-1); \draw[postaction={decorate,decoration={markings,
    mark=at position .8 with {\arrow[scale=1.3]{<}}}}] (0.2,0) +(1,-1) to[out=145,in=280] +(-1,1); \draw[wei] (0.2,0) +(0,1) -- +(0,-1); 
        \end{tikzpicture}};
      \end{tikzpicture}
    \end{equation*}

Thus, we only need to argue for the relations involving right cups and
caps.
The way we have
defined the right cup/cap means that the relations
\crefrange{QHA1}{flight2} are automatic. The remaining
relations \cref{pitch,color-opp-cancel} are actually
redundant when the right cap and cup are defined in terms of the left
cup and cap. From this perspective, \cref{pitch} is the definition of the upward red/black or downward blue/black crossings.  For
\cref{color-opp-cancel}, assume that we are considering the red
version; the blue version follows similarly.  We must consider two different cases. Let $\mu$ be the label
of the region at the left of the picture.
\begin{itemize}
\item If $\mu^i\geq 0$, then we have a loop at the left with $\mu^i$ dots.
  Pulling this through and applying \cref{dcost2}, then undoing this bubble, we obtain the desired relation.
\item If $\mu^i\leq 0$, then we start with the diagram with a leftward cup at the bottom and rightward cup at the top, and compare the result of applying \cref{flight1}  to these two strands to the left and right of the red strand.  Using the relations
  \cref{fig:pass-through,red-triple}, we can move
  the bigon to the right side of the red line, and using \cref{dcost2}
  to remove bigons between red and downward strands.  The left- and right-hand sides now have the same pattern of black strands, but in one the upward strands make a bigon with the red strand, and in the other, they do not.  This can only hold if \cref{color-opp-cancel} is true.\qedhere
\end{itemize}
\end{proof}

\begin{proof}[Proof of \cref{equivalence}]  \refstepcounter{dummy}\label{proof-equivalence}
\mybox{Existence of the functor $\Phi$:} Both the source and the target
are generated by a signed highest weight object, so
\cite[Prop. 3.25]{Webmerged} shows that a strongly
equivariant functor is induced whenever there is a map
\begin{equation}
\End_{\hl^{\tilde\la}}(\emptyset,0) \cong \bar{\mathbb{K}}\to  \End_{\hldbK}(\emptyset,0)\label{eq:highest-action}
\end{equation}
compatible with the action of fake bubbles.  Since $\mathsf{y} \colon \eF_{i,u}\to
\eF_{i,u}$ is nilpotent, the fake bubbles act trivially in both cases,
and the $\bar{\mathbb{K}}$-algebra structure on
$\End_{\hldbK}(\emptyset,0)$ induces the functor.  The functor is an equivalence if and only if the map of \cref{eq:highest-action}
 is an isomorphism.  This can only fail if $ \hldbK=0$.

\mybox{Inverse functor:} Thus, we need to rule out the possibility that $\hldbK=0$. 
By \cite[Prop. 3.25]{Webmerged}, the functor $\Phi$ is an equivalence if $\End(\emptyset,0)\cong \K$.  Thus, the only issue is that the object $(\emptyset,0)$ might simply be 0 (in which case the entire category  $\hldbK$ is 0).  The functor $\Xi$ sends $(\emptyset,0)$
to $(\emptyset,0)$.  The existence of this functor establishes that $\hldbK$ is not 0, so $\Phi$ must be an equivalence.  In fact, we can easily see that $\Xi$ must be quasi-inverse to $\Phi$ when composed with the equivalence $\hl^{\tilde{\bla}}\cong \hl^{\tilde \la}$.
\end{proof}

\subsection{The proof of \cref{tricolore-nondegenerate}}
\begin{proof}[Proof of \cref{tricolore-nondegenerate}]  \refstepcounter{dummy}\label{proof-tricolore-nondegenerate}
 \mybox{$B$ spans:}  The proof that these are a spanning set is essentially equivalent to
 that of \cite[Prop. 3.11]{khovanovCategorificationQuantum2010}.  
First, note that any two minimal diagrams for the same matching are equivalent modulo those with fewer crossings (using the relations \cref{QHA3,blue-triple,red-triple}).  Similarly, moving the dots to the chosen positions only introduces diagrams with fewer crossings.  

Thus, we only need to show that minimal diagrams span.  Of course, if
a diagram is non-minimal, then it can be rewritten in terms of the
relations as a sum of diagrams with fewer crossings, using the relations
to clear all strands out from a bigon, and then the relations
\cref{QHA2}, \crefrange{flight1}{flight2}, \cref{dcost2}
to remove it.  Thus, by induction,
this process must terminate at an expression in terms of minimal
diagrams. Thus, these elements span, and it suffices to show that these elements are linearly independent when $\Bz$ are generic, that is, after base change to $\mathbb{\bar K}$.

\mybox{Linear independence:}    Assume $\EuScript{L}=\mu$. 
We consider how the elements in $B$ act on a quadruple with $\Bi=\emptyset$ in the deformed category $\hld^{\boldsymbol{\la}}$ with
$\boldsymbol{\la}$ chosen so that $\sum\la_i=\mu$. 

It suffices to check that these elements act linearly independently on $\hldbK$ for some $\bla$; in the course of the proof we will modify $\bla$ as necessary to achieve this.  Note the enormous advantage obtained by having both dominant and anti-dominant weights, as we can add canceling pairs of these without changing the total sum.  

As in \cite[4.17]{Webmerged}, we can compose with the diagram $\eta_{\kappa}$ pulling all black strands
 to the left and $\dot\eta_{\kappa}$, its vertical
 reflection\footnote{We use $\eta$ instead of $\theta$ here since we are pulling left rather than right.}.  This will send a non-trivial relation between the diagrams in $B$ to a non-trivial relation between diagrams where $\kappa(i)=n$ for all $i$.  

 \mybox{Choice of eigenvalues:} We can now project this relation to the subspace where we fix the eigenvalue of each dot acting at the top and bottom.  The
 formulas  \crefrange{eq:nu-inv1}{eq:nu-inv2} and 
\crefrange{eq:cup/cap}{eq:4} defining the functor $\Xi$ show that this projection is the image under $\Phi$ of a diagram with an equal or smaller number of crossings, and we can only have equality if we choose eigenvalues so that they coincide at opposite ends of a strand.  Now, fix a matching $D$ such that an associated basis vector appears in our relation, and the corresponding diagram has a maximal number of crossings among those that appear. 

Let us number the arcs $\{a_1,\dots, a_n\}$ in the
diagram $D$, and let $i_{a_m}$ be the index that labels this arc.  As noted before, we are free to add cancelling pairs
$(\lambda,-\lambda)$ of highest and lowest weights to   $\boldsymbol{\la}$, and to reorder our
weights, so we can
assume that $\la_m$ is
\begin{enumerate}
\item dominant if the corresponding arc $a$ is downward oriented at its
 lefthand edge and $\la_a$ anti-dominant if the arc $a$ is upward
 oriented at its lefthand edge,
\item $\la_m^{i_{a_m}}\neq 0$ for all $m$, that is, $\la_m$ is not
 perpendicular to the coroot corresponding to the arc $a_m$.
\end{enumerate}
For simplicity, let $\la_{a_m}=\la_{m}$ and $\zw_{a_m}=\zw_m$.

For example, if \[  D=
\begin{tikzpicture}[baseline,very thick]
\draw [postaction={decorate,decoration={markings,
   mark=at position .3 with {\arrow[scale=1.3]{>}}}}] (-.5,-1)
to[out=90,in=180] node[below,at start]{$i$} node [below, pos=.9,green!50!black]{1}(.25,0)
 to[out=0,in=90] node[below,at end]{$i$} (1,-1);
\draw [postaction={decorate,decoration={markings,
   mark=at position .2 with {\arrow[scale=1.3]{<}}}}] (.5,-1)
to[out=90,in=-90]  node[above,at end]{$j$} node [right, pos=.7,green!50!black]{2} node[below ,at start]{$j$} (1,1); 
 \draw[postaction={decorate,decoration={markings,
   mark=at position .5 with {\arrow[scale=1.3]{<}}}}]  (.5,1) to[out=-90,in=-90] node [below, pos=.5,green!50!black]{3} node[above,at start]{$i$} node[above,at end]{$i$}(0,1);
\end{tikzpicture}
\] with the numbering shown, then we must have $\la_1$
antidominant and $\la_2,\la_3$ dominant, with \[\la_1^i<0\qquad
 \la_2^j>0\qquad \la_3^i>0,\]
and can take $\la_4=-\la_1,\la_5=-\la_2,\la_6=-\la_3$.  

 \mybox{Projection to eigenspaces:} Now, let us take the projection to the subspace where the eigenvalue of the dot at each end of the arc $a$ is the variable $\zw_a$.  Let $\Bj$ and $\Bj'$ be the associated sequences in $\tilde I$ at the
bottom and top of the diagram.  In the above example, we would have
$\Bj=\Big ((i,\zw_1),(-j,\zw_2),(-i,\zw_1) \Big)$ and $\Bj'=\Big ((-i,\zw_3),(i,\zw_3),(j,\zw_2)\Big)$.

Note that $D$ gives the only way to match the terminals in $\Bj$ and $\Bj'$ to produce a legal tricolore diagram.  Thus, all diagrams in our relation that give a different matching from $D$ project to 0, since there is no matching which has the same eigenvalue at both ends of each strand and fewer crossings than $D$.

Therefore, this must be the projection of a relation in $\hldbK$ where all
terms have the underlying matching $D$ with some number $d_a$ of dots on
the arc $a$,
times some monomial $M$ in the bubbles at the left of the diagram.  If we show that no such relation
exists, then for each choice of $d_a$ and $M$, the corresponding term must have had coefficient 0 in the
original relation.  It follows that the original relation must have been trivial.

\mybox{Ruling out projected relations:} First, we consider bubbles.  Any bubble at the left of the diagram evaluates to a scalar, using
the relations \cref{color-opp-cancel} and \cref{dcost2}. The
clockwise bubbles with the label $i$ evaluate to the coefficients of the
power series $\prod_{i}(z-\zw_k)^{\la^i_k}$ and the counterclockwise bubbles
to the coefficients of its formal inverse
$\prod_{i}(z-\zw_k)^{-\la^i_k}$.  By adding new pairs of red and blue
strands with labels $\nu$ and $-\nu$ for a strictly dominant weight
$\nu$, we can ensure that any finite set of monomials in clockwise bubbles are sent to
elements of $\mathbb{\bar K}$ which are algebraically independent over $\K$.  

Furthermore, we can explicitly describe the space $\Hom_{\EuScript{T}}
\big((\bla,\Bj,\kappa),(\bla,\Bj',\kappa')\big)$; it has a basis over
$\mathbb{\bar K}$ given
diagrams with matching $D$ and with $d_a<|\la_a^{i_a}|$.  This follows from \cref{eq:cyc-quotient}: using the fact that functors labeled by different components commute, and adjunction as appropriate, we can reduce to the case where $D$ is a set of vertical strands with no crossings, where this is just the obvious basis of a tensor product of truncated polynomial rings.  Thus,
for any finite number of ways of choosing $d_a$ and $M$, we can choose
$\la_a$'s so that the corresponding diagrams lie in this basis. 
Since these diagrams remain linearly independent after acting in
$\hldbK$, they must be linearly independent, so
we must have that the coefficient of this diagram in $\mathbb{\bar K}$
is 0.  This, in turn,
supplies a polynomial relation between the values of the clockwise
bubbles.  We can rule out this possibility by choosing $\bla$ so that the bubbles which
appear in the relation are
algebraically independent.  Thus, we see that the relation we chose
is trivial.  This establishes the linear independence of our
prospective basis and establishes the result.  
   \end{proof}	

 \appendix

\renc{\theequation}{\Alph{section}.\arabic{equation}}
\section{Valued graphs}
\label{sec:value}
We will follow the conventions of Lemay \cite{Lemaygraph} in this
section.  For simplicity, ``graph'' will always mean a graph without loops.
\begin{definition}
  A {\bf relatively valued graph} is an oriented graph with vertex set
  $I$ with a pair of positive integers
  $(\eta_e,\nu_e)$ assigned to each edge such that there exist $d_{i}\in
  \Z_{>0}$ for each
  $i\in I$ such that $d_i\eta_e=d_j\nu_e$ for $e\colon i\to j$.  

  An {\bf absolutely valued graph} is an  oriented graph as above with
  a fixed choice of positive real numbers $d_i$ for each vertex $i$ and $m_e$ for each edge
  $e\colon i\to j$. 
\end{definition}
Each absolutely valued graph has an associated relatively valued graph
with \[\eta_e=\frac{m_e}{d_i}\qquad  \nu_e=\frac{m_e}{d_j}\] for an
  edge $e\colon i\to j$, and every relatively valued graph has this
  form. Note that relatively valued graphs have a natural notion of
{\bf Langlands duality}, given by switching $\eta_e$ and $\nu_e$.  
We attach a Cartan matrix to each such graph without loops, with $c_{ii}=2$ and 
\[c_{ij}=-\sum_{e\colon i\to j} \eta_e-\sum_{e\colon j\to i}
\nu_e=-\frac{1}{d_i}\Big(\sum_{e\colon i\to j} m_e+\sum_{e\colon j\to i}
m_e\Big).\] Note that Langlands duality transposes this Cartan matrix.  

Given a choice of Cartan datum, we have coefficients $d_i$ as defined before.  Now, choose polynomials  $P_{ij}(x,y)$  for each pair $i,j\in I^2$ satisfying the conditions we require for $P_{ij}$'s in \cref{sec:klr-algebra};  that is, their product $Q_{ij}(x,y)=P_{ij}(x,y)P_{ji}(y,x)$ is homogeneous of degree $-2\langle\al_i,\al_j\rangle= -2d_ic_{ij}=-2d_jc_{ji}$ when
$x$ is given degree $2d_i$ and $y$ degree $2d_j$ with $Q_{ij}(1,0)$  a unit for all $i,j$.

We can canonically
associate an absolutely valued graph with vertex set $I$ by adding an edge $i\to j$ whenever $P_{ij}(x,y)$
is non-constant with $m_e=\deg
P_{ij}(x^{d_i},0)=\deg P_{ij}(0,x^{d_j})$.  
In the associated relatively valued quiver, the values we add to this edge are $(\deg
P_{ij}(x,0), \deg P_{ij}(0,x))$.  The Cartan matrix of the result is
our original Cartan matrix $C$ attached to the Cartan datum. 

Given a graph homomorphism between two valued graphs, we can consider
various forms of compatibility between the valuings on the two
graphs.  One notion considered by Lemay \cite{Lemaygraph} is a
morphism of valued graphs: this is a homomorphism of graphs where the
appropriate statistics ($\eta_*,\nu_*,m_*,d_*$) are preserved; this is too inflexible for our purposes.
Instead, we will consider a set of maps which are more analogous to
topological covers.
Consider relatively valued graphs $X$ and $Y$ with vertex sets $I(X)$ and $I(Y)$.
\begin{definition}
  We call a map $f\colon X\to Y$ of graphs a {\bf furling} if
  given any $y,y'\in I(Y)$, and $x\in f^{-1}(y)$, we have that for each edge
  $d\colon y\to y'$ and each edge $e\colon y'\to y$, 
    \[\nu_d=\sum_{x'\in f^{-1}(y')}\sum_{\substack{d'\colon x\to x'\\ f(d')=d}} \nu_{d'}
    \qquad \eta_e=\sum_{x'\in f^{-1}(y')}\sum_{\substack{e'\colon x'\to x\\ f(e')=e}} \eta_{e'}.\] 
\end{definition}
The notion of a furling is closely related to a ``folding,''
but we will not use this term, since it usually applies to the Langlands
dual of the operation above, and implies the existence of a group
action.  We will call $Y$ a {\bf furling} of $X$ and $X$ an {\bf
  unfurling} of $Y$. A morphism of relatively valued graphs which is also a topological cover is a furling.  See \cref{example1,example2} for some examples of furlings which will be relevant for us. Note that since this is a condition on preimages, it is satisfied when $I(X)=\emptyset$;  however, it does require that if $y\in Y$ has a preimage, then any $y'$ adjacent to it along an edge has a preimage.  Thus, the image of any furling is a union of connected components.

Note that if $X$ has a compatible absolute valued structure such that $d_x$ and $m_e$ are constant on the fibers of $f$ and these fibers are finite, then for an edge
$e\colon y'\to y$, we have that
\[\eta_e=\sum_{x'\in
f^{-1}(y')}\frac{1}{d_{x'}}\sum_{\substack{e'\colon x'\to x\\
    f(e')=e}} m_{e'}\qquad \nu_e=\sum_{x\in
  f^{-1}(y)}\frac{1}{d_{x}}\sum_{\substack{e'\colon x'\to x\\
    f(e')=e}} m_{e'}.\]
Thus, we can choose an absolute valued structure on $Y$ such that
\[d_y=\frac{d_x}{|f^{-1}(y)|}\qquad m_e=\frac{\sum_{
    f(e')=e} m_{e'}}{|f^{-1}(y)|\cdot |f^{-1}(y')|}.\]  As defined
here, $d_y$ and $m_e$ may not be integers, but
  if $Y$ is finite, then we can always just multiply every $d_y$ and
  $m_e$ by $\operatorname{lcm}(|f^{-1}(y)|)$ to clear denominators.

One special case of particular interest is when $X$ is given the
trivial valuation $d_x=m_e=\nu_e=\eta_e=1$ and is equipped
with an admissible automorphism $\sigma$; 
recall that we call an automorphism of a graph {\bf admissible} if no
edges connect two vertices in the same orbit under the action.  We let
$Y$ be the quotient graph $X/\sigma$ and $f\colon X\to Y$ the obvious
projection map.  In this case, we have
\begin{equation}
d_y=\frac{1}{|f^{-1}(y)|}\qquad
m_e=\frac{|f^{-1}(e)|}{|f^{-1}(y)|\cdot |f^{-1}(y')|}.\label{eq:absolute}
\end{equation}
This is the Langlands dual of the ``folding'' 
  discussed in \cite[\S 1]{Lemaygraph} (which is the more common way
  of associating a Cartan matrix to a graph with automorphism).
  
In particular, letting $\sigma$ be the unique nontrivial automorphism of the $A_{2n-1}$ Dynkin diagram, we obtain the type $C_n$ Cartan datum, with the obvious projection map being a folding;  similarly, for the $D_{n+1}$ Dynkin diagram, we have a folding to the $B_n$ Cartan datum.

\begin{lemma}\label{lem:Cartan}
Given a furling $f\colon X\to Y$, for any fixed $y,y'\in I(Y)$ and $x'\in f^{-1}(y')$ we have that:
\[ c_{yy'}=\sum_{x\in f^{-1}(y)} c_{xx'}.\]
\end{lemma}
\begin{proof}
    \begin{align*}
    c_{yy'}&=-\sum_{e\colon y\to y'} \eta_e-\sum_{e\colon y'\to y}
             \nu_e\\
           &=-\sum_{x\in f^{-1}(y)}\Bigg(\sum_{e\colon x\to x'} \eta_e+\sum_{e\colon x'\to x}
             \nu_e\Bigg)\\
           &=\sum_{x\in f^{-1}(y)} c_{xx'}\qedhere
  \end{align*}
\end{proof}

From now on, for all valued graphs appearing, we assume that the matrix $C$ is a generalized Cartan matrix (in
particular, all off-diagonal entries are negative integers).   
\begin{definition}
  Given a valued graph $X$, let $\check{\fg}_X$ be the Lie algebra generated by $E_i,F_i,H_i$ for $i\in I(X)$ with the relations \begin{equation}
    [H_i,E_{j}]=c_{ij}E_j \qquad [H_i,F_j]=-c_{ij}F_j\qquad
    [E_i,F_j]=\delta_{ij}H_i\qquad  [H_i,H_j]=0\label{eq:serre1}
  \end{equation}
We also have the 
 associated universal derived
  Kac-Moody algebra $\fg_X$ where we further quotient by the relations  
  \begin{equation}
    \operatorname{ad}_{E_i}^{1-c_{ij}}E_j=\operatorname{ad}_{F_i}^{1-c_{ij}}F_j=0.\label{eq:serre2}
  \end{equation}
\end{definition}
Since $I(X)$ might be infinite, we need to define a cofinite topology on the Lie algebras $\check{\fg}_X,\fg_X$.  Both of these algebras are spanned by iterated Lie brackets of the elements $E_i,F_i,H_i$. Let $\check{\fg}_X^{I_0,r},\fg_X^{I_0,r}$ for a finite subset $I_0\subset I(X)$ be the span of all monomials of length $\leq r$ in these elements that involve at least one $E_i,F_i,$ or $H_i$ for $i\notin I_0$.   Let the cofinite topology be the topology on $\check{\fg}_X,\fg_X$ with these sets as a basis of neighborhoods of the identity.  

Note the analogy to the finiteness conditions of 1-morphisms in $\htU$.  For any representation of ${\fg}_X$ where on any given vector $v$, we have $E_iv=F_iv=H_iv=0$  for all but finitely many $i\in I(X)$, we will have $\fg_X^{I_0,r}v=0$ for some $I_0\subset I(X)$, so the action on any such representation factors through the completion $\widehat{{\fg}}_X$ with respect to this topology.  

A straightforward extension of \cite[7.9]{kacInfinitedimensionalLie1990} shows that:
\begin{proposition}\label{prop:furling-homomorphism}
  If $f\colon X\to Y$ is a furling of valued graphs, there is an induced
  homomorphism of Kac-Moody algebras $\fg_Y\to \widehat{\fg}_X$ given by the
  formulas:
\[F_y\mapsto \sum_{x\in f^{-1}(y)} F_x\qquad E_y\mapsto \sum_{x\in
  f^{-1}(y)} E_x\qquad H_y\mapsto \sum_{x\in f^{-1}(y)} H_x. \]
\end{proposition}
\begin{proof}
The relations \cref{eq:serre1} are straightforward computations using
\cref{lem:Cartan}.  We have that:
\[\Big[\sum_{x\in f^{-1}(y)} H_x,\sum_{x'\in
   f^{-1}(y')} E_{x'}\Big ]=\sum_{x'\in
  f^{-1}(y')} \Bigg(\sum_{x\in f^{-1}(y)} c_{xx'}\Bigg)E_{x'}=\sum_{x'\in
  f^{-1}(y')} c_{yy'}E_{x'}\]
  \[\Big[\sum_{x\in f^{-1}(y)} H_x,\sum_{x'\in
  f^{-1}(y')} F_{x'}\Big ]=\sum_{x'\in
  f^{-1}(y')} \Bigg(\sum_{x\in f^{-1}(y)}-c_{xx'}\Bigg)F_{x'}=\sum_{x'\in
  f^{-1}(y')}- c_{yy'}F_{x'}\]
\[\Big[\sum_{x\in f^{-1}(y)} E_x,\sum_{x'\in
  f^{-1}(y')} F_{x'}\Big ]=\delta_{y,y'}\sum_{x\in f^{-1}(y)} H_x\]
This shows that we have a homomorphism 
$\check{\fg}_Y\to \widehat{\check{\fg}}_X$, which we wish to show descends.  If there are any components of $Y$ with no preimages, this map kills all the attached $E_i, H_i,F_i$, so this is just augmentation map and indeed descends.  

Thus, without loss of generality, we may assume that  $X\to Y$ is surjective.  Consequently, the induced map on Cartan subalgebras $\mathfrak{h}_Y \to \mathfrak{h}_X$ is injective, and so this map sends elements of non-zero weight to elements of non-zero weight.  In particular, the kernel of the map $\check{\fg}_Y\to {\fg}_Y$ is an ideal of this Lie algebra with no non-zero vectors of weight zero.  Thus, the same is true of its image in $\check{\fg}_X$.  
By the Gabber-Kac theorem \cite[Cor. 2]{gabberDefiningRelations1981}, any such ideal lies in the ideal generated by the relations \cref{eq:serre2}, so this completes the proof.
\end{proof}

As in the main body of the paper, we fix a Cartan datum and corresponding polynomials $P_{ij}(x,y)$.  We let $Y$ be the  valued graph defined above on the Dynkin diagram.  Given a complete choice of spectra $U_i$, we define the set $\tilde{I}$ and its induced graph structure as in \cref{def:tilde-I}.

\begin{proposition}
  If $U_i$ is a complete choice of spectra, then the map $\tilde{I}\to
  Y$ is a furling of valued graphs.
\end{proposition}
\begin{proof}
  The unique edge $e\colon i\to i'$ in $Y$ has preimages corresponding to each
  element $u\in U_i$ and each root of $P_{ii'}(u,x)$ as a polynomial in
  $x$.  Thus the number of preimages $e'$ is the degree of this polynomial
  in $x$, and each has $\nu_{e'}=1$, so this agrees with $\nu_e=\deg
  P_{ii'}(0,x)$.  Similarly, if we consider the edge $d\colon i'\to i$, the
  edges are in bijection with the roots of $P_{i'i}(x,u)$, and
  $\eta_d=\deg P_{i'i}(x,0)$.  This completes the proof.
\end{proof}

\begin{example}\label{example1}
	  The most important example is the so-called ``geometric'' parameters
  for the symmetric Cartan matrix for an oriented graph, where $\mathsf{P}_{ii'}(x,y)=(x-y)^{\#
  i\to i'}$.  In this case, $a^{(k)}_{ii'}=1$ for all $k$ and $(i,u)$
is connected to $(i',u')$ by the same number of edges as $i$ and $i'$
if $u=u'$ and none otherwise.  Thus, if we choose $U_i=U$ for some
fixed set $U\subset\K$, this is a complete choice of spectra and
$\tilde{I}=I\times U$ with the obvious graph structure.

If $X$ is simply-laced, but not simply-connected, then we can obtain
non-trivial covers as $\tilde{I}$.  For example, if $X$ is an
$n$-cycle with its vertex set identified with
$\Z/n\Z$ with edges $i\to i+1$.   Fix
some $q\in \K$ and choose
$Q_{i,i+1}(x,y)=qx-y$.  If we fix $U_0\subset \K$ to be any subset closed
under multiplication by $q^n$,  then we have a complete choice of
spectra with $U_i=q^{i}U_0$.  The components of the graph with vertex set
$\tilde{I}$ correspond to the orbits of multiplication by $q^n$; these
will be cycles if $q$ is a root of unity, or $A_\infty$ graphs if $q$
is not. 
\end{example}

\begin{example}\label{example3}
The most interesting examples are when:
\begin{enumerate}
  \item The Cartan matrix is of type $B_n$.  In this case, $d_1=1$ and $d_2=\cdots =d_n=2$.  If we take $P_{12}(x,y)=x^2-y$ and $P_{i,i+1}(x,y)= x-y$ for $i=2,\ldots,n-1$, and $P_{ij}(x,y)=1$ otherwise.  The components of the graph with vertex set $\tilde{I}$ are given by $(1,\pm a)$ and $(i, a^2)$ for $i=2,\dots,n$ and $a\in \K$, giving type $D_{n+1}$ diagrams.  
  \begin{center}
  \begin{tikzpicture}[>=To,x=1.2cm,y=0.9cm,font=\footnotesize]
  \node[draw,rectangle,inner sep=2pt] (onep) at (0,0.7) {$(1,a)$};
  \node[draw,rectangle,inner sep=2pt] (onem) at (0,-0.7) {$(1,-a)$};
  \node[draw,rectangle,inner sep=2pt] (two) at (2,0) {$(2,a^2)$};
  \node[draw,rectangle,inner sep=2pt] (three) at (4,0) {$(3,a^2)$};
  \node[draw,rectangle,inner sep=2pt] (n) at (6.2,0) {$(n,a^2)$};
  \path (three) -- (n) node[midway] (dots1) {$\cdots$};
  \draw[->,thick] (onep) -- (two);
  \draw[->,thick] (onem) -- (two);
  \draw[->,thick] (two) -- (three);
  \draw[thick] (three) -- (dots1);
  \draw[->,thick] (dots1) -- (n);

  \node[draw,circle,inner sep=1.2pt] (b1) at (0,-2.2) {};
  \node[draw,circle,inner sep=1.2pt] (b2) at (2,-2.2) {};
  \node[draw,circle,inner sep=1.2pt] (b3) at (4,-2.2) {};
  \node[draw,circle,inner sep=1.2pt] (bn) at (6.2,-2.2) {};
  \node at (0,-2.65) {1};
  \node at (2,-2.65) {2};
  \node at (4,-2.65) {3};
  \node at (6.2,-2.65) {$n$};
  \path (b3) -- (bn) node[midway] (dotsB) {$\cdots$};
  \draw[double distance=1pt,->] (b2) -- (b1);
  \draw (b2) -- (b3);
  \draw (b3) -- (dotsB);
  \draw (dotsB) -- (bn);
  \end{tikzpicture}
  \end{center}
  \item The Cartan matrix is of type $C_n$.  In this case, $d_1=2$ and $d_2=\cdots =d_n=1$.  If we take $P_{12}(x,y)=x-y^2$ and $P_{i,i+1}(x,y)= x-y$ for $i=2,\ldots,n-1$, and $P_{ij}(x,y)=1$ otherwise.  The components of the graph with vertex set $\tilde{I}$ are given by $(1,a^2)$ and $(i, \pm a)$ for $i=2,\dots,n$ and $a\in \K$, giving type $A_{2n-1}$ diagrams.  
  \begin{center}
  \begin{tikzpicture}[>=To,x=1.2cm,y=0.9cm,font=\footnotesize]
  \node[draw,rectangle,inner sep=2pt] (one) at (0,0) {$(1,a^2)$};
  \node[draw,rectangle,inner sep=2pt] (twoP) at (2,0.7) {$(2,a)$};
  \node[draw,rectangle,inner sep=2pt] (twoM) at (2,-0.7) {$(2,-a)$};
  \node[draw,rectangle,inner sep=2pt] (threeP) at (4,0.7) {$(3,a)$};
  \node[draw,rectangle,inner sep=2pt] (threeM) at (4,-0.7) {$(3,-a)$};
  \node[draw,rectangle,inner sep=2pt] (nP) at (6.2,0.7) {$(n,a)$};
  \node[draw,rectangle,inner sep=2pt] (nM) at (6.2,-0.7) {$(n,-a)$};
  \path (threeP) -- (nP) node[midway] (dotsP) {$\cdots$};
  \path (threeM) -- (nM) node[midway] (dotsM) {$\cdots$};
  \draw[->,thick] (one) -- (twoP);
  \draw[->,thick] (one) -- (twoM);
  \draw[->,thick] (twoP) -- (threeP);
  \draw[->,thick] (twoM) -- (threeM);
  \draw[thick] (threeP) -- (dotsP);
  \draw[thick] (threeM) -- (dotsM);
  \draw[->,thick] (dotsP) -- (nP);
  \draw[->,thick] (dotsM) -- (nM);
  \node[draw,circle,inner sep=1.2pt] (c1) at (0,-2.2) {};
  \node[draw,circle,inner sep=1.2pt] (c2) at (2,-2.2) {};
  \node[draw,circle,inner sep=1.2pt] (c3) at (4,-2.2) {};
  \node[draw,circle,inner sep=1.2pt] (cn) at (6.2,-2.2) {};
  \node at (0,-2.65) {1};
  \node at (2,-2.65) {2};
  \node at (4,-2.65) {3};
  \node at (6.2,-2.65) {$n$};
  \path (c3) -- (cn) node[midway] (dotsC) {$\cdots$};
  \draw[double distance=1pt,->] (c1) -- (c2);
  \draw (c2) -- (c3);
  \draw (c3) -- (dotsC);
  \draw (dotsC) -- (cn);
  \end{tikzpicture}
  \end{center}
  \item The Cartan matrix is of affine type $C_n^{(1)}$.  In this case, $d_0=d_n=2$ and $d_1=\cdots =d_{n-1}=1$.  If we take $P_{01}(x,y)=x-y^2$, $P_{n-1,n}(x,y)=x^2-y$ and $P_{i,i+1}(x,y)= x-y$ for $i=1,\ldots,n-2$, and $P_{ij}(x,y)=1$ otherwise.  The components of the graph with vertex set $\tilde{I}$ are given by $(0,a^2),(n,a^2)$ and $(i, \pm a)$ for $i=1,\dots,n-1$ and $a\in \K$, giving affine type $A_{2n}$ diagrams.  
  \begin{center}
  \begin{tikzpicture}[>=To,x=1.2cm,y=0.9cm,font=\footnotesize]
  \node[draw,rectangle,inner sep=2pt] (zero) at (0,0) {$(0,a^2)$};
  \node[draw,rectangle,inner sep=2pt] (oneP) at (2,0.7) {$(1,a)$};
  \node[draw,rectangle,inner sep=2pt] (oneM) at (2,-0.7) {$(1,-a)$};
  \node[draw,rectangle,inner sep=2pt] (twoP) at (4,0.7) {$(2,a)$};
  \node[draw,rectangle,inner sep=2pt] (twoM) at (4,-0.7) {$(2,-a)$};
  \node[draw,rectangle,inner sep=2pt] (nm1P) at (6.2,0.7) {$(n-1,a)$};
  \node[draw,rectangle,inner sep=2pt] (nm1M) at (6.2,-0.7) {$(n-1,-a)$};
  \node[draw,rectangle,inner sep=2pt] (nnode) at (8.0,0) {$(n,a^2)$};
  \path (twoP) -- (nm1P) node[midway] (dots2P) {$\cdots$};
  \path (twoM) -- (nm1M) node[midway] (dots2M) {$\cdots$};
  \draw[->,thick] (zero) -- (oneP);
  \draw[->,thick] (zero) -- (oneM);
  \draw[->,thick] (oneP) -- (twoP);
  \draw[->,thick] (oneM) -- (twoM);
  \draw[thick] (twoP) -- (dots2P);
  \draw[thick] (twoM) -- (dots2M);
  \draw[->,thick] (dots2P) -- (nm1P);
  \draw[->,thick] (dots2M) -- (nm1M);
  \draw[->,thick] (nm1P) -- (nnode);
  \draw[->,thick] (nm1M) -- (nnode);
  \node[draw,circle,inner sep=1.2pt] (a0) at (0,-2.2) {};
  \node[draw,circle,inner sep=1.2pt] (a1) at (2,-2.2) {};
  \node[draw,circle,inner sep=1.2pt] (a2) at (4,-2.2) {};
  \node[draw,circle,inner sep=1.2pt] (an1) at (6.2,-2.2) {};
  \node[draw,circle,inner sep=1.2pt] (an) at (8.0,-2.2) {};
  \node at (0,-2.65) {0};
  \node at (2,-2.65) {1};
  \node at (4,-2.65) {2};
  \node at (6.2,-2.65) {$n-1$};
  \node at (8.0,-2.65) {$n$};
  \path (a2) -- (an1) node[midway] (dotsA) {$\cdots$};
  \draw[double distance=1pt,->] (a0) -- (a1);
  \draw (a1) -- (a2);
  \draw (a2) -- (dotsA);
  \draw (dotsA) -- (an1);
  \draw[double distance=1pt,->] (an) -- (an1);
  \end{tikzpicture}
  \end{center}
  \item In affine types $A_n^{(1)}$, there is a further deformation induced by deforming $P_{ij}$.  In type $A_n^{(1)}$, identifying the elements of $I$ with $\Z/n\Z$, we can take $P_{i,i+1}(x,y)=y-x-h$ for $h\in \K\setminus \{0\}$, in which case the components of $\tilde{I}$ are isomorphic to $A_{\infty}$, given by $(i, a+hk)$ where $k\equiv i \pmod n$, with the adjacencies $(i,a+hk)\to (i+1,a+h(k+1))$.
  \begin{center}
  \begin{tikzpicture}[>=To,x=1.2cm,y=0.9cm,font=\footnotesize]
  \coordinate (lcap) at (-2.0,0);
  \coordinate (rcap) at (8.0,0);
  \node[draw,rectangle,inner sep=2pt] (left) at (.5,0) {$(n-1,a-h)$};
  \node[draw,rectangle,inner sep=2pt] (mid) at (3.0,0) {$(0,a)$};
  \node[draw,rectangle,inner sep=2pt] (right) at (5.5,0) {$(1,a+h)$};
  \path (lcap) -- (left) node[midway] (dotsL) {$\cdots$};
  \path (right) -- (rcap) node[midway] (dotsR) {$\cdots$};
  \draw[->,thick] (dotsL) -- (left);
  \draw[->,thick] (left) -- (mid);
  \draw[->,thick] (mid) -- (right);
  \draw[->,thick] (right) -- (dotsR);
  \end{tikzpicture}
  \end{center}
	\item  One can combine the two deformations above by taking $P_{01}(x,y)=x-(y-h)^2, P_{n-1,n}(x,y)=x^2-y$ and $P_{i,i+1}(x,y)= x-y$ for $i=1,\ldots,n-2$, and $P_{ij}(x,y)=1$ otherwise. In this case, the component containing $(0, a^2)$ is given by $(0, (a+hk)^2)$ for $k$ even and $(n,(a+hk)^2)$ and $ (i, \pm (a+hk))$ for $k$ odd.  The map is obtained by considering the unfurling map of (3) and then taking the universal cover of the cycle as in (4).  Note the similarity of this to the content system considered in \cite[Ex. 3A.2(e)]{evseevContentSystems2024}.
  \begin{center}
  \begin{tikzpicture}[>=To,x=1.2cm,y=0.9cm,font=\footnotesize]
  \node[draw,rectangle,inner sep=2pt] (oneM) at (3,-1) {$(1,-a+h)$};
  \node[draw,rectangle,inner sep=2pt] (twoM) at (5.6,-1) {$(2,-a+h)$};
  \node[draw,rectangle,inner sep=2pt] (nm1M) at (8.4,-1) {$(n-1,-a+h)$};
  \path (twoM) -- (nm1M) node[midway] (dots3M) {$\cdots$};

  \node[draw,rectangle,inner sep=2pt] (oneB) at (3,-3) {$(1,a-h)$};
  \node[draw,rectangle,inner sep=2pt] (twoB) at (5.6,-3) {$(2,a-h)$};
  \node[draw,rectangle,inner sep=2pt] (nm1B) at (8.4,-3) {$(n-1,a-h)$};
  \node[draw,rectangle,inner sep=2pt] (nB) at (10.0,-2) {$(n,(a-h)^2)$};
  \path (twoB) -- (nm1B) node[midway] (dots3B) {$\cdots$};

  \node[draw,rectangle,inner sep=2pt] (oneTop) at (3,3) {$(1,-a-h)$};
  \node[draw,rectangle,inner sep=2pt] (twoTop) at (5.6,3) {$(2,-a-h)$};
  \node[draw,rectangle,inner sep=2pt] (nm1Top) at (8.4,3) {$(n-1,-a-h)$};
  \path (twoTop) -- (nm1Top) node[midway] (dots3Top) {$\cdots$};

  \node[draw,rectangle,inner sep=2pt] (zeroUp) at (1.5,4) {$(0,(a+2h)^2)$};
  \node[draw,rectangle,inner sep=2pt] (zeroDown) at (1.5,-4) {$(0,(a-2h)^2)$};
  \node[draw,rectangle,inner sep=2pt] (oneUp) at (3,5) {$(1,a+3h)$};
  \node[draw,rectangle,inner sep=2pt] (oneDown) at (3,-5) {$(1,-a+3h)$};
  \coordinate (upNext) at (7,5);
  \coordinate (downNext) at (7,-5);
  \path (oneUp) -- (upNext) node[midway] (dotsUp) {$\cdots$};
  \path (oneDown) -- (downNext) node[midway] (dotsDown) {$\cdots$};

  \node[draw,rectangle,inner sep=2pt] (zero) at (1.5,0) {$(0,a^2)$};
  \node[draw,rectangle,inner sep=2pt] (oneP) at (3,1) {$(1,a+h)$};
  \node[draw,rectangle,inner sep=2pt] (twoP) at (5.6,1) {$(2,a+h)$};
  \node[draw,rectangle,inner sep=2pt] (nm1P) at (8.4,1) {$(n-1,a+h)$};
  \node[draw,rectangle,inner sep=2pt] (nnode) at (10.0,2) {$(n,(a+h)^2)$};
  \path (twoP) -- (nm1P) node[midway] (dots3P) {$\cdots$};

  \draw[->,thick] (zeroUp) -- (oneUp);
  \draw[->,thick] (zeroDown) -- (oneDown);
  \draw[->,thick] (oneUp) -- (dotsUp);
  \draw[->,thick] (oneDown) -- (dotsDown);

  \draw[->,thick] (zero) -- (oneP);
  \draw[->,thick] (oneP) -- (twoP);
  \draw[thick] (twoP) -- (dots3P);
  \draw[->,thick] (dots3P) -- (nm1P);
  \draw[->,thick] (nm1P) -- (nnode);

  \draw[->,thick] (zero) -- (oneM);
  \draw[->,thick] (oneM) -- (twoM);
  \draw[thick] (twoM) -- (dots3M);
  \draw[->,thick] (dots3M) -- (nm1M);

  \draw[->,thick] (oneB) -- (twoB);
  \draw[thick] (twoB) -- (dots3B);
  \draw[->,thick] (dots3B) -- (nm1B);
  \draw[->,thick] (nm1B) -- (nB);
  \draw[->,thick] (nm1M) -- (nB);

  \draw[->,thick] (zeroDown) -- (oneB);

  \draw[->,thick] (oneTop) -- (twoTop);
  \draw[thick] (twoTop) -- (dots3Top);
  \draw[->,thick] (dots3Top) -- (nm1Top);
  \draw[->,thick] (nm1Top) -- (nnode);
  \draw[->,thick] (zeroUp) -- (oneTop);

  \node[draw,circle,inner sep=1.2pt] (a0c) at (1.5,-7.5) {};
  \node[draw,circle,inner sep=1.2pt] (a1c) at (3.0,-7.5) {};
  \node[draw,circle,inner sep=1.2pt] (a2c) at (5.6,-7.5) {};
  \node[draw,circle,inner sep=1.2pt] (an1c) at (8.4,-7.5) {};
  \node[draw,circle,inner sep=1.2pt] (anc) at (10.0,-7.5) {};
  \node at (1.5,-7.95) {0};
  \node at (3.0,-7.95) {1};
  \node at (5.6,-7.95) {2};
  \node at (8.4,-7.95) {$n-1$};
  \node at (10.0,-7.95) {$n$};
  \path (a2c) -- (an1c) node[midway] (dotsAc) {$\cdots$};
  \draw[double distance=1pt,->] (a0c) -- (a1c);
  \draw (a1c) -- (a2c);
  \draw (a2c) -- (dotsAc);
  \draw (dotsAc) -- (an1c);
  \draw[double distance=1pt,->] (anc) -- (an1c);
  \end{tikzpicture}
  \end{center}
\end{enumerate} 
\end{example}

Let $d=\operatorname{lcm}(d_i)_{i\in I}$.  Let
$U_i$ be the $d/d_i$th roots of unity; there are $d/d_i$ distinct
roots of unity since $d$ is coprime to the characteristic of $\K$ and $\K$ is algebraically closed.
Assume that each $a_{ij}^{(k)}$ is a $d$th root of unity for all
$i,j,k$.  For example, we can assume that 
\begin{equation}
Q_{ij}(x,y)=\pm (x^{h_{ij}}-y^{h_{ji}})^{g_{ij}}\label{eq:6}
\end{equation}
in which case,
$Q_{ij}(x^{d_i},1)=\pm (x^{d_i{h_{ij}}}-1)^{g_{ij}}$, so the multiset of
$a_{ij}^{(k)}$ and $b_{ij}^{(k)}$ is given by the $d_i{h_{ij}}=d_jh_{ji}=\operatorname{lcm}(d_i,d_j)$ roots
of unity each with multiplicity $g_{ij}$.   
We have that each $u\in U_i$ is connected by $c_{ji}$ edges to
elements of $U_j$, given by the $h_{ji}$th roots of $u^{h_{ij}}$ with  multiplicity $g_{ij}$.  If $Q_{ij}$ is as in \cref{eq:6},
then for each $d/d_ih_{ij}$th root of unity $\xi$, we connect each
$h_{ij}$th root of $\xi$ in $U_i$ to each $h_{ji}$th root of $\xi$
in $U_j$ with $g_{ij}$ edges (with orientation depending on
$P_{ij}$).  We obtain the $a=1$ cases of (1-3) of \cref{example3} from this construction.

Let $\zeta$ be a primitive $d$th root of unity.
\begin{proposition}
  Assuming $a_{ij}^{(k)}$ is a $d$th root of unity for all
$i,j,k$ and $U_i$ is the $d/d_i$th roots of unity, the map
$\sigma\colon (i,u)\mapsto (i,\zeta^{d_i}u)$ is an admissible automorphism of the
graph $\tilde{Y}$ with vertex set $\tilde{I}$; the map $\tilde{Y}\to Y=\tilde{Y} /\sigma$ induces
the relative valued structure on $Y$ associated to the polynomials $P_{ij}$.   
\end{proposition} \ifanindex
\IndexOfNotation
\fi 
{\renewcommand{\markboth}[2]{}\printbibliography}
\end{document}